\documentclass{amsart}
\pdfoutput=1
\usepackage[letterpaper, margin=1in]{geometry}
\usepackage[utf8]{inputenc}
\usepackage{amsmath,amssymb,latexsym,amsthm}
\usepackage{MnSymbol}
\usepackage{graphicx}
\usepackage{tikz}
\usetikzlibrary{positioning}
\usepackage[alphabetic]{amsrefs}
\usepackage{verbatim}
\usepackage{xspace}
\usepackage{enumitem} 
\usepackage{pinlabel}

\ifx\pdfoutput\undefined
\usepackage{hyperref}
\else
\usepackage[pdftex,colorlinks=true,linkcolor=blue,urlcolor=blue,citecolor=violet]{hyperref}
\fi
\usepackage[nameinlink]{cleveref}

\newcommand{\commf}[1]{\ignorespaces}%
\newcommand{\commb}[1]{\ignorespaces}%
\newcommand{\commp}[1]{\ignorespaces}%
\newcommand{\compat}[1]{\ignorespaces}%
\newcommand{\compatnew}[1]{\ignorespaces}%

\newcommand{\Angle}{\mathit{Ang}}
\newcommand{\Aut}{\mathrm{Aut}}
\newcommand{\C}{\mathbb{C}}
\newcommand{\conj}[1]{\lsem #1 \rsem} 
\newcommand{\Cl}{\mathrm{Cl}}

\newcommand{\bbH}{\mathbb{H}}
\newcommand{\Homeo}{\mathrm{Homeo}}
\newcommand{\slope}{\mathrm{slope}}
\newcommand{\R}{\mathbb{R}}
\newcommand{\N}{\mathbb{N}}
\newcommand{\op}{o}
\newcommand{\prpi}{{\pi_1^{\mathrm{PR}}}}
\renewcommand{\plus}{+}
\newcommand{\Prongs}{\mathit{Pr}}
\newcommand{\sA}{{\mathcal{A}}} 

\newcommand{\sB}{{\mathcal{B}}} 

\newcommand{\sE}{{\mathcal{E}}} 
\newcommand{\sF}{{\mathcal{F}}} 
\newcommand{\sG}{{\mathcal{G}}} 
\newcommand{\sH}{{\mathcal{H}}} 
\newcommand{\sP}{{\mathcal{P}}} 
\newcommand{\sQ}{{\mathcal{Q}}} 
\newcommand{\sR}{{\mathcal{R}}} 
\newcommand{\sS}{{\mathcal{S}}} 
\newcommand{\sT}{{\mathcal{T}}} 
\newcommand{\TS}{\mathrm{TA}} 
\newcommand{\SL}{\mathrm{SL}}
\newcommand{\bbV}{\mathbb{V}} 
\newcommand{\bbS}{\mathbb{S}} 
\newcommand{\bbU}{\mathbb{U}} 

\newcommand{\TL}{\mathrm{TL}}

\newcommand{\sV}{{\mathcal{V}}} 
\newcommand{\Z}{\mathbb{Z}}

\newcommand{\PSL}{\mathrm{PSL}}
\newcommand{\GL}{\mathrm{GL}}

\newcommand{\0}{\boldsymbol{0}}
\newcommand{\seg}{\mathrm{ta}}

\renewcommand{\setminus}{\smallsetminus}

\newcommand{\gaussbonnet}{\hyperref[Gauss-Bonnet]{Gauss-Bonnet Theorem}\xspace}%
\newcommand{\surgery}{\hyperref[surgery]{Surgery Theorem}\xspace}%

\theoremstyle{plain}
\newtheorem{theorem}{Theorem}
\numberwithin{theorem}{section}
\newtheorem{corollary}[theorem]{Corollary}
\newtheorem{proposition}[theorem]{Proposition}
\newtheorem{conjecture}[theorem]{Conjecture}

\newtheorem{lemma}[theorem]{Lemma}
\newtheorem{definition}[theorem]{Definition}
\newtheorem{remark}[theorem]{Remark}
\newtheorem{question}[theorem]{Question}

\title[Zebra surfaces]{Stellar foliation structures on surfaces}
\date{\today}

\author{W. Patrick Hooper}
\address{Dept. of Mathematics, City College of New York and CUNY Graduate Center, New York, NY, USA}
\email{whooper@ccny.cuny.edu}
\urladdr{\url{http://wphooper.com}} 

\author{Ferr\'an Valdez}
\address{Centro de Ciencias Matem\'aticas, UNAM Campus Morelia, M\'exico}
\email{ferran@matmor.unam.mx}
\urladdr{\url{https://www.matmor.unam.mx/~ferran/}}

\author{Barak Weiss}
\address{Dept. of Mathematics, Tel Aviv University, Tel Aviv, Israel}
\email{barakw@tauex.tau.ac.il}
\urladdr{\url{http://www.math.tau.ac.il/~barakw/}}

\begin{document}

\begin{abstract}
We introduce the notion of a zebra structure on a surface, which is a more general geometric structure than a translation structure or a dilation structure that still gives a directional foliation of every slope. We are concerned with the question of when a free homotopy class of loops (or a homotopy class of arcs relative to endpoints) has a canonical representative or family of representatives, either as closed leaves or chains of leaves joining singularities. We prove that such representations exist if the surface has a triangulation with edges joining singularities (in the zebra structure sense). In the special case when the surface is closed, we describe several geometric conditions that are equivalent to the existence of canonical representations in every homotopy class of closed curves.
\end{abstract}

\maketitle

\section{Introduction}

A quadratic differential $q$ on a Riemann surface $X$, possibly with simple poles, naturally endows that surface with a half-translation structure via coordinate charts obtained by integration. Associated to $q$ is its divisor, which may be interpreted as a function
\begin{equation}
\label{eq:alpha}
\alpha:X \to \Z_{\geq -1} = \{-1,0,1, \ldots\}
\end{equation}
whose support is discrete and closed. \compat{Changed discretely supported to mention closed. See note below.}
We have $\alpha(x)=-1$ if $x$ is a simple pole, and $\alpha(x)=k \geq 1$ if $x$ is a zero of order $k$. The half-translation structure gives charts from a neighborhood of each $x \in X$ to the $\alpha(x)+2$-fold branched cover of $\C/\langle z \mapsto -z\rangle$ branched over the origin, and transition maps are given by translations or $180^\circ$ rotations in local coordinates.  Geometrically $x \in X$ has the local structure of a Euclidean cone point with cone angle $\pi\big(\alpha(x)+2\big)$. Thus the support $\Sigma$ of $\alpha$ is the collection of {\em singularities} of the half-translation structure. Throughout this paper, we allow our surfaces to be noncompact, though the closed surface case is of special interest and we prove new results in this setting as well.

Given a half-translation structure on an oriented topological surface $S$, for each slope  $m \in \hat \R=\R \cup \{\infty\}$, the foliation of the plane by lines of slope $m$ pulls back under charts to give a singular foliation $\sF_m$ of the surface by leaves of slope $m$. Formally $\sF_m$ is a foliation of $S \setminus \Sigma$ whose local behavior near a point $p$ is governed by the value $\alpha(p)$. In particular, each $\sF_m$ has $\alpha(p)+2$ prongs at each point $p \in S$.

More generally, a {\em singular foliation} $\sF$ of a topological surface $S$ is a foliation of $S \setminus \Sigma$, where $\Sigma \subset S$ is a closed discrete subset and such that there is a {\em singular data function} $\alpha:S \to \Z_{\geq -1}$ whose support is $\Sigma$ such that $\sF$ is locally homeomorphic at $p \in \Sigma$ to the horizontal foliation of a half-translation surface in a neighborhood of a cone point with cone angle $\pi\big(\alpha(p)+2\big)$. \commb{there is some discrepancy in the literature as to the precise meaning of the word discrete. Some (including Wikipedia) say it is a set on which the induced topology is discrete. Others say a set which intersects every compact subset in a finite set (according the previous definition, this would mean discrete and closed). We need the stronger definition. E.g. we don't want to allow singularities to be on the points {1/n: n a positive integer} on the real line. } \compat{I think I've only heard the first definition which means this is an error. I changed all occurrences I could find from discrete to discrete and closed.''}
Note that $\sF$ determines both $\Sigma$ and $\alpha$: The subsurface $S \setminus \Sigma$ is the union of leaves, and $\alpha$ can be determined by the number of prongs at a point.

Singular foliations need not come from quadratic differentials. Indeed, those that come from quadratic differentials carry the additional structure of a measured foliation, which appears because half-translation surfaces have a natural path metric obtained by pulling back the Euclidean metric on the plane. In this paper, we investigate what happens when these additional structures are not required, but where we still have a family of singular foliations that fit together nicely.

\begin{definition}
Let $S$ be an oriented topological surface. Consider a family $\{\sF_m:~m \in \hat \R\}$ of singular foliations on $S$ indexed by slope that determine the same singular set $\Sigma$ and the same singular data function $\alpha$.
We say such a family is {\em stellar} if each $p \in S$ has a neighborhood $U$ such that $U \setminus \{p\}$ is
foliated by segments of leaves containing $p$ in their closure in a manner homeomorphic to the standard half-translation model associated to $\alpha(p)$ as depicted in \Cref{fig:stellar}.
If $\{\sF_m\}$ is stellar, we say it induces a {\em stellar foliation structure} or a {\em zebra\footnote{The name zebra surface was inspired by the Zebra Slot Canyon in Grand Staircase-Escalante National Monument.} structure} on $(S,\alpha)$. We call a pair $(S, \{\sF_m\}_{m \in \hat \R})$ where $\{\sF_m\}$ is stellar a {\em zebra surface}.
\end{definition}

\begin{figure}[htb]
\centering
\includegraphics[width=\textwidth]{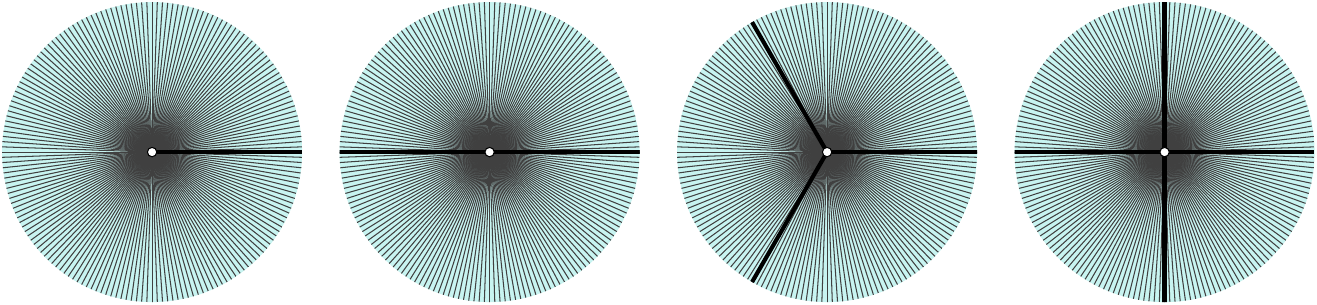}
\caption{The local structure of leaves through $p$, with $\alpha(p) \in \{-1,0,1,2\}$ from left to right. Here the bold leaves are of slope zero and slopes cyclically increase in the counterclockwise direction.}
\label{fig:stellar}
\end{figure}

See \Cref{sect:formal} for a more detailed and formal definition. Observe that half-translation structures induce zebra structures.
Another geometric structure on a surface that induces a zebra structure is a dilation structure; see \Cref{sect:contexts}.
This paper originated from our interest in dilation structures on surfaces. These structures have recently generated a lot of interest, see for example~\cites{ABW, BGT, DFG19, W21}. We discuss some relevant results in Sections \ref{dilation with cone singularities}, \ref{sect:dilation singularities}, and \ref{sect:questions}. In \Cref{sect: not dilation}, we show that there are zebra structures that do not arise from dilation or translation structures.
 	A broader discussion about foliations on surfaces can be found in~\cite{Nikolaev}. \compat{Rather than cite things here, maybe it is better to point out where we discuss results in the subject as I do here. Dec 20, 2022.}


It is a fundamental observation in Teichmüller theory that the bundle $\sQ_g$ of quadratic differentials (without poles) over the moduli space of closed Riemann surfaces of genus $g$ is naturally identified with the cotangent bundle of that space. The geodesic flow in the Teichmüller metric is conjugate via the identification of quadratic differentials with half-translation structures
to the diagonal action of $\PSL(2,\R)$ on the space of half-translation surfaces, acting affine-linearly by simultaneous post-composition with coordinate charts. This action of $\PSL(2,\R)$ on $\sQ_g$ is an active area of research \cite{WrightSurvey}. This action extends to an action of $\PSL(2,\R)$ on zebra surfaces: In fact, the $\PSL(2,\R)$-action extends to a $\Homeo_+(\hat \R)$-action given by reindexing the foliations; see \Cref{sect:homeo action}.

Our initial motivation when writing this paper was to extend fundamental facts about length-minimizing representatives of curves on translation surfaces to the more general context of dilation surfaces, where there is no natural notion of length.
One motivation for studying this question is Thurston's theory of simple closed curves and their relation with the classification of surface homeomorphisms.
When looking into this, we realized that surprisingly little is known about distinguished representatives of curves for related structures (such as dilation structures with cone-type singularities and translation structures on noncompact surfaces) and an answer can be given in the very general context of zebra surfaces. Working in this more general context makes some things challenging, but the limited tools available lead to a natural and general approach to the problems under consideration. \compat{I added mention of Thurston. Ferran had suggested we mention this as a motivation. He also suggested we say something along the lines of wanting to understand it in these contexts and extend the theory. I wasn't sure how to say this.}

\subsection{Main results}
\label{sect:main results}
We briefly introduce some important definitions so that we can state our results.

Let $(S, \{\sF_m\})$ be a zebra surface, with $S$ any oriented topological surface.
As indicated above, this information determines a singular set $\Sigma$ and a singular data function $\alpha:S \to \Z_{\geq -1}$ whose
support is $\Sigma$.

A {\em leaf} is a leaf of any of the foliations $\sF_m$. Leaves are contained in $S \setminus \Sigma$ and so do not contain singularities. A leaf is {\em closed} if it is homeomorphic to a circle. If a leaf is not closed, then it is homeomorphic to an open interval, and such a leaf can have singular endpoints in its closure. A {\em saddle connection} is a leaf together with two singular endpoints. \compatnew{Old phrasing: A {\em leaf triangulation} of $S$ is a triangulation of $S$ whose edges are saddle connections. We require that triangles to meet edge-to-edge and that the union of the triangles be all of $S$. In such a triangulation, only finitely many triangles can meet at each vertex.}

A {\em trail} is a maximal bi-infinite parameterized path that follows a sequence of leaves, transitioning between leaves only at singularities in such a way that the two angles made at the singular transitions are at least $\pi$. (Angles made between leaves meeting at a point can be measured using the stellar neighborhood of the point.) We call the angles made at singular transitions {\em bending angles}. Each bending angle appears either on the right or left side of the trail. We require trails to ``bounce off'' poles, returning along the leaf through which it arrived. (This ``bouncing off'' is not allowed at other singularities.) A trail is {\em closed} if it can be reparameterized to be periodic.  

A {\em zebra plane} is a zebra structure $(Z, \{\tilde \sF_m\})$ where $Z$ is an open disk and the singular data function $\tilde \alpha$ is nonnegative. (Examples of simply connected zebra surfaces which are not zebra planes can be found in \cite{Panov}.) If $(S, \{\sF_m\})$ is any connected zebra surface, and $\alpha$ is nonnegative, then the structure lifts to the universal cover to give a zebra plane. If $S$ has poles, then there is a larger cover which is a zebra plane, where we require double branching over poles. We call this the {\em pole-resolved universal cover} (the {\em PRU cover}), see \Cref{sect:maximal cover}. Note that because of the double branching, preimages of poles are nonsingular points in the zebra plane. The PRU cover coincides with the universal cover if $S$ has no poles.

\compatnew{This is the new phrasing. Note I introduce the concept of a ``leaf saddle connection.''}
A {\em leaf saddle connection} is a path $\gamma:[0,1] \to S$ that lifts to a saddle connection in the PRU cover. Because preimages of poles are nonsingular, a pole cannot be the endpoint of a leaf saddle connection. But a pole can lie in the {\em interior}, $\gamma\big((0,1)\big)$, of a leaf saddle connection that bounces off the pole as described above. If $S$ has no poles, then the notions of ``leaf saddle connection'' and ``saddle connection'' coincide. A {\em leaf triangulation} of a zebra surface $S$ a collection of leaf saddle connections with disjoint interiors such that the complementary components contain no singularities and are each bounded by three distinct leaf saddle connections. This forces each pole in $S$ to be contained in the interior of a unique leaf saddle connection in the collection. The collection of lifts of leaf saddle connections in a leaf triangulation of a zebra surface gives a leaf triangulation of its PRU cover.

A basic question in the geometry of metric spaces is whether any two points can be joined by a geodesic, and if so, whether this geodesic is unique. For example, the Cartan–Hadamard theorem guarantees that any two points can be connected by a unique geodesic segment in a complete nonpositively curved simply-connected metric space \cite{BH}. In our setting, we observe that distinct points on a zebra plane can be joined by at most one arc of a trail (\Cref{no bigons}). We say a subset of a zebra plane is {\em convex} if any two distinct points can be joined by an arc of a trail contained in that subset. We prove:

\begin{theorem}
\label{thm:convex}\label{convexity}
If a zebra plane $Z$ has a leaf triangulation, then $Z$ is convex.
\end{theorem}

The primary case of interest is when the zebra plane is the PRU cover of a closed surface. As a consequence if we can ``triangulate'' the closed surface in a manner that lifts to a leaf triangulation on the cover, then that cover is convex. \Cref{fig:amalgamated_cylinder} shows a surface with no leaf triangulation whose PRU cover is not convex.

\begin{figure}[htb]
\centering
\includegraphics[width=3in]{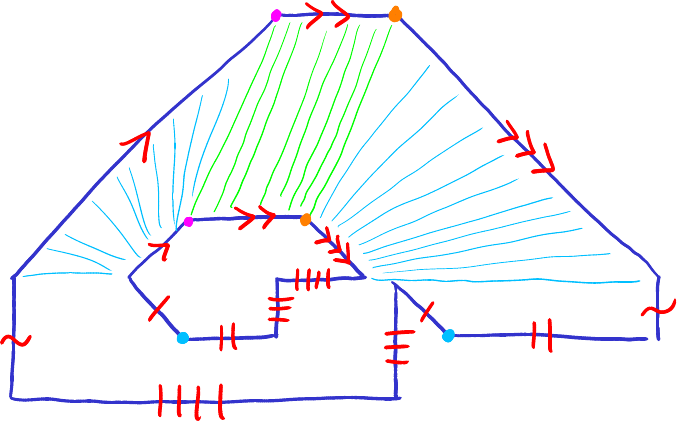}
\caption{A dilation surface: edges are glued by dilations and translations.
The orange, pink, and blue points are dilation singularities, which are not singularities in the induced zebra structure. The foliated region is a full zebra cylinder (see \Cref{sect:dilation singularities}), and so no closed trail can cross this cylinder by \Cref{full cylinder}.}
\label{fig:amalgamated_cylinder}
\end{figure}

Of course, it is only possible for a zebra plane to have a leaf triangulation when $\alpha$ takes positive values, (that is, when there are singularities). We also provide a criterion for convexity of the PRU cover of a closed surface when $\alpha$ is nonpositive; see \Cref{torus cover convex}. Some examples, like $\R^2$ (which covers the square torus) are convex, but others like the universal cover of a Hopf torus is not; see \Cref{dilation with cone singularities}.

We will need to extend the notion of homotopy of paths, to the case of zebra surfaces with poles. See \Cref{sect:curves and deck transformations} for details on this construction. We call this extended notion of homotopy {\em pole-resolved homotopy} or PR homotopy, and it coincides with the usual notion if there are no poles. The PR free homotopy classes of loops are in natural bijective correspondence with conjugacy classes in the deck group of the PRU covering. We say a PR free homotopy class of closed curves is {\em polar} if there is a simple closed curve in the class that bounds a disk whose only interior singularity is a pole.
We say that a PR free homotopy class $\conj{\gamma}$ of closed curves is {\em a power} if there is another PR free homotopy class $\conj{\beta}$ and a $k \geq 2$ such that $\conj{\gamma}$ is homotopic to the $k$-fold cover of $\conj{\beta}$. \compat{Added Nov 18, 2022. I'm doing this because if $\conj{\gamma}$ is a $k$th power, then for example a closed leaf in $\conj{\gamma}$ will be a $k$-fold cover of a simple curve. I don't want to worry about this in \Cref{thm:closed trails}. Also: I started using $\conj{\gamma}$ to denote free homotopy class or equivalently a conjugacy class of the (PR) fundamental group.}

The {\em closed standard cylinder} is $C = [-1,1] \times \bbS^1$ where $\bbS^1=\R/\Z$. This cylinder comes equipped with its {\em vertical foliation} by fibers of the projection $C \to [-1,1]$, and has two boundary components $\partial_\pm C=\{\pm 1\} \times \bbS^1$. Its interior is the subset $C^\circ = (-1,1) \times \bbS^1$. We split $C^\circ$ into two halves:
$$H^\circ_- = (-1, 0] \times \bbS^1 \quad \text{and} \quad H^\circ_+ = [0, 1) \times \bbS^1.$$
We have:

\begin{theorem}[Closed trails]
\label{thm:closed trails}
Let $(S, \{\sF_m\})$ be a zebra surface, where $S$ is any connected oriented topological surface. Fix a PR free homotopy class of closed curves $\conj{\gamma}$ which is nontrivial, non-polar, and not a power. Then one of the following mutually exclusive statements holds:
\begin{enumerate}
\item[(NR)] (Non-realization case) There is no closed trail in $\conj{\gamma}$.
\item[(TF)] (Toral foliation case) The surface $S$ is the torus, the zebra structure has no singularities, and the collection of all closed trails in $\conj{\gamma}$ is a collection of simple closed leaves that foliate $S$.
\item[(Cyl)] (Cylinder case) There is an embedding $\epsilon:C^\circ \to S$ such that the closed leaves of $S$ in $\conj{\gamma}$ are precisely the images under the embedding of the vertical closed leaves of $C^\circ$.
\item[(UT)] (Unique trail case) There is a unique closed trail in $\conj{\gamma}$ and this closed trail has at least one bending angle greater than $\pi$ on each side.
\end{enumerate}
Furthermore,
\begin{enumerate}
\item If the PRU cover of $S$ is convex, then case {\em (NR)} cannot occur.
\item In case {\em (Cyl)}, define $\sigma$ to be the collection of signs $s \in \{\pm \}$ such that
$\epsilon(H^\circ_s)$ has compact closure. Set
$$\ddot{C} = C^\circ \cup \bigcup_{s \in \sigma} \partial_s C.$$ Then there is a continuous map
$\ddot{\epsilon}: \ddot{C} \to S$ whose restriction $\ddot{\epsilon}|_{C^\circ}$ satisfies {\em (Cyl)} such that for each $s \in \sigma$, the curve $\ddot{\epsilon}|_{\partial_s C}$ is a closed trail in $\conj{\gamma}$ passing through a nonempty collection of singularities, and every bending angle made when passing through such a singularity on the side of $\ddot{\epsilon}(C^\circ)$ has measure $\pi$. Furthermore, all closed trails in $\conj{\gamma}$ are obtained as restrictions of $\ddot{\epsilon}$ to vertical circles in $\ddot{C}$.
\compatnew{Added the word ``continuous'' to the definition of $\ddot{\epsilon}$. June 7, 2023.}
\end{enumerate}
\end{theorem}

In order to prove this theorem, we prove a criterion for existence of a closed trail that does not require convexity of the PRU cover; see \Cref{closed trails exist}.

In the context of closed zebra surfaces with singularities and convex PRU covers, \Cref{thm:closed trails} specializes to the following:

\begin{corollary}
\label{cor:closed trails}
Suppose that $S$ is a closed surface and $\{\sF_m\}$ is a zebra structure on $S$ with a nonempty singular set $\Sigma$. Suppose also that the PRU cover of $S$ is convex. Then, if $\conj{\gamma}$ is a PR free homotopy class of closed curves that is nontrivial, non-polar, and not a power, then either there is a unique closed trail in $\conj{\gamma}$ as in case {\em (UT)} or there is a continuous map from the closed standard cylinder $\ddot{\epsilon}:C \to S$ whose restriction to $C^\circ$ is an embedding as in case {\em (Cyl)} and whose restriction to each boundary component is a closed trail as described in {\em (2)}, and such that all closed trails in the homotopy class are given by restriction as in {\em (2)}.
\end{corollary}

However, there certainly are some closed surfaces for which there is a $\conj{\gamma}$ that falls into case (NR). The PRU cover of such a surface is not convex.
An example of this situation is depicted in \Cref{fig:amalgamated_cylinder}. The surface is constructed by starting with a polygonal annulus in $\R^2$ and making boundary identifications (which can be chosen to give the surface a dilation structure). The homotopy class of a homotopically nontrivial loop traveling around the interior of this annulus gives an example of a homotopy class satisfying (NR). The problem with this homotopy class is that it crosses a cylinder which can be foliated by leaves of all possible slopes. In what follows we explain the notion that captures this obstruction in the context of zebra surfaces.

A {\em zebra cylinder} in a zebra surface $S$ is the union of all closed trails in a PR free homotopy class of closed curves $\conj{\gamma}$ that contains at least two such closed trails. By \Cref{thm:closed trails}, these closed trails must have the structure of a cylinder, and each closed trail has constant slope (though these slopes vary with the closed trail). A zebra cylinder is {\em full} if every slope in $\hat \R$ is the slope of one of the closed trails in the cylinder.
The following Proposition states that full zebra cylinders are an obstruction for the convexity of the PRU cover of a zebra surface.

\begin{proposition}
\label{full cylinder}
Suppose $S$ is a zebra surface and that $S$ has a PR free homotopy class of closed curves $\conj{\gamma}$ whose closed trails constitute a full zebra cylinder. If $\conj{\beta}$ is any PR free homotopy class of closed curves whose geometric intersection number with $\conj{\gamma}$ is nonzero, then $\conj{\beta}$ contains no closed trails. Furthermore if such a $\conj{\beta}$ exists, then the PRU cover of $S$ is not convex.
\end{proposition}
\begin{proof}
Suppose $\conj{\gamma}$ and $\conj{\beta}$ are as stated. Let $\ddot{\varepsilon}:{\ddot C} \to S$ be the full zebra cylinder containing the closed trails in $\conj{\gamma}$.
Suppose to the contrary that $\beta$ is a closed trail in $\conj{\beta}$. Because of the intersection number condition, there must be an arc $\beta' \subset \ddot{\varepsilon}(\ddot C)$ of $\beta$ that crosses every closed trail in the cylinder. Since $\ddot \varepsilon(C^\circ)$ contains no singularities, the slope of $\beta'$ must be constant. Let $m$ denote this slope.
Let $\ell \subset \ddot C$ be a vertical closed leaf whose image $\ddot \varepsilon(\ell)$ is a closed trail whose constant slope is also $m$. We will derive a contradiction from the fact that $\beta'$ must both contain points in $\ddot \varepsilon(\ell)$ and not in $\ddot \varepsilon(\ell)$. This is clearly impossible when $\ell$ is contained in the interior of $C^\circ$, because in this case both $\ddot \varepsilon(\ell)$ and $\beta'$ are leaves of the same foliation of slope $m$ restricted to the cylinder. If $\ell$ is one of the two boundary curves of $C$, then $\ddot \varepsilon(\ell)$ is a trail of constant slope $m$ and the bending angles along $\ddot \varepsilon(\ell)$ on the side of $\ddot \varepsilon(C^\circ)$ are all $\pi$. Because of these bending angles and the local structure at the singular points, no leaf of slope $m$ emanating from a singularity on $\ddot \varepsilon(\ell)$ enters $\ddot \varepsilon(C^\circ)$, again contradicting the existence of $\beta'$. This completes the proof that $\conj{\beta}$ contains no closed trail. The last statement follows directly from statement (1) of  \Cref{thm:closed trails}.
\end{proof}

The last main result we present is a characterization of the convexity of the PRU cover for
closed zebra surfaces having at least one singularity that is not a pole.

\begin{theorem}
\label{conj:zebra case}
Suppose $S$ is a closed surface with a zebra structure and at least one singularity that is not a pole. Then, the following are equivalent:
\compatnew{Added ``that is not a pole.'' June 2023.}
\begin{enumerate}
\item[(a)] $S$ has a leaf triangulation.
\item[(b)] The PRU cover of $S$ is convex.
\item[(c)] Every PR free homotopy class $\conj{\gamma}$ of closed curves that is nontrivial and non-polar either contains a unique closed trail or contains closed leaves (as described in \Cref{cor:closed trails}). (In particular, case
{\em (NR)} of \Cref{thm:closed trails} does not occur.)
\item[(d)] $S$ contains no full cylinders.
\end{enumerate}
\end{theorem}

Section~\ref{sect:loop} is dedicated to the proof of this result.


\subsection{Contexts}
\label{sect:contexts}
Here we consider our main theorems in specific contexts moving roughly from more specific structures to more general structures. The chart below depicts the various geometric structures on surfaces that we consider, together with arrows from one structure to another to indicate that a surface with the first structure is also a surface with the second structure. (E.g., a translation surface atlas is also a half-translation surface atlas.)

\begin{center}
\begin{tikzpicture}
\matrix [column sep=7mm, row sep=5mm] {
  &
  \node (t) [draw, shape=rectangle] {Translation}; &
  \node (dc) [draw, shape=rectangle, align=left] {Dilation with\\ cone singularities}; &
  \node (dd) [draw, shape=rectangle, align=left] {Dilation with\\ dilation singularities}; &
  \\
  \node (ec) [draw, shape=rectangle, align=left] {Euclidean\\ cone}; &
  \node (ht) [draw, shape=rectangle] {Half-translation}; &
  \node (hdc) [draw, shape=rectangle, align=left] {Half-dilation with\\ cone singularities}; &
  \node (hdd) [draw, shape=rectangle, align=left] {Half-dilation with\\ dilation singulartites}; &
  \node (z) [draw, shape=rectangle] {Zebra}; \\
};
\draw[->, thick] (t) -- (dc);
\draw[->, thick] (dc) -- (dd);
\draw[->, thick] (ht) -- (ec);
\draw[->, thick] (t) -- (ht);
\draw[->, thick] (dc) -- (hdc);
\draw[->, thick] (dd) -- (hdd);
\draw[->, thick] (ht) -- (hdc);
\draw[->, thick] (hdc) -- (hdd);
\draw[->, thick] (hdd) -- (z);
\end{tikzpicture}
\end{center}

\subsubsection{Closed translation surfaces and cone surfaces}
\label{sect:translation surfaces}
A {\em translation surface} is an oriented surface with an atlas of charts to the plane whose transition functions are translations, where we allow cone points with cone angles that are integer multiples of $2 \pi$ (so, in our notation, $\alpha$ takes even values). We briefly explain why our main results are true in the case of a closed translation surface and in a related case.

A {\em (Euclidean) cone surface} is an oriented surface with an atlas of charts to the plane where transition functions are in the orientation-preserving isometry group, and where we allow cone singularities with any positive real cone angle. Thus a translation surface is a special case of a cone surface. \compat{Added this definition.}

If $S$ is a closed translation surface, its universal cover is a Hadamard space, i.e., a complete metric space that is nonpositively curved in the $\mathrm{CAT}(0)$ sense. More generally, we could consider a closed cone surface all of whose cone singularities have cone angles greater than $2 \pi$. The universal cover is again a Hadamard space. For details see \cite[\S 2.1]{BL18}\commf{Not happy with this reference.}\compat{Why? Do you want to suggest an alternative?}. This also works for closed half-translation surfaces, but if the surface has poles, then we have to replace the universal cover with the PRU cover described in this article.

From local considerations, a curve on such a surface is a geodesic if and only if it satisfies our definition of a trail. (Indeed, this is the motivation for our definition.) Hadamard spaces are well known to be geodesically convex, so this gives \Cref{thm:convex} in this context.

The strategy for deducing \Cref{thm:closed trails} in this context is to use facts about isometries of Hadamard spaces. Isometries of metric spaces can categorized based on their attained translation lengths. Here the {\em translation length} of a point $p \in X$ under an isometry $\varphi:X \to X$ is $\TL(x)=d\big(x, \varphi(x)\big)$.
The isometry $\varphi$ is {\em elliptic} if it has a fixed point, {\em hyperbolic} if $\TL$ attains a strictly positive minimum, and {\em parabolic} if the infimum of values of $\TL$ is not attained. An isometry of a locally $\mathrm{CAT}(0)$ space translates along some geodesic if and only if the isometry is hyperbolic \cite[II, Thm 6.8]{BH}.

Given a PR homotopy class $\conj{\gamma}$ on a closed surface $S$ as above, we get a deck transformation $\Delta_\gamma:\tilde S \to \tilde S$ where $\tilde S$ is the Hadamard cover described above, which by hypothesis is not elliptic (because $\conj{\gamma}$ is nontrivial and non-polar). A cocompact group of isometries acting on a Hadamard space cannot contain parabolic isometries \cite[II, Prop 6.10]{BH}. Therefore, $\Delta_\gamma$ is hyperbolic and translates along a geodesic. It is not hard to move from this point to the description in \Cref{thm:closed trails} using elementary facts about Euclidean cone surfaces, though in the cone surface case cylinders are immersed rather than embedded. Also the closed trails in a cylinder in this case are parallel (globally in the translation surface case and locally in the cone surface case).

\subsubsection{Noncompact translation surfaces and cone surfaces}

We will briefly explain how the argument from \Cref{sect:translation surfaces} proving special cases of our main results breaks when we consider noncompact translation surfaces and Euclidean cone surfaces whose cone angles are larger than $2 \pi$.
\compat{I removed (commented out) the questionable sentence ``We are not sure if these arguments can be upgraded to be as broadly applicable as our results in the setting of noncompact translation surfaces.''}

There are two difficulties.
First, there are noncompact translation surfaces, all of whose singularities are finite cone singularities, that admit leaf triangulations (i.e., triangulations by saddle connections joining singularities) but whose universal covers are not complete; see \Cref{fig:disk_branched_cover} for an example. Second, since the surface is not compact, it is unclear how to rule out parabolic isometries in the deck group.

\begin{figure}[htb]
\centering
\includegraphics[width=3in]{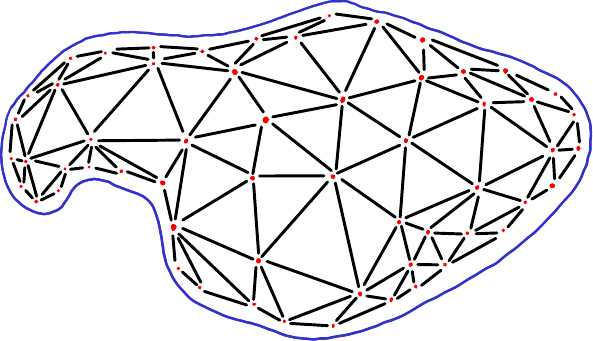}
\caption{Part of an infinite triangulation of a connected open subset of the plane is depicted. Let $S$ be the double cover of this disk which is branched over the vertices but has no other branching.
The surface $S$ is naturally a translation surface with a leaf triangulation, but is not complete as a metric space. \Cref{thm:convex} and \Cref{thm:closed trails} apply to $S$, with the latter giving a closed trail in every nontrivial free homotopy class.  \compatnew{Edited the first sentence slightly.}}
\label{fig:disk_branched_cover}
\end{figure}

It is interesting to note that the open unit disk in $\R^2$ has a complete metric for which geodesics are straight lines: The Klein disk model of the hyperbolic plane. We wonder:

\begin{question}
\label{q:cat0}
Which zebra surfaces have a complete $\mathrm{CAT}(0)$ metric whose geodesics are the trails?
\end{question}

Translation structures are special cases of zebra structures, so our results hold in this setting.
We have the following consequence: \compat{This corollary and proof was added in response to the following comment of Ferran. Perhaps I over reacted, or we can shorten the proof?} \commf{Suggestion remove this and explain how exactly Thm 1.3 is used to conclude that the deck group of the universal cover is `pure elliptic.'}

\begin{corollary}
\label{noncompact translation surface}
If $S$ is a noncompact translation surface whose universal cover $\tilde S$ is geodesically convex (or has a leaf triangulation), then every nontrivial deck transformation is a hyperbolic isometry of $\tilde S$.
\end{corollary}

We will explain that \Cref{fig:tractrix} illustrates a failure of our conclusions to hold in the context of noncompact cone surfaces with all singularities having cone angles larger than $2 \pi$. \commf{This last sentence was not clear enough.}\compat{Okay, I think the issue is that the sentence is explained by the paragraph, so I now preface with ``We will explain that...''. Is this sufficient?}
Note that this surface is not a zebra surface, because it has cone singularities with cone angle strictly between $2\pi$ and $3 \pi$. The figure illustrates a Euclidean cone structure on the annulus, but with no geodesic core curve, because the annulus continues to get thinner as we move towards one boundary. This surface is depicted with a decomposition into quadrilaterals, which when cut along their diagonals gives a leaf triangulation, in the sense that edges are saddle connections and vertices are singularities with cone angle larger than $2 \pi$. It follows that one of the following two implications must be incorrect in this context: A leaf triangulation implies convexity of the universal cover, or convexity of the cover implies the existence of a geodesic representative in every nontrivial free homotopy class of loops. However, we conjecture:

\begin{conjecture}
Fix $\epsilon>0$. Let $S$ be noncompact Euclidean cone surface such that all cone singularities have angle at least $2 \pi+\epsilon$. Suppose that $S$ admits a triangulation by saddle connections. Then the universal cover $\tilde S$ is convex and every nontrivial free homotopy class of closed curves in $S$ contains a geodesic representative.
\end{conjecture}

\begin{figure}[htb]
\centering
\includegraphics[width=3in]{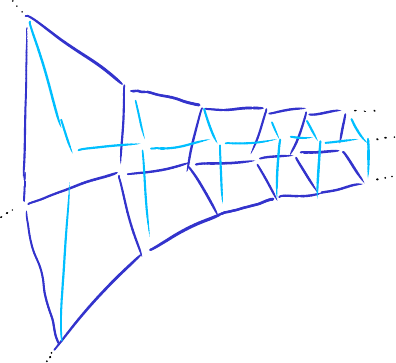}
\caption{A Euclidean cone structure on the annulus built using infinitely many trapezoids.}
\label{fig:tractrix}
\end{figure}

\subsubsection{Dilation surfaces with cone singularities}
\label{dilation with cone singularities}
A {\em half-dilation surface with cone singularities} is a surface with an atlas of charts to the plane and finite covers of $\C/\langle z \mapsto -z\rangle$ branched over the origin such that transition maps are in the group generated by translations, dilations, and rotations by $\pi$. A {\em dilation surface with cone singularities} is the same, only covers should be of the plane branched over the origin and the transition maps are in the group generated by translations and dilations.

The {\em (dilational) holonomy} around an oriented loop in such a surface is the ratio of lengths of a segment parallel translated around a loop and the original segment, measured in a fixed local coordinate chart and interpreted as an element of $\R_+$ \cite{W21}. Because our singularities are cone singularities, the holonomy around any contractible loop is trivial and the notion of holonomy around a loop gives rise to the  {\em holonomy homomorphism} $\pi_1(S) \to (\R_+, \times)$.

Earlier we mentioned the {\em Hopf tori}, given by $\C^\ast/\langle z \mapsto \lambda z\rangle$ where $\C^\ast=\C \setminus \{0\}$ and $\lambda$ is a positive real number, because these surfaces have non-convex universal covers. The fibers of the map $\arg: \C^\ast \to \R/2\pi\Z$ give a foliation of a Hopf torus. To see a Hopf torus has non-convex universal cover, observe that the torus has two closed leaves of every slope, and if $p$ and $q$ are points from distinct closed leaves of the same slope then there is no trail connecting $p$ with $q$. See \cite{DFG19} for background on Hopf tori and related constructions.

Consider an immersion of the cylinder $C=[-1,1] \times \bbS^1$ into a Hopf torus $T$ that sends vertical closed leaves of $C$ to the foliation of $T$ given by fibers of $\arg$. The pullback of the dilation structure to $C$ is called a {\em dilation} (or {\em affine}) {\em cylinder}. The universal cover of $C$ with this structure can be seen to be isomorphic to a sector in a branched cover of the plane, and we call the angle of this sector the {\em angle} of the dilation cylinder. It is not hard to see that a dilation surface (or half-dilation surface) containing an affine cylinder with angle $\pi$ or more cannot be convex, and no closed curve crossing such a cylinder can have a closed trail representing it. See \Cref{full cylinder}. We note that dilation surfaces can also contain {\em flat cylinders}, isomorphic to rotations of $[0,w] \times \R/c\Z$ for $c$ and $w$ positive.

We were unable to find a simple argument for proving Theorems \ref{thm:closed trails} and \ref{conj:zebra case} and in the context of dilation surfaces with cone singularities that bypasses the technique we use for zebra surfaces.

\subsubsection{Dilation surfaces with dilation singularities}
\label{sect:dilation singularities}
A {\em dilation singularity} is more general than a cone singularity: We allow a loop around a dilation singularity to have nontrivial dilational holonomy.

The dilation singularities must locally look like certain natural models. Let $\bbU$ denote the closed upper half-plane, $\{z \in \C:~\Im z \geq 0\}$.
One example of such a model is given by $M_\lambda = \bbU/\sim$ with $\lambda>0$, where $\sim$ is the finest equivalence relation on $\bbU$ where for every positive $x \in \R$, $-x \sim \lambda x$.
Here the singularity of $M_\lambda$ is at the origin, and the dilational holonomy of a counterclockwise loop around the origin is $\lambda$. In general, the model singularities are given by branched covers $M^n_\lambda$ of $M_\lambda$ of degree $n \geq 1$ branched over the origin $\0$. The dilational holonomy around the singularity in $M^n_\lambda$ is $\lambda^n$ and the {\em angle} at the singularity is $n \pi$. These models have natural local coordinate maps from $M^n_\lambda \setminus \{\0\}$ to $\C$ whose transition functions are in the group generated by translations, dilations, and rotations by $\pi$. If $n$ is even, we can specify local coordinate maps to $\C$ where the transition functions are in the group generated by translations and dilations.

A {\em half-dilation surface with dilation singularities} is a surface together with an atlas of charts to the plane and the spaces $M^n_\lambda$, where the transition maps are in the group generated by translations, dilations, and rotations by $\pi$ in local coordinates. A {\em dilation surface with dilation singularities} is the same, only the transition maps are in the group generated by translations and dilations, and the model singularities are of the form $M_\lambda^{2n}$. \compat{The three paragraphs above were rewritten. One point is that I want to allow cone singularities to be considered dilation singularities.}

Note that the universal cover of a dilation surface with dilation type singularities has no natural metric, because now there is dilational holonomy around loops in the cover. This makes metric methods to deduce convexity and existence of closed trails seem unlikely to work. It is conceivable there is a different metric worth considering. See \Cref{q:cat0}.

Interestingly, some papers in the field of dilation surfaces only allow cone singularities, while others allow dilation singularities. Veech was probably the first to consider dilation surfaces, though he worked in the more general context of (singular) complex affine structures on surfaces \cite{V93} \cite{V97}. Veech proved fundamental results on the moduli spaces of these structures. This understanding was recently improved in \cite{ABW} which specifically considers moduli spaces of dilation surfaces allowing dilation singularities, and proving (among other results) that the moduli space of dilation surfaces with singular data fixed is an orbifold covering of the usual moduli space of the corresponding punctured surface.

The directional foliations on the spaces $M_\lambda$ constructed above are isomorphic to the directional foliations of $\C^\ast/\langle z \mapsto - z\rangle$, and the foliations on $n$-fold branched covers $M^n_\lambda$ are isomorphic to those on the $n$-fold branched covers of
$\C^\ast/\langle z \mapsto - z\rangle$. Thus, dilation surfaces with dilation singularities still induce zebra structures on the surface. \commf{Weird phrase. We mean equivalent foliations?}\compat{I'm now saying isomorphic. This seems consistent with things later in the paper. It seems to be used in foliation theory (though I'm not 100\% certain this is the best  term).}

Veech proved that a dilation surface has a triangulation by saddle connections if and only if it contains no dilation cylinders with angle $\pi$ or more
\cite[Appendix]{DFG19}. Note however that dilation singularities with angle $2 \pi$ are singularities in the dilation surface sense but are nonsingular on the induced zebra structure. Therefore a triangulation by saddle connections may not be a leaf triangulation. Nonetheless, by Theorem~\ref{conj:zebra case} we obtain:
\begin{theorem}
\label{thm:dilation case}
Suppose $S$ is a closed surface with a dilation structure and at least one singularity, but without dilation singularities with angle $2 \pi$. Then, the following are equivalent:
\begin{itemize}
\item $S$ has a triangulation by saddle connections.
\item The universal cover $\tilde S$ is convex.
\item Every nontrivial free homotopy class $\conj{\gamma}$ of closed curves is realized by either a unique closed trail or a cylinder.
\item $S$ contains no dilation cylinder with angle $\pi$ or more.
\end{itemize}
\end{theorem}

Remark that \Cref{thm:dilation case} does not hold if one allows dilation singularities with angle $2 \pi$. When dilation singularities with angle $2 \pi$ are allowed in a dilation surface, there can be multiple dilation and flat cylinders with homotopic core curves that join together to form a full cylinder in the zebra sense (dilation singularities with angle $2 \pi$ are not considered to be singularities on the zebra surface). An example of a full zebra cylinder made from two dilation cylinders and one flat cylinder is shown in \Cref{fig:amalgamated_cylinder}.

\subsection{Outline of paper}
\label{sect:outline}

We will now explain what is done in this paper. Because the paper proves statements about zebra surfaces but some reader's interests will only include translation surfaces or dilation surfaces, we try to point out what can be skipped for such a reader. \commb{In this outline do you want to distinguish dilation singularities with only cone singularities versus those with dilation singularities?} \compat{I think not, but if you notice something useful, please point it out.}

In \Cref{sect:formal}, we carefully define what a zebra surface is and establish basic terminology. All results here are well known for translation and dilation surfaces. We prove that the Gauss-Bonnet theorem holds for zebra surfaces and subsurfaces with polygonal boundaries.

In \Cref{sect:basic}, we formally define the pole-resolved universal cover of a zebra surface. This is also natural for half-translation surfaces with poles, but we have not seen it in the literature. We consider basic geometric objects on zebra planes (such as PRU covers) such as polygons and trails. We prove basic results about these objects, which are all obvious when working with half-translation and half-dilation surfaces. For example, we construct rectangles, prove arcs of trails have maximal extensions as trails, and show trails on zebra planes are proper maps.

In \Cref{sect:boundary}, we define the notion of a zebra structure on a surface with boundary. Our definition allows for a polygonal boundary, generalizing the natural idea of a dilation surface with piecewise-linear boundary. This is important for laying a rigorous foundation for the next section. \compat{Edits to this paragraph and the next one related to the splitting the surgery section into a section on zebra surfaces with boundary and a second section on surgery. Oct 31.}

In \Cref{sect:surgery}, we consider surgical constructions on zebra surfaces, building new zebra surfaces from subsurfaces of others with polygonal boundary. Because of the flexibility of the zebra structure, we allow gluing subsurfaces together by homeomorphism of edges in the boundary. This is useful for simplifying several arguments appearing later in the paper. In \Cref{sect: not dilation}, we use surgery to show that there are zebra surfaces that do not arise from half-dilation structures.

\Cref{sect:foliations} focuses on producing foliations of polygons in zebra planes
by combining leaves from the foliations $\sF_m$ with $m$ varying. These foliations form the foundation of our later arguments, and some such foliations seem interesting even in the Euclidean plane (though proofs would be easier in this context where analytic methods are available).
In \Cref{sect:foliating triangles}, we show that a triangle in a zebra surface can be foliated by leaves emanating from a vertex, and that slopes of these leaves vary monotonically. In \Cref{sect:foliating polygons}, we prove that a polygon can be foliated by leaves passing through one edge, where the slopes of leaves passing through a given point on that edge are given by a monotone function (subject to obvious constraints). This second result has a slick proof using surgery on zebra surfaces. \compat{Rewrote this paragraph in response to some of Barak's comments. Dec 31, 2022.}

In \Cref{sect:connecting points with trails}, we investigate the behavior of trail rays emanating from a point in a zebra plane $Z$, and also arcs of trails joining two points. This work is fundamental for our convexity arguments later in the paper. Considering all the trail rays emanating from a point in $Z$ leads to a foliation of an open subset of $Z$ with a different singular structure. We use this structure to prove that polygonal regions in $Z$ all of whose exterior angles are at least $\pi$ are convex. This is clear from $\mathrm{CAT}(0)$ arguments when the polygon is in a translation surface, but seems unclear for polygons in dilation surfaces with dilation-type singularities in the interior of the polygon. This result allows us to prove a criterion for convexity of PRU covers of closed zebra surfaces where $\alpha$ is nonpositive; see \Cref{torus cover convex}. This statement applies to dilation tori, all of whose singularities are dilation-type with angle $2\pi$. In \Cref{sect:continuity of trail arcs}, we prove a continuity statement for the map sending a pair of points to the arc of a trail between the two points.

In \Cref{sect:convexity proof}, we prove \Cref{thm:convex}, which says that a zebra plane $Z$ with a leaf triangulation is convex. We choose an arbitrary point $p \in Z$ and consider all trail rays emanating from $p$. We inductively show that these rays cover every triangle in our triangulation. Given the results from the prior section, the proof is largely combinatorial. This argument is likely of interest to anyone interested in noncompact translation surfaces or in dilation structures.

In \Cref{sect:closed trails}, we consider the question of finding closed trails in a zebra surface. We begin with a discussion of PR free homotopy classes of curves and a pole-resolved version of the fundamental group in \Cref{sect:curves and deck transformations}. In \Cref{sect:existence of closed trails}, we prove a theorem that guarantees the existence of a closed trail. This result, \Cref{closed trails exist}, is of interest to those thinking about dilation surfaces. Later subsections are concerned with developing the remainder of the structure described in \Cref{thm:closed trails}. Arguments should be readable to experts interested in the contexts of translation or dilation surfaces. We prove \Cref{noncompact translation surface} in \Cref{sect:noncompact translation surface}. \compat{I made some minor changes to this paragraph on Oct 19th.}

\Cref{sect:loop} is dedicated to the proof of \Cref{conj:zebra case}. As explained in~\Cref{secc:proof-conj-zebra case} implications $(a)\Rightarrow (b) \Rightarrow (c) \Rightarrow (d)$ of the statements in \Cref{conj:zebra case} can be deduced from previous results in this article. The main result of this section is to show $(d)\Rightarrow(a)$, that is, every closed zebra surface $S$ with at least one singularity that is not a pole and that contains no full cylinders has a leaf triangulation. In~\Cref{ssec:triangulating-polygons-cylinders} we discuss first the triangulation of polygons and cylinders. \Cref{sect:minimal triangulations} introduces two new crucial ideas: \emph{preleaf} and \emph{minimal} triangulations. There we also
prove $(d)\Rightarrow(a)$ under extra the hypothesis that $S$ has no poles (\Cref{no full implies leaf}). In \Cref{ssec:producing-new-triangulations} we extend this proof for the case where there are poles. The proof of \Cref{conj:zebra case} is detailed in ~\Cref{secc:proof-conj-zebra case}.

\Cref{sect:questions} provides a list of open questions. We know very little about zebra surfaces.

\section{Formal definitions}
\label{sect:formal}

\subsection{Surfaces and structures}
\label{surfaces and structures}
For us a {\em surface} is a second countable Hausdorff space that is locally homeomorphic to $\R^2$. Throughout this paper, all surfaces are oriented.
A {\em closed surface} is a compact connected surface without boundary.
\commb{Isn't it true that non-orientable surfaces can't admit zebra structures? The zebra structure defines an orientation. It makes sense to go around the center of a stellar neighborhood in the counterclockwise direction}
\compat{I agree. Some things (e.g., foliations) make sense even for non-oriented surfaces, but I don't see the point of considering them.}
\compat{Earlier versions required a surface to be connected. I removed this, so it will be important to indicate this where required.} \commb{Connected is still being required.} \compat{Thanks. It is now really not being required!}

Let $X$ be a topological space and $S$ be a surface. An {\em atlas of charts} from $S$ to $X$ is a collection of {\em charts} of the form $\phi:U \to X$ whose domains are open and cover $S$ and such that each chart $\phi:U \to X$ has an open image $\phi(U)$, and is a homeomorphism from $U$ to its image. A {\em transition map} between two charts with intersecting domains $\phi_1:U_1 \to X$ and $\phi_2:U_2 \to X$ is the restriction of $\phi_2 \circ \phi_1^{-1}$ to $\phi_1(U_1 \cap U_2)$.

In general, geometric structures are specified by defining a pseudogroup of homeomorphisms between the open sets of $X$, and insisting that transition maps lie in the pseudogroup. We call $X$ the {\em model space}. We refer the uninitiated reader to Chapter 3 of \cite{ThurstonBook}. Foliations can be considered to be a particular case of a geometric structure, see \cite[Example 3.1.9]{ThurstonBook}.

\subsection{Foliated surfaces}
The {\em horizontal foliation} $\sH$ of $\R^2$ is the collection of all horizontal lines in the plane.

If $U \subset \R^2$ is open, we say that two points $(x_1,y_1)$ and $(x_2, y_2)$ are {\em horizontally equivalent} in $U$ if
$y_1=y_2$ and the horizontal line segment between the points is contained in $U$. The {\em horizontal foliation} of $U$ is the
collection $\sH|_U$ of horizontal equivalence classes. We call these equivalence classes {\em leaves}.

The {\em horizontal foliation pseudogroup} of $\R^2$ is the collection of homeomorphisms $h:U \to V$ between open subsets of $\R^2$ that induces a bijection from the leaves of $\sH|_U$ to the leaves of $\sH|_V$.

A {\em foliation atlas} on a surface $S$ is an atlas of charts to $\R^2$ whose transition functions lie in the horizontal foliation pseudogroup. A foliation atlas determines a {\em foliation equivalence relation} on $S$, namely the finest one such that given any chart $\phi:U \to \R^2$, preimages of points in the same leaf of $\sH|_{\phi(U)}$ are equivalent. A {\em foliation} of $S$ is the collection of foliation equivalence classes obtained from a foliation atlas. We call the equivalence classes {\em leaves}.

If $S$ is a surface with a foliation $\sF$, and $A \subset S$ is a subsurface (possibly with boundary), then the {\em restricted foliation} on $A$ is the collection $\sF|_A$ of connected components of intersections $A \cap \ell$, where $\ell$ varies over the leaves of $\sF$. If $A$ is an open set, then it is a surface and the restricted foliation is a foliation on $A$, because a foliation atlas for $A$ can be obtained by restricting each chart $\phi:U \to \R^2$ in the atlas for $\sF$ to the function $\phi|_{A \cap U}: A \cap U \to \R^2$.

\compat{I removed a remark indicating that $A$ will be a sector based on Ferran's concern that sector was not defined yet.}

\begin{definition}[Leaf topology]
\label{def:leaf topology}
Let $S$ be a topological surface, perhaps with boundary, and let $\ell \subset S$ be a subset. The {\em leaf topology} on $\ell$ is the coarsest topology such that for each open $U \subset S$, each connected component of $U \cap \ell$ is open. If $\ell \in \sF$ is a leaf of a foliated surface without boundary, then each point of $\ell$ has a neighborhood homeomorphic to an open interval. This gives
$\ell$ the structure of a connected $1$-manifold.
\end{definition}

A {\em local homeomorphism} $f:S_0 \to S_1$ is a map such that for every point $p \in S_0$, there is an open neighborhood $U$ of $p$ such that $f(U)$ is open in $S_1$ and $f|_U:U \to f(U)$ is a homeomorphism.
Suppose $S_1$ is a space with a foliation $\sF_1$. Let $S_0$ be another topological space and suppose $f:S_0 \to S_1$ is a local homeomorphism.
Then there is a natural pullback equivalence relation, namely the finest equivalence relation on $S_0$ such that
the points $p$ and $q$ of $S_0$ are equivalent when
there is an open set $U \subset S_0$ containing $p$ and $q$ such that $f(U)$ is open, $f|_U:U \to f(U)$ is a homeomorphism, and $f(p)$ and $f(q)$ lie on the same leaf of $\sF_1 |_{f(U)}$. The {\em pullback foliation} $f^{\ast}(\sF_1)$ is the collection of equivalence classes of the pullback equivalence relation. If $(S_0,\sF_0)$ and $(S_1, \sF_1)$ are two foliated spaces and $f:S_0 \to S_1$ is a homeomorphism such that $\sF_0 = f^{\ast}(\sF_1)$, then we say that $f$ is an {\em isomorphism}. That is, $f$ must
induce a bijection from $\sF_0$ to $\sF_1$.

Given a collection $\{(S_i, \sF_i)\}$ of foliated spaces as above, the disjoint union $\bigsqcup_i \sF_i$ is a foliation on $X = \bigsqcup_i S_i$.

In any of the foliated spaces $(X, \sF)$ constructed as above, the {\em foliation pseudogroup} consists of all isomorphisms between open subsets of $X$ endowed with restricted foliations.

\begin{remark}
One can define a foliated surface $(S,\sF)$ to be the geometric structure determined by a foliation atlas. But, treating a foliation as its collection of leaves seems more natural, and from the collection of leaves derived from such an atlas we can recover an atlas defining the structure. The charts can be taken to be the collection of all $\phi:U \to \R^2$ where $U \subset S$ is open and $\phi$ is an isomorphism from $(U, \sF|_U)$ to $\big(f(U),\sH|_{f(U)}\big).$
\end{remark}

\subsection{Standard singularities}
\label{sect:standard singularities}
Note that the action of multiplication by $-I$ on $\R^2$ preserves the horizontal foliation. Let $\Pi_{-1}$ denote $\R^2/{-I}$, which has a cone point with cone angle $\pi$ at the image of the origin. We call this cone point the {\em origin} $\0 \in \Pi_{-1}$ and write $\Pi_{-1}^\ast=\Pi_{-1} \setminus \{\0\}$. Note that a collection of inverses of restrictions of the covering map $\R^2 \setminus \{\0\} \to \Pi_{-1}^\ast$ gives a foliation atlas on $\Pi_{-1}^\ast$. We call the foliation associated to this atlas the {\em horizontal foliation} $\sH_{-1}$ of $\Pi_{-1}^\ast$.

For each integer $n \geq 0$, we define $\Pi_n$ to be the branched cover of $\Pi_{-1}$ of degree $n+2$ branched over the origin. In all these spaces, we use $\0$ to denote the unique preimage of $\0 \in \Pi_{-1}$, and call $\0$ the {\em origin}. Note that geometrically $\0 \in \Pi_n$ is a cone singularity with cone angle $(n+2) \pi$. We define $\Pi_n^\ast=\Pi_n \setminus \{\0\}$.
The pullback of the horizontal foliation on $\Pi_{-1}^\ast$ under the covering map $\Pi_n^\ast \to \Pi_{-1}^\ast$
is the {\em horizontal foliation} $\sH_n$ of $\Pi_{n}^\ast$.

Note that $\Pi_0$ is naturally homeomorphic to $\R^2$, and the foliation $\sH_0$ of $\Pi_0$ is carried by this homeomorphism to the horizontal foliation of $\R^2 \setminus \{\0\}$.

We further define $\Pi_{-2}$ to be the plane $\R^2$ equipped with the foliation of $\R^2 \setminus \{\0\}$ by circles with center $\0$. For convenience we call this foliation the {\em horizontal foliation} of $\Pi_{-2}^\ast =\Pi_{-2} \setminus \{\0\}$ and denote it by $\sH_{-2}$. This is locally homeomorphic to one of the straight-line foliations that arises near a double pole of a quadratic differential. Singularities of this form show up in some of our arguments, but are not allowed in zebra surfaces.

A {\em prong} of $\Pi_n$ is a leaf of the horizontal foliation with $\0$ as an endpoint. There are $n+2$ prongs in $\Pi_n$.

\subsection{Singular foliations}
\label{sect:singular foliations}
Consider the model space
\begin{equation}
\label{eq:model space}
X=\Pi_{-1} \sqcup \bigsqcup_{n \geq 1} \Pi_n, \quad \text{and let} \quad X^\ast=\Pi_{-1}^\ast \sqcup \bigsqcup_{n \geq 1} \Pi_n^\ast \subset X,
\end{equation}
which comes equipped with its horizontal foliation $\sH_X=\bigsqcup \sH_n$. \compat{Zero was removed from the indices above to accommodate the removal of removable singularities.}

Let $h:U \to V$ be an orientation-preserving homeomorphism between open subsets of $X$. We say $h$ is in the {\em horizontal pseudogroup} (of $X$) if :
\begin{enumerate}
\item We have $h(U \cap X^\ast) = h(U) \cap X^\ast$.
\item The restriction $h|_{U \cap X^\ast}$ is in the foliation pseudogroup of $(X^\ast, \sH_X)$.
\end{enumerate}

Observe that if $U_i \subset U$ is a connected component then $U_i \subset \Pi_m$ for some $m=m(i)$ and $h(U_i) \subset \Pi_n$ for some $n=n(i)$.
Note that statement (1) implies that $\0 \in U_i$ if and only if $\0 \in h(U_i)$ and that $h(\0)=\0$ if $\0 \in U_i$.

A {\em singular foliation atlas} on a surface $S$ is an atlas of charts to $X$ whose transition functions lie in the horizontal pseudogroup. We say a point $p \in S$ is a {\em singularity} if there is a chart $\phi:U \to \Pi_n$ such that $\phi(p)=\0$. Because elements of the pseudogroup send origins to origins, the notion of being a singularity is independent of the chart. The {\em singular set} $\Sigma \subset S$ is the collection of all singularities.
Every point $p$ has a neighborhood $U$ such that $U \setminus \{p\}$ contains no singularities, so $\Sigma$ is a closed discrete subset. \compat{Edits related to the point that $\Sigma$ is not just discrete but closed and discrete. Thanks Barak. Oct 31st.}

Observe that a singular foliation atlas determines a foliation on $S \setminus \Sigma$. We call this foliation a {\em singular foliation} on $S$. The {\em singular data} of the atlas consists of $\Sigma$ and the function $\alpha:S \to \Z_{\geq -1}$ whose support is $\Sigma$ and which sends $p \in \Sigma$ to the $n$ such that there is a chart from a neighborhood of $p$ to a neighborhood of $\0 \in \Pi_n$. Observe that $\alpha$ is well-defined, because a single chart tells you that there are $n+2$ prongs emanating from $p$. Here a {\em prong} emanating from a point $p \in S$ is the germ of an injective path $\gamma:(0,1) \to \ell$ into a leaf $\ell$ with $\lim_{t \to 0^+} \gamma(t)=p$, i.e., an equivalence class of such paths where two such paths $\gamma_1:(0,1) \to \ell$ and $\gamma_2:(0,1) \to \ell$ are equivalent if there
are constants $\epsilon_1, \epsilon_2 \in (0,1)$ such that the restrictions $\gamma_1|_{(0, \epsilon_1)}$ and $\gamma_2|_{(0, \epsilon_2)}$ are the same up to orientation-preserving reparameterization. (Our definition of prong is slightly non-standard in that typically prongs are only defined at singular points, but we define them at all points.)
If $\alpha(p)=-1$, the point $p$ is called a {\em pole}.

\begin{remark}[Removable singularities]
\compat{This is new.}
Since allowing removable singularities would make statements of our main results more technical and since we don't have much use for them in this paper, we have purposely not allowed removable singularities in our zebra surfaces.
However, if the model space $X$ is altered to include $\Pi_0$ with its horizontal foliation, then a point mapping to $\0 \in \Pi_0$ would be a {\em removable singularity} or {\em marked point}. A marked point in a singular foliation can be removed by replacing a chart to $(\Pi_0, \sH_0)$ with a chart to $(\R^2, \sH)$ or an isomorphic subset of $X$.
So, removable singularities are much the same as regular points. However, treating a regular point as a singularity creates some problems, because including a regular point in $\Sigma$ alters the space being foliated. Thus, this change alters the notion of what a leaf is. In particular, if we allow removable singularities, the statements of \Cref{thm:convex} and \Cref{thm:closed trails} or the definitions they depend on need to be suitable altered to make the results still true.
\compat{I truncated the remark, by editing and commenting out here. I did make some corrections before commenting out as suggested by Barak, if we want to expand this comment.}
\end{remark}

Given a prong contained in a leaf $\ell$ emanating from a singular point $p$, a parameterization of the leaf $(0,1) \to \ell$ can be extended to include $p$. In such a case we call the singularity an {\em endpoint} of the leaf. If a leaf has two endpoints, then we call the union of the leaf with its endpoints a {\em saddle connection}. If the leaf has one endpoint, then we call this union a {\em separatrix}. A leaf with no endpoints is called {\em bi-infinite}. A leaf that is homeomorphic to a circle is said to be {\em closed}. \compat{Based on a comment of Ferran, I removed the statement that we include singular endpoints in leaves of singular foliations. I mention this in case you notice a place where I say the endpoint is in a leaf. I am including the endpoints in saddle connections and separatrices.}


Sometimes we will allow double poles to appear in our singular foliations, though we do not allow these more general singular foliations in our zebra surfaces. A {\em generalized singular foliation atlas} on a surface is defined as above but including $\Pi_{-2}$ in the model space with its horizontal foliation $\sF_{-2}$. In this case $\alpha$ takes values in $\Z_{\geq -2}$ and we call a point $p$ where $\alpha(p)=-2$ a {\em double pole}. Note that a singular foliation is also a generalized singular foliation.

\begin{proposition}[Euler–Poincaré Formula]
\label{Euler-Poincare}
Let $\sF$ be a generalized singular foliation on a closed surface $S$ with singular set $\Sigma$ and singular data $\alpha$. Then,
$$\sum_{p \in \Sigma} \alpha(p)=-2 \chi(S),$$
where $\chi(S)$ denotes the Euler characteristic of $S$.
\end{proposition}

This result follows from the Poincaré-Hopf theorem, which gives the Euler characteristic in terms of the sums of indices of a zeros of a vector field with isolated zeros. For a proof see Proposition 5.1 of \cite{FLP}. The main idea is to pass to a double cover, so the foliation becomes oriented and gives rise to a vector field.
Note that here we allow $\alpha$ to take the values $-2$ and $-1$ while \cite{FLP} does not (though the proof goes through in this case).

\subsection{The extended real numbers}
Let $\hat \R=\R \cup \{\infty\}$, which is homeomorphic to a circle. The usual increasing order on $\R$ extends to a cyclic order on $\hat \R$.

We use interval notation to denote subsets of $\hat \R$. If $m_0 \neq m_1$, we use interval notation such as $(m_0, m_1)$ to denote the set of slopes $m \in \hat \R$ for which the triple $(m_0, m, m_1)$ is in strictly increasing cyclic order. Then $[m_0, m_1]$ will denote the associated closed interval, $(m_0, m_1) \cup \{m_0, m_1\}$.

\subsection{The stellar functions}
\label{sect:stellar function}
The {\em stellar function} on $\R^2 \setminus \{\0\}$ is the function which sends a point $p$ to the slope of the line joining $p$ to $\0$:
\begin{equation}
\label{eq:standard stellar foliation}
\rho:\R^2 \setminus \{\0\} \to \hat \R; \quad (x,y) \mapsto \frac{y}{x}.
\end{equation}
Observe that the value of $\rho$ is preserved by the action of $-I$, so $\rho$ descends to a well-defined map $\rho_{-1}:\Pi_{-1}^\ast \to \hat \R$. Then for $n \geq 0$, we can obtain functions $\rho_{n}:\Pi_{n}^\ast \to \hat \R$ by composing the covering map $\Pi_{n}^\ast \to \Pi_{-1}^\ast$ with $\rho_{-1}$. We call $\rho_n$ the {\em stellar function} on $\Pi_{n}^\ast$.

A {\em ray} in $\Pi_n$ of slope $m$ is a connected component of $\rho_n^{-1}(m)$.

\subsection{Definition of zebra structure}
\label{sect:definition of zebra}
Let $\{\sF_m:~ m \in \hat \R\}$ be a collection of singular foliations on a connected surface $S$ with the same singular set $\Sigma$ and the same singularity data function $\alpha:S \to \Z_{\geq -1}$.

A {\em stellar neighborhood} of a point $p \in S$ is an open neighborhood $U$ of $p$ such that there is an integer $n \geq -1$ and a homeomorphism $h:U \to \Pi_n$ such that $h(p)=\0$ and the following statements hold for each slope $m \in \hat \R$. \compat{The statements below were corrected and made to align with what is in \S 4 before Sept 28, 2022.}
\begin{enumerate}
\item For each ray $r \subset \Pi_n^\ast$ of slope $m$, $h^{-1}(r)$ is contained in a leaf of $\sF_m$.
\item For each prong of $\sF_m$ emanating from $p$, there is a ray $r \subset \Pi_n^\ast$ of slope $m$ such that for any path $\gamma:(0,1) \to r$ with $\lim_{t \to 0^+} \gamma(t)=\0$, the preimage $h^{-1} \circ \gamma$ represents the prong.
\label{prongs as preimage of rays}
\end{enumerate}
The above two statements guarantee that for each $m$, $h$ induces a bijection from prongs of $\sF_m$ emanating from $p$ and rays of slope $m$ in $\Pi_n$. We call $h$ a {\em stellar homeomorphism}. It follows by counting prongs and rays that $n=\alpha(p)$.

We say that the collection of singular foliations $\{\sF_m\}_{m \in \hat \R}$ is a {\em stellar foliation structure} or a {\em zebra structure} on $S$ if the foliations have the same singular sets, the same singular data, and every $p \in S$ has a stellar neighborhood. We call a surface together with a zebra structure a {\em zebra surface}.

\subsection{The action by homeomorphisms of the circle}
\label{sect:homeo action}

Our foliations are parameterized by the topological circle $\hat \R$. There is a natural action of the group $\Homeo_+(\hat \R)$ of all orientation-preserving homeomorphisms of $\hat \R$ on zebra surfaces defined as follows. Suppose $\{\sF_m\}$ defines a zebra structure on $S$. If $\varphi \in \Homeo_+(\hat \R)$, then for each $m \in \hat \R$ we can define $\sF'_m = \sF_{\varphi^{-1}(m)}$, and $\{\sF'_m\}_{m \in \hat \R}$ will define another zebra structure on $S$. Note the original structure and the new structure have the same singular data.

Let $\Homeo_-(\hat \R)$ denote the collection of all orientation-reversing homeomorphisms of $\hat \R$. Then $\Homeo_+(\hat \R) \cup \Homeo_-(\hat \R)$ forms the full homeomorphism group of $\hat \R$. For $\varphi \in \Homeo_-(\hat \R)$, we define
$$\varphi\big(S, \{\sF_m\}\big) = \big(\bar S, \{\sF_{\varphi^{-1}(m)}\}\big)$$
where $\bar S$ denotes $S$ with its opposite orientation.

\subsection{Leaves on zebra surfaces}
As in the introduction, a leaf of a zebra surface is a leaf of one of the foliations $\sF_m$. The {\em slope} of a leaf on $S$ is the $m$ for which the leaf belongs to $\sF_m$. The word {\em horizontal} means slope zero, and {\em vertical} means slope $\infty$. Terms like {\em saddle connection}, {\em separatrix} and {\em bi-infinite leaf} all make sense on $S$.

Observe:
\begin{proposition}[Transversality]
\label{transversality}
If leaves $\ell_1$ and $\ell_2$ have distinct slopes and intersect at a non-singular point, then they cross transversely in the sense that there is an open disk containing the intersection point such that there is only one intersection between $\ell_1$ and $\ell_2$ in this disk and the disk is cut in two by $\ell_1$ with points in $\ell_2$ in both halves.
\end{proposition}
\begin{proof}
Since we are on a zebra surface, the intersection point has a stellar neighborhood that gives the desired properties.
\end{proof}

\subsection{Angles and Gauss-Bonnet}
\label{sect:gauss-bonnet}

Zebra surfaces have a natural notion of {\em angle}.
Let $\overline{pq}$ and $\overline{qr}$ be two segments of leaves on a zebra surface, where we allow any of these three points to be singular.
Let $U$ be a stellar neighborhood of $q$ and $h:U \to \Pi_{\alpha(q)}$ be the corresponding stellar homeomorphism. Then $\measuredangle pqr$ indicates the counterclockwise angle measured at the origin of $\Pi_{\alpha(q)}$ from $h(\overline{pq} \cap U)$
to $h(\overline{qr} \cap U)$. We normalize this measurement so
\begin{equation}
\label{eq:angle normalization}
0 \leq \measuredangle pqr < \big(\alpha(q)+2\big)\pi.
\end{equation}

Suppose $S'$ is a subsurface of a zebra surface. Let $q \in \partial S'$, and suppose that in a neighborhood of $q$, $\partial S'=\overline{rq} \cup \overline{qp}$ where $\overline{rq}$ and $\overline{qp}$ are segments of leaves and $S'$ is on the left as we move from $r$ to $q$ to $p$ along this boundary curve. Then the {\em interior angle} of $S'$ at $q$ is $\measuredangle pqr$.

\begin{theorem}[The Gauss-Bonnet Theorem for zebra surfaces]
\label{Gauss-Bonnet} 
Let $K$ be a compact subsurface of a zebra surface with a boundary consisting of a union of disjoint simple closed curves that are piecewise given by segments of leaves (with finitely many pieces). Let $\alpha$ be the singular data function on the surface containing $K$.
For a boundary point $q$, let $\theta_q$ denote the interior angle at $q$. Then,
\begin{equation}
\label{eq:Gauss-Bonnet}
\sum_{q \in \partial K} \left(\pi - \theta_q\right) - \sum_{p \in \Sigma \cap K^\circ} \pi \alpha(p)= 2 \pi \chi(K).
\end{equation}
\end{theorem}

Note that in \eqref{eq:Gauss-Bonnet}, there are only finitely many points in each sum for which the contribution to the sum is non-zero.

\begin{proof}
\compat{This proof was rephrased due to comments by Barak, and further improved after Ferran's comments.}
Consider $K$ as equipped with a foliation of slope $m$ which does not coincide with the slope of any of the finitely many leaves in the boundary of $K$.
Double $K$ across its boundary to obtain a closed surface $X$ to which we can apply \hyperref[Euler-Poincare]{the Euler–Poincaré Formula}. The doubled surface satisfies $\chi(X)=2\chi(K)$. The surface $X$ inherits a foliation from the two copies of $K$, where we allow our leaves to pass between the copies of $K$ through the boundary.
We will see that this foliation of slope $m$ of $K$ lifts to a generalized singular foliation of $X$. We allow our leaves to pass between copies of $K$ through the boundary, so this will be a foliation of the complement of the singular points in the interior of $K$ and endpoints of boundary edges. Let $\tilde \alpha$ denote the singularity data on $X$ for the lifted foliation. (Checking that $\tilde \alpha$ is well defined will prove that the lifted foliation to $X$ is a generalized singular foliation.) Each singular point $p \in K^\circ$ has two lifts $\tilde p_1, \tilde p_2 \in X$, and we have $\tilde \alpha(\tilde p_1)=\tilde \alpha(\tilde p_2)=\alpha(p)$. For $q \in \partial K$ an endpoint of a boundary edge, we have only one lift $\tilde q$ and by considering a stellar neighborhood of $q$ we see that $\tilde \alpha(\tilde q)=2 v_q-2$ where $v_q$ is the number of prongs in $K$ terminating at $q$. Then by \hyperref[Euler-Poincare]{the Euler–Poincaré Formula}, we have
$$\sum_{q \in \partial K} (2-2v_q) - \sum_{p \in \Sigma \cap K^\circ} 2 \alpha(p) = 2 \chi(X) = 4 \chi(K).$$
\commb{Is $\tilde{\alpha}(q)$ equal to zero for points $q$ in the boundary of $K$ which are in the interior of boundary edges? What is the stellar neighborhood for such a point $q$?}
\compat{Yes. But there is no stellar foliation. We are just considering the foliation of slope $m$ on $K$ doubled. This is a singular foliation, and those have singular data functions. The space $X$ has no zebra structure.}
Dividing by $2$ and multiplying by $\pi$ yields:
\begin{equation}
\label{eq:GB1}
\sum_{q \in \partial K} \pi (1-v_q) - \sum_{p \in \Sigma \cap K^\circ} \pi \alpha(p) = 2 \pi \chi(K).
\end{equation}
Now consider a single boundary component $\gamma$ of $K$ whose vertices are $q_i$ for $i \in \Z/n\Z$ written in increasing cyclic order as we travel around $\gamma$
with the region $K$ on the left. Given any $i$, choose a stellar homeomorphism $h_i:U_i \to \Pi_{\alpha(q_i)}$. By possibly shrinking $U_i$, we can assume that $h(K \cap U_i)$
is a sector $\sigma_i \subset \Pi_{\alpha(q_i)}$.
Then the angle of this sector coincides with the interior angle $\theta_{q_i}$, and the starting ray of the sector has the same slope as $\overline{q_i q_{i+1}}$ and the ending ray has the same slope as $\overline{q_{i-1} q_{i}}$.
Consider the union of these sectors $\sigma_i$ with the starting ray of $\sigma_i$ glued to the ending ray of $\sigma_{i+1}$. Note that the glued rays have the same slopes, so this gluing of sectors produces a copy of $\Pi_k$ where $\sum \theta_{q_i} = (k+2)\pi$. It follows that the total number of prongs $\sum v_{q_i}=k+2$ and so we have
\begin{equation}
\label{eq:GB2}
\sum_{i=0}^{n-1} \theta_{q_i} = \sum_{i=0}^{n-1} \pi v_{q_i}.
\end{equation}
Substituting \eqref{eq:GB2} into \eqref{eq:GB1} (for each boundary component) yields \eqref{eq:Gauss-Bonnet}.
\end{proof}

It is useful to observe the following result for future Gauss-Bonnet calculations:
\begin{proposition}
\label{loop contribution}
Let $\gamma$ be a closed curve in a zebra surface that is piecewise given by segments of leaves. Let $\{\theta_i: i=1, \ldots, n\}$ denote the measures of angles at the transitions between the segments of leaves, measured uniformly on one side of $\gamma$. Then the sum
$\sum_{i=1}^n \theta_i$ is an integer multiple of $\pi$.
\commb{Did you use that $\gamma$ is simple here? Wouldn't this work for any closed curve?} \compat{You are right. I modified it. Maybe I'll point out that there are some issues with exactly how to measure these angles if the path doubles back on itself, but the choice in this case has no effect on the conclusions, so I'm just ignoring the issue.}
\end{proposition}
\begin{proof}
Let $q_i$ denote the endpoints of the segments of leaves making up $\gamma$, with the angle $\theta_i$ being measured at $q_i$. Orient $\gamma$, and consider a vector traveling around $\gamma$ in the direction of the orientation and pointed along the curve on the interiors of the arcs making up $\gamma$, and turning at each $q_i$ between the arcs. Then, the unit vector turns by a signed angle equivalent to $\pi \pm \theta_i$ modulo $\pi \Z$ at $q_i$,
where the sign only depends on the side of $\gamma$ where measurements were made. Since when the vector travels completely around $\gamma$, it ends pointing in a direction with the same slope as its start, we have that $\sum_i (\pi \pm \theta_i)$ is an integer multiple of $\pi$. Since the signs are uniform, the conclusion follows.
\end{proof}

\section{Basic observations, definitions, and results}
\label{sect:basic}

\subsection{The pole-resolved universal cover}
\label{sect:maximal cover}
Let $S$, $\Sigma$, and $\alpha$ be as in \Cref{sect:formal}, but add the hypothesis that $S$ is connected. Define $\Sigma_{-1} = \alpha^{-1}(-1)$ to be the set of poles. The points in $\Sigma_{-1}$ are our only source of ``positive curvature.'' The goal here is to define a variant of the universal cover but with no ``positive curvature.'' The cover described here was also considered in \cite[\S 3]{Frankel}. \compat{Added this reference Jan 3, 2023.}

\commb{Still having trouble, so when you defined it in the preceding page there was something wrong with your definition? There is the possibility that it is not well-defined.} \compat{It was a definition potentially like `Let $n$ be the largest integer...' I adjusted this section to make it more formal.}

Let $S^\plus=S \setminus \Sigma_{-1}$. Similar to language used in the introduction, call a loop in $S^\plus$ {\em polar} if it is freely homotopic in $S^\plus$ to a simple loop bounding a disk in $S$ containing exactly one point in $\Sigma_{-1}$. Choose a basepoint $p_0 \in S^\plus$ and define
\begin{equation}
\label{eq:N}
N = \langle \gamma^2:~\text{$\gamma$ is polar}\rangle \subset \pi_1(S^\plus,p_0).
\end{equation}
Then, $N$ is a normal subgroup of $\pi_1(S^\plus,p_0)$, because being polar is a conjugacy invariant.

\begin{proposition}
\label{maximal cover is a disk}
There is a largest branched cover $\tilde S$ of $S$ which is at most doubly branched over each point in $\Sigma_{-1}$ and is unbranched over other points.
\commb{I don't understand this. Isn't there just one way to form a branched cover where there is a double branch at all the poles and no other branch points? I am not very good with basic topology so maybe this is not clear.} \compat{No. For example, the square pillowcase has both the torus and the plane as a branched cover doubly branched over the singularities. ``Double branched'' refers to the fact that the map is two-to-one locally in a neighborhood of a preimage.}
The restriction of the covering $\pi:\tilde S \to S$ to $\pi^{-1}(S^+) \to S^+$ is the normal cover of $S^+$ associated to $N \subset \pi_1(S^\plus,p_0)$.
The surface $\tilde S$ is homeomorphic to a disk, and the covering $\pi$ is doubly branched over every pole, i.e., the covering map is locally $2-1$ near any preimage of a pole.
\end{proposition}

We define $\tilde S$ to be the {\em pole-resolved universal cover (PRU cover)} of $S$. It follows from the result above that $\tilde S$ is a branched cover of the usual universal cover. If $\Sigma_{-1} = \emptyset$, then the PRU cover coincides with the universal cover.

\begin{proof}[Proof of \Cref{maximal cover is a disk}]
Let $\tilde S^\plus$ denote the cover of $S^\plus$ associated to the subgroup $N$, as in covering space theory.
We claim that $\tilde S^\plus$ is $\tilde S$ with the preimages of points in $\Sigma_{-1}$ removed.

Suppose $\tilde T$ is some other branched cover of $S$ which is only branched over points in $\Sigma_{-1}$ and at most doubly branched over these points. Puncturing $\tilde T$ at preimages of $\Sigma_{-1}$, we obtain a covering $\tilde T^\plus$ of $S^\plus$, which is associated to a subgroup $G \subset \pi_1(S^\plus,p_0)$. To show that $\tilde S^\plus$ covers $\tilde T^\plus$, it suffices to prove that $N \subset G$. To this end, let $\gamma:[0,1] \to S^\plus$ be a polar loop with $\gamma(0)=\gamma(1)=p_0$.  Then there is a homotopy $h_s:[0,1] \to S^\plus$ such that $h_0=\gamma$ and $h_1$ is a loop in $S$ bounding a disk enclosing exactly one point in $\Sigma_{-1}$. Let $\eta \in \pi_1(S^\plus,p_0)$ be the loop which follows $s \mapsto h_s(0)$ for $s \in [0,1]$, then follows $t \mapsto h_1(t)$ for $t \in [0,1]$ and returns to the basepoint following $s \mapsto h_s(0)$ parameterized backward from $s=1$ to $s=0$. It is not hard to show that $\eta$ is homotopic rel endpoints to $\gamma$ in $S^\plus$, thus determining the same element of $\pi_1(S^\plus,p_0)$. Since $\tilde T$ is at most doubly branched over points in $\Sigma_{-1}$, the square of the element of $\pi_1(S^\plus,p_0)$ associated to the common class of $\gamma$ and $\eta$ lies in $G$. Since $\gamma$ was an arbitrary polar curve $N \subset G$ as claimed, proving that $\tilde S^\plus$ is the largest such cover. This argument also shows that $\tilde S^\plus$ is locally at most a double cover in neighborhoods of $\Sigma_{-1}$, and we can fill in these points to form the branched cover $\tilde S$.

It remains to show that the cover $\tilde S$ is actually doubly branched over points in $\Sigma_{-1}$ and that $\tilde S$ is a topological disk. To see that $\tilde S \to S$ is doubly branched over some point $p \in \Sigma_{-1}$, it suffices to find a branched cover $\tilde T$ as above which is doubly branched over $p$. To see that $\tilde S$ is a disk, it suffices to find a $\tilde T$ whose universal cover is a disk. (If $\tilde T$ satisfies the double branching condition, the so does its universal cover.) If $S$ has positive genus, this is clear since its universal cover is a disk and given any $p \in \Sigma_{-1}$, we can find a linear map
$H_1(S^\plus; \Z/2\Z) \to \Z/2\Z$
sending a loop wrapping once around $p$ to $1$. Covering space theory associates this linear map to a double cover of $S^\plus$ which is doubly branched over $p$. If $S$ is a sphere, then by the \gaussbonnet there are at least four points in $\Sigma_{-1}$. Choosing four points including our favorite point $p$, we can puncture only at those four points and define a linear map as before such that the homology classes of the loops around each of these points are sent to one. The associated double cover is a torus which is doubly branched over these four points as desired. Since this torus has a disk as its universal cover, this also proves that $\tilde S$ is a disk in this case.
\end{proof}

Now suppose $\{\sF_m\}_{m \in \hat \R}$ is a family of foliations determining a zebra structure on $S$. Let $\tilde \sF_m$ be the singular foliation of $\tilde S$ whose leaves are lifts of leaves of $\sF_m$. We call $\{\tilde \sF_m\}_{m \in \hat \R}$ the {\em lifted family of foliations}. These foliations have common singularity data $\tilde \alpha$, where if $\tilde p \in \tilde S$ projects to $p \in S$, we have
$$\tilde \alpha(\tilde p)=\alpha(p) \text{ if $\alpha(p) \geq 0$} \quad \text{and} \quad
\tilde \alpha(\tilde p)=0 \text{ if $\alpha(p) = -1$.}
$$

\subsection{Zebra planes}

A {\em zebra plane} is a zebra structure on the open topological disk such that the singular data function $\alpha$ is non-negative. From the discussion above,
the PRU cover of a zebra surface is always a zebra plane.

\subsection{Polygons}
\label{sect:polygons}

A {\em polygon} $p_0 p_1 \ldots p_{n-1}$ in a zebra plane is a topological disk bounded by a simple closed curve of the form $\overline{p_0 p_1} \cup \overline{p_1 p_2} \cup \ldots \cup \overline{p_{n-2} p_{n-1}}\cup \overline{p_{n-1} p_0}$,
where each edge $\overline{p_i p_{i+1}}$ with $i \in \Z / n\Z$ is a segment of a leaf from a directional foliation. The Jordan Curve Theorem guarantees that this curve bounds a topological disk, and we'll use the counterclockwise ordering when describing polygons so that the polygon is on the left as we move from $p_i$ to $p_{i+1}$ along $\overline{p_i p_{i+1}}$. We call the $p_i$ {\em vertices}. The interior angle of $P$ at $p_i$ is $\measuredangle p_{i+1} p_i p_{i-1}$. The {\em external angle} is $\measuredangle p_{i-1} p_i p_{i+1}$. The sum of the interior and exterior angles at $p_i$ is $\alpha(p_i)\pi+2\pi$, the total angle at $p_i$.

We'll call a vertex {\em straight} if the interior angle equals $\pi$. As we are typically interested in the internal geometry of a polygon, we will typically ignore straight vertices. So, a {\em $k$-gon} is a polygon with $k$ vertices that are not straight. A {\em triangle} is a $3$-gon, and we'll use other similarly obvious terminology coming from plane figures to describe objects in $\tilde S$.

Suppose $S$ is a zebra surface and $\tilde S$ is its PRU cover. Let $\pi:\tilde S \to S$ denote the covering map.
If $\tilde P$ is a polygon in $\tilde S$ such that the restriction $\pi|_{\tilde P}:\tilde P \to S$ is injective on the interior of $\tilde P$,
then we'll call the restriction $\pi|_{\tilde P}$ a polygon $P$ in $S$. These are maps rather than subsets of $S$, because it gives the right notion of the interior of $P$ (the image of the interior) and boundary (the further restriction to the boundary of $\tilde P$). These notions are confusing even when considering the square in the center of the usual square torus (in that the closed square is the whole torus, but you still want it to have boundary for instance).

The next proposition shows that interior angles and slopes of edges of a zebra triangle behave as they do in plane geometry.

\begin{proposition}
\label{triangle1}
Triangles contain no singularities in their interiors and the sum of the interior angles of a triangle in a zebra plane is always $\pi$.
Let $(m_0, m_1, m_2) \in \hat \R^3$ be a triple of slopes of edges of a triangle, listed in counterclockwise order as we travel around the boundary of the triangle. Then the triple of slopes is distinct and appear in decreasing cyclic order on $\hat \R$.
\end{proposition}

\begin{proof}
Let $\theta_i$ for $i=0, \ldots, 2$ be the interior angles of a triangle $T$. Since the curve is simple, we have $\theta_i>0$ for all $i$. By the \gaussbonnet, we have
$$\theta_0+\theta_1+\theta_2=\pi- \sum_{p \in T^\plus} \pi \alpha(p) > 0,$$
where the sum is taken over all interior singularities of $T$. Since $\alpha(p) \geq 1$ at all singularities in a zebra plane, there must be no singularities in $T^\plus$ and thus $\theta_0+\theta_1+\theta_2=\pi$ as desired. It follows the slopes of the sides of a triangle must be the same as the slopes of the sides of a Euclidean triangle. Therefore, the slopes appear in decreasing cyclic order.
\end{proof}

\begin{proposition}
\label{quadrilateral}
The sum of the interior angles of a quadrilateral $Q$ in a zebra plane is either $\pi$ or $2 \pi$. If the sum is $2 \pi$, then $Q$ contains no singularities in its interior. If the sum is $\pi$, then $Q$ contains a singularity $q$ in its interior with $\alpha(q)=1$.
\end{proposition}
\begin{proof}
From the \gaussbonnet, the interior angles satisfy
$$\theta_0+\theta_1+\theta_2+\theta_3=2\pi- \sum_{p \in T^\plus} \pi \alpha(p) > 0,$$
immediately giving the result.
\end{proof}

We also have the following more general result:

\begin{proposition}
\label{ngons}
Any closed $n$-gon $P$ in a zebra plane must have at least three interior angles whose measure is less than $\pi$.
\end{proposition}
\begin{proof}
Again by the \gaussbonnet,
$$\sum_{q \in \partial P} \left(\pi - \theta_q\right) = 2 \pi + \sum_{p \in K^\circ} \pi \alpha(p) \geq 2 \pi.$$
Since for each $q$, we have $\pi-\theta_q < \pi$, there must be at least three $q \in \partial P$ for which $\pi-\theta_q > 0$.
\end{proof}

\subsection{Arcs of trails}
Let $Z$ be a zebra plane and let $\{\tilde \sF_m\}$ denote the associated foliations.

Let $p \in Z$ be a singularity and $\overline{pq}$ and $\overline{pr}$ be segments of leaves. We say these arcs satisfy the {\em angle condition} at $p$ if
$$\measuredangle qpr \geq \pi \quad \text{and} \quad \measuredangle rpq \geq \pi.$$
\commb{Here I would end the formula and add another line "that is, the angle made by $\gamma$ at $\gamma(t)$ is at least pi on both sides."} \compat{I added this to the second bullet below. (I wasn't clear where you suggested adding it.)}

A {\em parameterized arc of a trail} on $Z$ is a parameterized curve $\gamma:I \to Z$, where $I \subset \R$ is a nondegenerate interval,
such that if $t$ is any point in the interior of $I$ then the following statements are satisfied:
\begin{itemize}
\item If $\gamma(t)$ is not singular, then in a neighborhood of $t$, $\gamma(t)$ moves injectively along a leaf.
\item If $\gamma(t)$ is singular, then the two arcs made at $\gamma(t)$ formed by increasing and decreasing $t$ satisfy the angle condition.
That is, the angle made by $\gamma$ at $\gamma(t)$ is at least $\pi$ on both sides.
\end{itemize}
Suppose $\gamma:I \to Z$ and $\eta:J \to Z$ are parameterized arcs of trails. We say that $\gamma$ is a {\em subarc} of $\eta$ if there is an orientation-preserving continuous injective map $\phi:I \to J$ such that $\gamma = \eta \circ \phi$. We say $\gamma$ is a {\em proper subarc} of $\eta$ if the map $\phi$ is not surjective.

The condition that two arcs of trails are each subarcs of the other (i.e., reparameterizations of one another) is an equivalence relation, and we'll call an equivalence class an {\em arc of a trail}. The notion of subarc induces a well-defined partial ordering on arcs of trails.

An {\em arc of a trail} on $S$ is the image of an arc of a trail on the PRU cover $\tilde S$ under the covering $\tilde S \to S$. We likewise use the cover to define the other notions above.

\begin{proposition}[No monogons]
\label{no monogons}
If $\tau:I \to Z$ is a parameterized arc of a trail in a zebra plane, then $\tau$ is injective.
\end{proposition}
\begin{proof}
Suppose $\tau$ is not injective. Then we can assume without loss of generality that $I=[a,b]$ and $\tau(a)=\tau(b)$. Observe that because $I$ is closed and bounded and $\tau$ is continuous and locally injective, the set $J$ of all $t \in [a,b]$ for which there is a $t' \in [a,t)$  such that $\tau(t)=\tau(t')$ is closed. Let $x = \inf J$. Then there is an $x' \in [a,x)$ such that $\tau(x)=\tau(x')$. Restricting $\tau$ to $[x',x]$ yields a simple closed curve, which by the Jordan Curve Theorem bounds a disk $D$. Observe that $D$ must be a polygon, and since $\tau$ is an arc of a trail, all its interior angles are larger than $\pi$ except possibly at $\tau(x')=\tau(x)$. This violates \Cref{ngons}.
\end{proof}

\begin{proposition}[No bigons]
\label{no bigons}
Suppose $\tau_1$ and $\tau_2$ are arcs of trails in a zebra plane $Z$. Then $\tau_1 \cap \tau_2$ is the empty set, is a single point, or
is a common subarc (possibly with different induced orientations).
\end{proposition}
\begin{proof}
Assume $\tau_1$ and $\tau_2$ intersect. Consider each $\tau_i$ to be parameterized by functions with the same name, $\tau_i:I_i \to Z$. The statement can be observed to be true as long as $\tau_1^{-1}(\tau_2)$ is connected. If $\tau_1^{-1}(\tau_2)$ is disconnected, we can let $J_1 \subset I_1 \setminus \tau_1^{-1}(\tau_2)$ be a bounded open interval such that both boundary points lie in $\tau_1^{-1}(\tau_2)$. Since $\tau_2$ is injective, there is also a unique interval $J_2 \subset I_2$ such that $\tau_2(\partial J_2)=\tau_1(\partial J_1)$. By construction $\tau_1(J_1) \cup \tau_2(J_2)$ forms a simple closed curve, which again by the Jordan Curve Theorem bounds a disk $D$. This time the only possible interior angles less than $\pi$ are the two points in $\tau_1(\partial J_1)$, again violating \Cref{ngons}.
\end{proof}

\begin{proposition}
\label{no returning trails}
Let $P$ be a polygon in a zebra plane all of whose exterior angles are at least $\pi$. Then there is no parameterized arc of a trail $\tau:[0,1] \to Z$ such that $\tau(0),\tau(1) \in \partial P$ and $\tau\big((0,1)\big) \cap P = \emptyset$.
\end{proposition}
\begin{proof}
If this were the case, the union of $\tau$ and an arc of $P$ bound a polygon $Q$ whose interior is in the complement of $P$. Points in $\partial Q \cap \tau$ have interior angles at least $\pi$ since $\tau$ is an arc of a trail, and points on $\partial Q \cap P$ that are not endpoints of $\tau$ have interior angles for $Q$ which are the same as the exterior angles for $P$. Thus, $Q$ can have at most two interior angles less than $\pi$ (namely, the endpoints of $\tau$). Again this violates \Cref{ngons}.
\end{proof}

\subsection{Trapezoids}

We speak of two segments of leaves in a zebra plane $Z$ as being {\em parallel} if they are segments of leaves coming from the same foliation $\tilde \sF_m$.
A {\em trapezoid} in $Z$ is a $4$-gon with a pair of opposite edges that are parallel. A parallelogram is a $4$-gon such that both opposite pairs of edges are parallel.

\begin{proposition}
\label{trapezoid observation}
Let $T$ be a trapezoid in $Z$. Then the angles of $T$ add to $2 \pi$ and there are no singularities in the interior of $T$.
\end{proposition}
\begin{proof}
Suppose our trapezoid is $pqrs$, with vertices ordered counterclockwise and with $\overline{pq}$ parallel to $\overline{rs}$. Then
$\measuredangle rqp + \measuredangle srq=\pi$. The other pair of angles add to $\pi$ as well, so the sum of all the angles is $2\pi$.
Then \Cref{quadrilateral} tells us that $T$ has no singular points in its interior.
\end{proof}

We say a trail has {\em constant slope} if there is an $m$ such that every segment of a leaf contained in the trail has slope $m$.
\compat{This definition and the proposition below are new. They handle the issue of showing that the trapezoid construction produces an embedded rather than immersed polygon.}

\begin{proposition}
\label{building a trapezoid}
Let $\overline{pq}, \overline{rs} \subset Z$ be arcs of trails of the same constant slope $m$.
Suppose that $\overline{qr}$ and $\overline{sp}$ are disjoint segments of leaves whose slopes are not $m$.
Then the curve
$\overline{pq} \cup \overline{qr} \cup \overline{rs} \cup \overline{sp}$ is a $4$-gon (trapezoid) whose vertices are $p$, $q$, $r$ and $s$. In particular, all interior angles at singularities in the interior of segments $\overline{pq}$ and $\overline{rs}$ are $\pi$, and the interior angles at $p$, $q$, $r$ and $s$ are each less than $\pi$.
\end{proposition}
\begin{proof}
The curve $\overline{pq} \cup \overline{qr} \cup \overline{rs} \cup \overline{sp}$ is a closed curve. If we can show it is simple, then by the Jordan curve theorem it bounds a disk, which is our polygon. Suppose it has $n$ sides.
As in Euclidean geometry, the \gaussbonnet guarantees that the sum of the interior angles is $(n-2)\pi$. Because $\overline{pq}$ and $\overline{rs}$ are parallel, we have
$$\measuredangle q + \measuredangle r=a \pi
\quad \text{and} \quad
\measuredangle s + \measuredangle p=b \pi \quad
\text{for some integers $a, b \geq 1$.}$$
Suppose $t_1, \ldots, t_{n-4}$ are the singularities in the interiors of $\overline{pq}$ and $\overline{rs}$. Then $\measuredangle t_i=k_i \pi$ for some integer $k_i \geq 1$, because both these arcs of trains have constant slope. Thus the sum of the interior angles is $(a+b+\sum_{i=1}^{n-4} k_i)\pi$. The smallest this sum can be is $(n-2)\pi$ and so we must have $a=b=k_1=\ldots=k_{n-4}=1$.

Now we will argue that $\overline{pq} \cup \overline{qr} \cup \overline{rs} \cup \overline{sp}$ is simple.
We know that the arcs $\overline{pq}$, $\overline{qr}$, $\overline{rs}$, and $\overline{sp}$ are simple by \Cref{no monogons}. Using the slope conditions, we see that \Cref{no bigons} guarantees that the intersection of adjacent edges (e.g., $\overline{pq} \cap \overline{qr}$) consists only of the common vertex. Therefore, the only way that the curve can fail to be simple is if opposite edges intersect in their interiors. By hypothesis $\overline{rq} \cap \overline{ps}=\emptyset$. We claim that $\overline{pq} \cap \overline{rs} = \emptyset$. Suppose to the contrary that $\overline{pq} \cap \overline{rs} \neq \emptyset$. Then by \Cref{no bigons}, they intersect in either a single point or a common compact subarc. Let $x \in \overline{pq}$ be the point closest to $q$ in $\overline{pq} \cap \overline{rs}$
(where ``closest'' is measured in a parameterization of $\overline{pq}$). Then the path $\gamma=\overline{xq} \cup \overline{qr} \cup \overline{rx}$ is simple. Observe that $\gamma$ has at most two points at which there are angles whose measure is less than $\pi$ (on either side of the curve), namely the points $q$ and $r$. (It could be that $x$ coincides with either $q$ or $r$, but otherwise because the arcs of $\gamma$ on both sides of $x$ are parallel, the angle at $x$ must be at least $\pi$.) This contradicts \Cref{ngons}, proving our claim and completing the proof.
\end{proof}

\begin{lemma}[Trapezoid construction lemma]
\label{trapezoid construction}
Let $\overline{pq} \subset Z$ be an arc of a trail, where the angle measured on the left side as we move from $p$ to $q$ at any singularities in the interior of $\overline{pq}$ is $\pi$. Let $U \subset Z$ be an open set containing $\overline{pq}$. Let $\overline{ps}$ and $\overline{qr}$
be additional segments such that $0<\measuredangle qps < \pi$ and
$0<\measuredangle rqp < \pi$. Then there exist $s' \in \overline{ps} \setminus \{p\}$ and $r' \in \overline{qr} \setminus \{q\}$ and a segment $\overline{s'r'}$ parallel to $\overline{pq}$ forming a trapezoid $pqr's'$ contained in $U$. Furthermore, we can construct the trapezoid in such a way so that leaves parallel to $\overline{pq}$ passing through the interior of the trapezoid pass through the interiors of $\overline{ps'}$ and $\overline{qr'}$.
\end{lemma}

See \Cref{fig:trapezoid construction} for an illustration of the trapezoid construction.

\begin{figure}[htb]
\centering
\includegraphics[width=5in]{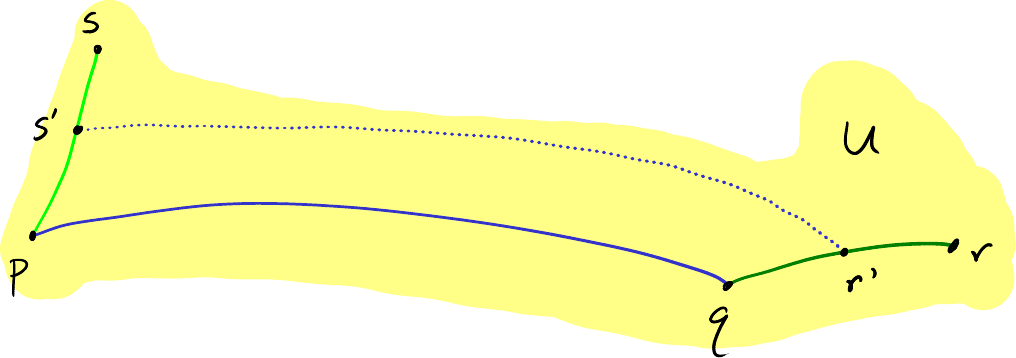}
\caption{Depiction of the trapezoid produced by \Cref{trapezoid construction}.}
\label{fig:trapezoid construction}
\end{figure}

We briefly discuss the idea of the proof before giving a detailed proof. By compactness, we can produce a finite covering of $\overline{pq}$ by foliation charts for the foliation parallel to $\overline{pq}$.
We can find stellar neighborhoods of $p$ and $q$ contained in the charts containing $p$ and $q$. The angle conditions at $p$ and $q$ and the fact that there are only finitely many charts can be used to guarantee that we can find a leaf parallel to $\overline{pq}$ that joins a point of $\overline{ps}$ in the stellar neighborhood of $p$ to a point of $\overline{qr}$ in the stellar neighborhood of $q$.
A technical point is that the charts might contain a singularity. \compat{Changed flow boxes to chart here following comments of Barak.}

\begin{proof}
Assume without loss of generality that $\overline{pq}$ is horizontal.
For each point $x \in \overline{pq}$ there is a foliation chart from a neighborhood of $x$ to $(-1,1) \times (-1,1)$ (or if $x$ is singular, a singular chart to $\Pi_n$ carrying $x$ to $\0$) for the horizontal foliation that intersects $\overline{pq}$ in an interval. By possibly shrinking the chart, we can assume that this neighborhood of $x$ is contained in $U$. For this proof, all our foliation charts lie in $U$. \compat{I added mention of $U$ in the previous two sentences. Thanks Barak.} We can normalize these charts so that the portion of $\overline{pq}$ in the domain maps to a subset of $(-1,1) \times \{0\}$ and so that the portion of the chart to the left of $\overline{pq}$ when moving from $p$ to $q$ is mapped into $(-1,1) \times (0,1)$. In case $x$ is singular, since the angle made on the left side of $\overline{pq}$ at $x$ is $\pi$, the portion of $U$ in the chart to the left of $\overline{pq}$ is mapped into a half-plane of $\Pi_n$. We can put coordinates on this closed half-plane of the form $\R \times [0,+\infty)$, and like in the nonsingular case we can restrict the chart and rescale it so that the portion of $\overline{pq}$ in the domain can be mapped to $(-1,1) \times \{0\}$ and the portion to the left of $\overline{pq}$ in the domain can be mapped to $(-1,1) \times (0,1)$ in these coordinates. Using compactness of $\overline{pq}$ we can produce a minimal finite subcovering of $\overline{pq}$. Restrict these charts to the points whose images have nonnegative $y$-coordinates. We can order this collection of restricted charts $\phi_i:B_i \to (-1,1) \times [0,1)$ for $i=0, \ldots, n$ such that $B_i \cap \overline{pq}$ gives a sequence of open subintervals of $\overline{pq}$ such that $p \in B_0$, $q \in B_n$ and $\overline{pq} \cap B_i \cap B_j \neq \emptyset$ if and only if $|i-j|\leq 1$. \compat{In this paragraph, I changed the box coordinates so that the $x$-coordinates are now $(-1,1)$, and spelled out the singular charts more carefully. (The previous version erroneously claimed that that the only singularities would be at $p$ or $q$.) Because of the singular charts, I now restrict the charts to half-planes.}

We claim that we can define intervals $J_i \subset [0,1)$ which are open as subsets of $[0,1)$
together with continuous strictly increasing functions $\psi_i:J_i \to [0,1)$ such that for all $y \in J_i$, the leaf $\phi_0^{-1}\big((-1,1) \times \{y\}\big)$ can be continued across $B_1$, $B_2$, \ldots, $B_i$ as
\begin{equation}
\bigcup_{k=0}^i \phi_k^{-1}\big((-1,1) \times \{\psi_k(y)\}\big).
\label{eq:continue leaves}
\end{equation}
We do this by induction.
Define $J_0=[0,1)$ and $\phi_0:J_0 \to [0,1)$ to be the identity map. Now assuming $J_i$ and $\phi_i$ are defined, we can let $U_i$ be the connected component of $B_i \cap B_{i+1}$ containing $\overline{pq} \cap B_i \cap B_{i+1}$ and there is a continuous strictly increasing function
$$h_i: \pi_y \circ \phi_i(U_i) \to \pi_y \circ \phi_{i+1}(U_i), \quad \text{where} \quad \pi_y(x,y)=y$$
coming from the transition between the charts such that
$\phi_i^{-1}\big((-1,1) \times \{y\}\big)$ continues across $U_{i}$ as $\phi_{i+1}^{-1}\big((-1,1) \times \{h_i(y)\}\big)$. Then by defining $J_{i+1}=J_i \cap \psi_i^{-1}(U_i)$ and $\psi_{i+1}=h_i \circ \psi_i$ we see that \eqref{eq:continue leaves} is satisfied, completing the induction and giving definitions for $J_n$ and $\psi_0, \ldots, \psi_n$ satisfying \eqref{eq:continue leaves}. This proves the claim. \compat{I made minimal changes to this paragraph. Adjusting $x$ coordinates and redefining the $J_i$ so that they are subsets of $[0,1)$ since we don't need negative $y$-coordinates.}

Let $N$ be a stellar neighborhood at $p$ and $h:N \to \Pi_{n'}$ be the stellar homeomorphism. Then the connected component of $N \cap \overline{pq}$ containing $p$ maps under $h$ to the closure of a horizontal ray $r_0 \subset \Pi_n$. Let $H \subset \Pi_n$ be the closed half-space consisting of $\0$ and all rays $r \subset \Pi_{n'}$ such that the counterclockwise angle from $r_0$ to $r$ lies in $[0, \pi]$. Observe that $p \in B_0$. By restricting to a smaller neighborhood $N$ and rescaling the stellar homeomorphism $h$, we can assume that $h^{-1}(H) \subset B_0$. Let $\ell$ be the connected component of $(\overline{ps} \setminus \{s\}) \cap N$ containing $p$. Then $h(\ell)$ is a ray contained in $H$. Since $h^{-1}(H) \subset B_0$, we see that $\ell \subset B_0$. Since $\ell$ is not horizontal and contains $p$, the horizontal leaves in $B_0$ meet $\ell$ transversely and so $\pi_y \circ \phi_0(\ell)=[0,b_0)$ for some $b_0>0$.
Similarly, there is a half-open arc $\ell' \subset \overline{qr} \cap B_n$ containing $q$ such that
interval $\pi_y \circ \phi_n(\ell')=[0,b_n)$ for some $b_n>0$. Since $J_n$ is an open subset of $[0,1)$ containing $0$ and $\psi_n$ is strictly increasing and preserves zero, we can find $y \in (0,b_0)$ so that $\psi_n(y) \in (0, b_n)$. Then we have a segment of a leaf that cuts across $B_0$, \ldots, $B_n$ as described in \eqref{eq:continue leaves}. Let $s'$ be the place this leaf crosses $\overline{ps}$ and $r'$ be the place this leaf crosses $\overline{qr}$. We conclude that $pqr's'$ is a trapezoid using \Cref{building a trapezoid}. Furthermore, if $0<y'<y$, then the leaf as constructed in \eqref{eq:continue leaves} (with $y'$ replacing $y$) cuts across this trapezoid passing through the interiors of edges $\overline{ps'}$ and $\overline{qr'}$ as desired. \compat{I expanded the first sentence. It is technical :(.}
\end{proof}

\begin{corollary}[Generalized rectangles]
\label{generalized rectangles}
Let $p \in Z$ and let $U \subset Z$ be an open set containing $p$. Then there is a $2\alpha(p)+4$-gon $P \subset U$ with alternating horizontal and vertical sides and all interior angles of $\frac{\pi}{2}$ such that $p \in P^\circ$. Furthermore, there is a bijection between the horizontal prongs at $p$ and the vertical edges of $P$ such that each horizontal prong has a realization as a horizontal path joining the corresponding edge to $p$. These $\alpha(p)+2$ paths cut $P$ into $\alpha(p)+2$ rectangles, each of which has $p$ in the interior of an edge formed by two of the prong realizations. Every horizontal leaf that enters the interior of $P$ either crosses through the interior of one of the rectangles joining opposite vertical sides of the rectangle, or follows one of the $\alpha(p)+2$ horizontal paths and terminates at $p$.
\end{corollary}

We call the polygon $P$ a {\em generalized rectangle} because of the alternating horizontal and vertical sides and angles of $\frac{\pi}{2}$. An example is depicted in \Cref{fig:generalized rectangle}.

\begin{figure}[htb]
\centering
\includegraphics[height=1.5in]{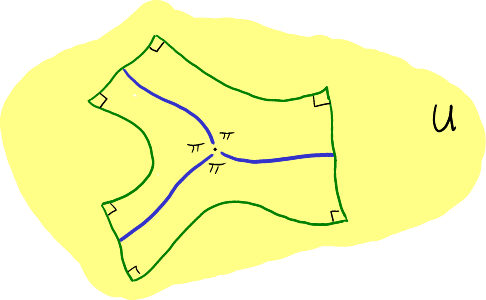}
\caption{A generalized rectangle surrounding a point $p$ where $\alpha(p)=1$.}
\label{fig:generalized rectangle}
\end{figure}

\begin{proof}
Construct horizontal segments of leaves with endpoints at $p$ realizing every prong. By possibly shortening them, we can assume they are pairwise disjoint. There are $n=\alpha(p)+2$ such arcs. Denote them $\{\beta_i~:~i \in \Z/n\Z\}$ and order them counterclockwise.
Then the counterclockwise angle from $\beta_i$ to $\beta_j$ at $p$ is $\pi$ if and only if $j=i+1$.
Let $e_i$ denote the path $\beta_i \cup \{p\} \cup \beta_{i+1}$. Using \Cref{trapezoid construction}, for each $i$ we can produce a rectangle $R_i$ one of whose edges is $e_i$ such that the counterclockwise angle from $\beta_i$ to $\beta_{i+1}$ at $p$ is interior to $R_i$.

Define $Q=(\bigsqcup R_i)/\sim$ where $\sim$ identifies the corresponding points in the common subarcs $\{p\} \cup \beta_i \subset R_{i-1} \cap R_i$. Observe that $Q$ is homeomorphic to a closed disk.
Let $\pi:Q \to Z$ be the natural map induced by the inclusions of $R_i$ into $Z$.
Observe that $\pi$ restricted to each $R_i$ is injective. By considering horizontal foliation charts at $p$ and at points of each $\beta_i$, we can see that $\pi$ is locally injective.
If we knew $\pi$ were globally injective, then $P=\pi(Q)$ will be a polygon satisfying the statements in the corollary, with the statement about the horizontal leaves following from \Cref{trapezoid construction}.

Now assume $\pi:Q \to Z$ is not injective. We will alter the construction to produce a smaller $Q' \subset Q$ with the same properties such that $\pi$ restricted to $Q'$ is injective. This will complete the proof.

We will actually construct a sequence of subsets $Q^n$ playing the role of $Q'$. For each $i$, construct a sequence of rectangles $R_i^n \subset \tilde R_i$ such that are nested ($\tilde R_i^{n+1} \subset \tilde R_i^n$ for all $n$) and satisfy $\bigcap_{n=0}^\infty R_i^n=e_i$. These $R^i_n$ can be produced from $R_i$ by cutting along a horizontal leaf through the interior.
We define $Q^n \subset Q$ to be the union over $i$ of the $R_i^n$.
If the restriction of $\pi$ to $Q^n$ is injective, then we can define $P=\pi(Q^n)$ to be our generalized rectangle, completing the proof.
Now suppose to the contrary that the restriction of $\pi$ to $Q^n$ is not injective for any $n$.
Then for each $n$, we can find distinct points $x_n,y_n \in Q^n$ such that $\pi(x_n)=\pi(y_n)$. By passing to a subsequence, we can assume that $\lim x_n=x$ and $\lim y_n=y$ both exist in $Q$. Then $\pi(x)=\pi(y)$ by continuity of $\pi$.
Observe that $\bigcap_n Q^n=\{p\} \cup \bigcup_i \beta_i$, and each $Q^n$ is closed, so we have $x,y \in \{p\} \cup \bigcup_i \beta_i$.
But $\pi$ restricted to $\{p\} \cup \bigcup_i \beta_i$ is injective because the $\beta_i$ were constructed to be pairwise disjoint realizations of prongs of $p$. Thus we must actually have $x=y$. But then the facts that $x_n \neq y_n$, $\pi(x_n)=\pi(y_n)$ and $\lim x_n=\lim y_n$ violates the local injectivity of $\pi$.
\end{proof}

\subsection{Properness of leaves}

Let $\ell$ be a separatrix. Then it has a parameterization of the form $\ell : [0,+\infty) \to Z$ where $\ell(0)$ is its singular endpoint. Bi-infinite leaves have a parameterizations of the form $\ell:\R \to Z$. Recall that a {\em proper map} between topological spaces is one for which preimages of compact subsets are compact.

\begin{proposition}
\label{leaves are proper maps}
A separatrix or bi-infinite leaf in a zebra plane that is parameterized as above is a proper map.
\end{proposition}

\begin{remark}
Singular foliations of the disk where $\alpha$ is allowed to take the value $-1$ can have separatrices and bi-infinite leaves that fail to be proper in the above sense even if there are only finitely many singularities \cite{Rosenberg83}. This phenomenon also occurs with periodically arranged singularities where $\alpha=-1$ \cite{Panov}. Avoiding this phenomenon is a reason for requiring $\alpha$ to be nonnegative in a zebra plane and for our definition of the PRU cover.
\commb{Have we ever discussed this reference? Is it related to $\pi$ cylinders? You may recall my request to put in a discussion about $\pi$ cylinders and things that can go wrong, if there is a connection we might mention it here.} \compat{My apologies: I think we haven't discussed this, though it would be good to. I added the reference to the Bibliography folder in dropbox. I don't think there is any connection to $\pi$ cylinders. I'm actually on the fence about whether we include it. I think it justifies the definition of zebra plane and the PRU cover- the universal cover might be hard to understand when there are poles. I was actually on the fence over whether this should be included or not.} \compat{Ferran said he supports including this.} \commf{Cite Panov also.} \compat{Added!}
\end{remark}

\begin{proof}[Proof of \Cref{leaves are proper maps}]
Let $\ell$ be as above, and assume without loss of generality that its slope is horizontal. If $\ell$ is not proper, there is a compact set $K \subset Z$ whose preimage is not compact. Since its preimage is necessarily closed, the preimage must not be bounded. So, we can find a sequence $t_n$ in the domain of $\ell$ such that each $\ell(t_n) \in K$ and $|t_n|\to +\infty$. Then by passing to a subsequence, we can assume that $\ell(t_n)$ converges to some point $p \in K$. Using \Cref{generalized rectangles}, we can produce a generalized rectangle $P$ containing $p$ in its interior. Therefore there are infinitely many $\ell(t_n) \in P$. Fix such an $n$. Since $\ell$ is horizontal and doesn't approach $p$ as $|t| \to +\infty$, \Cref{generalized rectangles} guarantees that the portion of the horizontal leaf through $\ell(t_n)$ must enter, cut across one of the rectangles making up $P$ and then exit. In particular the connected component of $\ell^{-1}(P)$ containing $t_n$ is a closed and bounded interval $I$. But then by hypothesis there is a $t_m$ in our sequence such that $\ell(t_m) \in P$ and $t_m \not \in I$. This contradicts \Cref{no returning trails}, which tells us that $\ell$ cannot later return to $P$. \compat{Changes here from using parallelogram covers to generalized rectangles. Nov 4.}
\end{proof}

\begin{corollary}
\label{trapezoid foliation}
Let $T=pqrs$ be a trapezoid with $\overline{pq}$ and $\overline{rs}$ parallel and of slope $m$. Then the restriction of $\sF_m$ to the interior of $T$ consists of segments of leaves joining $\overline{ps}$ to $\overline{qr}$. \compat{I added this on Dec 5th, and moved it to its current position on Dec 22, 2022.}
\end{corollary}
\begin{proof}
Consider a parameterized leaf $\ell:(a_-, a_+) \to Z$ of slope $m$ that intersects the interior $T^\circ$. We claim that $\ell(t)$ must exit $T$ or approach a point in $\partial T$  as $t$ approaches either endpoint.
If $s$ is a sign and $\lim_{t \to a_s} \ell(t)$ is a singularity, this follows from the fact that $T^\circ$ has no singularities by \Cref{trapezoid observation}. On the other hand, if
$\lim_{t \to a_s} \ell(t)$ is not a singularity, then $\ell$ must be a bi-infinite leaf or must be a separatrix with the limit to the other endpoint $\lim_{t \to a_{-s}} \ell(t)$ a singularity. Then, \Cref{leaves are proper maps} guarantees that this parameterization can be made proper. Setting $t_s$ to be the closest element of $(a_-, a_+)$ to $a_s$ such that $\ell(t_s) \in T$, we see $\ell$ exits $T$ at $\ell(t_s)$ and never returns.
\compat{I made some improvements here in the evening of Dec 22, 2022.}

Now consider a maximal segment of a leaf of $\sF_m$ in $T^\circ$. From the previous paragraph, traversing such a maximal segment in either direction approaches a point in $\partial T$. The segment can't approach a point in $\overline{pq}$ or in $\overline{rs}$ because these are trails of slope $m$ with interior angles of $\pi$; see \Cref{building a trapezoid}. This means there are no prongs of slope $m$ approaching points in $\overline{pq}$ or $\overline{rs}$. Now observe that the two points approached in the two different directions can't lie on the same edge by \Cref{no bigons}. Thus each maximal segment must join $\overline{ps}$ to $\overline{qr}$ as claimed.
\end{proof}

\subsection{Trails}

A {\em trail} is an arc of a trail which is maximal with respect to the subarc partial order. The following result tells us that trails exist, and every arc of a trail can be extended to a trail:

\begin{theorem}
\label{trails exist}
If $\gamma$ is an arc of a trail in a zebra surface, then $\gamma$ is a subarc of a trail.
\end{theorem}

\compat{I made some minor improvements here on Dec 22, 2022.}
Let $I \subset \R$ be an interval, $\bar I \subset \R \cup \{\pm \infty\}$ be its closure, and $\gamma:I \to Z$ be a parameterized arc of a trail.
Let $\ell$ be a leaf which, since $Z$ is a zebra surface, is not closed and thus is homeomorphic to an open interval. If $\gamma(a)$ is in $\ell$ then $\ell \setminus {\gamma(a)}$ has two connected components.
We'll say that $\gamma$ {\em finishes a leaf $\ell$ in the positive direction} if there are $a \in I$ and $b \in \bar I$ with $a<b$ such that $\gamma(a) \in \ell$ and $\ell \setminus \gamma\big([a,b)\big)$ has one connected component. That is, there is a $c \in (a,b]$ such that $\gamma\big((a,c)\big)$ is one of the connected components of $\ell \setminus {\gamma(a)}$.
Then,
any further extension of $\gamma$ in the positive direction will require adding points not in $\ell$ (e.g., a singularity and a portion of a new leaf). \compat{Finishing $\ell$ does not mean that the parameterization stops at an endpoint of $\ell$. It could go beyond $\ell$...} We make a similar definition of finishing a leaf in the negative direction.
The following is the main ingredient in the proof of this theorem:

\begin{lemma}
\label{lem:extension}
Let $Z$ be a zebra plane, let $I \subset \R$ be an interval with endpoints $-1$ and $1$, and let $\gamma:I \to Z$ be a parameterized arc of a trail.
\begin{enumerate}
\item {\em (Right limit)}  If $\lim_{t \to 1^-} \gamma(t)$ exists, then there is a parameterized arc of a trail
$\eta:I \cup [1,2) \to Z$ extending $\gamma$ such that $\eta$ finishes a leaf in the positive direction that $\gamma$ does not.
\item {\em (Left limit)} If $\lim_{t \to -1^+} \gamma(t)$ exists, then $\gamma$ can be similarly extended to left as an arc of a trail. \compat{I made the left limit less formal in response to a suggestion of Barak. I want both statements so I can state the converse below easily.}
\end{enumerate}
Conversely if neither limit exists, then $\gamma$ is a trail.
\end{lemma}
\begin{proof}
We will prove statement (1). Statement (2) will follow by symmetry. Let $\gamma(1)$ denote the limit $\lim_{t \to 1^-} \gamma(t)$ (regardless of whether $1$ is formally in the domain of $\gamma$). If $\gamma(1)$ is not singular, let $m$ be the slope of $\gamma$ at $\gamma(1)$. Then we can extend $\gamma$ by following the leaf of $\sF_m$ through $\gamma(1)$ until it finishes the leaf. Now suppose $\gamma(1)$ is singular. Let $L$ be a leaf with an endpoint at $\gamma(1)$ which satisfies the angle condition at $\gamma(1)$. (One can see by inspection of the angle condition that such a leaf always exists.) Then $\gamma$ can be extended to finish $L$.

The final statement can be proved by showing the contrapositive. Suppose $\gamma:I \to Z$ is a proper subarc of an arc of a trail $\eta: J \to Z$. By a change of coordinates, we may assume that $I$ has endpoints $-1$ and $1$. Let $\phi:I \to J$ be the continuous orientation-preserving map satisfying $\gamma=\eta \circ \phi$. Since $\phi$ is not surjective, we may assume without loss of generality that the limit $\lim_{t \to 1^-} \phi(t)$ exists in $J$. \commb{Do you also want to say here that $1$ is not in $I$?} \compat{The argument I provide works even if $1 \in I$, though you are right that the case that $1 \in I$ is trivial. Anyway, I  think it is simpler not to say anything.} Then we have
$$\lim_{t \to 1^-} \gamma(t) = \lim_{t \to 1^-} \eta \circ \phi(t) = \eta \Big(\lim_{t \to 1^-} \phi(t)\Big)$$
by continuity of $\eta$, so the limit in statement (1) exists.
\end{proof}

\begin{proof}[Proof of Theorem \ref{trails exist}]
Observe that it suffices to prove the statement for zebra planes, because arcs of trails on zebra surfaces are defined to be images of arcs of trails on their PRU cover. So, throughout this proof, we will only consider trails in a zebra plane $Z$.

Let $\gamma_1:J_1 \to Z$ be an arc of a trail where $J_1 \subset \R$ is a bounded interval.
We produce $k \in \N \cup \{+\infty\}$ and a finite or infinite sequence of parameterized arcs of trails $\{\gamma_i:J_i \to Z\}_{1 \leq i < k}$, where $\{J_i\}_{1 \leq i < k}$
is a strictly increasing sequence of bounded open intervals.
We will produce this sequence inductively. Here, each $\gamma_{i+1}$ will be a parameterized arc of a trail extending $\gamma_i$, i.e., $\gamma_{i+1}|_{J_i} = \gamma_i$, whenever $\gamma_{i+1}$ is defined.
If the sequence is finite (i.e., $k < +\infty$) we will have that $\gamma_{k-1}$ is a trail and if $k=+\infty$ we will produce a trail using a limiting argument.

We produce the sequence inductively. Suppose $\gamma_i$ is defined for some $i \geq 1$. If $\gamma_i$ is a trail, we're done and we declare $k=i$. Otherwise, we can apply \Cref{lem:extension} to construct an arc of a trail $\gamma_{i+1}$ that has $\gamma_i$ as a proper subarc
and which finishes an additional leaf in any direction for which the limits (as described in (1) and (2) of \Cref{lem:extension}) exist.

\commb{What is confusing to me is what if $\gamma_i$ is the restriction of arctan to $(-i,i)$ or $[-i,i]$? Then the union is all of $\R$ but the map arctan doesn't give a trail.} \compat{The key idea is that whenever we extend, the arc $\gamma_{i+1}$ contains an additional leaf that was not in $\gamma_i$. So, in any direction where we extend infinitely often, the limiting ``trail'' passes through infinitely many leaves. I did correct typos and tried to clarify the argument a bit. But overall it is the same.}

If the sequence $\gamma_i$ is infinite (i.e., $k = +\infty$) then either the right or left limits for each $\gamma_i$ always exist. Assume the right limit always exists but at some point the left limits do not exist. Then we can reparameterize the curves $\gamma_i:J_i \to Z$ such that $\sup J_i = i$ for all $i$ (because each $\gamma_{i+1}$ is an extension of $\gamma_i$ including an additional leaf on the right) and still ensure that $\gamma_{i+1}|_{J_i} = \gamma_i$. Define $J_\infty = \bigcup_i J_i$ and define
$\gamma: J_\infty \to Z$ to be the limiting map, i.e., $\gamma(t)=\gamma_i(t)$ if $t \in J_i$. We claim that $\gamma$ is a trail. By \Cref{lem:extension}, it suffices to look at the limits of $\gamma(t)$ as $t$ approaches the endpoints. By hypothesis, the limit to the left endpoint of $J_\infty$ eventually stops existing. Consider the right endpoint, which by assumption is $+\infty$. Thus assuming to the contrary that $\gamma$ is not a trail, the limit $\lim_{t \to +\infty} \gamma(t)$ is some point $p \in Z$. Let $N$ be a stellar neighborhood of $p$. By possibly shrinking $N$, we can assume $N$ contains no complete leaves. But on the other hand because $\lim_{t \to +\infty} \gamma(t)=p$, there is an $m>0$ such that $\gamma\big((m,+\infty)\big)$ is contained in $N$. Observe that $\gamma\big((m,+\infty)\big)$ contains infinitely many complete leaves, giving us our desired contradiction to the statement that $\lim_{t \to +\infty} \gamma(t)=p$.

There are two other cases. The case where the left limit always exists and the forward limit eventually stops existing is symmetric. The case where both limits always exists is similar: we may simultaneously handle the right and left limits say by defining $J_i=(-i,i)$.
\end{proof}

We record the following basic consequence of the results above:

\begin{corollary}
Every trail on $Z$ has a parameterization $\gamma:\R \to Z$.
\end{corollary}
\begin{proof}
Let $\gamma:I \to Z$ be a parameterization of a trail. By \Cref{lem:extension}, $\lim_{t \to \inf I} \gamma(t)$ and $\lim_{t \to \sup I} \gamma(t)$ do not exist, so $I$ must be an open interval. By possibly reparameterizing, we may assume that $I=\R$.
\end{proof}

This allows us to strengthen \Cref{no returning trails}:

\begin{corollary}
\label{intersections with polygons}
Let $P \subset Z$ be a polygon all of whose exterior angles are at least $\pi$. Then given any trail $\tau:\R \to Z$, the preimage $\tau^{-1}(P)$ is either empty or a (possibly degenerate) closed and bounded interval.
\end{corollary}
\begin{proof}
\compat{This proof was rewritten because of comments of Barak.}
Since $\tau$ is continuous and $P$ is a closed set, $\tau^{-1}(P)$ is closed. By \Cref{no returning trails}, the preimage $\tau^{-1}(P)$ is an interval. It remains to show that $\tau^{-1}(P)$ is bounded.

Recall that a trail follows leaves, transitioning between leaves at singularities. By \Cref{leaves are proper maps}, if $\tau|_{I}: I \to Z$ parameterizes a bi-infinite leaf or a separatrix including its singular endpoint\compat{As Ferran noted, it is important we include the endpoint}, the preimage $(\tau|_I)^{-1}(P)$ is compact. Because transitions only happen at singularities, restrictions of $\tau$ can parameterize at most one bi-infinite leaf and at most two separatrices.

Now consider the collection of all parameterized saddle connections of the form $\tau|_{J}$ that intersect $P$. Such an interval $J$ must be bounded by \Cref{lem:extension}. By \Cref{no returning trails}, $\tau$ can't leave $P$ and later return, so there are at most two parameterized saddle connections that intersect $P$ but are not contained in $P$. By \Cref{no monogons}, trails are simple curves and therefore each singularity in $P$ is the endpoint of at most two parameterized saddle connections
of the form $\tau|_{J}$. Since $P$ is compact and the singularities are isolated, there are at most finitely many singularities in $P$. It follows that there are at most finitely many parameterized saddle connections of the form
$\tau|_{J}$ that are contained in $P$.

Putting it all together, we have shown that $\tau^{-1}(P)$ is contained in the union of at most two compact subsets of intervals $I$ such that $\tau|_{I}$ parameterizes a bi-infinite leaf or a separatrix and finitely many bounded intervals $J$ parameterizing saddle connections intersecting $P$. Thus $\tau^{-1}(P)$ is bounded.
\end{proof}

\subsection{Properness of trails}

\begin{theorem}
\label{proper}
If $\tau:\R \to Z$ is parameterized trail, then $\tau$ is a proper map.
\end{theorem}
\begin{proof}
Let $\tau: \R \to Z$ be a parameterized trail. Suppose to the contrary that $\tau$ is not proper. Then there is a compact subset $K \subset Z$ such that $\tau^{-1}(K)$ is not compact. Since $\tau$ is continuous, $\tau^{-1}(K)$ is closed. Therefore, $\tau^{-1}(K)$ must be unbounded. Without loss of generality, we may assume that there is a sequence $t_n$ with $\lim t_n \to +\infty$ and $\tau(t_n) \in K$ for all $n$. Since $K$ is compact, by passing to a subsequence we may assume that $\lim \tau(t_n)=p \in K$. Use \Cref{generalized rectangles} to
construct a generalized rectangle $P$ containing $p$ in its interior. Then there are infinitely many $n$ such that $\tau(t_n) \in P$, so $\tau^{-1}(P)$ is unbounded
violating \Cref{intersections with polygons}. \compat{Changes here from using parallelogram covers to generalized rectangles. Nov 4.}
\end{proof}

\begin{corollary}
\label{cut by a trail}
If $\tau$ is a trail in $Z$, then $Z \smallsetminus \tau$ has two connected components, both homeomorphic to an open disk.
\end{corollary}
\begin{proof}
Let $X$ be the one-point compactification of $Z$ with $x_\infty$ denoting the point added. Since $Z$ is an open topological disk, $X$ is homeomorphic to the $2$-sphere. Our trail can be parameterized by an injective proper map from $\R$, so it extends continuously to a simple closed curve $\bar \tau:\hat \R \to \tau \cup \{x_\infty\}$ with $\bar \tau(\infty)=x_\infty$. By the Jordan Curve Theorem, $X \smallsetminus \bar \tau$ consists of two components, each homeomorphic to a disk. We have $Z \smallsetminus \tau=X \smallsetminus \bar \tau$.
\end{proof}

\begin{corollary}
\label{trail through interior}
Let $\tau:\R \to Z$ be a parameterized trail and let $P \subset Z$ be a polygon all of whose interior angles are less than or equal to $\pi$. If $\tau$ passes through the interior of $P$, then $I=\tau^{-1}(P)$ is a nondegenerate closed and bounded interval and $\tau^{-1}(\partial P)=\partial I$.
\compat{This was added on Sept 29, 2022}
\end{corollary}
\begin{proof}
Since $P$ is closed and $\tau$ is continuous, $I=\tau^{-1}(P)$ is closed. Since $\tau$ is proper, $I$ is bounded. If $I$ were not an interval, there would be a bounded open interval $J$ that is a component of $\R \setminus I$, and the restriction of $\tau$ to the closure $\bar J$ would lead to a contradiction to \Cref{no returning trails}.
Observe that $I$ is nondegenerate because $U=\tau^{-1}(P^\circ)$ is open, nonempty and contained in $I$.

Observe that $\tau(t) \in \partial P$ if and only if $t \in I \setminus U$. Thus, it remains to show that $U$ is the interior of $I$. Let $U_0 \subset U$ be a connected component of $U$. Observe that $\tau(U_0)$ is an arc of $\tau$ in the interior of $P$ whose endpoints $\tau(\partial U_0)$ lie in $\partial P$. Let $p \in \tau(\partial U_0)$ be one of those endpoints. The interior angle at $p$ is at most $\pi$, so any path $\ell:[0,1) \to P$ with $\ell(0)=p$ such that $\ell\big((0,1)\big)$ is contained in a leaf must make an angle that is strictly less than $\pi$ with $\tau(U_0)$ at $p$. But by the angle condition of trails, the continuation of $\tau$ outside of $U_0$ cannot make such an angle with $\tau(U_0)$ at $p$, so after $\tau$ passes through $p$ it immediately leaves $P$. Therefore $U$ only has one connected component, whose endpoints are also endpoints of $I$. In particular $U$ is the interior of $I$.
\end{proof}

\section{Zebra surfaces with boundary}
\label{sect:boundary}

\commb{section 4 is very long and much of it is what some readers would think of as obvious. I am not advocating changing anything but I suggest dividing it into two sections. The first one, from the beginning to the end of subsection 4.7, could be called "zebra surfaces with boundary", and the rest could be called "surgeries of zebra surfaces". Also in the opening paragraph to the first of these, we should say some motivating and apologetic remarks, e.g. "surgeries are useful but require some preparations", "there are no surprises but there is no literature on topological foliations," etc. Of course this also requires a tiny change to the part of the introduction describing the organization of the paper} \compat{Seems reasonable. I'll go ahead and do it.}

Surgeries are useful but require some preparations, namely a clear definition of zebra structure on a surface with boundary. There are no surprises but we were unable to find work on topological singular foliations on surfaces with boundary that was sufficiently detailed for our needs.

\subsection{Surfaces with boundary}

Let $\bbU=\{(x,y) \in \R^2:~y \geq 0\}$ be the closed upper half-plane, whose boundary is $\partial \bbU=\{0\} \times \R \subset \bbU$.
A {\em surface with boundary} $S$ is a second countable Hausdorff space that is locally homeomorphic to $\bbU$. That is, for each point $p \in S$, there is an open neighborhood $N$
and an open subset $U \subset \bbU$ together with a homeomorphism $h:N \to U$. A point $p$ is said to be {\em in the boundary} of $S$ if $h(p) \in \partial \bbU$, and the set of points in the boundary is denoted $\partial S$. It is a standard observation that this definition is well-defined, that $\partial S$ is a $1$-manifold, and that $\partial S$ is a closed subset of $S$. The points in $S^\circ=S \setminus \partial S$ are said to be {\em interior points} of $S$.

\subsection{Sectors}
Recall the objects constructed in \Cref{sect:standard singularities}: The spaces $\Pi_n$ which are $n$-fold covers of $\Pi_{-1}=\R^2/-I$ branched over the origin, the stellar functions $\rho_n: \Pi_{n} \setminus \{\0\} \to \hat \R$, and the stellar foliations of $\Pi_n \setminus \{\0\}$. As in \Cref{sect:stellar function}, a ray in $\Pi_n$ is a connected component of $\rho_n^{-1}(m) \subset \Pi_{n} \setminus \{\0\}$, where $m \in \hat \R$ is some slope.

If $r_1$ and $r_2$ are distinct rays in some $\Pi_n$, then $r_1 \cup \{\0\} \cup r_2$ is a simple curve that divides $\Pi_n$ into two connected components. A {\em sector} $\sigma \subset \Pi_n$ is the union of $r_1 \cup \{\0\} \cup r_2$ and one of the connected components. This is an example of a surface with boundary, with $\partial \sigma = r_1 \cup \{\0\} \cup r_2$. As we move outward along one of the boundary rays, $\sigma$ is on the left. We call this ray the {\em initial ray}. As we move out along the other ray, the sector $\sigma$ is on the right. We call this second ray the {\em terminal ray}. A sector also has an interior angle defined in \Cref{sect:gauss-bonnet}, which measures the angle from the initial ray to the terminal ray. A {\em half-plane sector} is a sector whose interior angle is $\pi$. We write $\sigma^\ast$ for $\sigma \setminus \{\0\}$.

\subsection{Horizontal foliations of sectors}
Recall that $\sH_n$ denotes the horizontal foliation of $\Pi_n$. If $\sigma \subset \Pi_n$ is a sector, its horizontal foliation is $\sH_{\sigma}=\sH_n|_{\sigma^{\ast}}$, which is really a foliation of $\sigma^\ast$.
Note that the boundary rays of a sector are either horizontal leaves or everywhere transverse to the horizontal foliation. See \Cref{fig:boundary} for some examples.

\begin{figure}[htb]
\centering
\includegraphics[width=5in]{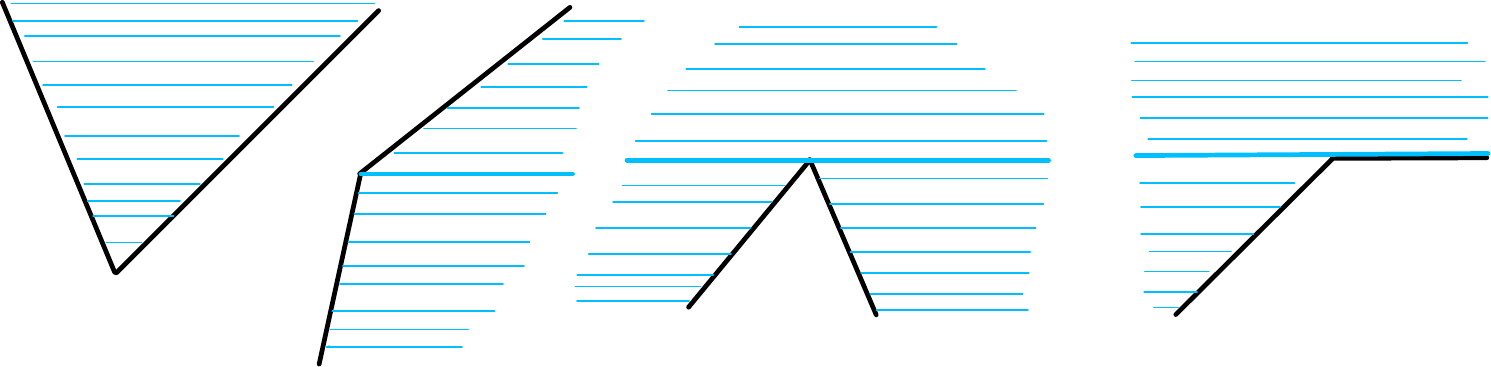}
\caption{Some sectors in $\Pi_0$ with their horizontal foliations.}
\label{fig:boundary}
\end{figure}

We say two sectors $\sigma_1 \subset \Pi_m$ and $\sigma_2 \subset \Pi_n$ have {\em isomorphic horizontal foliations} if there is an orientation-preserving homeomorphism $h:\sigma_1 \to \sigma_2$ such that $h(\0)=\0$ and $h|_{\sigma_1^\ast}:\sigma_1^\ast \to \sigma_2^\ast$ is an isomorphism between the horizontal foliations of these sectors.
Observe:

\begin{proposition}
\label{prop:foliation-equivalent}
Two sectors $\sigma_1$ and $\sigma_2$ have isomorphic horizontal foliations if and only if the two statements ``the initial ray of $\sigma_i$ is horizontal'' and ``the terminal ray of $\sigma_i$ is horizontal'' each have a truth value independent of the choice $i \in \{1,2\}$ and the number of horizontal rays in the two sectors is the same.
\end{proposition}
\begin{proof}
The listed conditions are clearly necessary for an isomorphism to exist. To see the converse, we need to break into cases. Recall that the Möbius action of $\PSL(2,\R)$ on $\hat \R$ (obtained from the action of lines in $\R^2$ through the origin) acts transitively on counterclockwise ordered triples in $\hat \R$. So, if $\sigma_1, \sigma_2 \subset \Pi_{-1}$ are two sectors that do not contain the horizontal ray in $\Pi_{-1}$, there is an affine map of $\Pi_{-1}$ which carries $\sigma_1$ to $\sigma_2$ and preserves the horizontal foliation. Every sector $\sigma \subset \Pi_n$ containing no horizontal rays is isomorphic to such a sector in $\Pi_{-1}$ under an isomorphism obtained by restricting the covering map $\Pi_n \to \Pi_{-1}$ to $\sigma$, so this handles the case when there are no horizontal rays in the sectors.

Now suppose that the same number of horizontal rays of $\sigma_1 \subset \Pi_m$ and $\sigma_2 \subset \Pi_n$ exist. By lifting to a common cover, we can assume that $m=n$ and the complementary angles of the sectors are at least $2 \pi$. By rotation of $\Pi_n$, we can assume that if the initial rays are both horizontal, then they are the same, and if they are not horizontal then they lie in the same half-plane $H_i$ with horizontal boundary. We can act affinely on $H_i$ as before while preserving the horizontal foliation to make the portions of the sectors in $H_i$ coincide. Then because the sectors contain the same number of horizontal rays, the terminal rays coincide if they are horizontal or lie in the same half-plane $H_r$ with horizontal boundary if not. Again, we can act affinely to make the sectors coincide. In this case, there is a piecewise affine map between the sectors that gives an isomorphism of the foliations.
\end{proof}

\subsection{The model space}
\label{sect:model space}
Let $\sS$ be a set of sectors, with one from each horizontal foliation isomorphism class.
Consider the model space
$$X_\partial = \Pi_{-1} \sqcup \bigsqcup_{n=1}^\infty \Pi_n \sqcup \bigsqcup_{\sigma \in \sS} \sigma
\quad \text{with its {\em horizontal foliation}} \quad
\sH_\partial = \sH_{-1} \sqcup \bigsqcup_{n=1}^\infty \sH_n \sqcup \bigsqcup_{\sigma \in \sS} \sH_\sigma.$$
We let $X_\partial^\ast$ denote the space $X_\partial$ with the origin removed from each $\Pi_n$ and each sector in $\sS$. Then, $\sH_\partial$ is a foliation of $X_\partial^\ast$.

Imitating the definition in \Cref{sect:singular foliations}, we define the pseudogroup $G_\partial$ to consist of all orientation-preserving homeomorphisms $h:U \to V$ between open subsets $U,V \subset X_\partial$ such that the following two statements hold:
\begin{enumerate}
\item $h(U \cap X_\partial^\ast) = h(U) \cap X_\partial^\ast$.
\item The restriction $h|_{U \cap X_\partial^\ast}$ is in the foliation pseudogroup of $(X^\ast_\partial, \sH_\partial)$.
\end{enumerate}
Note that elements of $G_\partial$ send boundary points of sectors to boundary points of sectors. As in \Cref{sect:singular foliations},
statement (1) guarantees that origins are sent to origins. In particular, if the origin in a sector is in the domain of an element of $G_\partial$, then it must be sent to the origin of another sector. \commb{There is something confusing about the terminology. "is in pseudogroup $G_{\partial}$" is a term with several quantifiers. You never define $G_{\partial}$ as an independent object. So I guess you are now referring to the foliation pseudogroup? At first reading I misunderstood and thought that $G_{\partial}$ could be a general pseudogroup, and then it was unclear what "this" referred to.} \compat{I added a clarifying remarks related to the way elements of $G_\partial$ must act on origins of sectors. This seemed to be a source of confusion in \Cref{sect:Singular foliations on surfaces with boundary}.}

It will be convenient to notice that $(X_\partial, \sH_\partial)$ looks fairly uniform, with points having standard neighborhood. We will use this to simplify our gluing arguments. Recall that $\bbU$ denotes the closed upper half-plane. Let $\bbV=\{(x,y) \in \R^2:~x \geq 0\}$ be the closed right half-plane. Both half-planes are surfaces with boundary that come with horizontal foliations obtained by restricting the horizontal foliation $\sH$ on $\R^2$.

\begin{proposition}
\label{prop:surjective}
Let $U \subset X_\partial$ be open and $p \in U$. Then:
\begin{enumerate}
\item If $p$ lies in the interior of $X_\partial$ and is not the origin of a $\Pi_n$, then there
is an open neighborhood $V \subset U \cap \Pi_n^\ast$ of $p$ such that $(V, \sH_\partial|_V)$ is isomorphic to $(\R^2, \sH)$
under a homeomorphism carrying $p$ to $\0$.
\item If $p \in \partial \sigma \setminus \{\0\}$ for some sector $\sigma$, then there is an open interval $I \subset U \cap \partial \sigma \setminus \{\0\}$ containing $p$ such that for any subinterval $J \subset I$ containing $p$ there is an open neighborhood of $p$, $V \subset U$, such that $V \cap \partial \sigma=J$ and $(V, \sH_\partial|_V)$ is isomorphic to $(\bbU, \sH|_\bbU)$ or $(\bbV, \sH|_\bbV)$ and such that $p \mapsto \0$ under this map. (See \Cref{fig:subspaces} for an example.)
\label{item:boundary ray}
\item If $p=\0 \in \sigma$ for some sector $\sigma$, then there is an open neighborhood $V \subset U$ of $p$ such that $(V, \sH_\partial|_V)$ is isomorphic to $(\sigma, \sH_\sigma)$. Furthermore, if $\sigma$ contains at least one horizontal ray, then there is an open interval $I \subset U \cap \partial \sigma$ containing $\0$ such that for any open subinterval $J \subset I$ containing $\0$ there is an open neighborhood $V \subset U$ of $p$ such that $V \cap \partial \sigma=J$ and $(V, \sH_\partial|_V)$ is isomorphic to $(\sigma, \sH_\sigma)$.
\label{item:surjective sector origin}
\commb{better "horizontal boundary ray" or better yet "if at least one of the initial and terminal ray is horizontal". At first I thought you mean "contains a horizontal ray", and the sector on the left side of the diagram does contain a horizontal ray. This threw me off}
\compat{I really mean $\sigma$ contains at least one ray that is horizontal. It doesn't need to be a boundary ray. In an attempt to emphasize this I changed "has a horizontal ray" to "contains a horizontal ray".}
\commb{It wasn't clear to me why you needed to treat these two cases separately. Is the stronger statement false for sectors which don't have a horizontal boundary ray?}
\compat{The stronger statement is only when the sector contains at least one horizontal ray (in the boundary or not). The issue is that if there are no horizontal rays in a sector $\sigma$, then the horizontal leaves join pairs of points in $\partial \sigma$. Any isomorphism must respect this pairing. Therefore in order for the stronger statement to hold in the case when $J$ has no horizontal rays, the interval $J$ would have to be preserved under the involution swapping points in these pairs.}
\item If $p=\0 \in \Pi_n$, then there is an open neighborhood $V \subset U$ of $p$ such that $(V, \sH_\partial|_V)$ is isomorphic to $(\Pi_n, \sH_n)$.
\end{enumerate}
\end{proposition}

Observe that in particular, this means that we can think of $(\R^2, \sH)$, $(\bbU, \sH|_\bbU)$ and $(\bbV, \sH|_\bbV)$ as part of our model space $X_\sigma$: These spaces are all realizable up to isomorphism by subsets of $X_\partial$.

\begin{figure}[htb]
\centering
\includegraphics[height=1.5in]{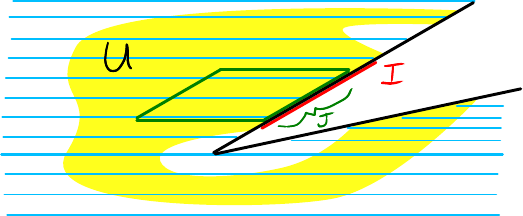}
\hspace{0.5in}
\includegraphics[height=1.5in]{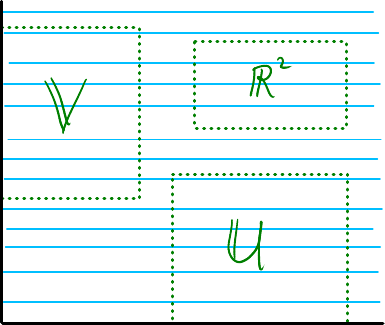}
\caption{{\em Left:} The situation of \hyperref[item:boundary ray]{statement (\ref{item:boundary ray})} of \Cref{prop:surjective}. The parallelogram has a foliation isomorphic to $(\bbV, \sH|_{\bbV}$).
{\em Right:} A sector containing subspaces isomorphic to $(\R^2, \sH)$, $(\bbU, \sH|_\bbU)$ and $(\bbV, \sH|_\bbV)$.}
\label{fig:subspaces}
\end{figure}

\begin{proof}
To see (1), we choose $V \subset U$ to be an open rectangle in $X_\partial^\ast$ with center $p$. Identify $V$ with $(-\frac{\pi}{2},\frac{\pi}{2})^2$ affinely and in these coordinates define the map to $\R^2$ by
\begin{equation}
\label{eq:tan}
(x,y) \mapsto (\tan x, \tan y).
\end{equation}
Observe that this map is an isomorphism as desired.

To see (2), observe that using \Cref{prop:foliation-equivalent}, we can assume that the boundary ray of $\sigma$ containing $p$ is either vertical or horizontal. Then we can find a rectangle $R \subset U$ containing $p$ that intersects $\partial \sigma$ in an interval $I \subset \partial \sigma \setminus \{\0\}$. Then for any $J \subset I$ containing $p$, we can define $V \subset R$ to be a smaller rectangle with boundary $J$. Place coordinates on $V$ of the form $(-\frac{\pi}{2},\frac{\pi}{2}) \times [0, \frac{\pi}2)$ or $[0,\frac{\pi}2) \times (-\frac{\pi}{2},\frac{\pi}{2})$ where $p$ is given coordinates of $(0,0)$. Then \eqref{eq:tan} defines a map to $\bbU$ or $\bbV$, respectively that is an isomorphism as claimed.

Now consider the first statement of (3) in the case when $\sigma$ has no horizontal rays. In this case $\sigma$ has isomorphic horizontal foliations to the sector depicted on the left side of \Cref{fig:boundary}. So, assume that all $\sigma \subset \Pi_0$ which we identify with $\R^2$. Rays in $\sigma$ are uniquely determined by their slope $m$, and the horizontal leaves are uniquely determined by their $y$-coordinate in $\R$. We can assume by possibly applying a $180^\circ$ rotation that $y \geq 0$ on $\sigma$. Then $(m,y)$ is a coordinate system on $\sigma$. Given $U$ containing $p=\0 \in \sigma$, there is a $y_0>0$ such that the triangle $V=\{(m,y) \in \sigma:~y <y_0\} \subset U$. Then the homeomorphism $V \to \sigma$ given by $(m,y) \mapsto (m, \tan \frac{\pi y}{2 y_0})$ is an isomorphism from $(V,\sH_\sigma|_V)$ to $(\sigma, \sH_\sigma)$ as desired.

In case $\sigma$ has a horizontal ray, the second statement of (3) implies the first. To see the second statement holds, suppose $p=\0 \in \sigma$ where $\sigma$ has a horizontal ray. Then by \Cref{prop:foliation-equivalent}, we can assume that $\sigma$ has horizontal and vertical boundary rays. The interior angle of $\sigma$ is $\frac{n \pi}{2}$ for some $n \geq 1$.
Then we can find a neighborhood $N \subset U$ of $\0$ that is a union of $n$ rectangles meeting edge-to-edge, with each rectangle in a subsector whose interior angle is $\frac{\pi}{2}$ with horizontal and vertical sides. See \Cref{fig:triple sector}. We let $I=N \cap \partial \sigma$ which is an interval. For any open subinterval $J \subset I$ containing $\0$, we can construct a neighborhood $V \subset N$ of $\0$ from a union of $n$ rectangles meeting edge-to-edge such hat $V \cap \partial \sigma=J$. Then each rectangle $R$ of $V$ can be mapped homeomorphically onto the subsector of $\sigma$ containing the rectangle via a map such as \eqref{eq:tan}. Observe that the maps agree along the common edges of rectangles, and so together they define a homeomorphism $V \to \sigma$ that gives the desired isomorphism.

\begin{figure}[htb]
\centering
\includegraphics[width=3.5in]{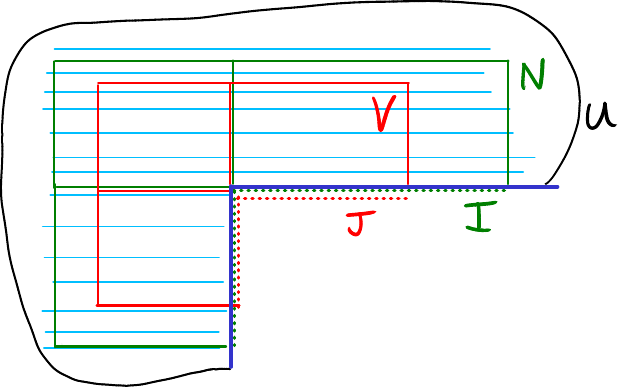}
\caption{A neighborhoods $V$ and $N$ of $\0$ built from rectangles in an open subset $U$ of a sector with horizontal and vertical sides and interior angle $\frac{3\pi}{2}$.}
\label{fig:triple sector}
\end{figure}

A similar argument proves (4), because a neighborhood $V \subset U$ of $\0 \in \Pi_n$ can be constructed from a union of $2n$ rectangles with $\0$ as a vertex meeting edge-to-edge.
\end{proof}

Let $Y$ be a surface possibly with boundary and $\sF$ be a foliation of a subset $Y^\ast \subset Y$. (We are only interested in the cases when $Y=\Pi_n$, $Y$ is a sector, or $Y \in \{\bbU,\bbV\}$.) Note that we can recover $Y^\ast$ from $\sF$ as it is the union of all leaves in $\sF$. We'll say that an {\em automorphism} of $(Y,\sF)$ is a homeomorphism $h:Y \to Y$ such that $h(Y^\ast)=Y^\ast$ and $h|_{Y^\ast}$ is an isomorphism from $(Y^\ast, \sF)$ to itself. The collection of orientation-preserving automorphisms forms the group $\Aut_+(Y,\sF)$.
\compat{I changed $X$ in this paragraph to $Y$ based on concerns that Barak brought up ($X$ is already defined in equation \eqref{eq:model space}). Similar changes below.}

\begin{proposition}
\label{automorphisms of models}
\begin{enumerate}
\item If $(Y,\sF) \in \{(\bbU, \sH|_{\bbU}), (\bbV, \sH|_{\bbV})\}$ then $Y^\ast=Y$ and any orientation-preserving homeomorphism $h:\partial Y \to \partial Y$ can be continuously extended to an element of $\Aut_+(Y,\sF)$. \commb{I don't understand this. We start with $X^*$ being some subset of X and suddenly it has to equal X? Why can't it be a proper subset? Or do you mean that X and $X^*$ are as defined in formula (2)?} \compat{$Y^\ast$ is always the union of leaves in the foliation. I added a sentence pointing this out to this effect to the paragraph above.}
\label{item:automorphism of half-space}
\item If $\sigma$ is a sector with at least one horizontal ray, then any orientation-preserving homeomorphism $h:\partial \sigma \to \partial \sigma$ such that $h(\0)=\0$ can be continuously extended to an element of $\Aut_+(\sigma,\sH_\sigma)$.
\label{item:automorphism of sector 1}
\item If $\sigma$ is a sector that contains no horizontal rays, then there is a unique involution $\iota:\partial \sigma \setminus \{\0\} \to \partial \sigma \setminus \{\0\}$, swapping the two components, such that for all $p \in \partial \sigma \setminus \{\0\}$, the leaf of $\sH_\sigma$ through $p$ also passes through $\iota(p)$. If $h:\partial \sigma \to \partial \sigma$ is any orientation-preserving homeomorphism such that $h(\0)=\0$ and such that $h|_{\partial \sigma \setminus \{\0\}}$ commutes with $\iota$, then $h$ can be continuously extended to an element of $\Aut_+(\sigma,\sH_\sigma)$. \compat{Added that $\iota$ is unique. To prevent back and forth, note that \href{https://livexp.com/blog/a-unique-or-an-unique-which-article-is-correct-in-this-case/}{a unique is correct}.}
\label{item:automorphism of sector 2}
\end{enumerate}
\end{proposition}
\begin{proof}
In statement (1), $Y$ is a subset of $\R^2$ and is naturally a product space $Y=\partial Y \times [0, +\infty)$ where leaves are fibers of one of the coordinate projections. Observe that if $f:[0,+\infty) \to [0,+\infty)$ is any homeomorphism then $h \times f \in \Aut_+(Y,\sF)$.

Consider the case of (2) when $\sigma$ has one horizontal boundary ray and one vertical boundary ray and an interior angle of $\frac{\pi}2$. Then we can naturally identify $\sigma$ with $\bar r_1 \times \bar r_2$ where $r_1$ and $r_2$ are the boundary rays whose closures $\bar r_i$ include $\0$. The foliation $\sH_\partial$ consists of the fibers of the perpendicular projection from $\sigma^\ast$ to the vertical boundary ray. Given $h$ as in statement (2), the map $h$ restricts to two homeomorphisms $h_1:\bar r_1 \to \bar r_1$ and $h_2: \bar r_2 \to \bar r_2$. The product $h_1 \times h_2$ lies in $\Aut_+(\sigma, \sH_\sigma)$.

Now consider the general case of a sector $\sigma$ with at least one horizontal ray. By \Cref{prop:foliation-equivalent}, we may assume that the boundary rays of $\sigma$ are either horizontal or vertical. Then the interior angle is $\frac{n\pi}{2}$ for some integer $n \geq 1$, and there are a total of $n+1$ horizontal and vertical rays, say $r_0, \ldots, r_n$ in counterclockwise order. As in the previous part, an $h$ as in (2) determines homeomorphisms $h_0:\bar r_0 \to \bar r_0$ and $h_n:\bar r_n \to \bar r_n$. Choose arbitrary homeomorphisms $h_i:\bar r_i \to \bar r_i$ for the other rays. Observe that the horizontal and vertical rays partition $\sigma$ into $n$-subsectors $\sigma_1, \ldots, \sigma_n$ each bounded by $\bar r_{i-1}$ and $\bar r_i$, and as in the previous part we can think of each $\sigma_i$ as the product $\bar r_{i-1} \times \bar r_i$. Then the map
which restricts to the product $h_{i-1} \times h_i$ on each $\sigma_i$ lies in $\Aut_+(\sigma, \sH_\sigma)$. This completes the proof of (2). \commb{any map}\compat{I think it should be `the map' because the $h_i$ are already defined and the products determine the map.}

Now let $\sigma$ be a sector that has no horizontal rays and consider (3). Then up to isomorphism, we may think of $\sigma$ as lying in the upper half-plane of $\Pi_0$ which we can identify with $\R^2$. The leaves in $\sH_\sigma$ are then fibers of the restriction to $\sigma^\ast$ of the projection $\pi_y:\R^2 \to \R$. Here $\pi_y(\sigma^\ast)=(0,+\infty)$ while $\pi_y(\0)=0$. Note that each $y \in (0,+\infty)$ has two preimages in $\partial \sigma \setminus \{\0\}$ and $\iota$ must swap them. This shows that $\iota$ exists and is unique as claimed. If $h$ is as in statement (3), then $h$ induces a homeomorphism $h_y:[0,+\infty) \to [0,+\infty)$ from its action on leaves. Now consider the stellar function $\rho_0:\Pi_0^\ast \to \hat \R$. Observe that $\rho_0(\sigma^\ast)$ is a closed interval $I$ with $0 \not \in I$, and
\begin{equation}
(\pi_y \times \rho_0)|_{\sigma^\ast}:\sigma^\ast \to (0,+\infty) \times I
\label{eq:dumb product}
\end{equation}
is a homeomorphism. If $f:I \to I$ is any orientation-preserving homeomorphism, then the homeomorphism of $\sigma$ whose restriction to $\sigma^\ast$ is conjugate under \eqref{eq:dumb product} to $h_y \times f$ is in $\Aut_+(\sigma, \sH_\sigma)$.
\end{proof}

\subsection{Definition of singular foliation on surfaces with boundary}
\label{sect:Singular foliations on surfaces with boundary}

Let $S$ be a surface with boundary. A {\em singular foliation atlas} on $S$ is an atlas of charts to $X_\partial$ such that transition functions lie in $G_\partial$. Observe that such an atlas induces a singular foliation on the interior $S^\circ$ in the sense of \Cref{sect:singular foliations}, because the restriction of a chart $\phi:U \to \sigma$ induces a chart from $U \cap S^\circ \to \Pi_n$ obtained by $i \circ \phi|_{U \cap S^\circ}$ where $i:\sigma^\circ \to \Pi_n$ denotes the inclusion of the interior of a sector into the $\Pi_n$ containing it. Thus we get a singular set $\Sigma \subset S^\circ$ that is discrete and closed and a singular data function $\alpha:S^\circ \to \Z_{\geq -1}$ supported on $\Sigma$ as before. Observe that $\Sigma$ is closed as a subset of $S$, because each point $p \in S$ has a neighborhood $N$ such that $N \setminus \{p\}$ contains no singular points. \compat{Made a minor improvement here, pointing out that $\Sigma$ is closed in $S$. Dec 23, 2022}

Given $S$ with a singular foliation atlas, a point $v \in \partial S$ is called a {\em vertex} if there is a chart $\phi:U \to X_\partial$ such that $v \in U$ and $\phi(v)$ is the origin in a sector.
\commb{shouldn't you also require that sigma is not a half plane? If sigma could be a half plane then any point on an edge would satisfy this definition of vertex.} \compat{No, the singular atlas is fixed. So, vertices are determined by the atlas. Furthermore, as part of the definition of $G_\partial$, transition functions must send origins to origins, so the notion of vertex is independent of the chart containing the point. A further subtlety  is that we want to allow ``removable vertices'' (even as we don't allow removable vertices). This will give us more flexibility with our gluing constructions.}
Again, each point $p \in S$ has a neighborhood $N$ such that $N \setminus \{p\}$ contains no vertices, so the collection $\sV \subset \partial S$ of all vertices is both closed and discrete. An {\em edge} of $S$ is a connected component of $\partial S \setminus \sV$. Recalling that $\partial S$ is a $1$-manifold, we see that each edge is also a $1$-manifold. It is important to note that edges can be homeomorphic to an open interval or to a circle (because a component of $\partial S$ might be homeomorphic to a circle and not contain any vertices).

An edge $e$ will be said to be {\em incident} to a vertex $v$ if $v \in \bar e$. Each $v$ has two incident edges (counting with multiplicity, since there could be an edge $e$ and a vertex $v$ such that $e \cup \{v\}$ is homeomorphic to a circle). We will name these two edges to match the terminology used for sectors: As we travel around $\partial S$ with $S^\circ$ on the left, we move from the {\em terminal incident edge} of $v$, to $v$, and then to the {\em initial incident edge to $e$}. Because these two edges may be counted with multiplicity, it is useful to distinguish them with an orientation: We orient them away from $v$. \compat{I made some improvements to the above two paragraphs: The chart $\phi$ wasn't formally correct before (it assumed the codomain was a sector), and I tried to address some of the technical issues with edges. Dec 23, 2022.}

Given a singular foliation atlas on a surface $S$ with boundary, we get a {\em foliation equivalence relation} on $S \setminus (\Sigma \cup \sV)$ as before: the coarsest one for which two points are equivalent if they can be connected by an arc whose image under a chart is contained in a leaf of the foliation $\sH_\partial$ of $X_\partial$. {\em Leaves} are again equivalence classes, and the {\em foliation} associated to the atlas is the collection of equivalence classes. Leaves get a topology by restricting the charts to connected components of intersections with charts as before. Observe that a leaf of $(\sigma, \sH_\sigma)$ that is transverse to a boundary ray has boundary, so leaves $\ell$ of our singular foliation are $1$-manifolds some of which have nonempty boundary $\partial \ell$.

Let $\sF$ be a singular foliation. We'll say $\sF$ is {\em transverse} to an edge $e$ of $\sF$ if for any $p \in e$, the point $p$ is in the boundary of the leaf containing $p$.

\begin{proposition}
\label{edges of singular foliations}
For each edge $e$, either $e$ is a leaf or the singular foliation is transverse to $e$. If a leaf $\ell$ is not an edge, then its interior $\ell^\circ$ is contained in $S^\circ$ and every point in $\partial \ell$ lies in an edge that is transverse to the singular foliation.
\end{proposition}

As a consequence we see that each edge $e$ is either a {\em leaf edge} meaning it is also a leaf, or a {\em transverse edge} meaning that the singular foliation is transverse to $e$.
\begin{proof}
Consider the first statement. Let $p \in e$ and $\phi:U \to X_\partial$ be a chart with $p \in U$. By
\hyperref[item:boundary ray]{(2)} of \Cref{prop:surjective}, there is a neighborhood $N_p \subset U$ containing $p$ such that
$(N_p, \sF|_{N_p})$ is isomorphic to either $(\bbU, \sH|_\bbU)$ or $(\bbV, \sH|_\bbV)$. Doing this for every point in $e$, gives local isomorphisms to these spaces covering $e$. But observe that whenever two neighborhoods $N_p$ and $N_q$ contain a common point of $e$, they must have the same type (either isomorphic to $\bbU$ or $\bbV$) because in $(\bbU, \sH|_\bbU)$ the boundary is a leaf and in $(\bbV, \sH|_\bbV)$ leaves hit the boundary transversely. Since $e$ is connected, it follows that all points have neighborhoods of the same type. If each $N_p$ is isomorphic to $(\bbU, \sH|_\bbU)$, then $e$ is a leaf edge and otherwise it is a transverse edge.

Now consider the last statement. Let $\ell \in \sF$ be any leaf and suppose there is a point $p \in \ell \cap \partial S$. Then $p$ must lie in some edge $e$. If $e$ is a leaf edge, then we must have $\ell=e$. Otherwise $e$ is a transverse edge, and so $p \in \partial \ell$.
Thus if $\ell$ is not an edge, every point of $\partial \ell$ is contained in a transverse edge and $\ell^\circ \subset S^\circ$.
\end{proof}

Now we give a more detailed description of how a singular foliation on a surface with boundary gives rise to a singular foliation of its interior, explaining exactly how the leaves of the two foliations are related:

\begin{corollary}
\label{restriction of singular foliation}
If $\sF$ is a singular foliation of a surface $S$ with boundary, then
$$\sF^\circ=\{\ell^\circ:~\text{$\ell \in \sF$ and $\ell \not \subset \partial S$}\}$$
is a singular foliation of $S^\circ$.
\end{corollary}
\begin{proof}
By \Cref{edges of singular foliations}, each point $p \in S^\circ \setminus \Sigma$ lies in the interior of a leaf of $\sF$. Thus $\sF^\circ$ is a partition of $S^\circ \setminus \Sigma$. To see $\sF^\circ$ is a singular foliation, we need to see that it arises from a singular foliation atlas on $S^\circ$. We'll obtain this atlas by altering the charts of the atlas determining $\sF$. If $\phi:U \to X_\partial$ is a chart for the structure on $S$ and $U_i \subset U$ is a connected component, then $\phi|_{U_i \cap S^\circ}$ is either a map to $\Pi_n$ or to the interior of a sector $\sigma^\circ \subset \Pi_n$ for some $n$. In the latter case, we obtain a new chart by post-composing $\phi|_{U_i \cap S^\circ}$ with the inclusion $\sigma^\circ \to \Pi_n$. To see the corresponding singular foliation $\sF^\circ$ of $S^\circ$ is as claimed, suppose that $\ell \in \sF$ is a leaf with $\ell \not \subset \partial S$. Then \Cref{edges of singular foliations} tells us that that $\ell^\circ \subset S^\circ$. Now observe that any interval $I \subset \ell^\circ$ that is contained in the domain of a chart $\phi:U \to X_\partial$ will also be contained in the domain of altered charts defined as above. Thus, $\ell^\circ \in \sF^\circ$ as desired.
\end{proof}

\subsection{Constructing singular foliations}
\label{sect:constructing singular foliations}

Formally a singular foliation $\sF$ is a partition of a subset of a surface that is determined from a singular foliation atlas. We offer a recipe for proving that a partition $\sF$ is a singular foliation:

\begin{lemma}
\label{construction}
Let $S$ be a surface with boundary. Let $\Sigma \subset S^\circ$ and $\sV \subset \partial S$ be sets such that $\Sigma \cup \sV$ is closed and discrete\compat{added closed Oct 31}. Suppose that $\sF$ is a partition of $S \setminus (\Sigma \cup \sV)$ into path connected subsets. Suppose further that there is an atlas $\sA$ of charts $\phi:U \to X_\partial$ from $S$ to $X_\partial$ such that each chart satisfies the two conditions:
\begin{enumerate}
\item We have $\phi(U^\ast)=\phi(U) \cap X_\partial^\ast$, where $U^\ast = U \setminus (\Sigma \cup \sV)$.
\item The partition of $U^\ast$ given by connected components of intersections of elements $\ell \in \sF$ with $U^\ast$ is the same as the pullback of $\sH_\partial|_{\phi(U^\ast)}$ under $\phi$.
\end{enumerate}
Then, $\sA$ is a singular foliation atlas on $S$ and $\sF$ is the singular foliation determined by $\sA$.
\end{lemma}
\begin{proof}
We claim that $\sA$ is a singular foliation atlas. To this end, let $\phi_1:U_1 \to X_\partial$ and $\phi_2:U_1 \to X_\partial$ be charts. Let $V \subset U_1 \cap U_2$ be a connected component of the intersection. We must show that the associated transition function $\phi_2 \circ \phi_1^{-1}$ lies in $G_\partial$. But, since $G_\partial$ is a pseudogroup, it suffices to show that $h=\phi_2 \circ \phi_1^{-1} |_{\phi_1(V)}$ is in $G_\partial$.

Since $V$ is connected, the images $\phi_1(V)$ and $\phi_2(V)$ must lie either in some $\Pi_n$ or in some sector $\sigma \in \sS$. Then statement (1) of the lemma guarantees that there can be at most one point of $\Sigma \cup \sV$ in $V$, since there is only one point in each $\Pi_n \setminus X_\partial^\ast$ and each $\sigma \setminus X_\partial^\ast$. Furthermore, if such a point exists, it must be sent to the origin in the containing $\Pi_n$ or $\sigma$. Thus $h$ satisfies part (1) of the definition of $G_\partial$.

To see that part (2) of the definition of $G_\delta$ is satisfied by $h$, set $V^\ast=V \setminus (\Sigma \cup \sV)$. Observe that $h$ must be an isomorphism from
$\big(\phi_1(V),\sH_\partial|_{\phi_1(V^\ast)}\big)$ to
$\big(\phi_2(V),\sH_\partial|_{\phi_2(V^\ast)}\big)$ because both foliations pull back to the same partition of $V^\ast$ by statement (2) of the Lemma. Thus $h|_{\phi_1(V^\ast)}$ is in $G_\partial$ as desired.

Let $\sF'$ be the singular foliation determined by the atlas $\sA$. We need to show that $\sF=\sF'$. Let $\ell \in \sF$. Then given any path $\gamma:[0,1] \to \ell$, $\gamma$ stays within the same leaf of $\sF'$ locally because given any $p \in \ell$ there is an open set $U$ and a chart $\phi:U \to X_\partial$ such that connected components of $\ell \cap U$ are obtained by pullback just as they are in the definition of $\sF'$. By hypothesis, elements of $\sF$ are path connected, so this local agreement guarantees that $\ell \subset \ell'$ for some $\ell' \in \sF'$.
Now suppose that $\ell$ was a proper subset of $\ell'$. Then, we can choose a point $p \in \ell$ and a point $q \in \ell' \setminus \ell$. Since $\ell'$ is a connected $1$-manifold, we can join $p$ to $q$ by a path $\eta$. Then we can cover $\eta$ by connected components of $U_i \cap \eta$ where each $U_i$ comes from a chart $\phi_i:U_i \to X_\partial$ in $\sA$. By compactness, we can pass to a finite subcover by these connected components $C_i \subset U_i \cap \eta$ where we now index by $i \in \{1, \ldots, n\}$. Observe that if $p$ and $q$ are in different elements of $\sF$, then some component $C_i$ contains points both in $\ell$ and in $\ell' \setminus \ell$. But this would violate statement (2) of the lemma for the corresponding chart. Thus in fact $\ell=\ell'$.

We've shown that each element of $\sF$ is an element of $\sF'$. But, since both $\sF$ and $\sF'$ are partitions of $S \setminus (\Sigma \cup \sV)$, they must coincide.
\end{proof}

We have the following consequence, which indicates that every singular foliation has a nice atlas:

\begin{proposition}
\label{nice atlas}
Let $\sF$ be a singular foliation on a surface $S$ with boundary. Then there is a singular foliation atlas for $\sF$ such that every chart is a homeomorphism from an open set $U$ to one of the spaces in
$$\{\R^2, \bbU, \bbV\} \cup \{\Pi_n:~n=-1 \text{ or } n \geq 1\} \cup \{\sigma:~ \sigma \in \sS\}$$
such that restriction to $U^\ast=U \setminus (\Sigma \cup \sV)$ is an isomorphism to the respective foliated space in
$$\{(\R^2, \sH), (\bbU, \sH|_{\bbU}), (\bbV, \sH|_{\bbV})\} \cup \{(\Pi_n^\ast,\sH_n):~n=-1 \text{ or } n \geq 1\} \cup \{(\sigma^\ast,\sH_\sigma):~ \sigma \in \sS\}.$$
Furthermore, if $D \subset S$ is any closed discrete subset, then such an atlas can be produced where each point of $D$ appears in the domain of exactly one chart, and such that each $p \in (D \cap S^\circ) \setminus \Sigma$ is sent to the origin $\R^2$, and each $p \in (D \cap \partial S) \setminus \sV$ is sent to the origin in $\bbU$ or $\bbV$. \compat{Added the condition that $D$ be closed Oct 31.}
\end{proposition}

We remark that above we are thinking of $\R^2$, $\bbU$ and $\bbV$ as subspaces of $X_\partial$, see \Cref{fig:subspaces}.

\begin{proof}
By definition $\sF$ is determined by an atlas $\sA$. We will define a new atlas $\sA'$ with maps as described determining the same foliation.

Suppose a closed discrete set $D$ has been provided. Fix a $p \in D$ and a chart $\phi:U \to X_\partial$ in $\sA$ with $p \in U$.
By \Cref{prop:surjective}, for each chart $\phi:U \to X_\partial$ in $\sA$ and each $p \in U$, we can define a new chart $\psi_{\phi,p}$ from a neighborhood $V_{\phi,p} \subset U \setminus (D \setminus \{p\})$ of $p$ to one of the spaces described and whose restriction to $V_{\phi,p}^\ast=V_{\phi,p} \setminus (\Sigma \cup \sV)$ is an isomorphism to the corresponding foliated space. Observe that these isomorphisms will satisfy the last statement. Each of these will be included in the atlas $\sA'$.

For every point $p \in X_\partial \setminus D$, define a chart $\psi_{\phi,p}$ from a neighborhood of $p$, $V_{\phi,p} \subset U \setminus D$, in the same way. Add these charts to $\sA'$. Now the domains of charts in $\sA'$ cover $S$, and there is only one chart containing each point in $D$.

Observe that statements (1) and (2) of \Cref{construction} are satisfied by $\sA'$ since the atlas came from restrictions of charts determining $\sF$. Thus, \Cref{construction} guarantees that $\sF$ is the foliation determined by the atlas $\sA'$ as desired.
\end{proof}

\subsection{Removable vertices}
\label{sect:removable vertices}
By \Cref{prop:foliation-equivalent}, there are only two half-plane sectors up to
isomorphism. These isomorphism classes are represented by the upper and right half-planes in $\Pi_0$. Given a singular foliation $\sF$ on a surface $S$ with boundary, a vertex $v \in \sV$ is {\em removable} if there is a chart $\phi:U \to X_\partial$ such that $v \in U$ and the connected component $C \subset U$ containing $v$ is mapped into a sector that is isomorphic to a half-plane sector. The edges on either side of a removable vertex $v$ are either both leaf edges or both transverse edges. In removing $v$ from the list of vertices, we either join these two leaf edges or add an endpoint to the single leaf with $v$ as an endpoint in the transverse case.

\Cref{construction} can be used to prove the following result:

\begin{proposition}
\label{removing marked points 1}
Let $\sF$ be a singular foliation on a surface $S$ with boundary. Let $\Sigma$ and $\sV$ be the singularity and vertex sets, respectively.
If $\sV' \subset \sV$ is any collection of removable vertices, then there is a unique singular foliation $\sF'$ on $S$ whose vertex set is $\sV \setminus \sV'$ such that every leaf of $\sF$ is contained in a leaf of $\sF'$.
\end{proposition}

\commf{Propositions \ref{removing marked points 1} and \ref{adding marked points 1} are `evident'. If we want at some point to reduce the length of the paper, we can leave the proofs to the reader.} \compat{I agree.}

\begin{proof}
By \Cref{nice atlas}, there is an atlas $\sA$ for $\sF$ such that every point of $\Sigma \cup \sV$ appears in exactly one chart, and such that charts are as described in that proposition. We will alter $\sA$ to define a new atlas $\sA'$. The only charts we alter are those whose domain contains a point in $\sV'$.
Each $v \in \sV'$ appears in the domain of a chart to a half-plane sector,
which is isomorphic to either the upper or right half-plane of $\Pi_0$. We replace this chart with one to $\bbU$ or $\bbV$ obtained by isomorphism to the half-plane in $\Pi_0$ followed by the natural map to $\R^2$. Observe that $\sA'$ is a singular foliation atlas that determines a foliation $\sF'$ as described.

Uniqueness follows from properties of the singular foliations. If $v \in \sV'$, then because $v$ has a neighborhood isomorphic to a half-plane sector, the two edges meeting at $v$ must be either both leaf edges or both transverse edges. If two leaf edges meet at $v$, these two leaf edges must be part of the same leaf of $\sF'$. If the two edges are transverse edges, then there is exactly one leaf whose closure contains $v$ and the union of $\{v\}$ and this leaf is contained in a leaf of $\sF'$. These conditions determine $\sF'$.
\end{proof}

We also have the reverse construction:

\begin{proposition}
\label{adding marked points 1}
Let $\sF$ be a singular foliation on a surface $S$ with boundary. Let $\Sigma$ and $\sV$ be the singularity and vertex sets, respectively.
If $\sV' \subset \partial S \setminus \sV$ is any collections of points such that $\sV \cup \sV'$ is a closed discrete subset of $\partial S$, then there is a unique singular foliation $\sF'$ of $S$ whose vertex set is $\sV \cup \sV'$ such that every leaf of $\sF'$ is contained in a leaf of $\sF$. \compat{Added closed to discreteness condition, Oct 31.}
\end{proposition}
\begin{proof}
Again by \Cref{nice atlas}, we can produce an atlas $\sA$ for $\sF$ as described there such that each point of $D=\Sigma \cup \sV \cup \sV'$ appears in exactly one chart. Then a point $p \in \sV'$ must be mapped to the origin in either $\bbU$ or $\bbV$. Recall that there are sectors $\sigma_U, \sigma_V \in \sS$ isomorphic to $(\bbU,\sH|_{\bbU \setminus \{\0\}})$
and $(\bbV,\sH|_{\bbV \setminus \{\0\}})$, respectively. We form a new atlas $\sA'$ by post-composing the chart associated to each $p \in \sV'$ with the isomorphism from the codomain ($\bbU$ or $\bbV$) to $\sigma_U$ or $\sigma_V$. This new atlas determines $\sF'$ as described. Uniqueness is clear: Leaves of $\sF'$ must be the connected components of $\ell \setminus \sV'$ taken over all $\ell \in \sF$.
\end{proof}

\section{Surgery on zebra surfaces}
\label{sect:surgery}

\subsection{Gluing singular foliations}
\label{sect:gluing singular foliations}
Let $S$ be an oriented topological surface with boundary and $\sF$ be a singular foliation on $S$. We do not require $S$ to be connected, and we could therefore obtain $S$ as the disjoint union of surfaces with singular foliations. We will explain how to glue edges of $S$ to obtain a new surface (perhaps with a smaller boundary) equipped with a singular foliation.
\commb{Do you mean ``edges of $\partial S$''.}\compat{This seems to just be semantics, but in the definition of edge, we say ``edge of $S$''. I think this is right because you say the edge of a polygon, not edge of the boundary of a polygon.}

Recall that $\partial S$ is a $1$-manifold and edges are connected components of $\partial S \setminus \sV$ where $\sV \subset \partial S$ is a closed discrete subset. \compat{Added closed to discreteness condition, Oct 31.} Let $\sE$ denote the collection of edges of $S$
. \commb{maybe "edges of $\partial S$" instead of "edges of $S$"? I am not sure which terminology is more consistent with what came before}\compat{I think it is fine as is. One says the "edges of a polygon" not the "edges of the boundary of a polygon". I went back and made a minor change for consistency earlier.} Then $\partial S=\sV \cup \bigcup_{e \in \sE} e$. The boundary of an edge $\partial e$ consists of any vertices in the closure $\bar e$ of $e$ viewed as a subset of $\partial S$. Thus $\partial e$ can consist of zero, one, or two vertices.

We will construct the quotient by gluing together closures of edges. Edges have a natural orientation: If we move in the direction of the orientation, the interior $S^\circ$ should be on the left. An {\em edge gluing} is an orientation-reversing homeomorphism between the closures of two edges $\bar e_1 \to \bar e_2$.

An {\em edge identification scheme} for $(S, \sF)$ consists of a collection of edges $E \subset \sE$, a fixed-point free involution $\varepsilon: E \to E$, and a choice of an edge gluing for each $e \in E$ of the form:
$$g_e:\bar e \to \overline{\varepsilon(e)} \quad \text{such that $g_{e} \circ g_{\varepsilon(e)}$ is the identity on $\bar e$ for each $e \in E$.}$$
An edge $e \in \sE$ will be called {\em glued} if $e \in E$ and {\em unglued} if $e \not \in E$.

An edge identification scheme determines an equivalence relation $\sim$ on $\bigcup_{e \in E} \bar e$: The finest equivalence relation such that each $p \in \bar e$ is equivalent to $g_e(p) \in \overline{\varepsilon(e)}$. We extend $\sim$ to an equivalence relation on all of $S$ by defining $\sim$-equivalence classes of a point $p \in S \setminus \bigcup_{e \in E} \bar e$ to be $[p]=\{p\}$. Let $\pi:S \to S/\sim$ be the quotient projection.

Observe that edge gluings send vertices to vertices, because these are endpoints of edges, so the $\sim$-equivalence class of a vertex $[v]$ consists only of vertices. Let $\sV/\sim$ denote the equivalence classes in $S/\sim$ containing vertices. We introduce several notions related to identified vertices. We say $[v]$ is {\em finite} if it is finite as an equivalence class of vertices of $\sV$. We say $[v]$ is {\em completely glued} if for each $e \in \sE$ containing a point of $[v]$ in its closure is glued. Finally, we let $\Prongs([v]) \in \Z_{\geq 0} \cup \{\infty\}$ denote the total number of prongs of $\sF$ taken over all $v \in V$, with each pair of glued leaf edges counting only once.

We will explain how an edge identification scheme satisfying certain requirements leads to a
singular foliation on the quotient surface $S'$. These requirements are:
\begin{enumerate}[label=(\alph*)]
\item Each $[v] \in \sV/\sim$ is finite.
\label{finiteness condition}
\item If $[v] \in \sV/\sim$ is completely glued, then $\Prongs([v])\geq 1$.
\label{prong condition}
\item
For each $e \in E$, the edge $e$ is a leaf edge if and only if $\varepsilon(e)$ is a leaf edge.
\label{type condition}
\end{enumerate}
Note that condition (c) says that glued edges must have the same type, because there are only two edge types: leaf edges and transverse edges.

\begin{theorem}[Surgery on singular foliations]
\label{surgery on singular foliations}
Let $S$ and $\sF$ be as above and suppose $(E, \varepsilon, \{g_e:e \in E\})$ is an edge identification scheme satisfying conditions (a)-(c). Then $S'=S/\sim$ is a surface with boundary $\partial S'=\sV' \cup \bigcup_{e' \in \sE'} e'$, where
$$\sV'=\{[v]\in \sV/\sim~:~\text{$[v]$ is not completely glued}\}
\quad \text{and} \quad
\sE'=\{\pi(e)~:~e \in \sE \setminus E\}.
$$
Also define
$$\Sigma' = \pi(\Sigma) \cup \{[v] \in \sV/\sim~:~\text{$[v]$ is completely glued and $\Prongs([v]) \neq 2$}\}.$$
Let $\sF'$ be the finest partition of $S' \setminus (\Sigma' \cup \sV')$ such that the following two statements hold:
\begin{enumerate}
\item Images of leaves of $\sF$ are contained in leaves of $\sF'$, i.e., for each $\ell \in \sF$, we have $\pi(\ell) \subset \ell'$ for some $\ell' \in \sF'$.
\label{images of leaves}
\item If $v \in \sV$ is a vertex such that $[v]$ is completely glued and $\Prongs([v]) = 2$, and the leaf $\ell \in \sF$ contains a prong emanating from $v$, then we have that $[v]$ is contained in the same leaf of $\sF'$ as $\pi(\ell)$. \compat{Ferran had some concerns about this statement. I tried to address them. Dec 7.}
\label{images of vertices with two prongs}
\end{enumerate}
Then $\sF'$ is a singular foliation on $S'$, with singular set $\Sigma'$, vertex set $\sV'$,
edge set $\sE'$ and singular data function $\alpha':(S')^\circ \to \Z_{\geq -1}$ whose only nonzero values are given by
$\alpha'([p])=\alpha(p)$ if $[p] \in \pi(\Sigma)$ and
$\alpha'([v])=\Prongs([v])-2$ if $[v] \in \Sigma' \setminus \pi(\Sigma)$.
\end{theorem}



\begin{proof}
Let $\sA$ be a singular foliation atlas for $(S,\sF)$. We will use $\sA$ to produce an atlas of charts $\sA'$ from $S'$ to $X_\partial$ satisfying \Cref{construction}. Existence of this atlas will prove that $S'$ is a surface with boundary and \Cref{construction} will guarantee that the atlas is a singular foliation atlas whose singular foliation is the partition $\sF'$ as described above. Since the connected components of the codomain of each chart consists of spaces homeomorphic to $\R^2$ or $\bbU$, it will follow that $S$ is a surface with boundary. Attention to the definition of the charts will verify that $\partial S'$ is as claimed.

In order for this to work, we need to check that the conditions of \Cref{construction} are satisfied. One condition that needs checking is that the elements of $\sF'$ are path connected. We verify this now. For $\ell \in \sF$, let $\ell^\bullet$ denote the union of $\ell$ and any $v \in \sV$ with $[v]$ completely glued and $\Prongs([v])=2$ such that there is a prong emanating from $v$ and contained in $\ell$. Observe that $\ell^\bullet$ is homeomorphic to an interval in $\R$, since the added points would have to be endpoints of an open end of $\ell$. Observe that $\sF'$ (as defined in the statement of the Theorem) can be seen to be the finest partition of $S' \setminus (\Sigma' \cup \sV')$ such that for each $\ell \in \sF$, the image $\pi(\ell^\bullet)$ is contained in a partition element of $\sF'$. Then by definition, for any two points $[p],[q] \in S'$
that lie in the same partition element of $\sF'$, there is a finite sequence
of leaves $\ell_1, \ldots, \ell_n \in \sF$ such that $[p] \in \pi (\ell_1^\bullet)$,
$[q] \in \pi (\ell_n^\bullet)$, and $\pi(\ell_i^\bullet) \cap \pi(\ell_{i+1}^\bullet) \neq \emptyset$ for $i =1, \ldots, {n-1}$. We can then explicitly describe a path from $[p]$ to $[q]$ in $S'$ by moving along $\pi (\ell_1^\bullet)$ from $[p]$ to a point of intersection with $\pi (\ell_2^\bullet)$ then along $\pi (\ell_2^\bullet)$ to a point of intersection with $\pi (\ell_3^\bullet)$, et cetera.

We will also need to check that the atlas $\sA'$ we produce satisfies statements (1) and (2) of \Cref{construction}. But these conditions pertain only to individual charts. So these conditions will be checked as we define our charts for $S'$. We will define charts covering any point not taking part in the gluing, then charts handling each point on a glued edge, and finally charts including any vertex taking part in the gluing.

The set of points that take part in the gluing construction is $G=\bigcup_{e \in E} \bar e$. We claim that $G$ is closed. Since $G \subset \partial S$ and $\partial S$ is closed, we know that $\bar G \subset \partial S$. Now recall
that every point of $\partial S$ is either in an edge or is a vertex. But edges are open in $\partial S$, and a vertex $v \in \sV$ has a neighborhood in $\partial S$ consisting of at most two edges $e$ with $v \in \partial e$. Thus, $\bar G \setminus G$ contains neither edges nor vertices, and therefore $\bar G=G$ as desired.

Because no gluing is taking place in the complement of $G$, the map $\pi$ induces an isomorphism between $(S \setminus G, \sF|_{S \setminus G})$ and $\big(\pi(S \setminus G), \sF'|_{\pi(S \setminus G)}\big)$. (By definition $\sF'|_{\pi(S \setminus G)}$ consists of the connected components of $\ell' \setminus \pi(G)$ taken over $\ell' \in \sF'$.) Given our atlas $\sA$, the collection of restricted charts
$\phi|_{U \setminus G}:U \setminus G \to X_\partial$
taken over all $\phi:U \to X_\partial$ in $\sA$ gives a singular foliation atlas for $\sF|_{S \setminus G}$. Thus pre-composing these charts with $\pi^{-1}$ gives a singular foliation atlas for
$\sF'|_{\pi(S \setminus G)}$, namely the collection of charts
$$\phi \circ \pi|_{S \setminus G}^{-1}:\pi(U \setminus G) \to X_\partial$$
taken over all $\phi:U \to X_\partial$ in $\sA$. We include each such chart in $\sA'$.
Since $\sA$ was a singular foliation atlas for $(S, \sF)$, these charts cover $\pi(S \setminus G)$ and form a singular foliation atlas for $\sF'|_{\pi(S \setminus G)}$. Since they are a singular foliation atlas, each individual chart satisfies statements (1) and (2) of \Cref{construction}.

\commf{(Written next to the paragraph below.) We could invoke statement (2) of Proposition 4.2 and directly work with the $V_i$s. Though formally correct, I find this can be shortened.} \compat{Okay I altered the paragraph to more quickly invoke statement (2) of Proposition 4.2.}

We will now define a chart of $\sA'$ for each pair $(e_1, p_1)$ consisting of an edge $e_1 \in E$ and a point $p_1 \in e_1$. Then $[p_1]=\{p_1,p_2\}$ where $p_2=g_{e_1}(p_1)$. Define $e_2=\varepsilon(e_1)$.
\hyperref[prop:surjective]{Statement (2)} of \Cref{prop:surjective} guarantees we can restrict charts of $\sA$ and post-compose these restricted charts with an isomorphism to obtain new charts for the original structure of the form
$$\psi_i:V_i \to \bbH \quad \text{for $i \in \{1,2\}$,}$$
where $V_1$ and $V_2$ are disjoint neighborhoods of $p_1$ and $p_2$ respectively, $\bbH \in \{\bbU, \bbV\}$ is independent of $i$, and $g_{e_1} \big(e_1 \cap V_1\big)=e_2 \cap V_2.$ Then using \hyperref[item:automorphism of half-space]{statement (\ref{item:automorphism of half-space})} of \Cref{automorphisms of models} we can alter $\psi_2$ by post-composition with an automorphism of $(\bbH,\sH|_{\bbH})$ to ensure
\begin{equation}
\label{eq:edge gluing}
\psi_2 \circ g_e(q)=-\psi_1(q) \quad \text{for any $q \in e_1 \cap V_1$}.
\end{equation}
Then we can define the glued neighborhood and corresponding chart
$$W=\pi(V_1 \cup V_2)
\quad \text{and} \quad
\chi:W \to \R^2; \quad \chi \circ \pi(q) = \begin{cases}
\psi_1(q) & \text{if $q \in V_1$,} \\
-\psi_2(q) & \text{if $q \in V_2$.}
\end{cases}$$
This map $\chi$ is a well-defined homeomorphism since $V_1$ and $V_2$ are disjoint, and the only gluing taking place is according to $g_e$ and our charts were built to satisfy \eqref{eq:edge gluing}. We include $\chi$ in our atlas $\sA'$ by identifying $\R^2$ with a subset of $X_\partial$ as in \Cref{prop:surjective}. The domain $W$ contains no singularities so statement (1) of \Cref{construction} is automatically satisfied. Also observe that statement (2) of \Cref{construction} is satisfied, i.e., $\chi$ induces a bijection between the collection of connected components of $\ell' \cap (W \setminus \{[v]\})$ taken over all $\ell' \in \sF'$ and the elements of the horizontal foliation $\sH$ of $\R^2$, because leaves of $\sF'$ pass through the glued edges according to the edge gluing.

Now suppose $[v]$ is an equivalence class of vertices that is not completely glued.
For each such $[v]$ we will define a single chart on a domain containing $[v]$.
\hyperref[finiteness condition]{Condition \ref{finiteness condition}} tells us that $[v]$ is finite. Because our edge gluings are required to be orientation-reversing, whenever two vertices $v$ and $v'$ are identified by a gluing map $g_e$, the initial incident edge of one of the vertices must be glued to the terminal incident edge of the other vertex.
Say $v' \in [v]$ is the {\em successor} of $v$ if the terminal incident edge of $v$ is glued to the initial incident edge of $v'$. From the preceding remarks, we can enumerate $[v]=\{v_1, \ldots, v_n\}$ such that $v_{i+1}$ is the successor of $v_i$ for $i=1, \ldots, n-1$.
Let $e_i^-$ denote the initial incident edge to $v_i$ and let $e_i^+$ denote the terminal incident edge.
Let $g_i:\overline{e_i^+} \to \overline{e_{i+1}^-}$ denote the gluing maps. For each $i$, fix a chart $\phi_i:U_i \to X_\partial$ such that $v_i \in U_i$. By restricting these charts to smaller open subsets we may assume the $U_i$ are pairwise disjoint. By \hyperref[item:surjective sector origin]{statement (\ref{item:surjective sector origin})} of \Cref{prop:surjective}, for each $i$, we can choose open intervals $I_i^- \subset e_i^-$ and $I_i^+ \subset e_i^+$ with $v_i \in \partial I_i^\pm$ such that:
\begin{enumerate}
\item We have $g_i(I_{i}^+)=I_{i+1}^-$ for $i=1, \ldots, n-1$,
\item For each $i \in \{1, \ldots, n-1\}$, there is an open subset $V_i \subset U_i$ such that $V_i \cap \partial S=I_i^- \cup \{v_i\} \cup I_i^+$ and such that there exists a sector $\sigma_i$ and a homeomorphisms $\psi_i:V_i \to \sigma_i$ that is an isomorphism from $(V_i, \sF|_{V_i})$ to $(\sigma_i, \sH_{\sigma_i})$.
\end{enumerate}
Observe that \hyperref[type condition]{condition \ref{type condition}} guarantees that for $i \in \{1, \ldots, n-1\}$, the terminal ray of $\sigma_i$ is horizontal if and only if the initial ray of $\sigma_{i+1}$ is horizontal. Let $\sigma'$ be the a sector such that the initial ray of $\sigma'$ is horizontal if and only if the initial ray of $\sigma_1$ is horizontal, such that the terminal ray of $\sigma'$ is horizontal if and only if the terminal ray of $\sigma_n$ is horizontal, and such that $\sigma'$ has the same number of horizontal rays as the total number of horizontal rays in $\{\sigma_1, \ldots, \sigma_n\}$ with glued rays counting only once. Using \Cref{prop:foliation-equivalent}, observe that $\sigma'$ can be be decomposed by cutting along rays into sectors $\sigma'_1, \ldots, \sigma'_n$ arranged in counterclockwise order such that each $\sigma_i$ has a horizontal foliation isomorphic to that of $\sigma'_i$ via a map $h_i:\sigma_i \to \sigma_i'$. We will define a chart
$$\chi:W \to \sigma' \quad \text{where} \quad W=\pi\left(\bigcup_{i=1}^n V_i\right)$$
by defining each restriction $\chi|_{\pi(V_i)}:V_i \to \sigma'_i$. We define
\begin{equation}
\label{eq:chi1}
\chi|_{\pi(V_1)} = h_1 \circ \psi_1 \circ \pi|_{V_1}^{-1}.
\end{equation}
Then proceeding by induction, assuming $i \in \{1, \ldots, n-1\}$ and $\chi|_{\pi(V_i)}$ has been defined, we define
\begin{equation}
\label{eq:chi2}
\chi|_{\pi(V_{i+1})} = \alpha_{i+1} \circ h_{i+1} \circ \psi_1 \circ \pi|_{V_1}^{-1},
\end{equation}
where $\alpha_{i+1}:\sigma_{i+1}' \to \sigma_{i+1}'$ is an automorphism chosen so that for all $q \in e_{i+1}^-$, we have
\begin{equation}
\label{eq:well-defined 2}
\chi|_{\pi(V_{i+1})}(q) = \chi|_{\pi(V_{i})} \circ g_{e_{i+1}^-}(q).
\end{equation}
Such an automorphism $\alpha_{i+1}$ exists by statement
\hyperref[item:automorphism of sector 1]{(\ref{item:automorphism of sector 1})}
or
\hyperref[item:automorphism of sector 2]{(\ref{item:automorphism of sector 2})}
of \Cref{automorphisms of models}. Equation \eqref{eq:well-defined 2} guarantees that $\chi$ is well defined. We include $\chi$ in our atlas $\sA'$. The only singular point in the domain is $[v]$ which is sent to $\0 \in \sigma'$, so $\chi$ satisfies statement (1) of \Cref{construction}. The identifications made by $\pi$ on $V_1 \cup \dots \cup V_n$ glue $[v]$ to one point and each $q \in e_{i+1}^-$ to $g_{e_{i+1}^-}(q) \in e_i^+$, just as the map $\chi$ does. By construction, the restriction $\chi|_{V_i}$ sends leaves of $\sF|_{V_i}$ to leaves of $\sH_{\sigma'_i}$. Since the only gluings happening on $\pi^{-1}(W)$ are between the edges already discussed, $\chi$ satisfies statement (2) of \Cref{construction}.

Now suppose $[v]$ is an equivalence class of vertices that is completely glued.
Again we will define a single chart for $\sA'$ on a domain containing $[v]$.
Much of the structure is the same, so we maintain as much of the notation as possible. Here because $[v]$ is completely glued, every $v \in [v]$ has a successor. Thus it is natural to enumerate $[v]$ as
$$[v]=\{v_i:~i \in \Z/n\Z\} \quad \text{where $n$ is the cardinality of $[v]$}$$
so that for all $i$, $v_{i+1}$ is the successor of $v_i$. We define the charts $\phi_i:U_i \to X_\partial$,
intervals $I_i^- \subset e_i^-$ and $I_i^+ \subset e_i^+$, gluing maps $g_i$, open subsets $V_i \subset U_i$,
sectors $\sigma_i$, and homeomorphisms $\psi_i:V_i \to \sigma_i$ as before except that now we require any statement relating an object with index $i$ to an object with index $i+1$ to be true for all $i \in \Z/n\Z$. Let $m=\Prongs([v])-2$. Recall that \hyperref[prong condition]{Condition \ref{prong condition}} guarantees that $m \geq -1$. Then there is a partition of $\Pi_m$ into sectors $\sigma_1', \ldots, \sigma_n'$ cyclically ordered such that $\sigma_i$ and $\sigma_i'$ always have an isomorphic horizontal foliation. Cyclically shift the indexing so that $v_n$ has a horizontal prong. Then $\sigma_n$ and $\sigma_n'$ have horizontal rays. Now we will define a chart
$$\chi:W \to \Pi_m \quad \text{where} \quad W=\pi\left(\bigcup_{i \in \Z/n\Z} V_i\right)$$
inductively. We first define $\chi|_{\pi(V_1)}$ to $\sigma_1'$ as in \eqref{eq:chi1}. Then we inductively define
$\chi|_{\pi(V_{i+1})}$ to $\sigma_{i+1}'$ for $i \in \{1, \ldots, n-2\}$ as in \eqref{eq:chi2} with each $\alpha_{i+1}$ defined
so that \eqref{eq:well-defined 2} holds as before. It remains to define $\chi|_{\pi(V_n)}:\pi(V_n) \to \sigma_n'.$
Recalling that $\sigma_n'$ has a horizontal ray, we can define $\chi|_{\pi(V_n)}$ as in \eqref{eq:chi2}
where this time $\alpha_n: \sigma_n' \to \sigma_n'$ is chosen to satisfy the stronger condition that for each $q_- \in e_{n}^-$ and each $q_+ \in e_n^+$ we have
\begin{equation}
\label{eq:well-defined 3}
\chi|_{\pi(V_{n})}(q_-) = \chi|_{\pi(V_{n-1})} \circ g_{e_{n}^-}(q_-) \quad \text{and} \quad
\chi|_{\pi(V_{n})}(q_+) = \chi|_{\pi(V_{1})} \circ g_{e_{n}^+}(q_+).
\end{equation}
To find such an $\alpha_n$, we need that $\sigma_n'$ has a horizontal ray so that we can use the stronger \hyperref[item:automorphism of sector 2]{statement (\ref{item:automorphism of sector 1})}
of \Cref{automorphisms of models}. Again $\chi$ is a well-defined homeomorphism because it respects all gluing maps. Assuming $m \neq 0$, we add $\chi$ to the atlas $\sA'$. In this case, $[v]$ is a singular point $\alpha'([v])=m$ as in the Theorem. Observe that statements (1) and (2) of \Cref{construction} are satisfied. If $m=0$, then we alter the chart slightly by defining $\chi'= \iota \circ \chi$, where $\iota:\Pi_0 \to \R^2$ is a homeomorphism that restricts to an isomorphism from $(\Pi_0 \setminus \{\0\},\sH_0)$ to $(\R^2 \setminus \{\0\}, \sH|_{\R^2 \setminus \{\0\}})$. We include $\chi':W \to \R^2$ into $\sA'$ by identifying $\R^2$ with a subset of $X_\partial$ as in the edge gluing case. In this case statement (1) of \Cref{construction} is satisfied because $W$ contains no singularities. Statement (2) is satisfied because $\chi'$ respects the edge identifications and because $\chi'$ joins the two leaves containing prongs of points in $[v]$.

Since we have covered $S'=S/\sim$ by charts, this is a surface with boundary consisting only of vertices that were not completely glued and the union of edges not taking part in the gluing. Furthermore, since statements statements (1) and (2) of \Cref{construction} are satisfied for all charts, $\sA'$ is a singular foliation atlas whose singular foliation is $\sF'$.
\end{proof}

\subsection{Stellar functions}
Let $\sigma \subset \Pi_n$ be a sector. Recall the definition of the stellar function on $\Pi_n$ in \Cref{sect:stellar function}. The {\em stellar function} on $\sigma$ is the restriction
$$\rho_\sigma=\rho_n|_\sigma:\sigma^\ast \to \hat \R.$$
This function maps each point in a ray of $\sigma$ to the slope of the ray.

Two sectors $\sigma_1$ and $\sigma_2$ are {\em stellar equivalent} if there is an orientation-preserving homeomorphism $h:\sigma_1 \to \sigma_2$ such that $h(\0)=\0$ and $\rho_{\sigma_1}=\rho_{\sigma_2} \circ h$ on $\sigma_1^\ast$. Observe:
\begin{proposition}
\label{prop:stellar-equivalent}
Two sectors are stellar equivalent if and only if their initial rays have the same slope and their interior angles are the equal.
\end{proposition}
\begin{proof}
If such an equivalence $h$ exists, then it must send the initial ray to the initial ray and so the slopes of the initial rays are the same. Also the internal angle of the sector can be determined by the number of rays of each slope. To see the converse assume $\sigma_1 \subset \Pi_m$ and $\sigma_2 \subset \Pi_n$ are two sectors with initial rays of the same slope and the same internal angles. Let $\pi_m:\Pi_m \to \Pi_{-1}$ and $\pi_n:\Pi_n \to \Pi_{-1}$ denote the covering maps. There is an $h:\sigma_1 \to \Pi_n$ satisfying $\pi_m=\pi_n \circ h$ obtained by lifting the restriction of $\pi_m$ to $\sigma_1$. Namely, both initial rays have the same image in $\Pi_{-1}$, so we can define $h$ on the initial ray of $\sigma_1$ so it sends this ray to the initial ray of $\sigma_2$. Then because $\sigma_1$ is simply connected, we can extend to all of $\sigma_1$ in a unique way. The image $h(\sigma_1)$ must be a sector with the same initial ray as $\sigma_2$ and have the same angle, so $\sigma_2=h(\sigma_1)$ and $h$ can be seen to be a stellar equivalence.
\end{proof}

\begin{corollary}
\label{stellar and isomorphic}
Stellar equivalence is a finer equivalence relation than the relation of being horizontal foliation isomorphic.
\end{corollary}
\begin{proof}
Suppose $\sigma_1$ and $\sigma_2$ are stellar equivalent. Then by \Cref{prop:stellar-equivalent}, the initial rays have the same slope and they have the same interior angles. It also therefore follows that the terminal rays have the same slope. Thus by statement (1) of \Cref{prop:foliation-equivalent}, the sectors $\sigma_1$ and $\sigma_2$ have isomorphic horizontal foliations.
\end{proof}

Later, we will need to glue sectors together in a manner respecting their stellar functions. Let $\sigma$ be a sector. We'll say that a {\em stellar epimorphism} $\psi:U \to \sigma$ is a homeomorphism whose domain $U$
is an open subset of $\sigma$ containing $\0$ such that $\psi(\0)=\0$ and such that for any ray $r \subset \sigma^\ast$, $\psi(r \cap U)=r$. It then follows that $\rho_\sigma \circ \psi(p)=\rho_\sigma(p)$ for $p \in \sigma^\ast$. The following three propositions ensure we can construct stellar epimorphisms that are useful for our gluing constructions.
None of these proofs are difficult, and we'll leave all but the last to the reader.
\compat{I think we're all supportive of omitting most of these. Ferran suggested keeping the last proof. The proofs of the others are commented out.}

\begin{proposition}
\label{dilating sectors}
If $\sigma$ is a sector and $C \subset \sigma$ is a closed subset that does not contain $\0$, then there is a stellar epimorphism $\psi:U \to \sigma$ with $U \cap C=\emptyset$.
\end{proposition}

\begin{proposition}
\label{interval fix}
If $\sigma$ is a sector and $I \subset \partial \sigma$ is an open interval containing $\0$, then there is a
stellar epimorphism $\psi:U \to \sigma$ with $U \cap \partial \sigma=I$.
\end{proposition}

\begin{proposition}
\label{stretching sectors}
If $\sigma$ is a sector and $\psi_0:\partial \sigma \to \partial \sigma$ is an orientation-preserving homeomorphism such that $\psi_0(\0)=\0$, then there is a homeomorphism $\psi:\sigma \to \sigma$ extending $\psi_0$ such that $\psi$ is also a stellar epimorphism.
\end{proposition}
\begin{proof}
Let $\{r_t:~t \in [0,1]\}$ denote the collection of rays of $\sigma$ ordered counterclockwise, so $r_0$ is the initial ray and $r_1$ is the terminal ray. Since $\psi_0$ is orientation-preserving and $\psi_0(\0)=\0$, we know $\psi_0(r_0)=r_0$ and $\psi_0(r_1)=r_1$. For $t \in \{0,1\}$, define $h_t:[0,+\infty) \to [0,+\infty)$ such that $h_t(0)=0$ and we have
$$d\big(\0,\psi_0(p)\big)=h_t\big(d(\0,p)\big) \quad \text{if $p \in r_t$.}$$
Then since $\psi_0$ is a homeomorphism, both $h_0$ and $h_1$ are homeomorphisms. We extend our definition of $h_t$ to include $t \in (0,1)$ by defining
$$h_t(x)=(1-t) h_0(x) + t h_1(x) \quad \text{for all $x \in [0,+\infty)$.}$$
Then each $h_t$ is a homeomorphism and the map $(t,x) \mapsto h_t(x)$ is continuous. Observe that the function $\psi: \sigma \to \sigma$ defined by $\psi(\0)=\0$ and $\psi(p)=q$
if $p \in r_t$ where $q \in r_t$ is the point such that
$d\big(\0,q\big)=h_t\big(d(\0,p)\big)$
is a stellar epimorphism such that the restriction to $\partial \sigma$ is $\psi_0$.
\end{proof}

\subsection{Zebra structures on surfaces with boundary}

Let $S$ be a surface with boundary and let $\{\sF_m\}_{m \in \hat \R}$ be a collection of singular foliations on $S$ with the same singular set $\Sigma$, the same singular data $\alpha:S \to \Z_{\geq -1}$, and the same vertex set $\sV$.

A stellar neighborhood of a point $p$ in the interior $S^\circ$ is defined exactly as in \Cref{sect:definition of zebra}. If $p \in \partial S$, a {\em stellar neighborhood} of $p$ is an open neighborhood $U$ of $p$ such that there is a sector $\sigma$ and a homeomorphism  $h:U \to \sigma$ such that $h(p)=\0$ and the following statements hold for each slope $m \in \hat \R$:
\begin{enumerate}
\item For each ray $r \subset \sigma$ of slope $m$, $h^{-1}(r)$ is contained in a leaf of $\sF_m$.
\item For each prong of $\sF_m$ emanating from $p$, there is a ray $r \subset \sigma^\ast$ of slope $m$ such that for any path $\gamma:(0,1) \to r$ with $\lim_{t \to 0^+} \gamma(t)=\0$, the preimage $h^{-1} \circ \gamma$ represents the prong.
\end{enumerate}
This homeomorphism $h$ is called a {\em stellar homeomorphism}. We remark that statement (1) guarantees that if $U$ is a stellar neighborhood $U$ of a point $p$, then there can be at most one singularity or vertex in $U$, namely $p$ itself.
\compat{There was a typo in the proof above (the last proof from the previous subsection) where I referred to stellar homeomorphism but meant epimorphism. I corrected this. This seemed to lead to complaints from both of you. Anyway, check if you still have objections related to this. (Stellar epimorphisms and stellar homeomorphisms are fairly different unfortunately- maybe less than ideal choice of terminology?)} \compat{Added the last sentence and split the paragraph here after a suggestion by Ferran related to the proof of \Cref{surgery}.}

We say $\{\sF_m\}$ is a {\em stellar foliation structure} or a {\em zebra structure} on a surface $S$ with boundary if the singular foliations all have the same singular set, same singular data, and same vertex set, and each point $p \in S$ has a stellar neighborhood. The pair $(S, \{\sF_m\})$ will be called a zebra surface with boundary.

Again a {\em leaf} in a zebra surface is a leaf of any of the foliations $\sF_m$. Because the foliations all have the same singular set and the same vertices, they also have the same edges. Each edge is a leaf:

\begin{proposition}
\label{edges of zebras with boundary}
Let $\{\sF_m\}$ be a zebra structure on a surface $S$ with boundary. Then:
\begin{enumerate}
\item For each edge $e$, there is a unique $m$ such that $e$ is a leaf of $\sF_m$.
\label{item:slope of an edge}
\item If $p$ is a point of an edge $e$, then the stellar homeomorphism associated to $p$ is a homeomorphism to a half-plane sector.
\end{enumerate}
\end{proposition}
The $m$ appearing in \hyperref[item:slope of an edge]{statement (\ref{item:slope of an edge})} is the {\em slope} of $e$.
\begin{proof}
Let $e$ be an edge. First we will establish that if $e$ is both a leaf of $\sF_{m_1}$ and of $\sF_{m_2}$ then $m_1=m_2$.
Choose any $p \in e$. Then the two ways of approaching $p$ within $e$ determine two prongs at $p$ of each of the two singular foliations $\sF_{m_1}$ and $\sF_{m_2}$. Let $h:U \to \sigma$ be a stellar homeomorphism associated to $p$. Each prong must arise from a ray, so we find there are two rays of slope $m_1$ and two rays of slope $m_2$ that give rise to the prongs. But observe that these rays must be contained in $\partial \sigma$ because $h$ is a homeomorphism and their preimages under $h$ are contained in $e \subset \partial S$.
We conclude that both pairs of rays are the boundary rays of $\sigma$. Thus, $m_1=m_2$ and boundary rays of $\partial \sigma$ have the same slope. This proves the uniqueness part of (1) and that the interior angle of $\sigma$ is an integer multiple of $\pi$ (since the boundary rays have the same slope).

We still need to show that $e$ is a leaf of some $\sF_m$. Again let $h:U \to \sigma$ be a stellar homeomorphism associated to $p \in e$.
Then $h(p)=\0$. Let $m$ be the slope of the initial ray $r$ of $\sigma$. From statement (1) of the stellar neighborhood definition, $h^{-1}(r)$ is a subset of a leaf $\ell \in \sF_m$. Because $h$ is a homeomorphism from an open subset of $S$ to $\sigma$, we have $h^{-1}(r) \subset \partial S$. Then \Cref{edges of singular foliations} guarantees that $\ell$ is a leaf edge. Observe that both $e$ and $\ell$ are edges containing $h^{-1}(r)$. Edges are pairwise disjoint, so we conclude that $e=\ell$. Thus $e$ is a leaf of $\sF_m$ proving (1).

It remains to prove that if $p \in e$ and $h:U \to \sigma$ is a stellar homeomorphism, then $\sigma$ is a half-plane. That is, we need to show that the interior angle of $\sigma$ is $\pi$. Let $m$ be such that $e$ is a leaf of $\sF_m$. From the first paragraph, the interior angle of $\sigma$ is $(n-1)\pi$ where $n$ is the number of prongs of $\sF_{m}$ at $p$. But any point in a leaf $\ell$ in a singular foliation has a single prong if the point lies in $\partial \ell$ and has two prongs if the point lies in $\ell^\circ$. Since edges are homeomorphic to open intervals, we see that $n=2$ and so the interior angle of $\sigma$ is $\pi$, proving (2).
\end{proof}

Now let $(S, \{\sF_m\})$ be a zebra surface with boundary and let $v \in \sV$ be a vertex. Since $v \in \partial S$, the stellar homeomorphism associated to $v$ must have the form $h:U \to \sigma$ with $h(v)=\0$. The {\em interior angle} of $S$ at $v$ is the same as the interior angle of $\sigma$ at $\0$. This notion is well-defined:

\begin{proposition}
For each $v \in \sV$, if $h_1:U_1 \to \sigma_1$ and $h_2:U_2 \to \sigma_2$ are two stellar homeomorphisms with $h_i(v)=\0$,
then $\sigma_1$ and $\sigma_2$ are stellar equivalent. In particular, the interior angle of $S$ at $v$ is well-defined.
\end{proposition}
\begin{proof}
Fix $v \in \sV$. Recall from \Cref{sect:Singular foliations on surfaces with boundary} that $v$ is isolated within $\sV$. Thus, as we travel around $\partial S$ through $v$ with $S^\circ$ on the left, we pass through an edge $e_1$ then through $v$ and finally through an edge $e_2$. By \Cref{edges of zebras with boundary}, each of these edges are leaves of some slope. Let $m_1$ and $m_2$ denote the slopes of $e_1$ and $e_2$, respectively. Now let $h_i:U_i \to \sigma_i$ be two stellar homeomorphisms with $h_i(v)=\0$ as in the statement. Then for each $i \in \{1,2\}$, the preimage of the terminal ray of $\sigma_i$ is contained in $e_1$ and the preimage of the initial ray is contained in $e_2$. The boundary slopes of $\sigma_i$ are therefore determined. Also, the interior angles of each $\sigma_i$ can be determined from knowing the number of rays of slope $m_1$ in $\sigma_i$. But by definition of stellar neighborhood this number of rays is the same of the number of prongs of $\sF_{m_1}$ at $v$. Since the boundary slopes are the same and the interior angles are the same, \Cref{prop:stellar-equivalent} tells us that $\sigma_1$ and $\sigma_2$ are stellar equivalent.
\end{proof}

\begin{proposition}
\label{zebra interior}
If $(S, \{\sF_m\})$ is a zebra surface with boundary, then $\{\sF_m^\circ\}$ is a zebra structure on $S^\circ$.
\end{proposition}

This is a direct consequence of \Cref{restriction of singular foliation} and the definitions.
Details are left to the reader.
\compat{Ferran suggested leaving this to the reader. I commented out the proof.}

\subsection{Removable vertices of zebra structures}
Let $S$ be a zebra surface with boundary and let $v \in \sV$ be a vertex. We say $v$ is {\em removable} if its internal angle is $\pi$. This is equivalent to the stellar homeomorphism for $v$ being a map to a half-plane sector. It then follows from \Cref{stellar and isomorphic} that a removable vertex for the zebra structure is also a removable vertex for each singular foliation $\sF_m$ in the structure.

\begin{proposition}[Removing marked points]
\label{removing marked points 2}
Let $(S, \{\sF_m\})$ be a zebra surface with boundary. Let $\Sigma$ and $\sV$ be the singularity and vertex sets, respectively. Suppose $\sV' \subset \sV$ is any collection of removable vertices. For each $m \in \hat \R$, let $\sF_m'$ be the singular foliation obtained by removing the removable vertices as in \Cref{removing marked points 1}. Then $(S, \{\sF_m'\})$ is a zebra surface with boundary.
\end{proposition}
\begin{proof}
The stellar neighborhoods and homeomorphisms for $(S, \{\sF_m\})$ still work for $(S, \{\sF_m'\})$.
\end{proof}

\begin{proposition}[Adding marked points]
\label{adding marked points 2}
Let $(S, \{\sF_m\})$ be a zebra surface with boundary. Let $\Sigma$ and $\sV$ be the singularity and vertex sets, respectively. Suppose $\sV' \subset \partial S \setminus \sV$ is a collections of points such that $\Sigma \cup \sV \cup \sV'$ is a closed discrete subset of $S$. \compat{Added closed to discreteness condition, Oct 31.}
For each $m \in \hat \R$, let $\sF_m'$ be the singular foliation obtained by adding each point in $\sV'$ to the collection of vertices as in \Cref{adding marked points 1}. Then $(S, \{\sF_m'\})$ is a zebra surface with boundary with singular set $\Sigma$ and vertex set $\sV \cup \sV'$. Furthermore, each $v \in \sV'$ is a removable vertex in $(S, \{\sF_m'\})$.
\end{proposition}
\begin{proof}
We must find a stellar neighborhood of each point $p \in S$ for the new structure. Let $h:U \to Y$ be a stellar homeomorphism for $p$ for the original structure, where $Y$ is either some $\Pi_n$ or a sector. Since $\sV'$ is closed and discrete, there is an $r>0$ such that there are no points of $h\big(U \cap (\sV' \setminus \{p\})\big)$ in the open ball $B$ consisting of all points whose Euclidean distance from $\0$ in $Y$ is less than $r$. Let $f:B \to Y$ be a homeomorphism such that $f(\0)=\0$ and for each ray $r$ of $Y$, $f(r \cap B)=r$. Then $h^{-1}(B)$ is a stellar neighborhood of $p$ and $f \circ h|_{B}$ is a stellar homeomorphism.
Observe that if $p \in \sV'$ then $p$ is a removable vertex because $p$ lies in an edge of the original structure and so $Y$ is a half-plane sector by \Cref{edges of zebras with boundary}.
\end{proof}

\subsection{Surgery on zebra surfaces}

Let $\{\sF_m\}$ be a zebra structure on a surface $S$ with boundary, possibly with multiple connected components.
Let $\Sigma$ denote the singular set and $\sV$ denote the vertex set. We will explain how to glue edges of $S$ to make a new zebra surface.

As in \Cref{sect:gluing singular foliations}, an edge gluing is an orientation-reversing homeomorphism between closures of distinct edges. As before an {\em edge identification scheme} is a collection $(E, \varepsilon, \{g_e\}_{e \in E})$, where $E$ is a collection of edges, $\varepsilon$ is a fixed-point free involution that describes which edges are to be glued, and $\{g_e\}$ is a choice orientation-reversing gluing homeomorphisms $g_e:\bar e \to \overline{\varepsilon(e)}$ such that $g_e^{-1}=g_{\varepsilon(e)}$ for all $e \in E$.

Again there is an equivalence relation $\sim$ on $S$ determined by the edge identification scheme. Let $[p] \in S/\sim$ denote the $\sim$-equivalence class of $p \in S$. As in \Cref{sect:gluing singular foliations}, vertices can only be identified with other vertices.
The notions of $[v] \in \sV/\sim$ being finite and being complete glued are defined as before. The {\em total angle}, $\Angle([v])$, of $[v]$ is the sum of the interior angles at $v$ over all $v \in [v]$. In order for our construction to produce a zebra surface, we require that:
\begin{enumerate}[label=(\alph*), start=4]
\item
Each $[v] \in \sV/ \sim$ is finite.
\label{finiteness condition 2}
\item For each edge $e \in E$, the slope of $e$ is the same as the slope of $\varepsilon(e)$.
\label{slope condition}
\end{enumerate}
As part of the proof of the result below, we will show that statements (b) and (c) from
\Cref{sect:gluing singular foliations}, which were required to glue singular foliations, follow from statements (d) and (e) above. \compat{Added this sentence in response to a comment of Ferran on Dec 7, 2022.}

\begin{theorem}[Surgery on zebra surfaces]
\label{surgery}
Let $(S,\{\sF_m\})$ be a zebra surface with boundary as above and let $(E, \varepsilon, \{g_e\})$ be an edge identification scheme satisfying statements \ref{finiteness condition 2} and \ref{slope condition} above.  Define
$$\Sigma'=\pi(\Sigma) \cup \big\{[v] \in \sV/\sim~:~\text{$[v]$ is completely glued and $\Angle([v]) \neq 2 \pi$}\big\} \quad \text{and}$$
$$\sV'=\{[v]\in \sV/\sim~:~\text{$[v]$ is not completely glued}\}.$$
Then $S'=S/\sim$ is a surface with boundary $\partial S'=\sV' \cup \bigcup_{e \in \sE \setminus E} \pi(e)$, and for each $m \in \hat \R$, the edge identification scheme satisfies the hypotheses of \Cref{surgery on singular foliations} for $(S, \sF_m)$ and $\sF_m'=\sF_m/\sim$ is a singular foliation $\sF_m'$ on $S'=S/\sim$ with singular set $\Sigma'$, vertex set $\sV'$, and singular data function $\alpha':(S')^\circ \to \Z_{\geq -1}$ whose only nonzero values are given by
$$\alpha'([p])=\alpha(p) \text{\quad if } [p] \in \pi(\Sigma) \quad \text{and} \quad
\alpha'([v])=\frac{1}{\pi}\Angle([v])-2\text{\quad if }[v] \in \Sigma' \cap \sV/\sim.$$
Furthermore $\{\sF_m':~m \in \hat \R\}$ is a zebra structure on $S'$, and if $[v] \in \sV'$,
then the interior angle at $[v]$ is $\Angle([v])$.
\end{theorem}
\begin{proof}
We begin by checking the hypotheses of \Cref{surgery on singular foliations} for $(S, \sF_m)$. \hyperref[finiteness condition]{Condition \ref{finiteness condition}} of \Cref{surgery on singular foliations} is the same as \hyperref[finiteness condition 2]{condition \ref{finiteness condition 2}} here. To check \hyperref[prong condition]{condition \ref{prong condition}} of \Cref{surgery on singular foliations}, suppose that $[v]$ is completely glued. For each $v_i \in [v]$, there is a stellar neighborhood $U_i$ and a stellar homeomorphism $h_i:U_i \to \sigma_i$ for some sector $\sigma_i$. Assume that $[v]$ is indexed $\{v_i:~i \in \Z/n\Z\}$ where the terminal incident edge of $v_i$ is glued to the initial incident edge to $v_{i+1}$ for all $i$. Our \hyperref[slope condition]{condition \ref{slope condition}} tells us that these pairs of edges have the same slopes. Then the terminal ray of $\sigma_i$ has the same slope as the initial ray of $\sigma_{i+1}$ for all $i$. Using stellar neighborhoods at each $v_i$, we see that each $v_i$ has prongs of varying slopes parameterized by a closed interval: these prongs are identified with rays in the corresponding sector. The collection of all prongs of all slopes at all $v_i$ with boundary prongs identified according to the edge gluings has a natural cyclic ordering coming from counterclockwise rotation of rays in the corresponding sectors, and the map from this collection to slopes is a covering map of $\hat \R$. It follows that $\Angle([v])$ is a positive integer multiple of $\pi$. Then by definition of stellar neighborhood, regardless of slope $\Prongs([v])=\frac{1}{\pi} \Angle([v]) \geq 1$, verifying \hyperref[prong condition]{Condition \ref{prong condition}}.
\hyperref[type condition]{Condition \ref{type condition}} of \Cref{surgery on singular foliations} follows trivially from \hyperref[slope condition]{condition \ref{slope condition}} here: An edge $e$ is a leaf edge for $\sF_m$ if and only if $e$ has slope $m$, and we only allow gluing edges of the same slope.

We have shown that $\sF'_m$ is well-defined for all $m$. It remains to show that $(S', \{\sF'_m\})$ is a zebra surface. This is a property that must be checked for every point on the surface.

Observe that the collection of points taking part in the gluing is $G=\bigcup_{e \in E} \bar e \subset \partial S$. For points $p \in S \setminus G$ a stellar neighborhood $U$ of $p$ for the zebra surface $(S, \{\sF_m\})$ can be produced that is disjoint from $G$. Then $\pi(U)$ is a stellar neighborhood for $\pi(p)$ on $(S', \{\sF'_m\})$. \compat{In response to a suggestion of Ferran, I omit some details here. The more detailed original argument is commented out below. Dec 7, 2022.}

Now let $p_1$ be a point in an edge $e_1 \in E$. Let $e_2=\varepsilon(e_1)$ and let $g:e_1 \to e_2$ denote the gluing map. Then $[p_1]=\{p_1, p_2\}$ where $p_2=g(p_1) \in e_2$. Choose stellar neighborhoods $U_i$ of $p_i$ for each $i \in \{1,2\}$. From (e) we know that $e_1$ and $e_2$ have the same slope. Then by \Cref{edges of zebras with boundary} we know that the stellar homeomorphisms $h_i:U_i \to \sigma_i$
are both maps to half-plane sectors with the same boundary slopes. We can assume that $U_1 \cap U_2 =\emptyset$. (Otherwise, we can find a closed set $C \subset S \setminus \{p_1, p_2\}$ such that $p_1$ and $p_2$ lie in different components of $S \setminus C$, and find new stellar neighborhoods in $U_i \setminus C$ with new stellar homeomorphisms defined by restricting $h_i$ to the smaller neighborhoods and post-composing with a stellar epimorphism obtained from \Cref{dilating sectors}. As the new neighborhoods are connected, they are necessarily disjoint.) Since the $\sigma_i$ are half-planes with the same slope, we can assume that $\sigma_1$ and $\sigma_2$ are complementary half-planes in $\Pi_0$. Thus, $\Pi_0=\sigma_1 \cup \sigma_2$. Now choose an open interval $I_1 \subset e_1$ containing $p_1$ such that
$I_1 \subset U_1$ and $I_2=g_{e_1}(I_1) \subset U_2$. Using \Cref{interval fix}, we can find an open $V_i \subset U_i$ such that $V_i \cap \partial S=I_i$ and a stellar epimorphism $\psi_i:h_i(V_i) \to \sigma_i$.
Observe that the two maps $\psi_i \circ h_i$ send $I_i$ to the common boundary $\partial \sigma_i$ by homeomorphism, but they do not yet respect the gluing map. By \Cref{stretching sectors}, there is a stellar epimorphism $\chi_2:\sigma_2 \to \sigma_2$ such that
\begin{equation}
\label{eq:chi3}
\chi_2 \circ \psi_2 \circ h_2 \circ g(q_1) = \psi_1 \circ h_1(q_1) \quad \text{for each $q_1 \in I_1$}.
\end{equation}
Define $W=\pi(V_1 \cup V_2)$. Then $W$ is an open neighborhood of $[p_1]$ in $S'$ since $\sim$ identifies $V_1$ and $V_2$ according to $g|_{I_1}:I_1 \to I_2$. We define
$$h:W \to \Pi_0; \quad \pi(q) \mapsto \begin{cases}
\psi_1 \circ h_1(q) & \text{if $q \in V_1$,} \\
\chi_2 \circ \psi_2 \circ h_2(q) & \text{if $q \in V_2$.} \\
\end{cases}$$
Then \eqref{eq:chi3} guarantees that $h$ is well-defined on $W$ and is a homeomorphism. We need to check that $h$ is a stellar homeomorphism. Observe that
$$h \circ \pi|_{V_1} = \psi_1 \circ h_1|_{V_1} \quad \text{and} \quad
h \circ \pi|_{V_2} = \chi_2 \circ \psi_2 \circ h_2|_{V_2}.$$
Since the original maps $h_i$ were stellar, for each $m \in \hat \R$, it induces a bijection between prongs of $\sF_m$ at $p_i$ and rays of slope $m$ in $\sigma_i$. Since $h \circ \pi|_{V_i}$ is obtained from $h_i$ by post-composition with a stellar epimorphisms, these bijections persist for $h$ between prongs of $\sF_m'$ at $[p_1]$ approaching $[p_1]$ within $\pi(V_i)$. The only identifications between the $V_i$ made by $\pi$ occur in the identification of $I_1$ with $I_2$, so these bijections induce a bijection between prongs of $\sF_m'$ at $[p_1]$ and rays of $\Pi_0=\sigma_1 \cup \sigma_2$. Thus $h$ is stellar as desired.

Now let $v \in \sV$ and assume that $[v]$ is not completely glued. We can order $[v]=\{v_1, \ldots, v_n\}$
such that the terminal incident edge $e_i^+$ of $v_i$ is glued to the initial incident edge $e_{i+1}^-$ of $v_{i+1}$ for $i \in \{1, \ldots, n-1\}$ under the gluing map $g_i:e_i^+ \to e_{i+1}^-$. As above, we can choose pairwise disjoint stellar neighborhoods $U_i$ of $v_i$ and a stellar homeomorphism $h_i:U_i \to \sigma_i$. For each $i$, choose open intervals $I_i^- \subset e_i^-$ and $I_i^+ \subset e_i^+$ with $v_i$ as an endpoint such that $g_i(I_i^+) = I_{i+1}^-$ for all $i \in \{1, \ldots, n-1\}$. As above, for each $i$ we can find an open subset $V_i \subset U_i$ such that $V_i \cap \partial S = I_i^+ \cup \{\0\} \cup I_i^+$ and a stellar epimorphism $\psi_i:h_i(V_i) \to \sigma_i$. Let $\sigma'$ be a sector whose initial ray has the same slope as the initial ray of $\sigma_1$, whose terminal ray has the same slope as the terminal ray of $\sigma_n$, and whose interior angle is the sum of the interior angles of the $\sigma_i$. Then by cutting $\sigma'$ along rays, we can partition $\sigma'$ into subsectors with the same internal angles as the $\sigma_i$ in counterclockwise order. Then each $\sigma_i$ is stellar equivalent to the corresponding subsector of $\sigma'$. Therefore, we redefine $\sigma_i$ so that it is this subsector of $\sigma'$. We have that
$$\psi_i \circ h_i(I_i^+)=\psi_{i+1} \circ h_{i+1}(I_{i+1}^-)$$
is the common boundary ray of $\sigma_i$ and $\sigma_{i+1}$, but the maps to do not yet respect the gluing maps $g_i|_{I_i^+}:I_i^+ \to I_{i+1}^-$.
Define $W=\pi(\bigcup_i V_i)$. Then $W$ is an open neighborhood of $[v]$. We will define $h:W \to \sigma'$ by induction. For $q \in V_1$, we define $h \circ \pi(q)=\psi_1 \circ h_1(q)$. Also define $\chi_1:\sigma_1 \to \sigma_1$ to be the identity map. Now assume that $h$ has been defined on $\pi(V_i)$. Using \Cref{stretching sectors}, we see there is a stellar epimorphism $\chi_{i+1}:\sigma_{i+1} \to \sigma_{i+1}$ such that
\begin{equation}
\label{eq:chi4}
\chi_{i+1} \circ \psi_{i+1} \circ h_{i+1} \circ g_i(q_i) = \chi_i \circ \psi_i \circ h_i(q_i) \quad \text{for each $q_i \in I_i^+$}.
\end{equation}
For $q \in V_{i+1}$ we define $h \circ \pi(q)=\chi_{i+1} \circ \psi_{i+1} \circ h_{i+1}(q)$. Equation \eqref{eq:chi4} guarantees that $h$ is a well-defined homeomorphism. Again $h$ was defined from the $h_i$ by post-composing with stellar epimorphisms, so $h$ is a stellar homeomorphism for $[v]$.

The case when $[v]$ is completely glued is similar. We will highlight the differences with the previous case.
This time we write $[v]=\{v_i:~i \in \Z/n\Z\}$ and write the elements of $\Z/n\Z$ as $1,\ldots, n$ to keep with the previous paragraph. We get a gluing map $g_i:e_i^+ \to e_{i+1}^-$ for all $i$ and we require that $g_i(I_i^+) = I_{i+1}^-$ hold for all $i$. As
noted in the first paragraph of the proof, in the completely glued case the sum of the interior angles of the sectors is of the form $(m+2)\pi$ for some $m \in \Z_{\geq -1}$. Thus, we can partition $\Pi_m$ into sectors that are stellar equivalent to the $\sigma_i$. We think of $\sigma_i \subset \Pi_m$ and order these sectors cyclically counterclockwise. The maps $h_i$, subsets $V_i \subset U_i$, and stellar epimorphisms $\psi_i$ can be defined as before. The set $W$ is defined as above, but we will define $h:W \to \Pi_m$. We again proceed inductively. We define the base case of $h$ on $\pi(V_1)$, and the inductive step of $h$ on $\pi(V_{i+1})$ when $i \in \{1, \ldots, n-2\}$ as before, but things change in a minor way when defining $h$ in the last step. This time using \Cref{stretching sectors} slightly differently, we define the stellar epimorphism $\chi_n:\sigma_n \to \sigma_n$ such that
\begin{align*}
\chi_{n} \circ \psi_{n} \circ h_{n} \circ g_{n-1}(q_{n-1}) &= \chi_{n-1} \circ \psi_{n-1} \circ h_{n-1}(q_{n-1}) \quad \text{for each $q_{n-1} \in I_{n-1}^+$}, \quad \text{and} \\
\chi_{n} \circ \psi_{n} \circ h_{n} \circ g_{n}^{-1} (q_{1}) &= \chi_{1} \circ \psi_{1} \circ h_{1}(q_{1}) \quad \text{for each $q_{1} \in I_{1}^-$}.
\end{align*}
Then we define $h$ on $\pi(V_n)$ to be $\chi_n \circ \psi_{n} \circ h_{n}$. The maps again respect the gluings so
$h:W \to \Pi_m$ is a homeomorphism and it is a stellar homeomorphism for the same reason as before.
\end{proof}

\subsection{Zebras not arising from dilation structures}
\label{sect: not dilation}
On a closed half-dilation surface, in any direction, there are at most finitely many isolated closed leaves. Indeed, this can be deduced from \cite{DFG19} as follows. It is enough to consider the case where the surface has a nonempty singular set or at least marked points. Each isolated closed leaf of the singular foliation of slope $m$ must lie in a dilation cylinder. Each dilation cylinder contains only finitely many closed leaves of slope $m$. Furthermore, the dilation cylinders that contain closed leaves of slope $m$ must have disjoint interiors, and must be bounded by saddle connections. The total of the interior angles of a dilation cylinder must be at least $2 \pi$, so this tells us that there can be only finitely many dilation cylinders that contain closed leaves of slope $m$. The finiteness result follows.
\compat{Added a citation and more details on the argument. DFG doesn't actually come out and say that any closed leaf gives rise to a cylinder. I wasn't able to find this explicitly in the literature...}

The same holds for surfaces in the $\Homeo_+(\hat \R)$-orbit of a zebra structure obtained by weakening a dilation surface structure. However, there are zebra surfaces whose directional foliations have countably many isolated closed leaves. Therefore:
\begin{proposition}
Let $g \geq 0$ and let $S$ be a closed and oriented topological surface of genus $g$. Suppose $\alpha:S \to \Z_{\geq -1}$ has finite support and satisfies $\sum_{p \in S} \alpha(p)=4g-4$. Then there is a zebra structure $\{\sF_m\}$ on $S$ with singular data $\alpha$, such that for every $\varphi \in \Homeo_+(\hat \R)$, the zebra structure $\varphi(S, \{\sF_m\})$ is not isotopic to a zebra structure obtained from a half-dilation surface.
\end{proposition}
\begin{proof}
There are square-tiled half-translation surfaces in each stratum determined by $\alpha$. (In fact they are dense; see \cite[\S 1.5.2]{HS06}.) Let $X$ be such a surface, and observe the horizontal foliation is periodic. Let $I \subset X$ be a vertical interval contained in the interior of a horizontal cylinder $C \subset X$. Slice $I$ open, obtaining two halves $I_+$ and $I_-$. Weakening the structure to a zebra structure, we see that the resulting object is a zebra surface $X'$ with two boundary edges, $I_+$ and $I_-$.
Let $g:I_+ \to I_-$ be a gluing homeomorphism that is monotone increasing in the $y$-coordinate such that the induced map $I \to I$ has countably many isolated fixed points. Then the zebra surface $X''$ obtained by gluing $\partial X'$ according to $g$ has the same singular data $\alpha$ and has countably many isolated closed leaves in the horizontal foliation $\sF_0$. Every surface in the $\Homeo_+(\hat \R)$-orbit also has a directional foliation with countably many isolated closed leaves, and so cannot be isotopic to a zebra structure arising from a half-dilation surface.
\end{proof}

\section{Constructing new foliations}
\label{sect:foliations}

\subsection{The Burp Lemma}

Later we will be constructing foliations consisting of leaves from directional foliations with varying slope. It is generally easy to prove that the proposed leaves do not cross, but it is more difficult to prove that there are no gaps (bubbles) between the proposed leaves. We call the following result the Burp Lemma because it is useful to rule out (burp) these bubbles:

\begin{lemma}[The Burp Lemma]
\label{burp}
Let $\overline{pq}$ be an arc of a trail such that all bending angles on the left side, as we move from $p$ to $q$, are $\pi$. For every segment $\overline{qr}$ such that $\measuredangle rqp < \pi$, there exists a point $x \in \overline{qr} \setminus \{q\}$ and a segment of a leaf $\overline{px}$. \compat{I rephrased the first sentence of the Burp lemma because I found the old version imprecise, and now with the notion of bending angle it is possible to be more precise. Dec 14, 2022.}
\end{lemma}
\begin{proof}
We may assume by acting by an element of $\Homeo_+(\hat \R)$ that $\overline{pq}$ is horizontal and $\overline{qr}$ is vertical. Using \Cref{trapezoid construction}, we can construct a rectangle $R$ whose vertices are $p$, $q$, a point $r' \in \overline{qr} \setminus \{q\}$ and some additional point $s$. The rectangle has no singularities in its interior by \Cref{trapezoid observation}. Our task is to construct a segment of a leaf $\overline{px}$ where $x \in \overline{qr'} \setminus \{q\}$.

Let $\gamma:[0,1] \to \overline{pq}$ be a parameterization of $\overline{pq}$ with $\gamma(0)=p$ and $\gamma(1)=q$. Say that $t \in [0,1]$ is {\em bad}, if there is a vertical segment $\overline{\gamma(t) \beta(t)}$ with $\beta(t)$ in the interior of $R$ such that there is no segment of a leaf joining $p$ to a point on $\overline{\gamma(t) \beta(t)} \setminus \{\gamma(t)\}$.

If the Burp Lemma is false, then there is a choice of points for which $1$ is bad. Assume this is the case, and we will derive a contradiction. Since there is a stellar neighborhood at $p$, every $t$ for which $\gamma(t)$ is in this stellar neighborhood is not bad, because the neighborhood is foliated by leaves emanating from $p$. Thus, setting $t_0$ equal to the infimum of the bad values of $t$ gives $t_0>0$.

First we claim that $t_0$ is not bad. Suppose to the contrary that $t_0$ is bad. Consider $\beta(t_0)$, which is a point on the vertical segment $\overline{\gamma(t_0) \beta(t_0)}$ below which no leaf from $p$ can cross. Let $\ell$ denote the leaf through $\beta(t_0)$ of slope $1$. Moving leftward along $\ell$, we must eventually leave $R$ by \Cref{leaves are proper maps}. We will consider where $\ell$ exits the rectangle. This situation is depicted on the left side of \Cref{fig:burp}. The leaf $\ell$ can't exit through the $\overline{r's} \setminus \{r'\}$ because then the polygon formed whose vertices are $\gamma(t_0)$, $q$, $r'$, $\ell \cap \overline{r's}$, and $\beta(t_0)$ would have interior angles adding to more than $3 \pi$, violating our \gaussbonnet which guarantees that the interior angles of a pentagon with no interior singular points add up to $3 \pi$. Similarly, we can see that $\ell$ cannot exit through $\overline{qr'} \setminus \{q\}$ (by applying \gaussbonnet to the quadrilateral with vertices $\gamma(t_0)$, $q$, $\overline{qr'} \cap \ell$ and $\beta(t_0)$)
and cannot exit through $\overline{\gamma(t_0) q} \setminus \{\gamma(t_0)\}$ (by applying 
\gaussbonnet to the triangle with vertices $\gamma(t_0)$, $\overline{\gamma(t_0) q} \cap \ell$ and $\beta(t_0)$). If $\ell$ exits through $\overline{sp}$, then it can't exit at $p$, or else it would violate the definition of $\beta(t_0)$. If $\ell$ exits at a point of $\overline{sp} \setminus \{p\}$, consider the segment $\tau$ of slope $1$ leaving $p$ and entering $R$. The segment $\tau$ enters and so must eventually exit the quadrilateral with vertices $p$, $\gamma(t_0)$, $\beta(t_0)$ and $\ell \cap \overline{sp}$. It can't exit through $\overline{pq}$ or $\overline{ps}$ because such an exit would create a bigon contradicting \Cref{no bigons} and cannot exit through $\ell$ because it is a leaf of the same foliation. So $\tau$ would have to exit through $\overline{\gamma(t_0) \beta(t_0)} \setminus \{\gamma(t_0)\}$, which would violate the definition of $\beta(t_0)$. We've ruled out the possibility that $\ell$ exits $R$ through $\overline{ps}$, and the only remaining segment of $R$ that remains for $\ell$ to exit is $\overline{p \gamma(t_0)} \setminus \{p\}$.
If $\ell$ exited at $\gamma(t_0)$ it would create a bigon with the vertical segment $\overline{\gamma(t_0) \beta(t_0)}$. So, it must exit through some point $u \in \overline{p \gamma(t_0)} \setminus\{p, \gamma(t_0)\}$, forming a new triangle $T=\triangle u \gamma(t_0) \beta(t_0)$. Pick a $t$ such that $\gamma(t)$ is in the interior of $\overline{u \gamma(t_0)}$. Choose $\beta(t)$ on the vertical leaf through $\gamma(t)$ such that $\beta(t)$ is in $T$. As $t<t_0$, the parameter $t$ cannot be bad. Therefore, there is a point $y \in \overline{\gamma(t) \beta(t)}$ and a segment of a leaf $\overline{py}$. Since $\overline{py}$ enters the interior of $T$, it must enter through an edge. By definition of $\beta(t_0)$, it cannot be through $\overline{\gamma(t_0) \beta(t_0)}$. It cannot be through $\overline {u \gamma(t_0)}$, because this is part of the boundary of $R$, and no trail can exit $R$ and later reenter (\Cref{intersections with polygons}). Therefore, it enters through $\ell=\overline{u\beta(t_0)}$. But then it must exit $T$ through a different edge, or else it would create a bigon. But we've already showed that the leaf continuing $\overline{py}$ cannot pass through any of the other edges. This contradicts our hypothesis that $t_0$ is bad, proving that $t_0$ is not bad. \compat{Edits were made based on Ferran's comments. I think it should be good now. Dec 10, 2022.}

\begin{figure}[htb]
\centering
\includegraphics[width=5in]{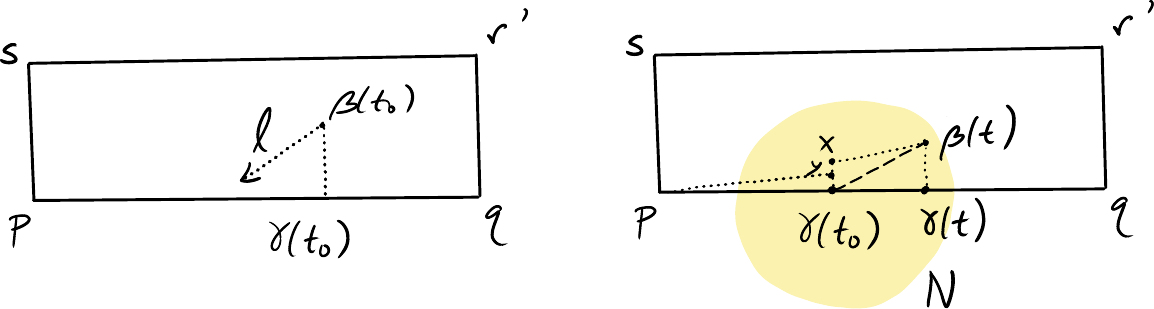}
\caption{Configurations discussed in the proof of \Cref{burp}.}
\label{fig:burp}
\end{figure}

We have shown that $t_0$ is not bad. We will also derive a contradiction from this. Construct a stellar neighborhood $N$ of $\gamma(t_0)$. Since $t_0$ is the infimum of the bad values of $t$, there is a $t>t_0$ such that $t$ is bad and $\gamma(t) \in N$. Let $\beta(t)$ be as in the definition of bad. We can assume, by possibly moving $\beta(t)$ closer to $\gamma(t)$ along $\overline{\gamma(t) \beta(t)}$ that $\beta(t) \in N$. 
Using the fact that $N$ is stellar, we can construct $\overline{\gamma(t_0) \beta(t)}$ forming a triangle $\triangle \gamma(t_0) \gamma(t) \beta(t)$. The slope of $\overline{\gamma(t_0) \beta(t)}$ is some $m>0$ by \Cref{triangle1}. Let $m'$ be a slope with $0<m'<m$, and construct the ray of slope $m'$ from $\beta(t)$ moving leftward through $R$. This leaf cannot cross over $\overline{\beta(t) \gamma(t_0)}$ so it must cross the vertical leaf through $\gamma(t_0)$ at some point, call it $x$. (The vertical leaf through $\gamma(t_0)$ must exit through $\overline{sr'}$, and the ray cannot exit through $\overline{\gamma(t) q} \cup \overline{q r'} \cup \overline{r's}$ by repeating analysis done in the previous paragraph.) Since $t_0$ is not bad, there must be a segment $\overline{py}$ where $y \in \overline{\gamma(t_0) x} \setminus \{\gamma(t_0)\}$. This leaf must have positive slope, and any leaf of smaller positive slope through $p$ must also intersect $\overline{\gamma(t_0) x} \setminus \{\gamma(t_0)\}$, so we can assume without loss of generality that the slope $m''$ of $\overline{py}$ satisfies $0<m''<m'$. Then continuing along $\overline{py}$ we enter the quadrilateral $\gamma(t_0) \gamma(t) \beta(t) x$ through the edge $\overline{x \gamma(t_0)}$. The continuation cannot exit through $\overline{\gamma(t_0) \gamma(t)}$ or else it would create a bigon, and cannot exit through $\overline{x \beta(t)}$ because it would create a triangle whose counterclockwise slope triple is $(\infty, m'', m')$ which is increasing and violates \Cref{triangle1}. Therefore, it must exit through $\overline{\gamma(t) \beta(t)}$, but this violates that $t$ was bad.

We have shown that the infimum of bad points cannot exist. It follows that all elements of $[0,1]$ are not bad. In particular $1$ is not bad, so there is a leaf joining $p$ to $\overline{qr'} \setminus \{q\}$. 
\end{proof}

\subsection{Foliating triangles}
\label{sect:foliating triangles}
Let $\triangle pqr$ be a triangle in a zebra plane $Z$. For $\ast \in \{p,q,r\}$, let $m_\ast$ denote the slope of the edge opposite $\ast$.

Triangles inherit restricted foliations $\sF_m$ of every slope $m$ from inclusion in $Z$. For each vertex $\ast$, each edge with $\ast$ as a vertex forms a {\em section} for the foliation of the triangle by leaves of slope $m_\ast$ (i.e., every maximal segment of a leaf contained the triangle passes transversely through the edge exactly once), because leaves of slope $m_\ast$ passing through the interior of the triangle cannot exit through the opposite edge.

\begin{lemma}
\label{triangle2}
Let $\triangle pqr$ be a triangle in $Z$ with vertices ordered counterclockwise.
Then the collection of leaves with slopes in $[m_r, m_q]$ through $p$ cover the triangle and are pairwise disjoint except for sharing the common vertex $p$.
Consider the function
\begin{equation}
\label{eq:triangle foliation}
h:\triangle pqr \setminus \{p\} \to [m_r, m_q] \times (\overline{pq} \setminus \{p\})
\end{equation}
sending $x$ to the pair consisting of the slope of $\overline{px}$ and the point on $\overline{pq}$ where the leaf of slope $m_p$ through $x$ exits the triangle.
This function is a homeomorphism whose restriction to $\overline{qr}$ maps to $[m_r,m_q] \times \{q\}$.
\end{lemma}

Equation \ref{eq:triangle foliation} places a coordinate system on the triangle somewhat analogous to polar coordinates.

\begin{proof}
For $m \in [m_r, m_q]$, let $\ell_m$ denote the arc of the leaf of slope $m$ starting at $p$ and entering the triangle until it exits the triangle. (Exiting is guaranteed by \Cref{leaves are proper maps}.) Then $\ell_{m_r}=\overline{pq}$
and $\ell_{m_q}=\overline{pr}$. In a stellar neighborhood of $p$, it is easy to observe that the segments $\ell_m$ are disjoint except at $p$, and by \Cref{no bigons}, these segments are disjoint except at $p$ as subsets of $\triangle pqr$.
Thus if $m \in (m_r, m_q)$, the arc of the leaf $\ell_m$ exits the triangle through a point on $\overline{qr}$. Thus for such $m$, $\triangle pqr \setminus \ell_m$ consists of two components, one containing $q$ and the other containing $r$. We'll say that $\ell_m$ {\em runs through} $x$ if $x \in \ell_m$, {\em runs above} $x$ if $x$ lies in the component of $\triangle pqr \setminus \ell_m$ containing $q$, and {\em runs below} $x$ if $x$ lies in the component containing $r$. These notions make sense for $m \in [m_r, m_q]$. We will first prove that for every $x \in \triangle pqr$ there is a $m \in [m_r, m_q]$ such that $\ell_m$ runs through $x$.

Observe that for $m<m'$ slopes in $[m_r, m_q]$, the segments $\ell_m$ and $\ell_{m'}$ together with a segment of $\overline{qr}$ form a triangle. Thus by \Cref{triangle1}, the counterclockwise order for the edges of this triangle is $\ell_m$, the segment of $\overline{qr}$, and finally $\ell_{m'}$. In particular, if $A_m$ denotes the set of points that $\ell_m$ runs above and $B_m$ denote the set of points that $\ell_m$ runs below we have
\begin{equation}
\label{eq:above or below}
A_{m} \subset A_{m'} \quad \text{and} \quad B_{m'} \subset B_{m}.
\end{equation}

Now assume that there is a point $x_0 \in \triangle pqr$ such that there is no $m \in [m_r, m_q]$ for which $\ell_m$ runs through $x_0$. Define
$$m_c=\sup\,\{m \in [m_r, m_q]:~\text{$\ell_m$ runs below $x_0$}\}$$
(where the supremum is taken using the increasing cyclic ordering on $[m_r, m_q]$). We may assume without loss of generality (by possibly applying an orientation-reversing homeomorphism of the circle as described in \Cref{sect:homeo action}) that $\ell_{m_c}$ runs below $x_0$.
Now define
$$X =  B_{m_c} \cap \bigcup_{m \in (m_c,m_q]} A_m.$$
Observe that $x_0 \in X$.
By \eqref{eq:above or below}, for $x \in X$ and $m \in [m_r, m_c]$, $m$ runs below $x$. 

Consider the leaf $\beta$ through $x_0$ which is parallel to $\overline{qr}$. (If $x \in \overline{qr}$, we take $\beta=\overline{qr}$.) Since $\beta$ is a leaf of the same foliation as $\overline{qr}$, it cannot cross $\overline{qr}$. Therefore, the leaf $\beta$ must intersect both $\overline{pq}$ and $\overline{qr}$ and so must cross each $\ell_m$ exactly once (or else it creates a bigon).
Let $y=\ell_{m_c} \cap \beta$ and consider the segment $\overline{x_0 y} \subset \beta$. We claim that $\overline{x_0 y} \setminus \{y\} \subset X$.
We have $\overline{x_0 y} \setminus \{y\} \subset B_{m_c}$ because $\ell_{m_c} \cap \beta=\{y\}$ and so $\overline{x_0 y} \setminus \{y\}$ lies in the same component of $\triangle pqr \setminus \ell_{m_c}$. Similarly if $m \in (m_c, m_q]$, because $x_0 \in A_m$ the leaf $\ell_m$ must intersect $\beta$ in the segment of $\beta$ connecting $\overline{pr}$ to $x_0$. Thus $\overline{x_0 y} \setminus \{y\} \subset A_{m}$ because $\overline{x_0 y}$ lies in the same component of $\triangle pqr \setminus \ell_{m}$ as $x_0$.

Observe that $\measuredangle x_0 y p < \pi$ but there is no leaf from $p$ that intersects $\overline{x_0 y} \setminus \{y\}$. This is a violation of \Cref{burp}, proving that the $\ell_m$ cover $\triangle pqr$.

Now we will show that the map $h$ is a homeomorphism. To see it is continuous, fix an $x$ in the triangle. The topology on the codomain has a basis consisting of rectangles. Suppose $h(x)$ is in the rectangle given by the product of an interval of slopes with endpoints $m_1$ and $m_2$, and points $y_1, y_2 \in \overline{pq}$. The preimage of this set consists of the intersection of the triangle bounded by $\ell_{m_1}$, $\ell_{m_2}$ and a segment of $\overline{qr}$ and a trapezoid consisting of the region between the leaves of slope $m_p$ through the points $y_1$ and $y_2$. Thus $h$ is continuous. It is onto since there must be an intersection between any $\ell_m$ and any leaf of slope $m_p$ passing through the triangle (since the boundary of $\ell_m$ separates $\overline{pq} \setminus \{p\}$ from $\overline{pr} \setminus \{r\}$ in the boundary of the triangle). It is one-to-one because of \Cref{no bigons} and the fact that the two leaves being intersected have distinct slopes. Thus $h^{-1}$ is well-defined. To show $h^{-1}$ is continuous, let $U$ be an open subset of the triangle. Let $x \in U$ and suppose $h(x)=(m,y)$. Let $\gamma$ be the leaf of slope $m_p$ through $x$ and $y$. Since $x$ is in the interior of $U$, we can find $x_1 \in U \cap (\overline{xy} \setminus \{x\})$ and $x_2 \in U \cap (\gamma \setminus \overline{xy})$, so that $\overline{x_1 x_2} \subset U$. Let $m_1$ and $m_2$ be the slopes of $\overline{px_1}$ and $\overline{px_2}$ respectively. Then \Cref{trapezoid construction} allows us to construct two trapezoids in $U$ one on each side of $\overline{x_1 x_2}$ with adjacent edges consisting of segments of $\ell_{m_1}$ and $\ell_{m_2}$. The union of these trapezoids is a larger trapezoid with $x$ in its interior. As we've already shown that the leaves $\ell_m$ vary monotonically, the image of the interior of this trapezoid under $h$ is a rectangle $(m_1,m_2) \times (y_1,y_2)$ with $y_1$ and $y_2$ being the places that the edges of the trapezoid parallel to $m_q$ intersect $\overline{pq}$.

The fact that $\overline{qr}$ maps to $[m_r,m_q] \times \{q\}$ is clear since $\overline{qr}$ is one of the leaves of slope $m_p$.
\end{proof}

\begin{corollary}[Vertex foliations]
\label{vertex foliations}
Let $P=p_0 p_1 \ldots p_{n-1}$ be a $n$-gon in a zebra plane with no interior singularities all of whose interior angles are less than $\pi$. Let $m_0$ be the slope of $\overline{p_0 p_1}$ and $m_1$ be the slope of $\overline{p_0 p_{n-1}}$. For $m \in [m_0, m_1]$, let $\ell_m$ be the leaf of slope $m$ entering $P$ from $p_0$. Then the leaves $\{\ell_m:~m \in [m_0, m_1]\}$ cover $P$, foliate the interior of $P$, and for any $m \in (m_0, m_1)$ the intersection $\ell_m \cap \partial P$ consists of two points.
\end{corollary}

\begin{figure}[htb]
\centering
\includegraphics[width=5in]{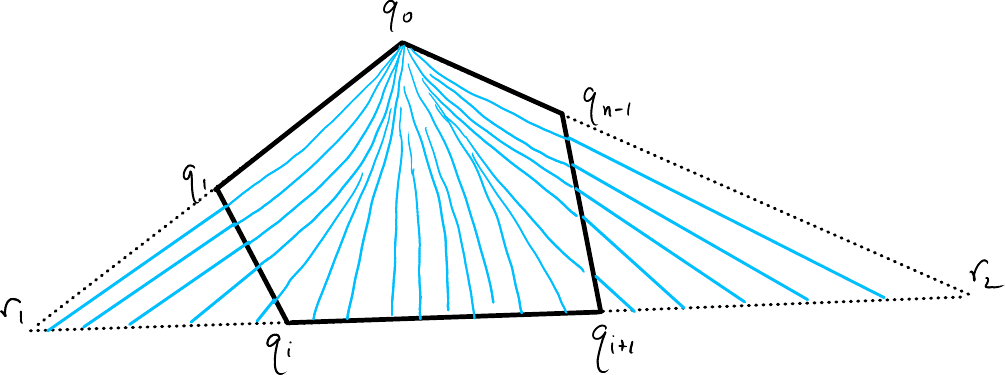}
\caption{\Cref{vertex foliations} and its proof.} \label{fig:vertex foliation}
\end{figure}

\begin{proof}
Let $Z$ denote the zebra plane containing $P$. Since $P$ has no singularities in its interior and all of its interior angles are less than $\pi$, there is a convex Euclidean $n$-gon $Q=q_0 \ldots q_{n-1} \subset \R^2$ with the same boundary slopes. Since $Q$ is convex,
there is an $i \in \{1, \ldots, n-2\}$ such that the line $\overleftrightarrow{q_i q_{i+1}}$ extending the side $\overline{q_i q_{i+1}}$ intersects the two rays $\overrightarrow{q_0 q_1}$ and $\overrightarrow{q_0 q_{n-1}}$.
Define $r_1=\overrightarrow{q_0 q_1} \cap \overleftrightarrow{q_i q_{i+1}}$, and $r_2=\overrightarrow{q_0 q_{n-1}} \cap \overleftrightarrow{q_i q_{i+1}}$.
Then the triangle $\triangle q_0 r_1 r_2$ contains $Q$. Using surgery (\Cref{surgery}), define
$Z' = (\R^2 \setminus Q) \cup P$.
Then applying \Cref{triangle2}, we see that the leaves through $p_0$ cover all of the triangle $\triangle q_0 r_1 r_2$ (which we are now viewing as in $Z'$),
and in particular these leaves cover $P$ and foliate its interior. Finally, if $m \in (m_0, m_1)$, we see by considering a stellar neighborhood at $p_0$ that $\ell_m$ immediately enters the interior of $P$ after leaving $p_0$. Thus, $\ell_m \cap \partial P$ contains only two points by \Cref{trail through interior}.
\end{proof}

As a consequence, we see some uniformity to our stellar neighborhoods:

\begin{corollary}
Suppose $P$ is polygon containing no singularities in its interior all of whose interior angles are less than $\pi$. If $x$ is in the interior of $P$,
then the interior of $P$ is a stellar neighborhood of $x$. The same holds if $P$ is a polygon satisfying the same properties, but with $x$ being the only singularity in the interior of $P$.
\end{corollary}
\begin{proof}
Consider the horizontal and vertical leaves emanating from $x$. These divide $P$ into subpolygons all of whose interior angles are less than $\pi$ with no singularities in its interior. Then \Cref{vertex foliations} allows us to foliate all these subpolygons using leaves emanating from $x$. We can use this to construct leaves from $x$ to the vertices of $P$, cutting $P$ into triangles. If $v$ and $w$ are consecutive vertices of $P$, then \Cref{triangle2} gives a homeomorphism from each triangle with $x$ removed to rectangles $[m_v,m_w] \times (0,1]$, carrying the foliation of the triangle by leaves through $x$ to the vertical foliation of the rectangle. These rectangles can be stitched together (using bump functions to continuously adjust the homeomorphisms so that the vertical sides can be identified by isometry), giving  homeomorphism from $P \setminus \{x\}$ to $\hat \R^k \times (0,1]$ where $\hat \R^k$ denotes the $k$-fold cover of $\hat \R$ and $k=\alpha(x)+2$. Then composing with a standard map from complex analysis gives the stellar homeomorphism.
\end{proof}

\subsection{Foliating polygons}
\label{sect:foliating polygons}
In this subsection, we explain two methods of foliating a polygon by leaves passing through an edge.

\begin{lemma}
\label{quadrilateral foliation}
Let $P$ be an $n$-gon with vertices $p_0$, \ldots, $p_{n-1}$ listed in counterclockwise order. Assume that there
are no singularities in the interior of $P$ and interior angles at each $p_i$ are less than $\pi$. For each $i \in \Z/n\Z$, let $m_i$ denote the slope of $\overline{p_i p_{i+1}}$. Let $\gamma:[0,1] \to \overline{p_0 p_{n-1}}$ be a parameterization such that $\gamma(0)=p_0$ and $\gamma(1)=p_{n-1}$. Suppose $\mu:[0,1] \to \hat \R \setminus \{m_{n-1}\}$ is continuous and, using the ordering on $\R \setminus \{m_{n-1}\}$, we have that $\mu$ is strictly increasing and satisfies
$$\mu(0) \leq m_0 \quad \text{and} \quad m_{n-2} \leq \mu(1).$$
Let $\ell_t$ denote the leaf of slope $\mu(t)$ through $\gamma(t)$. Then the collection of leaves $\{\ell_t:~t\in [0,1]\}$ foliates $P$. This foliation is transverse to $\overline{p_0 p_1}$ if $\mu(0)<m_0$, is transverse to $\overline{p_{n-2} p_{n-1}}$ if $m_{n-2}<\mu(1)$, and is always transverse to the other edges.
\end{lemma}

\begin{figure}[htb]
\centering
\includegraphics[width=3in]{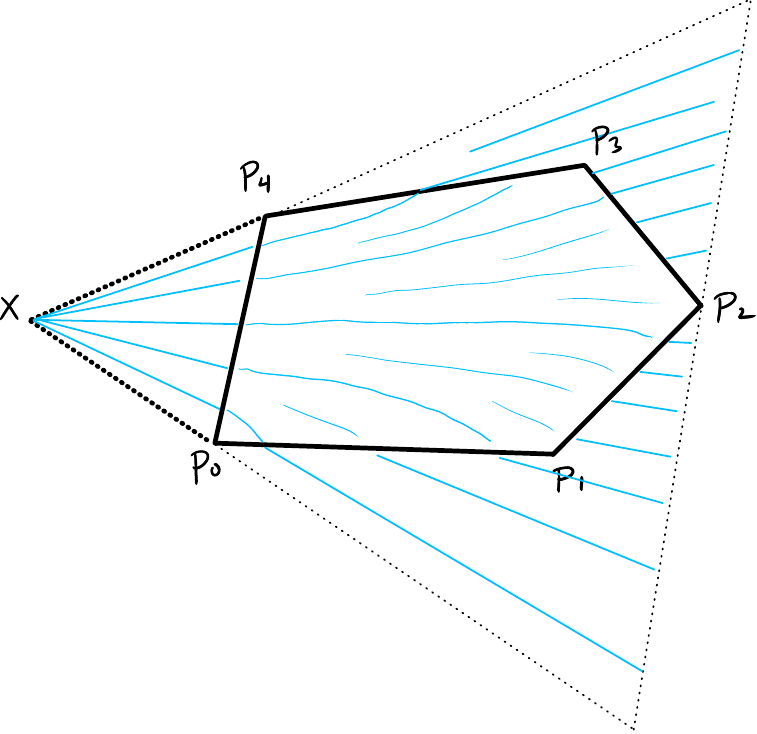}
\caption{A foliation of a polygon as described by \Cref{quadrilateral foliation}, depicted as constructed in the proof.} \label{fig:quadrilateral lemma}
\end{figure}

\commb{Just curious. You use strictly increasing in the proof, but is it necessary for the Lemma to hold?} \compat{It is probably not necessary, but, as you point out, our proof needs it. If $\mu$ was just increasing, then there would be intervals where slope is constant, but those constant intervals obviously correspond to foliated intervals. To extend the slick argument, maybe there is an argument that collapses intervals where $\mu$ is constant. (This is just a vague idea, and probably would be hard to make rigorous.)}

\begin{proof}
By the Gauss Bonnet Theorem, the sum of the interior angles of $P$ is $(n-2) \pi$. Thus we can find a strictly convex polygon $P'=p_0'\ldots p_{n-1}'$ in the plane $\R^2$ with the same edge slopes and the same interior angles as $P$. Let $x \in \R^2$ be the point at which the line of slope $\mu(0)$ through $p_0'$ intersects the line of slope $\mu(1)$ through $p_{n-1}'$. \commf{Maybe I am not understanding something...but why these intersect? why can't these lines be parallel?} \compat{I think we discussed. We have $\mu(0) \neq \mu(1)$ because the function $\mu:[0,1] \to \hat \R \setminus \{m_{n-1}\}$ is strictly increasing.} We can extend the rays $\overrightarrow{x p_0'}$ and $\overrightarrow{x p_{n-1}'}$ and join them by a segment forming a triangle $T'$ containing $P'$ as depicted in \Cref{fig:quadrilateral lemma}.

Consider the foliation of $T'$ by rays emanating from $x$. Let $\eta:[0,1] \to \overline{p_0' p_{n-1}'}$ be a parameterization such that $\eta(0)=p_0'$ and $\eta(1)=p_{n-1}'$. For $t \in [0,1]$, let $\nu(t)$ denote the slope of the segment $\overline{x \eta(t)}$. By a standard Euclidean geometry exercise, $\nu$ is a function from $[0,1]$ to $[\mu(0), \mu(1)]$ which is strictly increasing and satisfies $\nu(0)=\mu(0)$ and $\nu(1)=\mu(1)$. 

We use surgery. Let $X=(\R^2 \setminus P') \cup P$, where we reglue $P$ in place of $P'$ being sure to glue $\overline{p_0 p_{n-1}}$ to segment $\overline{p_0' p_{n-1}'}$ by the map
$$\gamma(t) \mapsto \eta\big(\nu^{-1} \circ \mu(t)\big),$$
and gluing the other edges more freely by arbitrary homeomorphisms.
The triangle $T'$ remains a triangle in $X$, and so by \Cref{triangle2} there is a foliation $\sG$ of $T$ by leaves through $x$. By construction the leaf of slope $\mu(t)$ enters $P$ through the point $\gamma(t)$ and so the restriction of $\sG$ to $P$ realizes the desired foliation of $P$.

To see the last statement, recall that on a zebra surface, any two foliations $\sF_m$ and $\sF_{m'}$ with $m \neq m'$ are transverse. Since interiors of edges are segments of leaves, to show an edge is transverse to $\sG$, it suffices to prove that the interior of the edge is not contained in a leaf of $\sG$. Observe that $\ell_m \cap P$ is connected by \Cref{no returning trails}. First consider the leaf $\ell_{\mu(0)}$ which passes through $p_0$. By considering a stellar neighborhood at $p_0$ we see that $\ell_{\mu(0)} \cap P=\{p_0\}$ if $\mu(0)<m_0$. Also if $\mu(0)=m_0$ then $\ell_{\mu(0)} \cap P=\overline{p_0 p_1}$ because the interior angles at $p_0$ and $p_1$ are both less than $\pi$. Thus the only edge that $\ell_{\mu(0)}$ can contain is $\overline{p_0 p_1}$ and only if $\mu(0)=m_0$. Similarly, the only edge $\ell_{\mu(1)}$ can contain is $\overline{p_{n-2} p_{n-1}}$ and only if $m_{n-2}=\mu(1)$. Now consider a leaf $\ell_{\mu(t)}$ with $t \in (0,1)$. By construction $\ell_{\mu(t)}$ intersects $\gamma(t)$ and because $\mu(t)$ is distinct from $m_{n-1}$, it crosses $\overline{p_{n-1} p_0}$ transversely and enters the interior of $P$. Then \Cref{trail through interior} guarantees that $\ell_{\mu(t)} \cap \partial P$ contains only two points and so $\ell_{\mu(t)}$ cannot contain any edges. We've shown that $\overline{p_0 p_1}$ and $\overline{p_{n-2} p_{n-1}}$ are the only edges that can fail to be transverse to $\sG$, and only in the circumstances allowed in the statement of the lemma.
\end{proof}

\section{Connecting points with trails}
\label{sect:connecting points with trails}

\subsection{Trail rays}
\label{sect:trail rays}

Let $Z$ be a zebra plane. A {\em trail ray} (or simply {\em ray}) with initial point $p \in Z$ is an arc of a trail with a parameterization of the form $\gamma:[0,+\infty) \to Z$ with $\gamma(0)=p$ which is maximal with respect to the subarc partial order among all such arcs of trails with initial point $p$. \Cref{trails exist} guarantees that any arc of a trail starting at $p$ can be extended to a ray (simply by extending to a trail and discarding the portion before $p$).

Let $\Sigma$ denote the set of singularities of $Z$, let $\gamma:[0, +\infty) \to Z$ be a trail ray, and let $p=\gamma(0)$. Then $J=\gamma^{-1}\big(Z \setminus (\Sigma \cup \{p\})\big)$ is an open subset of $\R$. For $I$ a connected component of $J$, we call $\ell=\gamma(I)$ a {\em leaf} of $\gamma$. It is either a leaf contained in the image of $\gamma$ or a connected component of a leaf with $p$ removed. Trail rays and their leaves are oriented away from $p$.

\compat{Barak rightly pointed out that there are issues with leaves and endpoints. {\bf Leaves should not contain singular endpoints}, but in earlier versions I had the opposite convention. If you notice anything amiss, please point it out.}

Fix a point $p$ and consider the collection $R_p \subset Z$ of all trail rays leaving $p$. We will show:
\begin{theorem}
\label{union of rays}
The union of all rays in $R_p$ is an open subset $U_p \subset Z$ containing $p$.
Let $U_p^\ast=U_p \setminus (\Sigma \cup \{p\})$. Then the collection $\sR_p$ of all leaves of rays leaving $p$ is an oriented foliation of $U_p^\ast$. \compat{Added oriented here. Dec 5, 2022. Thanks Barak. Orientation is mentioned in the previous paragraph.}
\end{theorem}

We remark that if $Z$ is convex, then $U_p$ will equal $Z$ because any point can be connected to $p$ by a trail.

The oriented foliation $\sR_p$ of $U_p^\ast$ should be considered to be an oriented singular foliation, but with a singular structure which is different from that of the singular foliations defined in \Cref{sect:singular foliations}. The initial point $p$ is special: the leaves agree locally with the stellar neighborhood foliation with leaves oriented outward. At the singularities $s \in (\Sigma \cap U_p) \setminus \{p\}$, the foliation has exactly one leaf $\ell_s$ oriented towards $s$. (If there were two such leaves, we'd get a contradiction to \Cref{no bigons}.) Nearby leaves pass by the singularity, and leaves emanating from $s$ making angle with $\ell_s$ of at least $\pi$ on each side are oriented away from $s$. See the left side of \Cref{fig:ray foliation}. \compat{This paragraph was rewritten on Dec 5, 2022.}

\begin{figure}[htb]
\centering
\includegraphics[height=1.6in]{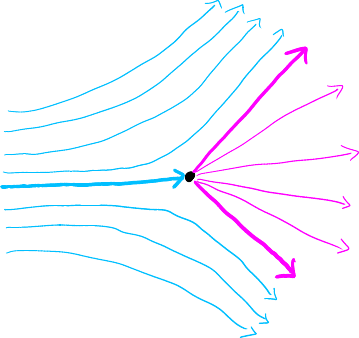}
\hspace{0.5in}
\includegraphics[height=1.6in]{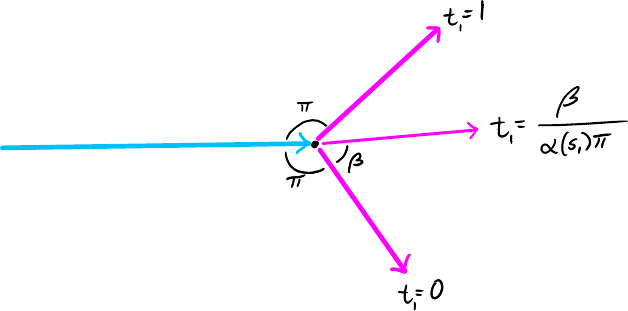}
\caption{{\em Left:} The foliation near a singularity by rays as in \Cref{union of rays}. Here the light blue curves represent rays passing near or hitting the singularity (bold), and the pink curves are emanating from the singularity and can be viewed as trail continuations of the ray hitting the singularity. {\em Right:} The correspondence between values of $t_1 \in [0,1]$ and angles for continuations through a singularity $s_1$. \compat{The right side of the figure was added on Dec 11, 2022 in response to a request of Ferran.}}
\label{fig:ray foliation}
\end{figure}

We'll say a simple path $\alpha:[0,1] \to U_p \setminus \{p\}$ is a {\em leaf transversal to $\sR_p$} if $\alpha\big((0,1)\big)$ is contained in a leaf of some directional foliation $\sF_m$ of $Z$ but not contained in a leaf of $\sR_p$. \commb{Doesn't simple follow from the leaf transversal definition?} \compat{I'm worried about backtracking. Just because the image is a path, doesn't meant hat the parameterization is simple.}

There is a natural slope map $\mu:U_p \setminus \{p\} \to \hat \R$
that assigns to each point $q \in U_p^\ast$ the slope of the leaf of $\sR_p$ through $q$. If $s$ is a singularity in $U_p \setminus \{p\}$, then we define $\mu(s)$ to be the slope of the unique leaf in $\sR_p$ that has $s$ as an endpoint and is oriented towards $s$. We prove the foliation associated to $R_p$ is monotonic in the following sense.

\begin{theorem}
\label{union of rays2}
If $\alpha:[0,1] \to U_p$ is a leaf transversal and is parameterized such that rays cross from the left side of $\alpha$ to the right, then the function
$$[0,1] \to \hat \R; \quad t \mapsto \mu \circ \alpha(t)$$
is continuous and (cyclically) strictly increasing.
\end{theorem}

To prove these theorems we will give an {\em address} to every ray in $R_p$.
The directions of straight line paths leaving $p$ can be parameterized by an angle in $\R/\big(\alpha(p)+2\big)\pi \Z$. We fix an identification between leaves emanating from $p$ and this circle. Then, every ray must initially travel in one of these directions, call it $\theta_0$. If the ray never hits a singularity, then $\theta_0$ is the complete address. Otherwise, it hits a singularity $s_1$ with cone angle $\big(\alpha(s_1)+2\big) \pi$. (Because $s_1$ is a singularity in a zebra plane, $\alpha(s_1) \geq 1$.)
There is a nondegenerate closed interval of directions in which the ray can continue, of total angle $\alpha(s_1) \pi$, because of the angle condition. Apply an affine change of coordinates so this interval becomes $[0,1]$. That is, $t_1 \in [0,1]$ represents continuing through $s_1$ in a direction so that the counterclockwise angle from the initial leaf leaving $p$ and the continuation through $s_1$ is $\pi + t_1 \alpha(s_1) \pi$; see the right side of \Cref{fig:ray foliation}. Continuing along our ray, we may encounter more singularities. We therefore have provided our ray with one of the following three kinds of addresses:
\begin{equation}
\label{eq:ray addresses}
(\theta_0;), \quad (\theta_0;t_1, t_2, \ldots, t_k), \quad \text{or} \quad (\theta_0;t_1, t_2, \ldots).
\end{equation}
In the first case, the ray never hits a singularity. In the second case, the ray hits a finite number of singularities (namely, $k$) and then follows a separatrix. In the third case, the ray follows an infinite sequence of saddle connections.

A {\em truncated address} is given by the choice of $\theta_0$ and a possibly empty finite sequence $t_1, t_2, \ldots, t_j$ which can be extended to a full address of a ray. The truncated addresses correspond to leaves of the directional foliations making up a ray. Namely, we reach the leaf by following a ray under the instructions given by the truncated address. So $(\theta_0; t_1, t_2, \ldots, t_j)$ corresponds to the leaf at obtained by first moving in direction $\theta_0$ away from $p$, we then hit a sequence of $j$ singularities $s_1, \ldots, s_j$ and in each case move in the direction as described by $t_j$. The leaf followed after crossing $s_j$ is the leaf
$L(\theta_0; t_1, t_2, \ldots, t_j)$ referred to by the truncated address.

\begin{proof}[Proof of \Cref{union of rays} and \Cref{union of rays2}]
Let $U_p$ denote the union of rays in $R_p$. Since $p$ has a stellar neighborhood, $p$ lies in the interior of $U_p$.

Now consider a nonsingular point $q \in U_p \smallsetminus \{p\}$. Then $q$ lies on one of the leaves $L(\theta_0; t_1, t_2, \ldots, t_j)$, where the truncated address can be extended to be the address of a ray. We consider several cases. Some arguments are used more than once, so we highlight some ideas for later use in bold. See \Cref{fig:trail rays open} for illustrations of these named arguments.

{\bf Argument 1.}
First, suppose the leaf through $q$ has the form $L(\theta_0;)$. We claim that the leaves of the form $L(\theta;)$ foliate a neighborhood of $q$. To see this, fix a slope $m$ distinct from the slope of $\overline{pq}$, and let $\alpha$ be the leaf of slope $m$ through $q$. Then by two applications of \Cref{burp} we can find points $r$ and $s$ on $\alpha$ such that $q \in \overline{rs}$ and such that $\overline{pr}$ and $\overline{ps}$ are segments of leaves. Then $\triangle prs$ is a triangle. Since $\alpha$ was a leaf, there are no singularities on $\overline{rs}$ and we can construct a quadrilateral $rr' s' s$ such that $\overline{rr'}$ extends $\overline{pr}$ and $\overline{ss'}$ extends $\overline{ps}$ using \Cref{trapezoid construction}, forming a larger triangle $\triangle pr's'$ that contains $q$ in its interior. Then \Cref{triangle2} guarantees that the leaves through $p$ foliate this triangle, covering a neighborhood of $q$ as desired. It also follows by applying \Cref{triangle2} to $\triangle prs$ that the function $\mu$ is continuous and monotonic on $\alpha$ as described in \Cref{union of rays2} on $\overline{rs}$. \commb{This last sentence also holds in the other cases involving regular points, but you don't mention this.}
\compat{I edited this last sentence to make it more concrete, and added a sentence to the end of Argument 2 below. I think the other cases were fine.}

\begin{figure}[htb]
\centering
\includegraphics[width=5in]{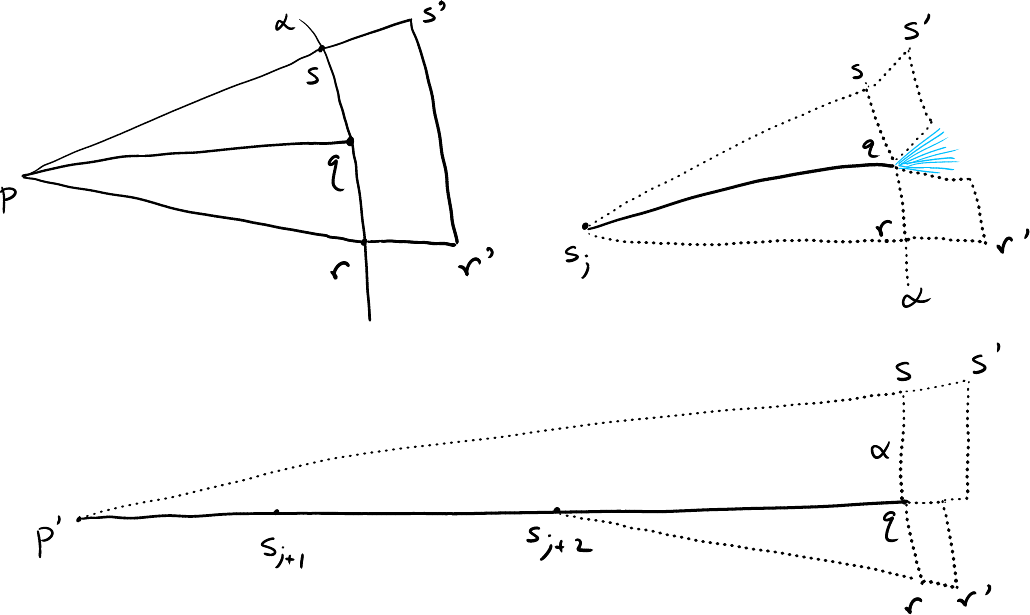}
\caption{Figures of the arguments for the proof of \Cref{union of rays}. Top left: Argument 1; Top right: Argument 3; Bottom: Argument 2 with $k=2$.}
\label{fig:trail rays open}
\end{figure}

\commb{I agree that
if English
was logical,
“First”
should be
continued
with
“Second”,
but this
doesn't
sound right. “Secondly”
is better, then ``Thirdly'',
etc.}
\compat{I added comments to First, Second and Third. It seems to be grammatically correct to use either First, Second, Third or Firstly, Secondly or Thirdly. I prefer not to use the extraneous ``ly'' for reasons explained here: \url{https://www.grammarly.com/blog/firstly/}. In particular, it would be wrong to add ``ly'' to only some of these words and not others.}

Second, assume that the address of the leaf through $q$ is $L=L(\theta_0; t_1, t_2, \ldots, t_j)$ where $t_j \not \in \{0, 1\}$. The leaf with this address starts at some singularity, call it $s_{j-1}$, and the leaf $L$ continues $\overline{s_{j-1} q}$. Repeating Argument 1 with $s_{j-1}$ playing the role of $p$ gives a foliation of a triangle containing $q$ in its interior by leaves emanating from $s_{j-1}$. The slopes of these leaves vary continuously and monotonically by \Cref{triangle2}, and so there is a neighborhood of $q$ in which points all lie on leaves with addresses $L(\theta_0; t_1, t_2, \ldots, t_j')$ with $t_j'$ near $t_j$.

{\bf Argument 2.}
Third, assume that the address of the leaf through $q$ is $L=L(\theta_0; t_1, t_2, \ldots, t_j, 1, 1, \ldots, 1)$
where there are exactly $k$ ones at the end of the address. (It can be that $j=0$, in which case the portion of the address after $\theta_0$ is all ones.) Let $p'=p$ if $j=0$ and let $p'=s_j$ otherwise. Then the trail visits $p'$ and then a sequence of $k$ singularities, making an angle of $\pi$ on the left as the trail moves through each of these singularities. Let $\alpha$ be a leaf through $q$ with a different slope than $L$. Repeating half of Argument 1 for the portion of $\alpha$ reachable from $\overline{s_{j+k} q}$ by making a right turn at $q$, gives a $\triangle s_{j+k} rq$. We can extend this triangle by a trapezoid through edge $\overline{qr}$. This triangle is foliated by leaves through $s_{j+k}$, and thus points in the triangle have addresses of the form
$L=L(\theta_0; t_1, t_2, \ldots, t_j, 1, 1, \ldots, t_{j+k}')$ for $t_{j+k}'$ in some interval of the form $[a,1]$. This foliates a half-neighborhood of $q$. On the other side of $\alpha$, because the angles at $s_{j+c}$ are $\pi$ for all $c \in \{1, \ldots, k\}$ we can apply \Cref{burp} to construct $\triangle p' q s$, where $\overline{qs} \subset \alpha$. We can extend the triangle through $\overline{qs}$, and foliate the triangle again. Points sufficiently close to $q$ in this triangle lie on leaves whose addresses have the form $L(\theta_0; t_1, t_2, \ldots, t_{j-1}, t_j')$ where $t_j' > t_j$. This gives a foliation of the second half of the neighborhood of $q$. \Cref{triangle2} can again be applied to
$\triangle p'qs$ and $\triangle p' r q$ to deduce continuity and monotonicity of $\mu$ on $\alpha$. \compat{Added this last sentence in response to Barak's comment. Jan 1, 2023.}

Fourth, if the address of the leaf through $q$ is $L=L(\theta_0; t_1, t_2, \ldots, t_j, 0, 0, \ldots, 0)$,
then we can repeat a symmetric version of Argument 2. (In fact, an orientation reversing map of the underlying surface has the effect of changing each $t_i$ to $1-t_i$.)

{\bf Argument 3.}
The above four paragraphs handle all possible addresses of regular points. A singular point $q \in U_p$ appears as the endpoint of some leaf $L=L(\theta_0; t_1, t_2, \ldots, t_j)$, i.e., $L=\overline{s_j q}$. The point $q$ has a stellar neighborhood $N$ with total angle being $n \pi$ where $n=\alpha(q)+2$. Say that the angle coordinate of a point $x \in N$ is the angle made with $L$: $\measuredangle x q s_j$.
Arguing as in the previous paragraphs, we can foliate neighborhoods of $q$ with total angle $\pi$, namely points sufficiently close to $q$ with angle coordinates in $[0,\pi]$ or $[(n-1)\pi,n\pi]$. Points in the stellar neighborhood whose angle coordinates lie in $[\pi, (n-1)\pi]$ lie on leaves with addresses of the form $L=L(\theta_0; t_1, t_2, \ldots, t_j, t_{j+1}')$ for some $t_{j+1}' \in [0,1]$. Since this is a stellar neighborhood, we have foliated a neighborhood of $q$ as desired.
\end{proof}

\subsection{Polygonal convexity}
The goal of this section is to prove:

\begin{theorem}
\label{thm:polygonal convexity}
A polygon $R$ in a zebra plane is convex if and only if all of the exterior angles of $R$ are at least $\pi$.
\end{theorem}

We will prove the ``if'' part of this result first. The following complements \Cref{union of rays}, giving a subspace in which the union of rays is closed.

\begin{lemma}
\label{lem:closed}
Let $K$ be a compact topological disk in a zebra plane and suppose that $p \in K$. For a ray $r \in R_p$, let $g_r$ denote the connected component of $r \cap K$ containing $p$, which is an arc of a trail from $p$ to a point on $\partial K$ or possibly just $\{p\}$ if $p \in \partial K$. Then the union $G=\bigcup_{r \in R_p} g_r$ is closed.
\end{lemma}
\begin{proof}
Let $q_i \in G$ be a sequence approaching some $q \in K$. We will prove that $q \in G$.
Since $p \in G$, we can assume $q \neq p$. Because singularities are isolated, we can assume (by passing to a subsequence) that each $q_i$ is nonsingular.

Let $L_i$ denote the leaf containing $q_i$ in a trail through $p$ and $q_i$. By compactness, there are only finitely many singularities in $K$, so up to passing to a subsequence, we can assume that each $L_i$ either passes through $p$ or starts at the same singularity $s \in K$. In the first case, the truncated address of each $L_i$ has the form $(\theta_i;)$. In the latter case, these addresses have the form $(\theta, t_1, t_2, \ldots, t_k, x_i)$ for some $x_i \in [0,1]$, where $\theta$ and $t_1, \ldots, t_k$ are fixed.
We'll continue the argument in the case of $(\theta, t_1, t_2, \ldots, t_k, x_i)$, but the same argument applies in the other case as well (taking $s=p$).

We may further assume by passing to a subsequence, that the sequence $(x_i)$ is monotonic. We may assume without loss of generality that it is increasing.
Let $x = \lim_{i \to \infty} x_i \in [0,1]$.

Suppose first that $x=x_i$ for some $i$. Then the sequence $(x_i)$ would be eventually constant and therefore for $i$ large enough we'd have $q_i$ and $q$ on the same leaf. In this case, we can fix $g_r$ that contains the
$q_i$ that lie on the same leaf as $q$, and since $g_r$ is closed we'd have $q \in g_r \subset G$ as well.

If $(x_i)$ is not eventually constant, we can pass to a subsequence such that the sequence $(x_i)$ is strictly increasing and converging to $x$. Thus in particular $x>0$. Since there are only finitely many singularities in $K$, we can find an $a \in [0,x)$ such that for any $y \in (a,x)$, the leaf $L(\theta; t_1, t_2, \ldots, t_k, y)$ does not terminate in a singularity in $K$. We will assume by passing to a further subsequence that each $x_i \in (a,x)$.

Consider the trail ray $r_\infty$ with address $(\theta; t_1, t_2, \ldots, t_k, x, 0, 0, 0, \ldots)$, where we add as many zeros as possible. We claim that following $r_{\infty}$ we hit $q$ before leaving $K$. To see this observe that all the arcs of trails from $p$ to $q_i$ follow the same path up to $s$ as $r_\infty$, because of the common start to their coding. The idea is that the segments $\overline{s q_i}$ accumulate on a segment ending at $q$ from the ray $r_\infty$, and since $K$ is compact it follows that this segment is in $K$ and therefore the connected component of $r_\infty \cap K$ that contains $p$ also contains $q$. See \Cref{fig:approach}.

\begin{figure}[htb]
\labellist
\small\hair 2pt
 \pinlabel {$s$} [ ] at 2 49
 \pinlabel {$q$} [ ] at 271 73
 \pinlabel {$q_i$} [ ] at 263 13
\endlabellist
\centering
\includegraphics[scale=1.0]{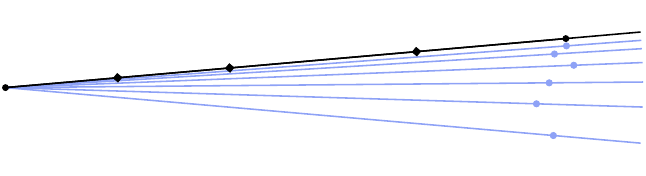}
\caption{The end of the argument in \Cref{lem:closed}. The squares denote locations of possible singularities along $r_\infty$.}
\label{fig:approach}
\end{figure}

To formalize the argument, observe that by \Cref{proper}, the trail ray $r_\infty$ must eventually exit the compact set $K$. Since $K$ is compact and the singularity set is discrete and closed, we can choose a point $y$ in the first open interval of $r_\infty \smallsetminus (K \cup \Sigma)$ after $g_{r_\infty}$. Since $K$ is compact, we can construct a stellar neighborhood $N$ of $y$ that is disjoint from $K$. Let $z$ be a point in $N$ such that the angle at $y$ from $\overline{p y} \subset r_\infty$ to $\overline{yz}$ is $\frac{\pi}{2}$. Then using \Cref{burp} we can construct a triangle $T=\triangle s z'y$ where $z' \in \overline{yz} \setminus \{y\}$. Using \Cref{triangle2} we can foliate $T$ by leaves emanating from $s$. The leaf space is homeomorphic to the interval $\overline{z'y}$, and the leaf
$\overline{st}$ with $t \in \overline{z' y}$ has slope that is continuous and strictly increasing as $t$ moves from $z'$ to $y$. For $i$ sufficiently large, $\overline{sq_i}$ extends to a leaf of this foliation of $T$, and because the $x_i$ from the addresses vary affine linearly with slope, these leaves must approach the edge $\overline{s y}$ of $T$
as $i \to \infty$. Thus we must have that $\overline{sq} \subset K$. (Every point on $\overline{sq}$ is approached by a sequence of points from $\overline{sq_i}$ with $i \to \infty$.) It follows that $\overline{sq} \subset g_{r_\infty}$ and therefore $q \in G$ as desired. \compat{This paragraph was simplified a bit by ideas of Barak. Jan 1, 2023.}
\end{proof}

\begin{proof}[Proof of \Cref{thm:polygonal convexity}]
Let $P$ be a polygon with all exterior angles at least $\pi$.
We will show that there is a trail in $P$ between any two points $p,q \in P$ by showing that $G=P$, where as in \Cref{lem:closed}, $G$ is the union over all rays $r \in R_p$ of the connected component of $r \cap P$ containing $p$.
By this lemma, $G$ is closed. By \Cref{intersections with polygons}, we have $g_r=r \cap P$. Thus $G$ coincides with the intersection of $P$ with the union of rays, which is open as a subset of $P$ by \Cref{union of rays}. Since $P$ is connected and $G$ is nonempty and both open and closed as a subset of $P$, we have that $G=P$. This proves the ``if'' part of the statement. \compat{Rephrased the first few sentences to make the logic more clear as suggested by Barak. Jan 1, 2023.}

To prove the converse, suppose $P$ is a polygon, and $a$, $b$, and $c$ are consecutive vertices in counterclockwise order such that the measure of the external angle $\measuredangle abc$ is less that $\pi$. Then we can find a slope $m$ and an open segment of a leaf or trail through $b$ of constant slope $m$ that is contained in the interior of $P$ except for touching the boundary at $b$. Considering a chart near $b$ of the singular foliation $\sF_m$, we can find segments of leaves of $\sF_m$ that join $\overline{ab}$ to $\overline{bc}$ that enter the complement of $P$; see \Cref{fig:convex_only_if}.
Let $a' \in \overline{ab}$ and $c' \in \overline{bc}$ be the intersections of such a leaf with these two edges of $P$. We see that $P$ cannot be convex, because there can be at most one arc of a trail from $a'$ to $c'$ by \Cref{no bigons}, but by construction $\overline{a' c'}$ exits $P$.
\end{proof}

\begin{figure}[htb]
\centering
\includegraphics[width=2.5in]{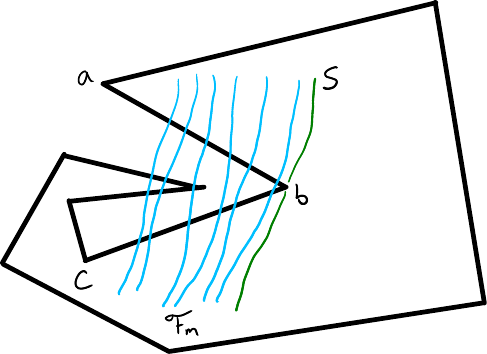}
\caption{The ``only if'' part of the proof of \Cref{thm:polygonal convexity}.}
\label{fig:convex_only_if}
\end{figure}

This theorem has a nice consequence for zebra tori without singularities:

\begin{corollary}
\label{torus cover convex}
Suppose $S$ is a closed surface with a zebra structure such that $\alpha$ is nonpositive. Let $S^\plus$ denote $S$ with its singularities removed. If $S$ has two closed leaves which are not homotopic in $S^\plus$, then
the PRU cover $\tilde S$ is convex.
\end{corollary}

\begin{remark}
The condition about the existence of nonhomotopic closed leaves is necessary.
Consider a Hopf torus, $\C^\ast/\langle z \mapsto \lambda z\rangle$, as described in \Cref{dilation with cone singularities}.
The foliation by lines through the origin descends to a foliation of a Hopf torus by homotopic closed leaves.
The universal cover of such a torus is also the universal cover of $\C^\ast$, which is not convex, so these are the only closed leaves. \compat{Edited on Oct 25, because it duplicated some discussion from the intro.}
\end{remark}

\begin{proof}[Proof of \Cref{torus cover convex}]
By the \gaussbonnet, $S$ is either a torus with $\alpha$ identically zero, or a sphere with four poles.

Consider the case of the sphere. Each of our closed leaves is simple and so bounds a pair of disks, and by the \gaussbonnet, each disk must contain two poles. If the two curves were disjoint, then they'd have to bound an annulus and would therefore be homotopic. Thus, they must intersect. The sphere has the torus as a double cover branched over the fours poles, and we can lift the zebra structure to the torus and our closed leaves to intersecting closed curves on the torus. The foliations on the torus are orientable, so the geometric intersection number coincides with the algebraic intersection number. The lifted closed leaves on the torus cannot be homotopic, and
thus it suffices to discuss the case of the torus.

\compat{The part of the proof from this paragraph to the end was modified to address an error pointed out by Ferran. Dec 12, 2022.}
Now let $\beta$ and $\gamma$ be the two nonhomotopic closed leaves for a zebra structure on the torus without singularities. Since these two curves are nonhomotopic simple closed curves, they must intersect. It follows that their slopes are distinct, and the intersection points are isolated.
Let $p_0 \in \beta \cap \gamma$. Parameterize $\beta:[0,1] \to S$ and $\gamma:[0,1] \to S$ such that both curves start and end at $p_0$.
Choose a lift $\tilde p_0 \in \tilde S$, and select lifts $\tilde \beta$ and $\tilde \gamma$ starting at $\tilde p_0$. Let $\Delta_\beta$ and $\Delta_\gamma$ denote the deck transformations that carry $\tilde p_0$ to the endpoint of $\tilde \beta$ and $\tilde \gamma$, respectively. Consider the union of the four segments of leaves
\begin{equation}
\label{eq:parallelogram}
\tilde \beta \cup \Delta_\beta(\tilde \gamma) \cup \Delta_\gamma(\tilde \beta) \cup \tilde \gamma.
\end{equation}
We claim that these four segments of leaves bound a parallelogram in $\tilde S$, and the segments are edges listed in cyclic order. Indeed consecutive segments above are not parallel and therefore, by \Cref{no bigons}, intersect at exactly one point.
Namely,
$\tilde \gamma \cap \tilde \beta = \{p_0\}$,
$\tilde \beta \cap \Delta_\beta(\tilde \gamma) = \{\Delta_\beta(p_0)\}$,
$$
\Delta_\beta(\tilde \gamma) \cap \Delta_\gamma(\tilde \beta) = \{\Delta_\beta \circ \Delta_\gamma(\tilde p_0)=\Delta_\gamma \circ \Delta_\beta(\tilde p_0)\}, \quad
\text{and}
\quad
\Delta_\gamma(\tilde \beta) \cap \tilde \gamma = \{\Delta_\gamma(p_0)\}.$$
Opposite edges must be disjoint: If they weren't disjoint then because they are parallel, they'd have to lie in the same leaf, but then one of the other edges would join a leaf to itself violating \Cref{no bigons}. Then \Cref{building a trapezoid} guarantees that the four curves in \eqref{eq:parallelogram} bound a parallelogram $P_1 \subset \tilde S$ as claimed above.

Observe that $\langle \Delta_\beta, \Delta_\gamma \rangle$ is a finite index subgroup of the fundamental group of the torus, that $\tilde S/\langle \Delta_\beta, \Delta_\gamma \rangle$ is a torus covering $S$, and that the parallelogram $P_1$ constructed above is a fundamental domain for the action of $\langle \Delta_\beta, \Delta_\gamma \rangle$ on $\tilde S$. For each $n \geq 1$, define
$$P_{2n+1}=\bigcup_{-n \leq i \leq n} \bigcup_{-n \leq j \leq n} \Delta_\beta^i \circ \Delta_\gamma^j(P_1).$$
Each $P_{2n+1}$ is a parallelogram formed from $(2n+1)^2$ copies of $P_1$ and bounded by four lifts of $2n+1$-fold covers of $\beta$ and $\gamma$. Since $P_1$ is a fundamental domain for the action of $\langle \Delta_\beta, \Delta_\gamma \rangle$ on $\tilde S$, it follows that $\bigcup_n P_{2n+1}=\tilde S$. Each $P_{2n+1}$ is convex by \Cref{thm:polygonal convexity}, and so $\tilde S$ is convex: If $x,y \in \tilde S$, then there is an $n$ such that $x,y \in P_n$ and so $\overline{xy}$ exists.
\end{proof}

\subsection{Continuity of arcs of trails}
\label{sect:continuity of trail arcs}
If $X$ is a topological space, then we use $\Cl(X)$ to denote the collection of all closed subsets of $X$ with its
{\em Fell topology}, which has a subbase given by the sets,
$$V^\cap = \{A \in \Cl(X):~A \cap V \neq \emptyset\} \quad \text{and} \quad K^{\not \cap}=\{A \in \Cl(X):~A \cap K= \emptyset\}$$
taken over all nonempty open sets $V \subset X$ and all proper compact subsets $K \subset X$. If $X$ is locally compact second countable, then $\Cl(X)$ is compact and metrizable \cite[Theorem 5.1.5]{Beer}. In particular, this holds when $X$ is a surface. \compat{In a prior version $\Cl(X)$ did not contain $\emptyset$. Including $\emptyset$ makes $\Cl(X)$ compact. Thanks Barak! I don't think this change affects anything, but I'm pointing this out in case you notice something.}

For a zebra plane $Z$, consider the set
$$\TS = \big\{\{p\}:p \in Z\big\} \cup \big\{\{x,y\}:~\text{$x,y \in Z$ are distinct points that can be joined by a trail arc}\}.$$
If $x$ and $y$ are distinct and $\{x,y\} \in \TS$, we use $\overline{xy}$ to denote the arc of a trail starting at $x$ and ending at $y$. A singleton $\{p\}$ can also be written $\{p,p\}$ and we define $\overline{pp}=\{p\}$. In this way we have associated each element of $\TS$ with a closed subset of $Z$.
\compat{There was some confusion (contributed to by minor errors in the exposition), so I tried to be more explicit about the definition of $\TS$. I also made some minor corrections to the exposition.}

\begin{theorem}
\label{segment continuity}
The set $\TS$ is an open subset of $Z^2$ modulo permutation. The function
$$\seg:\TS \to \Cl(Z); \quad \{x,y\} \mapsto \overline{xy}$$
is a homeomorphism onto its image.
\end{theorem}

\begin{lemma}
\label{approximating trail arcs}
If $\overline{xy}$ is an arc of a trail in $Z$ (possibly with $x=y$) then there is a decreasing sequence of convex polygons $P_n$ such that $\bigcap P_n=\overline{xy}$. In $\Cl(Z)$, we have $\lim P_n = \overline{xy}$.
\end{lemma}

\begin{proof}
In a locally compact and second countable space, a decreasing nested sequence of closed subsets approaches its intersection in the Fell topology. Thus the second sentence follows from the first. \compat{Inserted this here in response to Barak's comment that this is generally true. Previously, there was a paragraph at the end of the proof that proved this. Dec 23, 2022.}

First consider the case of $x=y$. Since $Z$ is a topological disk, $x$ has a countable neighborhood base which we can take to be nested, $U_1 \supset U_2 \supset \ldots$.
We produce our sequence of polygons by induction using \Cref{generalized rectangles} to produce generalized rectangles. First define $P_1$ to be a generalized rectangle contained in $U_1$ and containing $x$ in its interior. Then assuming that $P_k \subset U_k$ is generalized rectangle containing $x$ in its interior, define $P_{k+1}$ to be a generalized rectangle contained in $U_{k+1} \cap P_k^\circ$ and containing $x$ in its interior. To see why the resulting sequence $\{P_n\}$ approaches $\overline{xy}=\{x\}$, first observe that $\{x\} \in V^\cap$ implies $x \in V$ and so $P_n \in V^\cap$ for all $n$. Second observe that $\{x\} \in K^{\not \cap}$ implies there is a $U_N$ such that $U_N \cap K=\emptyset$ and therefore $P_n \in K^{\not \cap}$ for $n \geq N$. \compat{This paragraph was rewritten to make use of generalized rectangles on Nov 4.}

Now suppose that $x \neq y$. We will cover a neighborhood of $\overline{xy}$ by polygons of three types. Observe that that $\overline{xy}$ is the union of segments of leaves (or leaves) whose endpoints are in the union of $\{x,y\}$ and the collection of singularities in the interior of $\overline{xy}$. We'll call these segments {\em edges} in this proof. For each edge $e$ of $\overline{xy}$, use \Cref{trapezoid construction} to construct two rectangles with an edge of $e$, one on each side of $\overline{xy}$. These will be called {\em edge rectangles}. Now fix a singularity $z$ in the interior of $\overline{xy}$. Since $Z$ is a zebra plane, the cone angle at $z$ is at least $3 \pi$, and since $\overline{xy}$ is a trail, the angles made at $z$ are at least $\pi$. So, the constructed right angles of edge rectangles at $z$ cannot overlap, but also can't cover all of a neighborhood of $z$. The complement of these right angles consists of one or two positive angles (one occurs if one side of $\overline{xy}$ makes an angle of $\pi$ at $z$), which we will call {\em complementary angles}. Choose finitely many segments of leaves emanating from $z$ that cut these complementary angles into smaller angles all of whose measures are less than $\pi$. Fix two adjacent leaves emanating from $z$ bounding such an angle, and construct a triangle
with vertex $z$ such that two edges are arcs of these leaves, and the opposite angles are equal. We'll call these equal angles the {\em base angles} and these triangles {\em vertex triangles}. At endpoint $x$, the two edge rectangles with vertex $x$ cover an angle at $x$ of measure $\pi$. Using \Cref{trapezoid construction}, we can construct more rectangles to add to these two edge rectangles that cover a neighborhood of $x$ and have disjoint interiors. We call these rectangles {\em end rectangles} and also add them around $y$. A picture of this situation is depicted in \Cref{fig:approximating_tail_segments}.

\begin{figure}[htb]
\centering
\includegraphics[height=1in]{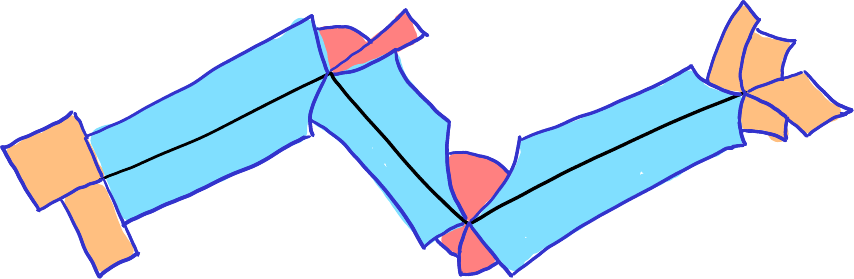}
\hspace{3em}
\includegraphics[height=1in]{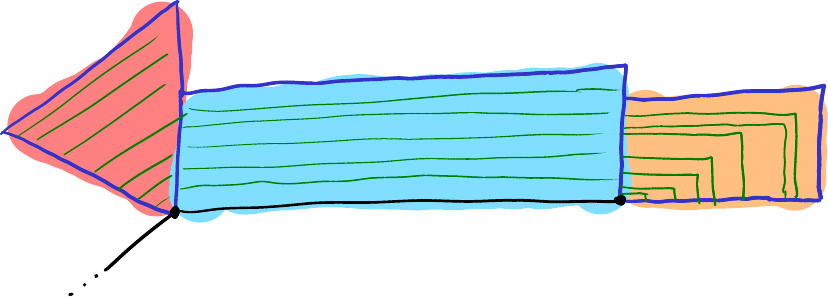}
\caption{Left: A neighborhood of a trail tiled by polygons as in the proof of \Cref{approximating trail arcs}. Edge rectangles are colored blue, vertex triangles are red, and end rectangles are orange. Right: Example foliations of the polygons, with the trail situated at the bottom of the polygons.}
\label{fig:approximating_tail_segments}
\end{figure}

We will foliate each of the polygons we have constructed, with the goal to enclose $\overline{xy}$ in a $1$-parameter family of boundaries of convex polygons that nest down to $\overline{xy}$. Each edge rectangle has one edge on $\overline{xy}$, and by \Cref{trapezoid foliation} we can foliate it by parallel leaves which run between the edges emanating from points on $\overline{xy}$. Each vertex triangle has $T$ one vertex $v$ that is a singularity on $\overline{xy}$. Let $m$ be the slope of the side opposite $v$. Restricting $\sF_m$ to $T$, we see from \Cref{no bigons} that maximal segments of leaves through the interior of $T$ must join the edges of the triangle emanating from $v$. We foliate the end rectangles in a slightly more technical way. These rectangles intersect $\overline{xy}$ in an endpoint (either $x$ or $y$), and suppose we are given a homeomorphism between the edges sharing this endpoint vertex, $h:e_1 \to e_2$ that preserves that endpoint (i.e., $h(x)=x$ or $h(y)=y$). A point $p \in e_1$, then picks out two segments of leaves of the directional foliations: the leaf through $p$ parallel to $e_2$ and the leaf through $h(p)$ parallel to $e_1$. By \Cref{trapezoid foliation}, these segments of leaves join distinct pairs of opposite sides and so intersect. Therefore, we can join $p$ to $h(p)$ by a path that follows these leaves, making the transition at the intersection point. (See the orange rectangle on the right side of \Cref{fig:approximating_tail_segments} for examples of such paths.)
\compat{This paragraph was edited on Dec 5, 2022 to reference the new \Cref{trapezoid foliation} and to provide more details on how we construct these polygons, as requested by Barak.}

In order to achieve our goal, we need the leaves of the foliations around $\overline{xy}$ to close up. (If we attempted to close the polygons using segments of say the boundaries of end rectangles, then the resulting curves might not bound convex polygons.) To ensure closing, we need to take advantage of the flexibility built into the foliations of the end rectangles. Pick one end rectangle with vertex $x$ and call it $R^\star$, and choose foliations as above, with arbitrary choices made for all end rectangles other than $R^\star$. Let $e^\star_1$ and $e^\star_2$ denote the edges of $R^\star$ with endpoint $x$.
Observe that there are closed intervals $I_j \subset e^\star_j$ with endpoint $x$ for $j \in \{1,2\}$ and a homeomorphism $h_\star:I_1 \to I_2$ such that $h_\star(x)=x$ and for $p \in I_1 \setminus \{x\}$ there is a concatenation of leaves joining $p$ to $h_\star(p)$. To foliate $R^\star$, extend $h_\star$ to a homeomorphism $e^\star_1 \to e^\star_2$, and use this homeomorphism to define the foliation as above. Observe that concatenations of leaves through our polygons passing through $I_1 \setminus \{x\}$ all close up.

We claim that the closed leaves constructed above bound convex polygons. To see this we can check that the interior angles all have measure less than $\pi$, by \Cref{thm:polygonal convexity} and the fact that all points in zebra planes have angle at least $2 \pi$. These interior angles only occur in the interiors of end rectangles, where all interior angles are right angles, and in the transition between two vertex triangles or between a vertex triangle and an edge rectangle. Observe that because our vertex triangles were constructed so that the base angles are equal, these base angles always measure less than $\frac{\pi}{2}$, while the contribution of edge rectangles to an interior angle is always $\frac{\pi}{2}$. So, all interior angles are less than $\pi$ as claimed and so each bounded polygon is convex.
\end{proof}

\begin{proof}[Proof of \Cref{segment continuity}]
To see that $\TS$ is open, let $\{x,y\} \in \TS$. Then $\overline{xy}$ exists, and \Cref{approximating trail arcs} guarantees that there is a convex polygon $P_1$ containing $\overline{xy}$ in its interior, $P_1^\circ$. Observe that convexity guarantees that $P_1^\circ \times P_1^\circ \subset \TS$, so $\TS$ is open.

To see $\seg$ is injective, let $A \in \seg(\TS)$. Observe that if $A$ consists of only one point, say $p$, then $\seg^{-1}(A)$ contains only $\{p\}$. Otherwise, $A$ is an arc and $\seg^{-1}(A)$ consists of the endpoints of $A$. (The endpoints of $A$ can be distinguished from the other points of $A$. If $x$ is an endpoint of $A$, then for any neighborhood $N$ of $x$, there is a smaller neighborhood $N'$ such that $N' \setminus A$ is connected. On the other hand, if $z \in A$ is not an endpoint, then there is a neighborhood $N$ of $z$ such that $N \setminus A$ is homeomorphic to the disjoint union of two open disks with $z$ in the common boundary. Therefore, any smaller neighborhood $N'$ must intersect both these disks and so $N' \setminus A$ is also disconnected.\compat{This comment has been expanded to say something correct! Thanks Ferran! Dec 12, 2022.})

Now consider the continuity of $\seg$. Fix $\overline{xy} \in \seg(\TS)$. Let $P_n$ be the sequence of convex polygons approaching $\overline{xy}$ guaranteed to exist by \Cref{approximating trail arcs}.
A basis for $\Cl(Z)$ is given by sets of the form
$$U = V_1^\cap \cap V_2^\cap \cap \ldots \cap V_n^\cap \cap K^{\not \cap}$$
(since finite unions of compact sets are compact). Suppose $U$ contains $\overline{xy}$.
We'll produce an open set containing $\{x,y\} \in \TS$ whose image is contained in $U$.

The case when $x=y$ is simple: Here we have that $x \in V_i$ for all $i$ and since $P_n \to \{x\}$ in $\Cl(Z)$, for $n$ large enough,
$P_n$ that is disjoint from $K$. The open set consisting of pairs from $P_n^\circ \cap \bigcap_{i=1}^n V_i$ suffices.

Now suppose that $x \neq y$. Then for any index $i \in \{1, \ldots, n\}$, we can find a segment of a leaf $\gamma_i \subset V_i$ such that $\gamma_i \cap \overline{xy}$ consists of a single point $p_i$ that is distinct from $x$ and $y$ and that the crossing at $p_i$ is transverse, in the sense that $\gamma_i$ crosses from one side of $\overline{xy}$ to the other at $p_i$.
This situation is depicted in \Cref{fig:segment_continuity}.
Let $E$ be the collection of endpoints of the collection of paths $\{\gamma_i:~i=1,\ldots, n\}$.
Then there is  a polygon $P_n$ that is disjoint from both $K$ and $E$.
Observe that by construction, $x$ and $y$ lie in distinct components of $P_n^\circ \setminus \gamma_i$.
Denote these components by $X_i$ and $Y_i$ respectively. Let
$X=\bigcap_{i=1}^n X_i$ and $Y=\bigcap_{i=1}^n Y_i$. For $x' \in X$ and $y' \in Y$, both points lie in $P_n^\circ$, so $\overline{x'y'}$ exists. Since $P_n$ is disjoint from $K$, so is $\overline{x'y'}$.
Furthermore, the path $\overline{x'y'}$ must intersect each of the $\gamma_i$ and so must intersect each $V_i$.
Thus $\overline{x' y'}$ is in our basis element $U$ as desired.

Finally, we need to show that $\seg^{-1}:\seg(TS) \to \TS$ is continuous. Let $U \subset \TS$ be open and pick any $(x_0, y_0) \in U$. We'll find an open neighborhood of $\overline{x_0 y_0}$ in $\Cl(Z)$ such that whenever this open set contains $\overline{xy}$, we have $\{x,y\} \in U$.

We consider two cases separately. First suppose $x_0=y_0$. Then there is an open subset $U' \subset Z$ containing $x_0$ such that $U' \times U' \subset U$. Using \Cref{approximating trail arcs}, we can produce a compact polygon $P$
containing $x_0$ in its interior which does not intersect $\partial U'$ and therefore is contained in $U'$. Observe that the set $(P^\circ)^\cap \cap (\partial P)^{\not \cap}$ is a neighborhood satisfying our condition.

Now assume that $x_0 \neq y_0$. In this case we can find disjoint open neighborhoods $U_x$ of $x_0$ and $U_y$ of $y_0$ such that the collection $U'$ of pairs $\{x,y\}$ with $x \in U_x$ and $y \in U_y$ satisfies $U' \subset U$.
Applying  \Cref{trapezoid construction}, we can produce produce a convex polygon $Q_x \subset U_x$ containing $x_0$ such that $\overline{x_0 y_0}$ passes through the interior of one of the edges of $Q_x$.
Define $K_x$ to be the union of the edges that do not intersect $\overline{x_0 y_0}$. Similarly define $Q_y \subset U_y$ and $K_y$. Consider the open subset of $\Cl(Z)$ defined by
$$(Q_x^\circ)^\cap \cap K_x^{\not \cap} \cap (Q_y^\circ)^\cap \cap K_y^{\not \cap}.$$
Fix an $\overline{xy}$ from this set. Observe that since $\overline{xy} \in (Q_x^\circ)^\cap$ it must contain points in $Q_x$. Because $K_x$ contains all but one edge of $Q_x$ and $\overline{xy} \in K_x^{\not \cap}$, if there was no endpoint in $Q_x$, then $\overline{xy}$ must enter and exit $Q_x$ through the same edge. This is impossible by \Cref{no bigons}, so $Q_x$ must contain either $x$ or $y$. By the same analysis, $Q_y$ must contain either $x$ or $y$. Thus $\{x,y\} \in U$ as desired.
\end{proof}

\begin{figure}[htb]
\centering
\includegraphics[width=3in]{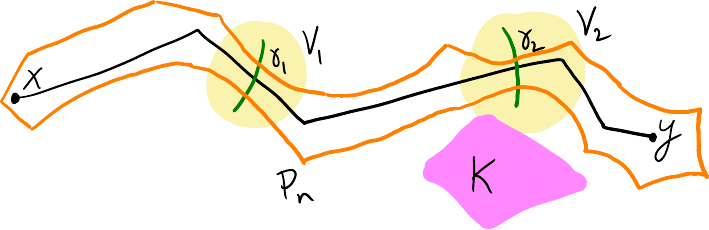}
\caption{Illustration of a configuration described in the proof of \Cref{segment continuity}.}
\label{fig:segment_continuity}
\end{figure}

\begin{corollary}
\label{compact segments}
If $K \subset \TS$ is compact, then so is the subset of $Z$ given by $\bigcup_{\{x,y\} \in K} \overline{xy}$.
\end{corollary}
\begin{proof}
Let $Q=\bigcup_{\{x,y\} \in K} \overline{xy} \subset Z$. Since $Z$ is metrizable, it suffices to prove that $Q$ is sequentially compact, so let $q_n \in Q$ be a sequence. Then for each $n$, we can find $\{x_n,y_n\} \in K$ such that $q_n \in \overline{x_n y_n}$. Since $K$ is compact, there is a subsequence $\{x_{n_k}, y_{n_k}\}$ that converges to some $\{x_\infty, y_\infty\} \in K$. By \Cref{approximating trail arcs}, we can find a convex polygon $P$ that contains $\overline{x_\infty y_\infty}$ in its interior. Then the interior $P^\circ$ intersects $\overline{x_\infty y_\infty}$ and the boundary $\partial P$ is disjoint from $\overline{x_\infty y_\infty}$. By \Cref{segment continuity}, we know that $\overline{x_{n_k} y_{n_k}}$ converges to $\overline{x_\infty y_\infty}$ in $\Cl(Z)$, and so there is an index $k'$ such that $k>k'$ implies that
$\overline{x_{n_k} y_{n_k}} \cap P^\circ\neq \emptyset$ and $\overline{x_{n_k} y_{n_k}} \cap \partial P=\emptyset$. Since $\overline{x_{n_k} y_{n_k}}$ is connected, we conclude that $\overline{x_{n_k} y_{n_k}} \subset P$ for $k>k'$.
Then since $P$ is compact, we know that $q_{n_k} \in \overline{x_{n_k} y_{n_k}}$ has a convergent subsequence.
\end{proof}


\begin{theorem}
\label{slope continuity}
The function $\mu$ that sends a pair of distinct points $(x,y) \in Z^2$ with $\{x,y\} \in \TS$ and $y$ nonsingular to the slope of $\overline{xy}$ measured at $y$ is continuous.
\end{theorem}
\begin{proof}
Let $I=(m_0, m_1) \subset \hat \R$ be an open interval. Suppose $\mu(x_0,y_0) \in I$. We need to find an open neighborhood of $(x_0,y_0)$ whose intersection with the domain maps into $I$.

Let $\alpha \subset \overline{x_0 y_0}$ be an arc on the same leaf as the leaf of $\overline{x_0 y_0}$ containing $y_0$, but such that $\alpha$ does not contain $y_0$. Using \Cref{trapezoid construction} twice, we can construct two rectangles with $\alpha$ as one edge, one on each side of $\overline{x_0 y_0}$. The union of these two rectangles is a larger rectangle $R$ such that $\alpha$ passes through the interior of $R$ and joins opposite sides. Recall that $R$ is convex and contains no singularities in its interior. Choose a point $z$ from the interior of $\alpha$, and through $z$ construct segments $\overline{ac}$ and $\overline{bd}$ contained in $R$ of slope $m_0$ and $m_1$, respectively. Let $Q$ denote the quadrilateral $abcd$. This construction is illustrated in \Cref{fig:slope_continuity}. Because $\mu(x_0,y_0) \in I$, direction considerations at $z$ tell us that $\overline{x_0 y_0}$ passes through opposite sides of $Q$, and up to swapping $a$ and $c$, we can assume these sides are $\overline{ab}$ and $\overline{cd}$. But $\overline{x_0 y_0}$ does not pass through $\overline{bc}$ or $\overline{cd}$. So, \Cref{approximating trail arcs} guarantees we can find a convex polygon $P$ containing $\overline{x_0 y_0}$ in its interior that does not intersect $\overline{bc} \cup \overline{ad}$. This lemma further allows us to ensure that $P$ contains no singularities not found on $\overline{x_0 y_0}$, since the singularities in any compact polygon form a compact set.
Observe that $x_0$ and $y_0$ lie in distinct components of $P^\circ \setminus Q$, call these components $X$ and $Y$. We claim the intersection of $X \times Y$ with the domain of $\mu$ maps into $I$. Fix $(x,y) \in X \times Y$. Then $\overline{xy}$ is forced to cross through edge $\overline{ab}$ into $Q$, then over both both $\overline{ac}$ and $\overline{bd}$, and exit through $\overline{cd}$. Assuming $\overline{xy}$ does not pass through $z$, $\overline{xy}$ forms a triangle with $\overline{ac}$ and $\overline{bd}$, and \Cref{triangle1} guarantees that the slope of $\overline{xy}$ measured in $Q$ is in $I$. (The same holds if it does pass through $z$ because parallel trajectories form such a triangle.) The slope is the same as the slope measured at $y$, because $Y$ cannot contain any singularities, because by construction all singularities lie in $X$. Thus $\mu(x,y) \in I$ as desired.
\end{proof}

\begin{figure}[htb]
\centering
\includegraphics[width=3.5in]{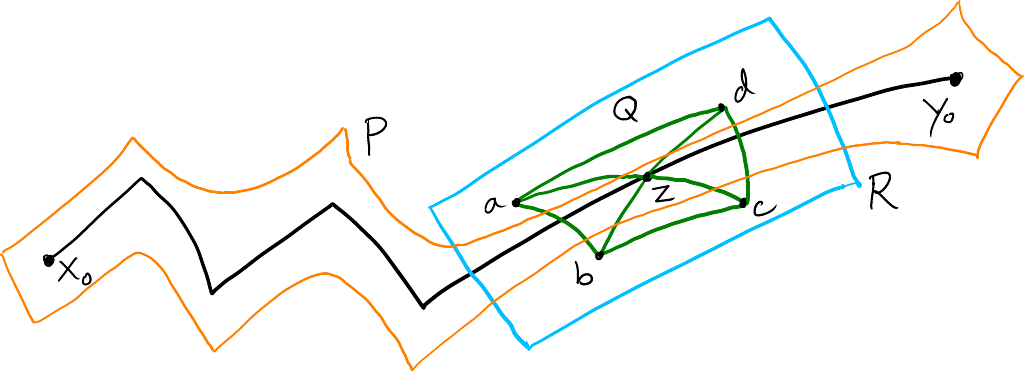}
\caption{Illustration for the proof of \Cref{slope continuity}.}
\label{fig:slope_continuity}
\end{figure}

\begin{theorem}
\label{compact extension}
\compat{This theorem and proof are new as of Nov 5.}
Let $K$ be a compact subset of the zebra plane $Z$. Suppose $\{x_n,y_n\} \in \TS$ is a sequence of pairs such that the limits
$x=\lim x_n$ and $y=\lim y_n$ exist. If $\overline{x_n y_n} \subset K$ for all $n$, then $\{x, y\} \in \TS$.
\compat{The first version on the arxiv said zebra surface here rather than zebra plane. This would be an error (at least the proof doesn't work. We only cite this once, and apply it to a zebra plane, so that seems okay.}
\end{theorem}
\begin{proof}
We offer a direct proof. Using \Cref{generalized rectangles}, for each point $p \in K$, there is a generalized rectangle $P_p$ such that $p \in P_p^\circ$ and $\partial P_p$ contains no singularities. Then $\{P_p^\circ:~p \in K\}$ is an open cover of $K$, so there is a finite subcover. Let $\sP$ denote the finite collection of polygons used in the finite subcover.

Let $Q_0 \in \sP$ denote a polygon containing $x$ in its interior. If $y \in Q_0$ then $\{x, y\} \in \TS$ by \Cref{thm:polygonal convexity}. Now suppose $y \not \in Q_0$. Then by discarding the first finitely many elements of our sequences, we may assume that $x_n \in Q_0^\circ$ and $y_n \not \in Q_0$ for all $n$. Therefore \Cref{intersections with polygons} guarantees that for all $n$ we have $Q_0 \cap \overline{x_n y_n} = \overline{x_n z^1_n}$ for some point $z^1_n \in \partial Q_0$. Since $\partial Q_0$ is compact, by passing to a subsequence, we can assume that $z^1 = \lim z^1_n \in \partial Q_0$ exists. Observe that
$\overline{x z^1}$ exists as an arc of a trail by \Cref{thm:polygonal convexity}.

Since $z_1^n \in \overline{x_n y_n}$ for all $n$, we see that $z^1 \in K$.
Let $Q_1 \in \sP$ be such that $z^1 \in Q_1^\circ$. If $y \in Q_1$ then we can construct the arc of a trail $\overline{z^1 y}$ by \Cref{thm:polygonal convexity}. In this case we claim that the concatenation $\overline{x z^1} \bullet \overline{z^1 y}$ is a trail.
Since both paths in the concatenation are arcs of trails, we only need to check the path at $z^1 \in \partial Q_0$. Since $\partial Q_0$ contains no singularities, we know $z^1$ is not singular. Similarly, the points $z^1_n$ are not singular. Since $z_n^1 \in \overline{x_n y_n}$, we have
$\overline{x_n y_n} = \overline{x_n z_n^1} \bullet \overline{z_n^1 y_n}.$
Let $m_n^1$ denote the slope of $\overline{x_n y_n}$ measured at $z_n^1$.
Applying \Cref{slope continuity} to the sequence $\{x_n, z_n^1\} \to \{x, z^1\}$ tells us that $m_n^1$ converges to the slope $m^1$ of $\overline{x z^1}$ measured at $z^1$.
Then applying the same result to $\{y_n,z_n^1\} \to \{y, z^1\}$ tells us that
$m^1$ is also the slope of $\overline{z^1 y}$ measured at $z^1$.
Therefore, $\overline{x z^1} \bullet \overline{z^1 y}$ is a trail, proving our claim.

Now assume that $y \not \in Q_1$. By discarding the first finitely many elements of our sequences we may assume that $z^1_n \in Q_1^\circ$ and $y_n \not \in Q_1$ for all $n$. Then \Cref{intersections with polygons} guarantees that for all $n$ we have $Q_1 \cap \overline{z^1_n y_n} = \overline{z^1_n z^2_n}$ for some point $z^2_n \in \partial Q_1$.
By passing to a subsequence we can assume that $z^2_n$ converges to some $z^2 \in \partial Q_1$. Then \Cref{thm:polygonal convexity} guarantees that we can construct $\overline{z^1 z^2}$. The argument from the previous paragraph guarantees that
$\overline{x z^1} \bullet \overline{z^1 z^2}$ is an arc of a trail $\overline{x z^2}$.

The previous two paragraphs can be repeated inductively, producing a sequence of polygons $Q_0, \ldots, Q_k$ in $\sP$. We also produce sequences $z_n^j \in \overline{x_n y_n} \cap \partial Q_{j-1}$ for $j=1, \ldots, k$ that converge to $z^j \in \partial Q_{j-1}$ such that there are arcs of trails $\overline{x z^j}$. At each such stage, either $y \in Q_k$ in which case we can construct a trail $\overline{xy}=\overline{x z^k} \bullet \overline{z^k y}$, or we can extend the sequence one more step (passing to a subsequence as we do). Since $\sP$ is finite, we must either produce the trail $\overline{xy}$ at some point, or some polygon must appear twice in our sequence. To prove the theorem, it suffices to show that no polygon can appear twice. Suppose to the contrary that $Q_j=Q_k$ where $j<k$. Temporarily fix $n$, which because of our passing to subsequences determines the points $z_n^1, \ldots, z_n^k$. In our construction, we define $z_n^{j+1}$ to be the point such that $Q_j \cap \overline{z_n^j y_n}=\overline{z_n^j z_n^{j+1}}$.  By construction we have the sequence of proper subsets $\overline{z_n^j y_n} \supset \overline{z_n^{j+1} y_n} \supset \ldots \supset \overline{z_n^k y_n}$. Thus $z_n^k \in \overline{z_n^j y_n} \setminus \overline{z_n^j z_n^{j+1}}.$
Now recall that by definition of $Q_k$ we have $z^k \in Q_k^\circ$. Therefore for $n$ large enough we have $z_n^k \in Q_k^\circ$.
But this is a contradiction: Since $Q_j=Q_k$, for such an $n$ we are supposed to have
$$z_n^k \in Q_k^\circ, \quad
Q_k \cap \overline{z_n^j y_n}=\overline{z_n^j z_n^{j+1}},\quad \text{and} \quad
z_n^k \in \overline{z_n^j y_n} \setminus \overline{z_n^j z_n^{j+1}}.$$
\nopagebreak
\end{proof}
\commb{Why is there a newline above before the QED symbol?} \compat{I tried reading this: \url{https://tex.stackexchange.com/questions/66152/pushing-qed-to-the-right-within-a-displayed-formula} and played a bit with it, and was unsuccessful addressing this issue. This is a preprint, so lets let it go. (When the equation splits with the QED symbol over a page, that is pretty bad... but again it is a preprint!) I tried inserting \\nopagebreak.}

\section{Convexity of triangulated zebra planes}
\label{sect:convexity proof}
In this section we prove \Cref{thm:convex}, which says zebra planes with a leaf triangulation are convex. From now on we assume our zebra plane $Z$ has a leaf triangulation and we fix such a triangulation $\sT$. Since $\sT$ is a leaf triangulation, its vertices are singularities $s \in \Sigma$ where $\alpha(s)>0$. (This guarantees that the angles at the singular vertices of our triangle are all $3 \pi$ or more. This point will be crucial for \Cref{entering a triangle} below.)
Thus edges $e$ of $\sT$ are saddle connections, including their singular endpoints. As sets, a triangle $T \in \sT$ is the union of its three edges and the interior, and so is homeomorphic to a closed disk.

\begin{proposition}
\label{vertices are finite degree}
In a leaf triangulation, only finitely many triangles meet at each vertex. \compat{I think this is unnecessary. I have a proof that this is more generally true for a triangulation whose vertex set is discrete. It would be nice to find a reference!}
\end{proposition}
\begin{proof}
Let $v$ be a vertex of a leaf triangulation. The zebra structure on the surface gives a bijection between the 
prongs emanating from $v$ and the circle $\R/\alpha(v) \pi \Z$. Each triangle with vertex $v$ corresponds to an interval in $\R/\alpha(v) \pi \Z$
representing the prongs with representations that are entirely contained in the triangle.Since the triangles are meeting edge-to-edge, these intervals of prongs meet endpoint-to-endpoint. Using compactness of the circle, it is not hard to show that any collection of nondegenerate closed intervals in the circle that 
have disjoint interiors, meet endpoint-to-endpoint, and cover the circle is necessarily a finite collection.
\end{proof}

To prove \Cref{thm:convex}, we pick a point $p \in Z$. As above, let $R_p$ denote the collection of all trail rays emanating from $p$, and let $U_p$ denote the union of all these trail rays. By \Cref{union of rays}, $U_p$ is open. Our goal is to prove that $U_p=Z$, which indicates that any point $q \in Z$ lies on a trail ray from $p$. Since $p$ was arbitrary, all pairs of points can be connected by an arc of a trail.

Our proof is an inductive argument showing that $R_p$ covers larger and larger collections of triangles from $\sT$ whose union is $Z$.  At the end of the day, the argument is largely combinatorial, so we begin by considering how the rays in $R_p$ can pass through a triangle $T \in \sT$.

Recall from \Cref{sect:trail rays} that $\sR_p$ is a singular foliation of an open subset $U_p \subset Z$ by the leaves of rays in $R_p$ with singularities at $p$ and at $\Sigma$. Let $e$ be an edge of a triangle $T$ not containing $p$. We'll say $e$ is {\em transverse} to $\sR_p$ if $e \subset U_p$ and as a leaf $e$ is transverse to $\sR_p$. If $e$ is an edge and $T$ is a triangle with edge $e$, then rays {\em enter $T$ through $e$} if some rays intersect $e$ before entering $T$ or {\em exit $T$ through $e$} if rays passing through $e$ have already passed through $T$. (There can be no transition between exiting and entering through $e$ because $e$ is either everywhere transverse or coincides with a leaf.)

We say that an edge $e$ of $\sT$ is a {\em leaf} of $\sR_p$ (or a {\em leaf edge}) if there is a trail ray containing $e$. In this case we also have $e \subset U_p$. Leaf edges get an orientation from inclusion in the ray containing them.

The following observations are the key to our induction.

\begin{lemma}
\label{entering a triangle}
Fix $p$ as above. Let $T$ be a triangle that does not contain $p$ and has an edge $e_0$ that is transverse to $\sR_p$. Assume that rays enter $T$ through $e_0$. Let $e_1$ and $e_2$ be the other two edges, labeled so that the list of edges $e_0, e_1, e_2$ is in counterclockwise order around $\partial T$. Then $T \subset U_p$ and one of the following three statements holds:
\begin{enumerate}
\item The edges $e_1$ and $e_2$ are transverse to $\sR_p$ and rays exit $T$ through them. There is a ray in $R_p$ that passes through the interior of $e_0$ and the opposite vertex of $T$.
\item The edge $e_1$ is transverse to $\sR_p$ and rays exit $T$ through $e_1$, but $e_2$ is a leaf of $\sR_p$ whose orientation agrees with the clockwise orientation on $\partial T$. \compat{Now I more carefully explain the orientation condition. Jan 1, 2023.}
\item The edge $e_2$ is transverse to $\sR_p$ and rays exit $T$ through $e_2$, but $e_1$ is a leaf of $\sR_p$ whose orientation agrees with the counterclockwise orientation on $\partial T$.
\end{enumerate}
\end{lemma}

Note that as a consequence of this lemma, no triangle has two edges through which rays enter $T$.

\begin{proof}
Normalize by a rotation so that edge $e_0$ is vertical. Parameterize $e_0$ clockwise around $T$ by $\gamma:[0,1] \to e_0$. Then by \Cref{union of rays2}, there is a strictly increasing continuous function $\mu:[0,1] \to \R$ such that the ray crossing $\gamma(t)$ has slope $\mu(t)$ at $\gamma(t)$. Let $m_1$ and $m_2$ be the slopes of $e_1$ and $e_2$ respectively. Then by \Cref{triangle1}, both are real and $m_2 < m_1$. Since $\mu(0)<\mu(1)$, we have three mutually exclusive possibilities:
\begin{enumerate}
\item[(i)] $\mu(0)<m_1$ and $m_2<\mu(1)$.
\item[(ii)] $\mu(0)<m_1$ and $\mu(1) \leq m_2$.
\item[(iii)] $m_1 \leq \mu(0)$ and $m_2 < \mu(1)$.
\end{enumerate}
These possibilities are depicted in \Cref{fig:entering a triangle}.

\begin{figure}[htb]
\centering
\includegraphics[width=5in]{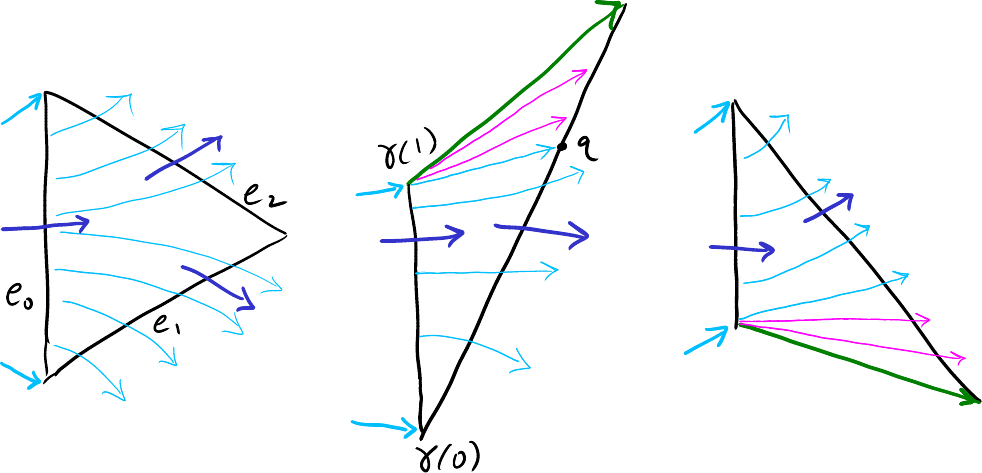}
\caption{Cases (1)-(3) of \Cref{entering a triangle} from left to right: fully transverse, clockwise leaf, and counterclockwise leaf triangles. Edge $e_0$ is always depicted at left, rays entering through $e_0$ are depicted in light blue with bold light blue arrows indicating $\mu(0)$ and $\mu(1)$. Pink arrows indicate portions of rays emanating from a vertex.}
\label{fig:entering a triangle}
\end{figure}

In case (i), we can apply \Cref{quadrilateral foliation} to foliate $T$ by the continuation of rays through $e_0$. These rays enter through $e_0$ and therefore must exit transversely through $e_1$ and $e_2$.

In case (ii), there are two subcases. First, it could be that $\mu(1)=m_2$. In this case, there is a leaf of $\sR_p$ of slope $\mu(1)$ hitting $e_2$, and the corresponding ray can be continued as a trail along $e_2$, making an angle of $\pi$ counterclockwise from $e_2$ to the ray hitting $\gamma(1)$. Thus $e_2$ is a leaf and is oriented clockwise. Applying \Cref{quadrilateral foliation} again shows that rays hitting $e_0$ exit transversely through $e_1$. The second subcase is when $\mu(1)<m_2$. In this case, we can continue the ray hitting $\gamma(1)$ into the interior of the triangle $T_1$, making an angle of $\pi$ at $\gamma(1)$. The ray cannot exit through $e_2$ or else it would violate the decreasing cyclic order of slopes promised by \Cref{triangle1}. Therefore, it exits through $e_1$ at some point, call it $q$. This forms triangle $T_1=\triangle \gamma(0) \gamma(1) q$, to which we can apply \Cref{quadrilateral foliation} again to foliate $T_1$ by leaves passing through $e_0$ and exiting through $\overline{\gamma(0) q} \subset e_1$.
Now consider the triangle $T_2=T \setminus T_1.$ This triangle has $\gamma(1)$ as one vertex and by \Cref{triangle2} we can foliate $T_2$ by leaves emanating from $\gamma(1)$. The counterclockwise angle from the ray hitting $\gamma(1)$ to each leaf emanating from $\gamma(1)$ into $T_2$ is in the interval $[\pi, 2 \pi)$ and since $\gamma(1)$ is a singularity with cone angle at least $3 \pi$, these are continuations of the trail hitting $\gamma(1)$. The edge $e_2$ is one of these leaves emanating from $\gamma(1)$, so $e_2$ is a leaf oriented clockwise. Also, all of edge $e_1$ is covered by trail rays exiting $T$, so $e_1$ is transverse to $\sR_p$ and rays exit through $e_1$.

Case (iii) is symmetric to case (ii) under an orientation reversing symmetry.
\end{proof}

We introduce some terminology associated to the triangles described in \Cref{entering a triangle}.
We'll call a triangle $T$ with an edge that is transverse to $\sR_p$ such that rays enter $T$ through the edge a {\em transverse triangle}. This includes any triangle covered by the lemma. We also introduce more specialized terminology for the three cases. We'll call $T$ satisfying (1) a {\em fully transverse triangle}. A $T$ satisfying (2) is a {\em clockwise leaf triangle}, and a $T$ satisfying (3) is a {\em counterclockwise leaf triangle}. Here ``clockwise'' and ``counterclockwise'' refer to the orientation the (already oriented by $\sR_p$) leaf edges induce on $\partial T$. \compat{added this last sentence}

There is a fourth type of triangle we need to understand. A {\em double leaf triangle} is one for which two edges are leaves of $\sR_p$, and both are oriented away from their common vertex, and the third edge is a transverse edge through which rays exit the triangle. These will show up as a consequence of the lemma below. Combinatorial pictures of each triangle type are shown in \Cref{fig:triangle types}.

\begin{figure}[htb]
\centering
\includegraphics[width=3in]{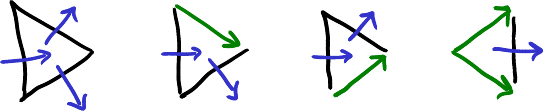}
\caption{Combinatorial representations of the four triangle types (from left to right): fully transverse, clockwise leaf, counterclockwise leaf, and double leaf. Green arrows indicate leaf edges and their orientations. Blue arrows indicate transverse edges and the direction rays cross.}
\label{fig:triangle types}
\end{figure}

We will now consider the local structure of the edges of $\sT$ sharing a vertex $v$, as they relate to $\sR_p$.
A {\em flag} with vertex $v$ is a pair $(e,v)$ where $e$ is an edge of $\sT$ with endpoint $v$.
Let $\sF(v)$ denote the collection of all flags with vertex $v$.
Assuming $e \subset U_p$, there are several possible relationships between the flag $(v,e)$ and the oriented foliation $\sR_p$.
First, it could be that $e$ is a leaf of $\sR_p$, in which case we call $e$ a {\em leaf edge}
and $(v,e)$ a {\em leaf flag}.
Recalling that rays are oriented away from $p$, we see that leaf edges inherit orientations.
Thus, a leaf flag can either be oriented towards the vertex or away from the vertex.
If $e$ is not a leaf, then viewing $v$ as the center of a clock, we see that leaves of $\sR_p$ either cross $e$ in the clockwise or counterclockwise direction. In these cases we call $(v,e)$ a {\em clockwise or counterclockwise transverse flag}, respectively.
See \Cref{fig:flags}. \compat{Rewrote this paragraph, improved the definitions, and added \Cref{fig:flags}.}

\begin{figure}[htb]
\centering
\includegraphics[width=4in]{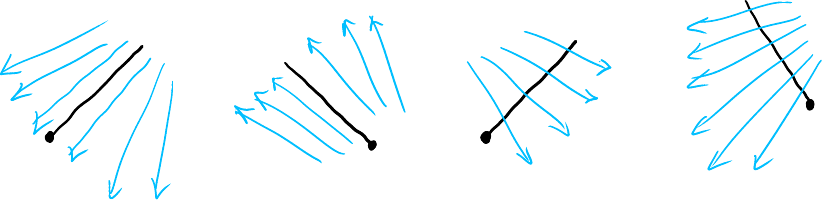}
\caption{Four examples of flags $(v,e)$ drawn as a black point and segment. Portions of the oriented foliation $\sR_p$ are drawn in blue.
From left to right: a leaf flag oriented towards the vertex, a leaf flag oriented away from the vertex, a clockwise transverse flag, and a counterclockwise transverse flag.}
\label{fig:flags}
\end{figure}

Since singularities have cone angle at least $3 \pi$, there are at least four edges meeting at any vertex, and by 
\Cref{vertices are finite degree} there can be only finitely many edges meeting at a vertex.
Thus, the flags in $\sF(v)$ come with a natural cyclic ordering. Given a flag $f=(v,e) \in \sF(v)$,
we'll use $f^{cc}$ to denote the next flag counterclockwise from $f$,
and use $f^c$ to denote the next flag clockwise from $f$.
We'll say a collection of flags $\sF \subset \sF(v)$ are {\em consecutive} if with at most one exception, $f \in \sF$ implies $f^{cc} \in \sF$. We have:

\begin{lemma}
\label{covering a vertex}
Let $v \in U_p \setminus \{p\}$ be a singularity. Suppose either:
\begin{enumerate}
\item[(a)] There is a fully transverse triangle in $U_p$ such that $v$ is vertex opposite the edge through which rays enter the triangle.
\item[(b)] There is an edge $e$ of the triangulation that is also a leaf of $\sR_p$ oriented towards $v$, and the two triangles sharing the edge $e$ are contained in $U_p$.
\end{enumerate}
Then each triangle with vertex $v$ is contained in $U_p$ and we have the following configuration of flags in $\sF(v)$:
\begin{enumerate}
\item In case (a) there are no leaf flags oriented towards $v$, and in case (b) there is exactly one leaf flag oriented towards $v$.
\item If $f \in \sF(v)$ is a leaf edge oriented towards $v$, then $f^{cc}$ is a counterclockwise transverse flag
and $f^c$ is a clockwise transverse flag.
\item The collections of clockwise transverse flags, counterclockwise transverse flags, and leaf flags oriented away from $v$ are all consecutive nonempty subsets of $\sF(v)$. 
\end{enumerate}
\compat{Rewritten in terms of flags.}
\end{lemma}

Observe that the lemma forces clockwise and counterclockwise transverse flags to move outward from the edges provided in the hypothesis (either a fully transverse triangle or leaf edge oriented towards $v$). At some point, in each direction these leaves have to transition to (one or more) leaf edges oriented away from $v$. See \Cref{fig:vertices} for examples. \compat{Minor changes here.}

\begin{figure}[htb]
\centering
\includegraphics[width=3in]{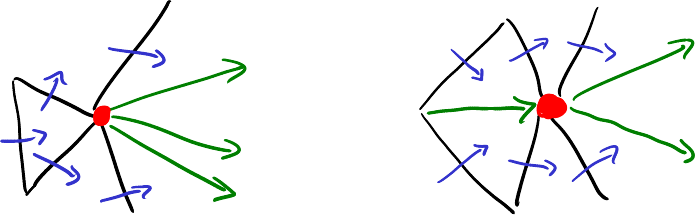}
\caption{Two examples of configurations by \Cref{covering a vertex} are shown. An example satisfying hypothesis (a) is shown on the left and an example of (b) is on the right. Blue arrows denote the direction $\sR_p$ crosses a transverse edge, and green arrows are leaf edges indicating their orientation.}
\label{fig:vertices}
\end{figure}

\begin{proof}[Proof of \Cref{covering a vertex}]
\compat{Changes to the proof to make use of flags.}
Note that if we have a leaf edge oriented towards $v$, we get an arc of a trail $\overline{pv}$.
Similarly, in case (a), from \Cref{entering a triangle}(1) we get an arc of a trail $\overline{pv}$. Statement (1) holds because there can be at most one trail from $p$ to $v$.

To see statement (2), suppose the leaf edge oriented towards $v$ is $e=\overline{wv} \subset \overline{pv}$. Let $T$ be the triangle to the right of $e$ as we move from $w$ to $v$. Let $x$ be the third vertex of $T$, making $(v,\overline{vx})=(v,e)^{cc}$. Let $\ell:[0,1]$ be a parameterized segment of a leaf perpendicular to $e$ such that $\ell(0)$ is in the interior of $e$ and $\ell\big((0,1]) \subset T^\circ$. By \Cref{segment continuity}, observe that for $t$ sufficiently small, we have $\overline{p \ell(t)} \cap \overline{vx}=\emptyset$ by definition of the Fell topology. (Since $\overline{p \ell(0)}$ misses $\overline{vx}$, so must $\overline{p \ell(t)}$ for $t$ small.)
Thus the ray extending such a $\overline{p \ell(t)}$ must exit through $\overline{vx}$, making $(v,\overline{vx})$ a counterclockwise transverse flag. A symmetric argument handles the edge clockwise from $\overline{wv}$. \compat{This paragraph is new. Jan 1, 2023.}

Either case (a) or (b) gives a counterclockwise transverse flag $f=(v,e)$ in the initial triangle(s). Let $T$ be the triangle such that both the edges from $f$ and $f^{cc}$ are edges of $T$. Then rays enter $T$ through $e$, so we can apply \Cref{entering a triangle} to see that rays cover $T$ and see that the $f^{cc}$ is either another counterclockwise transverse flag or a leaf flag edge oriented away from $v$. We can then repeat this inductively, forming a consecutive collection of counterclockwise transverse flags. We observe that this process must terminate in a leaf edge oriented away from $v$, because the trail $\overline{pv}$ can be continued as a trail by continuing the path through $v$ such that the counterclockwise angle from $\overline{pv}$ to the continuation is $\pi$. The next edge counterclockwise after this continuation (or perhaps the edge that coincides with this continuation) is a leaf edge oriented away from $v$, and prior edges are all counterclockwise transverse flags.

A similar analysis in the clockwise direction gives a consecutive collection of clockwise transverse flags followed by a leaf edge oriented away from $v$. All triangles considered so far are contained in $U_p$ by \Cref{entering a triangle}. If there are edges we haven't seen yet, then they make an angle of more than $\pi$ with $\overline{pv}$ on each side, so they are leaf edges oriented away from $v$ arranged in a consecutive collection as desired. Considering the remarks above, we've proved (3).

It remains to show that triangles with two leaf edges oriented away from $v$ are in $U_p$. Let $T$ be any such triangle. Observe that the angles made between the leaf edges and $\overline{pv}$ are all larger than $\pi$. Therefore, any leaf in $T$ emanating from $v$ is a trail continuation of $\overline{pv}$. We conclude using \Cref{triangle2} that $T \subset U_p$.
\end{proof}

\Cref{covering a vertex} provides us with an inductive step to prove \Cref{convexity}. Assuming a vertex satisfies the hypotheses of the lemma, it allows us to extend the known collection of triangles that are foliated by rays to include all triangles with the given vertex. We apply this inductively to prove all of $Z$ is covered.

To organize the induction used in the proof of \Cref{convexity}, we borrow the idea of a {\em queue} from computer programming \cite{Knuth}. A queue $Q$ is a data structure whose state is always a finite ordered list of elements of a set $V$. 
We'll write $Q_i$ to denote the state of the queue after $i \geq 0$ operations have been performed. Thus each $Q_i$ is an element of
$$V^\ast=\bigcup_{n=0}^\infty V^n.$$
For simplicity our queue will always start by representing the empty list, which we denote by $\varepsilon$. Queues support two operations. The {\em enqueue} operation ${\mathbf E}(w):V^\ast \to V^\ast$ takes as input an element $w \in V$ and appends it to the list. Assuming this is the $j$-th operation performed, 
$$Q_{j-1}=(w_0, w_1, \ldots, w_{n-1})
\quad \text{implies} \quad
Q_{j}=(w_0, w_1, \ldots, w_{n-1}, w).$$
The {\em dequeue} operation ${\mathbf D}:V^\ast \setminus \{\varepsilon\} \to V^\ast$ removes an element from the start of the list. Thus if the dequeue operation is applied to $Q_{j-1}=(w_0, w_1, \ldots, w_{n-1})$, we would remove $w_0$ from the list and define $Q_j=(w_1, \ldots, w_{n-1})$. 
A {\em dequeue error} is the attempt to apply the dequeue operation to $\varepsilon$. This is not defined and must be avoided.
Thus the set of operations we can perform is 
$${\mathcal O}=\{\mathbf D\} \cup \{{\mathbf E}(w):~w \in V\}.$$

Suppose that we choose an infinite sequence of operations,
$\{\op_j \in {\mathcal O}\}_{j=1}^\infty.$
Then we can inductively update the state of a queue $Q$ by defining $Q_0=\varepsilon$ and
defining $Q_{j}$ to be operation $\op_j$ applied to $Q_{j-1}$ for each $j \geq 1$.
The {\em queuing sequence} of $Q$ is the sequence $\{v_i \in V\}$ in which elements are enqueued. That is, if $\{j_i\}_{i=0}^\infty$ is the collection of $j \geq 1$ such that $\op_{j}$ is an enqueue operation and $\{j_i\}$ is enumerated in increasing order, we define 
$$\{v_i \in V\}_{i=0}^\infty \quad \text{where $v_i$ is such that} \quad \op_{j_i}={\mathbf E}(v_i).$$
Assuming no dequeue error occurs, the queuing sequence will be a well-defined infinite sequence.

We will use the following scheme to enumerate the vertices of an infinite triangulation of a connected surface:

\begin{lemma}
\label{queue lemma}
Let $G$ be an infinite connected graph all of whose vertices have finite degree. Let $V$ denote the vertex set of $G$. Suppose $Q$ is a queue storing finite lists of elements of $V$ with $Q_0=\varepsilon$. Let $\{\op_j\}_{j=1}^\infty$ be an infinite sequence of operations defined by induction according to the rules that:
\begin{enumerate}
\item $\op_1={\mathbf E}(v_0)$ for some $v_0 \in V$.
\item $\op_2={\mathbf D}$.
\item Whenever $j \geq 2$ and $\op_j={\mathbf D}$ is an operation that leads to the dequeuing of a vertex $w$, there is an enumeration of the set ${\mathcal E}(w)$ of all vertices that are adjacent to $w$ and have not already been enqueued in operations with index less than $j$, ${\mathcal E}(w)=\{w_1, \ldots, w_k\}$,
such that the next operations on the queue are given by
$\op_{j+i}={\mathbf E}(w_i)$ for $i=1, \ldots, k$ and $\op_{j+k+1}={\mathbf D}$.
\end{enumerate}
Then, there are no dequeue errors, the queuing sequence $\{v_i\}$ of $Q$ is an enumeration of $V$, and every vertex is eventually dequeued.
\compat{{\bf An earlier version contained the following statement which was used in the proof of \Cref{convexity} before it was rewritten (the proof of the statement was also commented out):}\\
Finally, if $i \geq 1$ and $V_{i}$ denotes the set of vertices that are enqueued before $v_i$ is dequeued, then 
\begin{enumerate}
\item[($\ast$)]
$\displaystyle V_{i}=\bigcup_{n=0}^{i-1} \sB_1(v_n)$ where $\sB_1(v_n)$ denotes the union of $\{v_n\}$ and the set of vertices adjacent to $v_n$.
\end{enumerate}
}
\end{lemma}
Note that we have some freedom in our choice of operations that satisfy (1)-(3): We are free to choose the order in which the vertices $\{w_1, \ldots, w_k\}$ are enqueued after each dequeue operation. 

From a programming perspective, the reason this lemma holds is that we are constructing a spanning tree formed inductively by including in the tree the edges from each dequeued vertex $w$ to each of its adjacent and not previously enqueued vertices (i.e., $w_1, \ldots, w_k$), while simultaneously traversing the tree according in the manner of a breadth-first search. We give an elementary formal proof:

\begin{proof}
The queuing sequence can be interpreted as a finite or infinite sequence $\{v_i\}_{i=0}^m$ for some $m \in \Z_{\geq 0} \cup \{+\infty\}$, being finite if there is a queuing error. We only enqueue vertices that have not already been enqueued, so the map $i \mapsto v_i$ is injective. We will show that every vertex will eventually be dequeued (before a dequeue error occurs). Since there are infinitely many vertices, this guarantees that no no dequeue error can occur and that $\{v_i\}$ is an infinite sequence enumerating the vertices.

Let $w$ be any vertex. Since $G$ is connected, there is a path of edges joining $v_0$ to $w$, with vertices:
$$w_0=v_0, w_1, \ldots, w_k=w.$$
We will prove by induction that each $w_i$ is dequeued at some point. Since $w_0=v_0$, we can see by inspection that it is dequeued when $\op_2$ is performed. Now assume that $w_i$ was dequeued at some point. Then because $w_{i+1}$ is adjacent to $w_i$, statement (3) ensures that $w_{i+1}$ is enqueued shortly afterward if it hasn't already been enqueued. Now observe that because queues remove from the front of the list and append to the back, the vertex $w_{i+1}$ will be removed after a finite number of dequeue operations (determined by its initial position in the list). Since only finitely many enqueue operations occur between each dequeue operation, $w_{i+1}$ is dequeued in finite time. Furthermore, no dequeue errors can occur before $w_{i+1}$ is dequeued, because a dequeue error requires an empty list and $w_{i+1}$ is in the list up to the point at which it is removed. By induction, we see that $w_k=w$ is successfully dequeued.
%
\end{proof}

\compat{Added a proof outline before the formal proof.} We will briefly outline the idea of the proof of \Cref{convexity}. Let $\sT$ be a leaf triangulation of a zebra plane $Z$ and let $p \in Z$. By a reduction, we'll be able to assume that $p$ is in the vertex set $V$ of $\sT$. We'll let $Q$ be a queue keeping track of finite lists of elements of $V$. The first operation we'll perform on the queue are $\op_1={\mathbf E}(p)$ and $\op_2={\mathbf D}$ which dequeues $p$. We'll follow the restrictions on queue operations given in \Cref{queue lemma}. Thus, we will be defining a queue sequence $\{v_i\}_{i=0}^\infty$ enumerating $V$. But, recall that \Cref{queue lemma} leaves us some freedom: the choice of the initial vertex, and the order in which vertices adjacent to the most recently dequeued vertex are enqueued. Therefore, the sequence $\{v_i\}$ is defined inductively according to our choices and initially we only know $v_0=p$. According to requirements of \Cref{queue lemma}, the operations immediately after $\op_2$ must enqueue the vertices adjacent to $p$, and we do so in arbitrary order. Then we are required to dequeue one of the adjacent vertices: vertex $v_1$ in our sequence. Proceeding inductively, once vertex $v_i$ is determined, so are:
\begin{equation}
\label{eq:A and E}
{\mathcal A}(v_i)=\{w \in V:~\text{$w$ is adjacent to $v_i$}\} \quad \text{and} \quad {\mathcal E}(v_i)=\{w \in {\mathcal A}(v_i):~\text{we have $\op_k \neq {\mathbf E}(w)$ for all $k \leq i$}\}.
\end{equation}
After dequeuing a $v_i \neq p$, we are required to enqueue the vertices in ${\mathcal E}(v_i)$. 
As part of our induction, we produce a proper consecutive subset ${\mathcal A}'(v_i) \subset {\mathcal A}(v_i)$ for $i \geq 1$ such that the vertices in the complement ${\mathcal A}(v_i) \setminus {\mathcal A}'(v_i)$ were already enqueued in previous steps of the induction, guaranteeing that ${\mathcal E}(v_i) \subset {\mathcal A}'(v_i)$. 
Recall that $\sF(v_i)$ denotes the flags with vertex $v_i$. There is a natural bijection
$$f_i:{\mathcal A}(v_i) \to \sF(v_i); \quad w \mapsto (v_i, \overline{v_i w}).$$
Using \Cref{covering a vertex}, we show that $f_i\big({\mathcal A}'(v_i)\big)$ splits into three consecutive groups of flags: 
the counterclockwise leaf flags, the leaf flags oriented away from $v_i$, and the clockwise leaf flags.
This allows us to enumerate the collection ${\mathcal A}'(v_i)$ as a list: 
\begin{itemize}
\item[($a$)] First we list vertices $w \in {\mathcal A}'(v_i)$ such that $f_i(w)$ is a counterclockwise transverse flag, in counterclockwise order.
\item[($b$)] Second we list those $w \in {\mathcal A}'(v_i)$ such that $f_i(w)$  is a clockwise transverse flag, in clockwise order.
\item[($c$)] Last we list the $w \in {\mathcal A}'(v_i)$ such that $f_i(w)$ is a leaf flag oriented away from $v_i$, in counterclockwise order.
\end{itemize}
Two examples of such enumerations are depicted by numbering vertices in \Cref{fig:induction}.
We enqueue the subcollection ${\mathcal E}(v_i) \subset {\mathcal A}'(v_i)$ in the order these vertices appear in the above list. Our inductive argument will show that if we enqueue in this order, then we can continue to enqueue in this order. 
Whenever we apply \Cref{covering a vertex} after dequeuing a vertex $v_i$, we also can conclude that all the triangles with vertex $v_i$ are all contained in $U_p$. Since our queue sequence enumerates $V$, we conclude that every triangle is contained in $U_p$ and therefore $Z \subset U_p$ as desired.

\begin{figure}[htb]
\centering
\includegraphics[width=\textwidth]{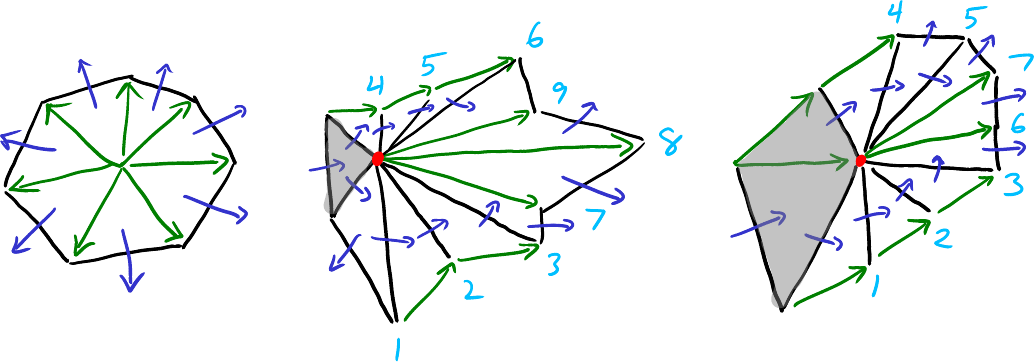}
\caption{Images related to the proof of \Cref{convexity}. Left: A possible region $X_0$. Middle and right: Collections of all triangles with vertex $v_i$ (red dot), in cases when statements $C^a_i$ (middle) or $C^b_i$ (right) are true. Dark gray regions illustrate $Y_i \subset X_{i-1}$. Numbers are written next to the vertices in
${\mathcal A}'(v_i)={\mathcal A}(v_i) \setminus Y_i$ and indicate the order in which we enumerate these vertices.}
\label{fig:induction}
\end{figure}

\begin{proof}[Proof of \Cref{convexity}]
\compat{Large parts of this proof were rewritten.}
Let $Z$ be a zebra plane with a leaf triangulation $\sT$. Let $p \in Z$. We'll show that $Z \subset U_p$, where $U_p$ as above denotes the union of all rays through $p$.

With no loss of generality, we may assume that $p$ is a singularity and a vertex of a triangulation.
To see why there is no loss of generality, suppose $p$ is not a singularity. We define a different triangulation $\hat \sT$ of $Z$. Observe that either $p$ is in the interior of a triangle $T$, or $p$ is in the interior of an edge $e$ where two triangles $T_1$ and $T_2$ meet. 
In the first case, partition $T$ into three triangles along leaf segments from $p$ to the vertices of $T$ to form $\hat \sT$. 
In the second case, cut $T_1$ and $T_2$ into two triangles each along leaf segments from $p$ to the vertices opposite $e$ in $T_1$ and $T_2$. Both collections of leaf segments exist by \Cref{triangle2}. The new triangulation $\hat \sT$ is not a leaf triangulation, because
$p$ is not singular. But, if $Z'$ is the double cover of $Z$ branched only at $p$, then the preimages of edges of $\hat \sT$ form a leaf triangulation $\sT'$ for which the preimage $p'$ of $p$ is a singularity and a vertex. Our argument will then show that $Z' \subset U_{p'}$. Because the image of a ray emanating from $p'$ under the covering map $Z' \to Z$ is a ray emanating from $p$, we also would have $Z \subset U_p$, explaining this reduction. \compat{The WLOG approach was suggested by Barak and made the argument much better! Added this paragraph Jan 1, 2023.}

Let $V$ denote the collection of vertices of $\sT$, and now by assumption $p \in V$.
Let $Q$ be a queue keeping track of finite lists of elements of $V$. The set $V$ together with the edges of triangles forms an infinite connected graph. We will perform operations as described in \Cref{queue lemma} to ensure that the queuing sequence is an enumeration $\{v_i\}_{i=0}^\infty$ of $V$. 
For each $i \geq 0$, let $X_i$ denote the finite union of triangles:
\begin{equation}
\label{eq:D}
X_i = \bigcup \big\{\text{triangles $T$}:~\text{$T$ has at least one vertex in $\{v_0, \ldots, v_i\}$}\big\}.
\end{equation}
Note that we are defining the vertices $\{v_i\}$ inductively, and we will be considering $X_i$ to be defined as soon as the vertices
$v_0, \ldots, v_i$ are defined.

As in the outline, we have $Q_0=\varepsilon$, $\op_1={\mathbf E}(p)$ and $\op_2={\mathbf D}$. Thus $v_0=p$.

It remains to specify how we enqueue vertices in a manner consistent with \Cref{queue lemma} after each dequeue.
The case of $i=0$ is special. Vertex $v_0=p$ is dequeued in operation $\op_2$. We list the collection of adjacent vertices
${\mathcal A}(v_0) = \{v_1, \ldots, v_{k_0}\}$ in arbitrary order. Then the next $k_0+1$ operations will be
\begin{equation}
\label{eq:next vertices}
\op_3={\mathbf E}(v_1), \quad \op_4={\mathbf E}(v_2), \quad \ldots, \quad \op_{k_0+2}={\mathbf E}(v_{k_0}), \quad \text{and} \quad \op_{k_0+3}={\mathbf D}.
\end{equation}
We named these vertices so that their names correspond to their place in the queuing sequence. This also determines unions of triangles
$X_1, \ldots, X_{k_0}$.

Consider the statements defined below for some $i \geq 1$:
\begin{enumerate}
\item[$A_i$\,:] The vertices $\{v_0, \ldots, v_i\}$ have been determined and vertex $v_i$ is dequeued in some operation $\op_{j}={\mathbf D}$.
\item[$B_i$\,:] We have $X_{i-1} \subset U_p$.
\item[$C_i$\,:] Either statement $C^a_i$ or statement $C^b_i$ is true, where:
\begin{enumerate}
\item[$C^a_i$\,:] There is a fully transverse triangle in $X_{i-1}$ such that $v_{i}$ is the vertex opposite the edge through which rays enter the triangle.
\item[$C^b_i$\,:] There is an edge $e$ of the triangulation that is also a leaf of $\sR_p$ oriented towards $v_{i}$, and the two triangles sharing this edge $e$ are contained in $X_{i-1}$.
\end{enumerate}
\end{enumerate}
We will inductively prove that for every $i \geq 1$, the statements $A_i$, $B_i$, and $C_i$ are true.

We will first prove some base cases. First observe that $A_1$ is true, because $v_1$ will be dequeued in operation $\op_{k_0+3}$ by
\eqref{eq:next vertices}. Second we claim that $B_1$ is true. To see this observe that $X_0$ is the union of triangles with vertex $v_0=p$, and \Cref{triangle2} tells us that these triangles are foliated by leaves emanating from $p$, proving $B_1$. It also follows that statements $C^b_1, \ldots, C^b_{k_0}$ are all true, because for $i=1, \ldots, k_0$, the edge
$\overline{p v_i}$ is a leaf edge oriented towards $v_i$, and the triangles sharing this edge are in $X_0 \subset X_{i-1}$. Clearly $C^b_i$ implies $C_i$, so we have shown that
\begin{equation}
\label{eq:base case}
A_1 \wedge B_1 \wedge C_1 \wedge C_2 \wedge \ldots \wedge C_{k_0} \quad \text{is true}.
\end{equation}

We will now prove the following implication for each $i \geq 1$:
\begin{equation}
\label{eq:implication 1}
(A_i \wedge B_i \wedge C_i) \implies (A_{i+1} \wedge B_{i+1}).
\end{equation}
Suppose that $A_i$, $B_i$ and $C_i$ are true. Since $A_i$ is true, $\{v_0, \ldots, v_i\}$ have been determined and vertex $v_i$ has been dequeued. Verifying $A_{i+1}$ involves ensuring a $v_{i+1}$ has been defined and dequeued, but before we can dequeue $v_{i+1}$ we must enqueue the vertices in ${\mathcal E}(v_i)$,
and so we also need to show we can enqueue these vertices according to statements (a)-(c) of the proof outline. 
When combined with the truth of statement $B_i$, statement $C^a_i$ implies statement (a) of \Cref{covering a vertex}
is true with $v=v_i$, and $C^b_i$ implies statement (b) of \Cref{covering a vertex}
is true. Using the conclusions of \Cref{covering a vertex}, we see that $B_i$ and $C_i$ together imply that all triangles with vertex $v_i$ are contained in $U_p$. That is, $B_{i+1}$ is true. To see $A_{i+1}$ is true, we need to show we can enqueue the vertices
in ${\mathcal E}(v_i)$ and so we need to more carefully consider the vertices in ${\mathcal A}(v_i)$. 
There are two cases. 
If $C^a_i$ is true, then there is a fully transverse triangle $T_a$ satisfying $C^a_i$. Otherwise $C^b_i$ must be true and the pair of triangles $\{T_b, T_b'\}$ sharing the edge $e$ from the statement satisfy the statement.
Define $Y_i=T_a$ if $C^a_i$ is true, and $Y_i=T_b \cup T_b'$ if $C^b_i$ is true. Then as a consequence of $C^a_i$ or $C^b_i$ (whichever is true) we see that $Y_i \subset X_{i-1}$. We define ${\mathcal A}'(v_i) = {\mathcal A}(v_i) \setminus Y_i$, and so  ${\mathcal E}(v_i) \subset {\mathcal A}'(v_i)$ as in the outline. By (1)-(3) of \Cref{covering a vertex}, with the possible exception of one leaf flag oriented towards $v_i$ which is necessarily contained in $Y_i$, the flags $\sF(v_i)$ fall into three consecutive groups: counterclockwise transverse flags, leaf flags oriented away from $v_i$, and clockwise transverse flags. Furthermore the order in which these edge types were just listed matches the counterclockwise cyclic order on the intervals. Therefore, we can enumerate the vertices in ${\mathcal A}'(v_i)$ as described in statements (a)-(c) of the proof outline. (Note that $f_i\big({\mathcal A}'(v_i)\big)$ excludes the leaf edge
oriented towards $v_i$ if it exists, as well as the first counterclockwise transverse flag in the counterclockwise order, 
and excludes the first clockwise transverse flag in the clockwise order.)
Because ${\mathcal E}(v_i) \subset {\mathcal A}'(v_i)$, the total ordering corresponding to this enumeration on ${\mathcal A}'(v_i)$ restricts to a total ordering on ${\mathcal E}(v_i)$ and we enqueue these vertices in this order.
Immediately after enqueuing all these vertices, we must carry out a dequeue operation.
Since we have been following the restrictions in \Cref{queue lemma}, this cannot cause a dequeue error so we must dequeue vertex $v_{i+1}$. This proves that $A_{i+1}$ is true, and completes the proof
of \eqref{eq:implication 1}.

We now show the following implication holds for each $i > k_0$:
\begin{equation}
\label{eq:implication 2}
\big(A_i \wedge (B_1 \wedge B_2 \wedge \ldots \wedge B_i) \wedge (C_1 \wedge C_2 \wedge \ldots \wedge C_{i-1})\big) \implies C_i.
\end{equation}
Assume the hypotheses. Since $A_i$ is true, we know that $\{v_0, \ldots, v_i\}$ have been defined. Therefore $v_i$ was enqueued at some point, and we must have $v_i \in {\mathcal E}(v_k)$ for some $k < i$. Since $i>k_0$, we must have $k \geq 1$.
Since $1 \leq k < i$, we know that $C_k$ is true. We can then define the collection of triangles $Y_k$ as in the previous paragraph.
Since $Y_k \cap {\mathcal E}(v_k) = \emptyset$, we know $v_i \not \in Y_k$. We wish to prove $C_i$, and to do so we break into cases depending on the type of the flag $f_k(v_i)=(v_k, \overline{v_k v_i})$, which was relevant to the order in which the vertices ${\mathcal E}(v_k)$ were enqueued. Recall that $f_k\big({\mathcal E}(v_k)\big)$ consisted only of counterclockwise transverse flags, clockwise transverse flags, and leaf flags oriented away from $v_k$.
First suppose $f_k(v_i)$ is a counterclockwise transverse flag.
Let $u \in {\mathcal A}(v_k)$ be such that the flag $f_k(u)=f_k(v_i)^{c}$ is next clockwise.
Because $v_i \not \in Y_k$, it follows from statements (1)-(3) of \Cref{covering a vertex} (which apply because $B_k$ and $C_k$ are true as noted in the previous paragraph) that $f_k(u)$ is also a counterclockwise transverse flag. (It might be helpful to look at \Cref{fig:induction}.)
Thus, rays enter the triangle $\triangle v_{k} u v_i$ through edge $\overline{v_k u}$.
We can then apply \Cref{entering a triangle} and see that either $\overline{u v_i}$ is a transverse edge through which rays exit the triangle
or $\overline{u v_i}$ is a leaf edge oriented towards $v_i$. 
In the first case, $\triangle v_{k} u v_i$ is a fully transverse triangle in $X_k \subset X_{i-1}$ so $C^a_i$ is true.
In the second case, observe that because of our choice of ordering, we have that $u=v_j$ for some $j<i$.
(If $u \not \in {\mathcal E}(v_k)$ then it was enqueued earlier than elements of ${\mathcal E}(v_k)$, 
and if $u \in {\mathcal E}(v_k)$ it comes before $v_i$ in our enumeration of ${\mathcal E}(v_k)$.)
Therefore in this second case, the two triangles sharing edge $\overline{u v_i}$ are in $X_j \subset X_{i-1}$ and so statement $C^b_i$ is true. The case when $f_k(v_i)$ is a clockwise transverse flag is handled by a mirror-symmetric argument.
The final possibility is that $\overline{v_k v_i}$ is a leaf edge oriented towards $v_i$. In this case, the two triangles sharing this edge are contained in $X_k \subset X_{i-1}$, so statement $C^b_i$ is true here. We have shown that $C_i$ is true in all possible cases.

Now observe that \eqref{eq:base case}, \eqref{eq:implication 1}, and \eqref{eq:implication 2} together imply that $A_i$, $B_i$ and $C_i$ are true for all $i \geq 1$. Statement $B_i$ guarantees that $X_{i-1} \subset U_p$ for all $i$. Since $\{v_i\}$ is an enumeration, we have $Z=\bigcup X_{i-1}$ and so $Z \subset U_p$ as desired.
\end{proof}

\section{Closed Trails}
\label{sect:closed trails}

\subsection{Curves and deck transformations}
\label{sect:curves and deck transformations}
Let $S$ be a zebra surface. Recall that in \Cref{sect:maximal cover} we defined the PRU cover $\tilde S$ as the largest cover which is at most doubly branched over the poles. As in that section, let $\Sigma_{-1}=\{p \in S:~\alpha(p)=-1\}$ be the set of poles. In \Cref{maximal cover is a disk} we proved that $\tilde S$ is a disk and a normal cover of $S$.

The homotopy lifting property does not work when the range of the homotopy includes points in $\Sigma_{-1}$ because of the branching. Therefore, for this section, we will only consider paths and loops in $\gamma:[0,1] \to S$ where either:
\begin{enumerate}
\item no poles lie in the interior of $\gamma$, i.e., $\gamma\big((0,1)\big) \cap \Sigma_{-1}=\emptyset$.
\item $\gamma$ is an arc of a trail on $S$.
\end{enumerate}
We allow $\gamma$ to be a trail, because the lifting property works here in the sense that given a lift of the beginning of $\gamma$ (e.g., $\gamma|_{(0,\epsilon)}$ for some $\epsilon>0$), there is a unique way to continue the lift through the preimage of a point in $\Sigma_{-1}$ such that the lift is still a trail. (The preimage of a point in $\Sigma_{-1}$ is nonsingular, and the lift must go straight through nonsingular points.)

Two paths in $S$ are {\em pole-resolved (PR) homotopy equivalent rel endpoints} if they have a lift to $\tilde S$ with the same start and end points. Because the cover is normal, if $\gamma_1$ is PR homotopy equivalent rel endpoints to $\gamma_2$, for any lift $\tilde \gamma_1$, there is a lift $\tilde \gamma_2$ with the same start and end points as $\tilde \gamma_1$.

Recall $S^\plus = S \setminus \Sigma_{-1}$. Choose a basepoint $p_0 \in S^\plus$ and a preimage of this point $\tilde p_0 \in \tilde S$. The pole-resolved fundamental group is $\prpi(S,p_0)$ is the collection of PR homotopy classes rel endpoints of loops starting and ending at $p_0$, with the operation of concatenation. This group is the same as $\pi_1(S^\plus,p_0)/N$ where $N$ is as in \eqref{eq:N}. In particular, if $\Sigma_{-1} = \emptyset$, then $\prpi(S,p_0)=\pi_1(S,p_0)$.

If $\alpha$ and $\beta$ are two parameterized curves such that the endpoint of $\alpha$ is the same as the starting point of $\beta$, let $\alpha \bullet \beta$ denote their concatenation which follows $\alpha$ and then $\beta$.
Two loops $\gamma_1, \gamma_2:[0,1] \to S^\plus$ will be said to be {\em pole-resolved (PR) free homotopy equivalent} if there are paths $\eta_1,\eta_2:[0,1] \to S^\plus$ with $\eta_i(0)=p_0$ and $\eta_i(1)=\gamma_i(0)$ for $i \in \{1,2\}$ such that the concatenations $\eta_1 \bullet \gamma_1 \bullet \eta_1^{-1}$ and $\eta_2 \bullet \gamma_2 \bullet \eta_2^{-1}$ are PR homotopy equivalent rel endpoints. Existence of such $\eta_1$ and $\eta_2$ is equivalent to the condition that for any curves $\beta_i$ with $\beta_i(0)=p_0$ and $\beta_i(1)=\gamma_i(0)$, the elements in $\prpi(S,p_0)$ given by
$[\beta_1 \bullet \gamma_1 \bullet \beta_1^{-1}]$ and $[\beta_2 \bullet \gamma_2 \bullet \beta_2^{-1}]$
are conjugate in $\prpi(S,p_0)$. Thus, PR free homotopy classes of closed curves in $S^\plus$ are in natural bijective correspondence with conjugacy classes in $\prpi(S,p_0)$. From the remarks above about trails, a closed trail also determines such a conjugacy class. The collection of all closed curves determining a conjugacy class in $\prpi(S,p_0)$ is a {\em PR free homotopy class} of closed curves. The notion of PR free homotopy equivalence coincides with the usual notion of free homotopy equivalence if $S$ contains no poles.

From standard covering space theory, there is an isomorphism $\Delta$ from the group $\prpi(S,p_0)$ to the deck group of the covering $\tilde S \to S$. To understand the associated deck group action, fix an element $[\gamma] \in \prpi(S,p_0)$ with representative $\gamma$ and a point $\tilde q \in \tilde S$. Choose a path $\tilde \eta$ starting at $\tilde p_0$ and ending at $\tilde q$ such that the image $\eta$ in $S$ is a path in the sense above.  Let $\widetilde{\gamma \bullet \eta}$ denote the lift of the concatenation $\gamma \bullet \eta$ starting at $\tilde p_0$. Then the image of $\tilde q$ under the deck transformation associated to $\gamma$, is the end point of $\widetilde{\gamma \bullet \eta}$ denoted $\Delta_\gamma(\tilde q)$.

Recall from \Cref{sect:maximal cover} the notion of a polar loop in $S^\plus$. We'll call an element $[\gamma] \in \prpi(S,p_0)$ {\em polar} if curves in $[\gamma]$ are PR free homotopy equivalent to a polar loop.

\begin{proposition}
\label{deck group}
Let $[\gamma] \in \prpi(S,p_0)$. The following statements are equivalent:
\begin{itemize}
\item $[\gamma]$ is polar.
\item $\Delta_\gamma$ is nontrivial and fixes a point in $\tilde S$.
\item $[\gamma]$ is nontrivial and $[\gamma]^2$ is the identity.
\item $\Delta_\gamma$ fixes a unique point in $\tilde S$.
\end{itemize}
If these statements are false and $[\gamma]$ is nontrivial, then $[\gamma]$ is infinite order and all orbits of $\Delta_\gamma$ are infinite.
\end{proposition}
\begin{proof}
First suppose $[\gamma]$ is polar and let $\gamma \in [\gamma]$.
We claim that $\Delta_\gamma$ is nontrivial and fixes a point in $\tilde S$.
By hypothesis, $\gamma$ is PR homotopic rel endpoints to a concatenation $\eta \bullet \ell \bullet \eta^{-1}$, where $\eta$ is a path in $S^\plus$ starting at $p_0$ and ending in an open disk $U \subset S$ with $U \cap \Sigma_1$ containing a single point $q$, and $\ell \subset U$ is a polar loop
enclosing $q$. Let $\beta \subset U$ be a path joining the endpoint of $\eta$ to $q$.
Let $\widetilde{\eta \bullet \beta}$ denote the lift of $\eta \bullet \beta$ starting at $\tilde p_0$, and let $\tilde q$ denote the endpoint of $\widetilde{\eta \bullet \beta}$. We claim that $\tilde q$ is fixed by $\Delta_\gamma$. From our description of the deck group action, $\Delta_{\gamma}(\tilde q)$ is given by the endpoint of the lift of
$$(\eta \bullet \ell \bullet \eta^{-1}) \bullet (\eta \bullet \beta)
\quad \text{which is homotopic rel endpoints to} \quad
\eta \bullet \ell \bullet \beta.$$
Thus $\Delta_{\gamma}(\tilde q)=\tilde q$ if and only if the lifts $\widetilde{\eta \bullet \beta}$ and $\widetilde{\eta \bullet \ell \bullet \beta}$ terminate at the same point. We will explain that this follows from the fact that both curves begin with $\eta$ and are concatenated with paths that stay within $U$ terminating at the singularity $q$. Let $\tilde U$ be the connected component of the preimage of $U$ in $\tilde S$ that contains $\tilde q$. By \Cref{maximal cover is a disk}, the restriction of the covering map is a map $\tilde U \to U$ which is double branched over $q$. Both $\widetilde{\eta \bullet \beta}$ and $\widetilde{\eta \bullet \ell \bullet \beta}$ begin by following the same lift $\tilde \eta$.
By definition $\widetilde{\eta \bullet \beta}$ ends at $\tilde q$, so the endpoint of $\tilde \eta$ is contained in $\tilde U$. Since $\ell \bullet \beta$ is contained entirely in $U$, the endpoint of the concatenation $\widetilde{\eta \bullet \ell \bullet \beta}$ must be in $\tilde U$. But, this endpoint must also be a lift of $q$, and the only lift of $q$ that is contained in $\tilde U$ is $\tilde q$, so $\widetilde{\eta \bullet \ell \bullet \beta}$ ends at $\tilde q$ as desired. This proves that $\Delta_\gamma(\tilde q)=\tilde q$.
To see $\Delta_\gamma$ is nontrivial, let $\tilde r$ be the endpoint of $\tilde \eta$. Then, $\Delta_\gamma(\tilde r)$ is the endpoint of $\widetilde{\eta \bullet \ell}$, which is distinct since the lift $\tilde \ell$ of $\ell$ starting at $\tilde r$ does not lift as a closed loop to $\tilde S$ because of the double branching. This completes the proof of our claim.

Now suppose $\Delta_\gamma$ is nontrivial but has a fixed point $\tilde q$.
Since the restriction of the PRU covering map to the preimage of $S^+$ is a covering map, no point in the preimage of $S^+$ can be fixed by a nontrivial deck transformation. So the image $q$ of $\tilde q$ must be a pole. Let $U \subset S$ be a disk such that $U \cap \Sigma_{-1}=\{q\}$ as in the previous paragraph and let $\tilde U$ be the connected component of the preimage containing $\tilde q$. Then the deck group of the restricted covering map cover $\tilde U \to U$ must be order two and so $\Delta_\gamma^2$ is trivial. Thus $[\gamma]$ is nontrivial but $[\gamma]^2$ is the identity in $\prpi(S,p_0)$.

Now suppose $[\gamma]$ is nontrivial but $[\gamma]^2$ is the identity. The quotient $\tilde S/\langle \Delta_\gamma \rangle$ must be an orientable surface (intermediate between $\tilde S$ and $S$). If $\Delta_\gamma$ has no fixed points then covering space theory guarantees that $\tilde S/\langle \Delta_\gamma \rangle$ is a surface with fundamental group isomorphic to $\Z/2\Z$, but all surfaces with finite fundamental group are simply connected so such a quotient cannot exist.
\compat{Ferran had concerns here, and later in the proof where this argument is repeated. I rephrased it slightly. Here is a discussion of this fact: \url{https://math.stackexchange.com/questions/4010415/noncompact-surface-with-finite-fundamental-group}}
We conclude that $\Delta_\gamma$ must have at least one fixed point. Suppose it has two, $\tilde q$ and $\tilde q'$. In $\tilde S/\langle \Delta_\gamma \rangle$ construct a simple path $\alpha$ starting at the image of $\tilde q$ and ending at the image $\tilde q'$ whose interior does not contain lifts of points in $\Sigma_{-1}$. Then $\alpha$ has two lifts $\tilde \alpha_1$ and $\tilde \alpha_2$, which are both paths from $\tilde q$ to $\tilde q'$. By construction, the union $\tilde \alpha_1 \cup \tilde \alpha_2$ is a simple closed curve, which by the Jordan Curve Theorem bounds a disk $D \subset \tilde S$. Observe that $\Delta_\gamma$ swaps $\tilde \alpha_1$ and $\tilde \alpha_2$, and therefore takes $D$ to the exterior of the curve $\tilde \alpha_1 \cup \tilde \alpha_2$. Because $\Delta_\gamma$ swaps the two curves, the union $D \cup \Delta_\gamma(D)$ is a $2$-sphere. But, this is impossible because this set is contained in $\tilde S$ which is a disk by \Cref{maximal cover is a disk}. We conclude that $\Delta_\gamma$ cannot have distinct fixed points.

To complete the equivalence of the four statements, suppose $\Delta_\gamma$ has a unique fixed point $\tilde q \in \tilde S$. Let $\tilde \zeta$ be a simple loop through the basepoint in $\tilde S$ and contained in the preimage of $S^\plus$ such that intersection of the enclosed disk with $\Sigma_{-1}$ is $\{\tilde q\}$. Let $\zeta \subset S^\plus$ be the image of $\zeta$.
Then from the first paragraph, $\Delta_\zeta$ fixes $\tilde q$ but is nontrivial. Since $\Delta_\gamma$ and $\Delta_\zeta$ fix $\tilde q$ and are nontrivial, they must agree in a neighborhood of $\tilde q$ (as the covering is double branched at $\tilde q$). Thus $\Delta_\gamma=\Delta_\zeta$. From covering space theory, $[\gamma]$ and $[\zeta]$ are equal in $\prpi(S,p_0)$, and so $[\gamma]$ is polar.

Finally suppose the four statements are false for a nontrivial $[\gamma] \in \prpi(S,p_0)$. Clearly if all orbits of $\Delta_\gamma$ are infinite, then $[\gamma]$ is infinite order. So, it suffices to show that $\Delta_\gamma$ has no periodic points. Suppose to the contrary that $\Delta_\gamma$ does have a periodic point. Then we can find a periodic point $\tilde q \in \tilde S$ of minimal period $n \geq 2$. First, it could be that $\Delta_\gamma^n$ is a trivial deck transformation. In this case there is a well-defined covering $\tilde S \to \tilde S/\langle \Delta_\gamma \rangle$, and covering space theory tells us that $\tilde S/\langle \Delta_\gamma \rangle$ has fundamental group $\Z/n\Z$. But, then $\tilde S/\langle \Delta_\gamma \rangle$ is a surface with a nontrivial finite fundamental group, which is impossible. If $\Delta_\gamma^n$ is nontrivial, then the four statements guarantee that $\Delta_\gamma^n$ fixes a unique point. Let $\tilde q$ be this fixed point. But then each point in the $\Delta_\gamma$ orbit of $\tilde q$ (namely $\{\tilde q, \Delta_\gamma(\tilde q), \ldots, \Delta_\gamma^{n-1}(\tilde q)\}$) is fixed by $\Delta_\gamma^n$, contradicting the uniqueness of $\tilde q$ (or that $n \geq 2$ is the minimal period).
\end{proof}

\begin{corollary}
\label{annulus}
If $[\gamma] \in \prpi(S,p_0)$ is nontrivial and non-polar, then $\tilde S/\langle \Delta_\gamma \rangle$ is a normal quotient of $\tilde S$ that is homeomorphic to an annulus.
\end{corollary}
\begin{proof}
Clearly $\tilde S/\langle \Delta_\gamma \rangle$ is an oriented surface because it is intermediate between $\tilde S$ and $S$. From \Cref{deck group} and covering space theory, $\tilde S/\langle \Delta_\gamma \rangle$ has fundamental group isomorphic to $\Z$. From the classification of surfaces, $\tilde S/\langle \Delta_\gamma \rangle$ must be homeomorphic to an annulus.
\end{proof}

\subsection{Topological considerations for closed trails}
\label{sect:topological}
Let $(S, \{\sF_m\})$ be a zebra surface with singular data function $\alpha:S \to \Z_{\geq -1}$. A {\em zebra automorphism} of $(S, \{\sF_m\})$ is a homeomorphism $\delta:S \to S$ such that $\alpha \circ \delta=\alpha$ and such that for each $m \in \hat \R$, the pullback of $\sF_m$ under $\delta$ is identical to $\sF_m$.

\compat{At some later point I use the Bigon Criterion of Farb and Margalit \cite{FM}, which states: Two transverse simple closed curves
in a surface S are in minimal position if and only if they do not bound a bigon. I think this is very useful, but currently do not use except towards the end of this section. If you notice another place to use it to shorten arguments, let me know.}

Let $Z$ be a zebra plane and let $\delta:Z \to Z$ be a zebra automorphism with no fixed points. (Such automorphisms naturally arise as deck transformations of PRU covers.) Let $A=Z/\langle \delta \rangle$ which is topologically an open annulus with a zebra structure.

The fundamental group of an annulus is isomorphic to $\Z$, and is isomorphic to $H_1(A; \Z)$. We say that a loop in the annulus is a {\em core curve} if its homology class generates $H_1(A; \Z)$. We'll say curve is {\em essential} if it is not homotopic to a point. A basic topological fact about curves in the annulus is:

\begin{proposition}
\label{core curve}
Any essential simple closed curve in an annulus is a core curve.
\end{proposition}
\begin{proof}
First, identify $A$ with $\C^\ast=\C \setminus \{0\}$. It is a simple check that a simple parameterization of the circle $C_r$ of radius $r$ centered at $0$ is a core curve of $\C^\ast$.
The general statement follows from the Jordan curve theorem. A simple curve $\gamma$ in $\C^\ast$ separates the sphere into two components. For the curve to be essential, each must contain one of the two points removed. For $r>0$ sufficiently small, $C_r$ and $\gamma$ must be disjoint and so $\C^\ast \setminus (C_r \cup \gamma)$ must have three components. One component must be compact so the homology class of $\gamma$ is the same as that of $C_r$ up to sign. So, $\gamma$ is also a core curve. \compat{Fixed the issue Barak pointed out: I didn't actually prove that $\gamma$ generated homology. There is probably a better proof using exact sequences in simplicial homology, but I didn't think about it. I think this is beneath us and suggest removing it.}
\end{proof}

\begin{proposition}
A closed trail in $A$ is homologically nontrivial.
\end{proposition}
\begin{proof}
If there were a homologically trivial closed trail, then it would lift to a closed trail in $Z$ in violation of \Cref{no monogons}.
\end{proof}

A {\em cover} of a closed curve is just a parameterization that wraps around the curve multiple times. A cover of a closed trail is still a closed trail. Up to covers, closed trails in annuli are simple and distinct closed trails do not intersect:

\begin{proposition}
\label{simple and disjoint}
Any closed trail in $A$ covers a simple closed trail in $A$. If two closed trails in $A$ intersect, they cover the same simple closed trail in $A$ (up to reparameterization).
\end{proposition}
\begin{proof}
Suppose $\gamma_1$ and $\gamma_2$ are homologous closed trails that intersect. Then we may parameterize the curves by $\gamma_i:\R/\Z \to A$ for $i=1,2$ such that $\gamma_1(0)=\gamma_2(0)$. Denote this common point by $p=\gamma_1(0)$. Consider the curve $\eta=\gamma_1 \bullet \gamma_2^{-1}$, which starts at $p$, wraps around $\gamma_1$ once, then wraps around $\gamma_2$ backward once. Then $\eta$ is homologically trivial and so lifts to a closed curve $\tilde \eta = \tilde \gamma_1 \bullet \tilde \gamma_2^{-1}$, where
$\tilde \gamma_1$ and $\tilde \gamma_2$ are trails in $Z$ that descend to $\gamma_1$ and $\gamma_2$ on $A$, respectively. Observe that \Cref{no bigons} guarantees that $\tilde \gamma_1=\tilde \gamma_2$ up to reparameterization fixing the endpoints,
so $\gamma_1=\gamma_2$ up to reparameterization fixing $0$.

We can use the fact we just proved to establish the proposition. First suppose $\gamma:\R/\Z \to A$ is a closed trail that is not simple. For the first assertion, it suffices to prove that $\gamma$ is a nontrivial cover. (Since this will decrease the absolute value of the homology class of $\gamma$, repeating the operation finitely many times must lead to a simple curve.) Since $\gamma$ is not simple there are distinct points $t_1$ and $t_2$ in the domain such that $\gamma(t_1)=\gamma(t_2)$. Let $\gamma_1=\gamma \circ \phi_1$ and $\gamma_2=\gamma \circ \phi_2$ be orientation-preserving reparameterizations of $\gamma$ sending $t_1$ and $t_2$ to zero, respectively. The curves $\gamma_1$ and $\gamma_2$ are homologous so the above paragraph applies and we get that $\gamma_1$ is a reparameterization of $\gamma_2$, $\gamma_1=\gamma_2 \circ \psi$, where $\psi$ fixes zero. It then follows that $\gamma=\gamma \circ \phi_2 \circ \psi \circ \phi_1^{-1}$. Since $\gamma$ is locally one-to-one,this implies that $\phi_2 \circ \psi \circ \phi_1^{-1}$ is a finite-order homeomorphism of $\R/\Z$.
\compat{Added the fact that this map is finite-order in response to Barak's concerns. Dec 30, 2022.}
The combined reparameterization described by $\phi_2 \circ \psi \circ \phi_1^{-1}$ sends $t_1$ to $t_2$ and so is nontrivial, and $\gamma$ covers the quotient curve
$$(\R/\Z)/\langle \phi_2 \circ \psi \circ \phi_1^{-1} \rangle \to A,$$
which sends the $\langle \phi_2 \circ \psi \circ \phi_1^{-1} \rangle$-orbit of $x \in \R/\Z$ to $\gamma(x)$.

Now suppose $\eta_1$ and $\eta_2$ are closed trails that intersect. From the previous paragraph, we know they cover simple closed trails $\gamma_1$ and $\gamma_2$ respectively. Since $\eta_1$ and $\eta_2$ intersect, we know that $\gamma_1$ and $\gamma_2$ also must. By \Cref{core curve}, up to reversing the orientation of one of the curves, we may assume that $\gamma_1$ and $\gamma_2$ are homologous. Then
the first paragraph again gives that $\gamma_1=\gamma_2$ up to reparameterization.
\end{proof}

The above focuses our attention on core curves when looking for closed trails. We will need the following topological fact for later arguments:

\begin{proposition}
\label{topological annulus}
Let $A$ be an open annulus and let $q,r \in A$ be distinct points. Let $\gamma$ be a simple closed curve that is a core curve of $A$ and passes through both $q$ and $r$. Let $A'$ be the union of $\gamma$ and one of the components of $A \setminus \gamma$. Then any simple curve in $A'$ joining $q$ to $r$ is homotopic rel endpoints to one of the arcs of $\gamma$ joining $q$ to $r$. Furthermore if $\beta$ is a simple closed curve in $A'$ that passes through $q$ and $r$ and is a core curve of $A$, then the two arcs of $\beta$ from $q$ to $r$ are homotopic rel endpoints to the two arcs of $\gamma$ from $q$ to $r$.
\end{proposition}
\begin{proof}
We can assume that $A$ is the $2$-sphere $S^2$ with two points $x,y \in S^2$ removed.
Let $\gamma$ be a simple core curve of $A$, and let $q$ and $r$ be distinct points on $\gamma$. Let $A'$ be the union of $\gamma$ and the component of $A \setminus \gamma$ containing $x$ in its boundary. Now suppose that $\alpha:[0,1] \to A'$ is a simple curve with $\alpha(0)=q$ and $\alpha(1)=r$. There is an ambient isotopy of $A$ that fixes $q$ and $r$ but moves the rest of $\gamma$ into the interior of $A'$. By moving $\alpha$ under this isotopy, we see that we can assume that $\alpha \cap \gamma=\{q,r\}$. Thus, the union of $\alpha$ and either of the arcs of $\gamma$ from $q$ to $r$ forms a simple closed curve. By the Jordan curve theorem, each choice bounds a closed disk contained in $A' \cup \{x\}$, and the union of the two disks is $A' \cup \{x\}$. These two disks intersect in the curve $\alpha$, so exactly one of the disks contains $x$. The disk that does not contain $x$ can be used to define a homotopy rel endpoints from $\alpha$ to the other boundary component of the disk, which is one of the two arcs of $\gamma$ joining $q$ to $r$.

Now consider the last case where $\beta$ is a simple core curve contained in $A'$ and passing through $q$ and $r$. In light of the previous paragraph, the arcs of $\beta$ from $q$ to $r$ are each homotopic rel endpoints to one of the arcs from $\gamma$. If the arcs of $\beta$ were homotopic to the same arc of $\gamma$, then $\beta$ would be contractible in $A'$. But this is impossible because $\beta$ is a core curve, so each homotopy class of arcs rel endpoints of $\gamma$ must be attained by an arc of $\beta$.
\end{proof}

\subsection{Existence of closed trails}
\label{sect:existence of closed trails}
As in the previous section, let $A=Z/\langle \delta \rangle$ be an annular quotient of a zebra plane.

Given a point $\tilde p \in Z$, if there is an arc of a trail joining $\tilde p$ to $\delta(\tilde p)$, then we'll denote it by $\tilde \gamma_{\tilde p}$. We use $\gamma_p$ to denote the image curve in $A$, which is a closed loop based at $p$ that is also a core curve of the annulus. Note that $\gamma_p$ is independent of the choice of lift $\tilde p$. The loop $\gamma_p$ satisfies the angle condition for trails, except possibly at the point $p$. So, in some sense it is close to being a closed trail. Our main result is that under mild hypotheses it passes through a closed trail:

\begin{theorem}
\label{closed trails exist}
Let $p \in A$ be arbitrary and suppose that $\gamma_p$ exists. If for all points $q \in \gamma_p$ the curve $\gamma_q$ exists, then for some $q \in \gamma_p$, the curve $\gamma_q$ is a simple closed trail.
\end{theorem}

Observe the following consequence:

\begin{corollary}
\label{cor: closed trails exist}
Suppose $S$ is a zebra surface and $[\gamma] \in \prpi(S,p_0)$ is nontrivial and non-polar.
Let $\tilde S$ be the PRU cover of $S$, and let $A$ be the annulus $\tilde S / \langle \Delta_\gamma\rangle$. Suppose there is a point $p \in A$ such that $\gamma_p$ exists. Suppose further that there is a convex subset of $\tilde S$ that contains two consecutive periods of the preimage of $\gamma_p$. Then, $S$ contains a closed trail that is PR free homotopic to the curves in $[\gamma]$. \compat{Thanks Barak for catching this. Fixed again Dec 30, 2022.}
\end{corollary}

We now turn our attention to proving the theorem. It will make our lives easier to assume that $\gamma_p$ is a simple curve. The following result allows us to do this:

\begin{lemma}[Lollipop lemma]
\label{lollipop}
Every $\gamma_p$ contains a simple $\gamma_q$ as a subarc. Moreover, if $\gamma_p$ is not simple, then $\gamma_p$ is (up to reparameterization) the concatenation of paths $\gamma_p=\alpha \bullet \gamma_q \bullet \alpha^{-1}$ for some $q \in \gamma_p$ such that $\gamma_q$ is a simple closed curve. Furthermore, $\alpha$ is a simple curve that is disjoint from $\gamma_q$ except for the common endpoint $q$.
\end{lemma}

We call this the Lollipop Lemma because it tells us that a non-simple $\gamma_p$ traces the pattern of a {\em lollipop}: $\gamma_p$ begins by traveling up the stick (following a simple curve $\alpha$), then travels around a circular candy ($\gamma_q$) and returns to $p$ by traveling back down the stick ($\alpha^{-1}$), and in addition both $\gamma_q$ and $\alpha$ are simple and disjoint except at their common endpoint. \commb{Def of a lollipop after statement of lemma 9.9: alpha is not allowed to be a point, that is simple closed curves do not qualify as lollipops. This is used later so should be made more explicit. lternatively, in first line of proof of Cor. 9.10, you could say “first assume gamma is not simple”.}\compat{Now I  explicitly say that $\alpha$ is a simple curve in the definition. A simple curve cannot be just a point.}

\begin{proof}
Assume $\gamma_p$ is not simple. Let $X$ denote the set of pairs $(a,b) \in [0,1] \times [0,1]$ with $a<b$.
Define
$$Y = \{(a,b) \in X:~\gamma_p(a)=\gamma_p(b)\}.$$
Then by continuity of $\gamma_p$, $Y$ is a closed subset of $X$. Since $\gamma_p$ is not simple, $Y$ is nonempty.

There is a natural partial ordering on $Y$ given by $(a,b) \leq (c,d)$ if $[a,b] \subset [c,d]$. We will apply Zorn's lemma to find a minimal element. (We remark that Zorn's lemma is more than we need here, because our $\gamma_p$ is combinatorially fairly simple. However, Zorn's lemma provides a framework enabling us to avoid thinking about the combinatorial details.) To see that Zorn's lemma applies, let $\{(a_i,b_i) \in Y:~i \in \Lambda\}$ be a totally ordered subset. We must find a lower bound.
Let $a=\sup \{a_i\}$ and $b=\inf \{b_i\}$. Then $[a,b]=\bigcap_{i \in \Lambda} [a_i,b_i]$. We claim that $a \neq b$. If this were not the case, we can choose a neighborhood of the common point $\gamma_p(a)=\gamma_p(b)$ that lifts to the zebra plane $Z$, and by continuity and the fact that the intervals nest down to this common point, we can find a pair $(a_i, b_i)$ such that $\gamma_p([a_i, b_i])$ is contained in this neighborhood. The lift of this segment to $Z$ violates the injectivity of trails on zebra planes. Thus, $a \neq b$ as claimed and because $Y$ is closed, we have $(a,b) \in Y$ giving us our needed lower bound.

Zorn's lemma guarantees the existence of a minimal $(a,b) \in Y$. The restriction $\gamma_p|_{[a,b]}$ is a simple curve in $A$.
The curve $\gamma_p|_{[a,b]}$ must be an essential simple closed curve, because otherwise its lift to $Z$ violates \Cref{no monogons}.
Then \Cref{core curve} tells us that $\gamma_p|_{[a,b]}$ is a core curve, and lifts of this restriction to $Z$ have endpoints differing by $\delta$. It follows that the restriction $\gamma_p|_{[a,b]}$ must coincide with $\gamma_q$ or $\gamma_q^{-1}$, where $q$ is the common point, $q=\gamma_p(a)=\gamma_p(b)$.
We have shown that $\gamma_p=\alpha \bullet \gamma_q^{\pm 1} \bullet \beta$ where $\alpha$ is a path from $p$ to $q$ and $\beta$ is a path from $q$ back to $p$.

\begin{figure}[htb]
\centering
\includegraphics[width=3in]{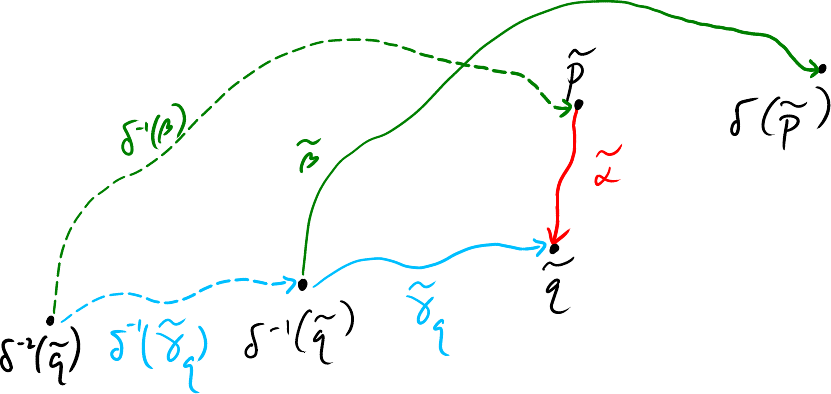}
\caption{Illustration relevant to the case $\gamma_p=\alpha \bullet \gamma_q^{-1} \bullet \beta$ in the proof of \Cref{lollipop}.}
\label{fig:badly_oriented}
\end{figure}

We will first show that the power of $\gamma_q$ must be positive. Suppose to the contrary that
$\gamma_p=\alpha \bullet \gamma_q^{-1} \bullet \beta$. We will obtain a contradiction by considering lifts of these curves to $Z$; see \Cref{fig:badly_oriented}.
Choose a lift $\tilde p \in Z$ of $p$, and let $\tilde \alpha$ denote the lift of $\alpha$ starting at $\tilde p$. Then $\tilde \alpha$ ends at some lift $\tilde q$ of $q$. Let $\tilde \gamma_q$ be the lift ending at $\tilde q$, so that following $\tilde \gamma_q^{-1}$ ends at $\delta^{-1}(\tilde q)$. Then define $\tilde \beta$ to be the lift of $\beta$ starting at $\delta^{-1}(\tilde q)$. This means that the concatenation
$\tilde \alpha \bullet \tilde \gamma_q^{-1} \bullet \tilde \beta$ is a lift of $\gamma_p$, and $\tilde \beta$ must end at $\delta(\tilde p)$. Now observe that both $\tilde \alpha \bullet \tilde \gamma_q^{-1}$ and $\delta^{-1}(\tilde \beta^{-1} \bullet \tilde \gamma_q)$ are arcs of trails running from $\tilde p$ to $\delta^{-1}(\tilde q)$. Therefore, by \Cref{no bigons}, the two curves coincide. Now observe that $\delta^{-2}(\tilde q)$ is in $\delta^{-1}(\tilde \beta^{-1} \bullet \tilde \gamma_q)$ (as the endpoint of $\delta^{-1}(\tilde \beta^{-1})$) and therefore it must also be in $\tilde \alpha \bullet \tilde \gamma_q^{-1}$. Since $\tilde \gamma_q^{-1}$ is the lift of a simple loop and joins $\tilde q$ to $\delta^{-1}(\tilde q)$, it must be that $\delta^{-2}(\tilde q) \in \tilde \alpha$.
Thus by uniqueness of trails, we must have $\tilde \alpha=\delta^{-1}(\tilde \beta^{-1}) \bullet \epsilon$ for some path $\epsilon$ from $\delta^{-2}(\tilde q)$ to $\tilde q$. Similarly, we see that $\tilde q$ must be in
$\delta^{-1}(\tilde \beta^{-1} \bullet \gamma_q)$ but can't be in $\delta^{-1}(\tilde \gamma_q)$ and so must be in
$\delta^{-1}(\tilde \beta^{-1})$. We conclude by uniqueness of trails that $\delta^{-1}(\tilde \beta^{-1})=\tilde \alpha \bullet \eta$ for some path $\eta$ joining $\tilde q$ to $\delta^{-2}(\tilde q)$. But we have shown
\begin{equation}
\label{eq:absurd1}
\tilde \alpha=\delta^{-1}(\tilde \beta^{-1}) \bullet \epsilon
\quad \text{and} \quad
\delta^{-1}(\tilde \beta^{-1})=\tilde \alpha \bullet \eta
\end{equation}
from which it follows that $\tilde \alpha=\tilde \alpha \bullet \eta \bullet \epsilon$, which is absurd: a compact path cannot be the concatenation of itself with a nontrivial curve.

We have shown that $\gamma_p=\alpha \bullet \gamma_q \bullet \beta$. Again let $\tilde p$ be a lift of $p$ and $\tilde \alpha$ be a lift of $\alpha$ starting at $\tilde p$. Define $\tilde q$ to be the endpoint of $\tilde \alpha$, which is a lift of $q$. Let $\tilde \gamma_q$ be the lift of $\gamma_q$ starting at $\tilde q$. Then $\tilde \gamma_q$ ends at $\delta(\tilde q)$. Let $\tilde \beta$ be the lift of $\beta$ starting at $\delta(\tilde q)$. By definition of $\gamma_p$, we have that $\tilde \beta$ ends at $\delta(p)$. Now observe that both $\tilde \alpha$ and $\delta^{-1}(\tilde \beta^{-1})$ run from $\tilde p$ to $\tilde q$, so they coincide because they are trails. We conclude that $\alpha=\beta^{-1}$ as desired.

Now we know that $\gamma_p=\alpha \bullet \gamma_q \bullet \alpha^{-1}$. With $(a,b) \in Y$ minimal as above, we have $\gamma_p(a)=\gamma_p(b)=q$. We claim the minimal pair $(a,b) \in Y$ is unique. Suppose to the contrary that $(c,d) \in Y$ is a distinct minimal pair. Let $q = \gamma_p(a)=\gamma_p(b)$ as above, and let $r = \gamma_p(c)=\gamma_p(d)$. Let $\alpha=\gamma_p|_{[0,a]}$ as above and $\beta=\gamma_p|_{[0,c]}$. We see that
$$\gamma_p=\alpha \bullet \gamma_q \bullet \alpha^{-1} = \beta \bullet \gamma_r \bullet \beta^{-1}.$$
By definition of the partial order, we cannot have $[a,b] \subset [c,d]$ or $[c,d] \subset [a,b]$. Up to swapping the minimal pairs, we have two possible configurations of the points $a,b,c,d \in [0,1]$, nonoverlapping or overlapping:
$$a<b<c<d \quad \text{or} \quad a<c \leq b<d.$$
First consider the nonoverlapping possibility, $a<b<c<d$. In this case,
$$\alpha^{-1} = \gamma_p|_{[b,c]} \bullet \gamma_r \bullet \beta^{-1} \quad \text{and} \quad \beta=\alpha \bullet \gamma_p \bullet \gamma_p|_{[b,c]}.$$
Combining these rules creates an absurdity as in \eqref{eq:absurd1}. In the overlapping case of $a<c \leq b<d$, we have
$$\beta=\alpha \bullet \gamma_p|_{[a,c]} \quad \text{and} \quad \alpha^{-1}=\gamma_p|_{[b,d]} \bullet \beta^{-1},$$
which again leads to an absurdity, ruling out the possibility of multiple minimal pairs.

Now we claim that the curve $\alpha=\gamma_p|_{[0,a]}$ is simple. If not, then there is a pair $(c,d) \in Y \cap [0,a]^2$. Then $Y \cap [0,a]^2$ is nonempty and we can apply Zorn's lemma to produce a minimal $(c',d') \in Y \cap [0,a]^2$. Such a minimal element would also be minimal in $Y$, in contradiction to the previous paragraph.

Finally suppose that $\alpha$ and $\gamma_q$ share a point in common other than $q$. Then there is a $c \in [0,a)$ and a $d \in (a,b)$ for which $\gamma_p(c)=\gamma_p(d)$. Thus $Y \cap [c,d]^2$ is nonempty and Zorn's lemma guarantees that there is a minimal element $(c',d') \in Y \cap [c,d]^2$. This minimal element cannot be $(a,b)$ because $(a,b) \not \subset [c,d]$, so this is also a contradiction to the uniqueness of the minimal element of $Y$.
\end{proof}

\begin{corollary}
\label{equality implies trail}
If $p$ and $q$ are distinct points in the annulus $A$ and $\gamma_p=\gamma_q$ as sets, then $\gamma_p$ is a simple closed trail.
\end{corollary}
\begin{proof}
First assume $\gamma_p=\gamma_q$ is not simple. Then this common curve must be a lollipop.
\compat{Minor changes in the first two sentences of the proof. Dec 31 2022.}
Observe that a lollipop is homeomorphic to a graph with two vertices: one of degree one and another of degree three.
If the lollipop is $\gamma_p$ then $p$ is the vertex of degree one. Thus $\gamma_p=\gamma_q$ implies that $p=q$ in the case of a lollipop.

Now consider the case when $p \neq q$ and $\gamma_p=\gamma_q$ is a simple closed curve.
Observe that $\gamma_p$ satisfies the angle condition to be a trail at every point other than $p$, and $\gamma_q$ satisfies the angle condition at every point other than $q$. So, $p \neq q$ and $\gamma_p=\gamma_q$ implies that this simple closed curve satisfies the angle condition at all points. \compat{Proof rewritten Dec 14, 2022 to address concerns raised by Ferran.}
\end{proof}

\begin{lemma}
\label{connected intersection}
Suppose $\gamma_p$ is a simple closed curve but is not a closed trail.
If $q \in \gamma_p$ and $\gamma_q$ is defined, then $\gamma_p \cap \gamma_q$ is connected.
\end{lemma}

To prove this, observe that since $\gamma_p$ is a simple closed core curve of the annulus $A$, it splits $A$ into two sub-annuli. Since $\gamma_p$ is not a closed trail, one of the angles made by $\gamma_p$ at $p$ must have measure smaller than $\pi$. Let $A_p$ denote the union of $\gamma_p$ and the component of $A \setminus \gamma_p$ containing the interior angle at $p$ with measure smaller than $\pi$. We have:

\begin{proposition}
\label{containment}
Suppose $\gamma_p$ exists and is a simple closed curve but is not a closed trail. Then for any $q \in A_p$ for which $\gamma_q$ exists, we have $\gamma_q \subset A_p$.
\end{proposition}
\begin{proof}
Suppose to the contrary that $\gamma_q \not \subset A_p$. Then there is an arc $\alpha$ of $\gamma_q \setminus A_p$ joining $A_p$ to itself. Thus, $\alpha$ has endpoints in $\gamma_p= \partial A_p$. Lift $\alpha$ to a curve $\tilde \alpha$ in the PRU cover $Z$. Observe that $\tilde \alpha$ is a trail and therefore simple. The preimage of $\gamma_p$ in $Z$ is a bi-infinite path, so there is a polygonal disk $\tilde D$ bounded by $\tilde \alpha$ and an arc of the preimage of $\gamma_p$.
But, there are only two possible points in $\partial \tilde D$ where the interior angles are smaller than $\pi$, namely the endpoints of $\tilde \alpha$. This violates \Cref{ngons}, giving us our desired contradiction.
\end{proof}


\begin{proof}[Proof of \Cref{connected intersection}]
Clearly $q \in \gamma_p \cap \gamma_q$, so to show $\gamma_p \cap \gamma_q$ is connected it suffices to show that any $r \in \gamma_p \cap \gamma_q \setminus \{q\}$ can be joined to $q$ by a segment in the intersection. Let $r \in \gamma_p \cap \gamma_q \setminus \{q\}$.
Let $\alpha \subset \gamma_p$ be an arc from $q$ to $r$ within $\gamma_p$ that does not contain $p$ in its interior. (At least one of the two arcs from $q$ to $r$ in $\gamma_p$ must work.) Observe that $\gamma_q \subset A_p$ by \Cref{containment}. If $\gamma_q$ is a simple closed curve, then \Cref{topological annulus} guarantees that there is an arc $\beta$ of $\gamma_q$ joining $q$ to $r$ that is homotopic rel endpoints to $\alpha$. The same holds if $\gamma_p$ is a lollipop, but for continuity we postpone the argument to the following paragraph. So, either way we get an arc $\beta$ of $\gamma_q$ that is homotopic to $\alpha$. Since $\alpha$ and $\beta$ are homotopic trails joining $q$ to $r$, \Cref{no bigons} guarantees that $\alpha=\beta$ up to reparameterization. The common curve $\alpha=\beta$ is an interval in $\gamma_p \cap \gamma_q$ joining $q$ to $r$ as desired.

Consider the case when $\gamma_q$ is a lollipop and $r \neq q$ is another point of $\gamma_p \cap \gamma_q$. We must show that two arcs of $\gamma_q$ from $q$ to $r$ are homotopic to the two arcs in $\gamma_p$ from $q$ to $r$. We will see that we can reduce to the case when $\gamma_q$ is a simple closed curve by performing a surgery on the annulus $A$. Viewing $\gamma_q$ as a graph, let $x$ denote the vertex of $\gamma_q$ that has degree three (i.e., the place where the stick meets the candy). The segment $\overline{qx}$ is the stick of the lollipop. Let $D$ be a closed topological disk such that $\gamma_q \cap D=\{x\}$ and such that $D$ is contained in the unbounded component of $A_p \setminus \gamma_q$. Let $T$ be a triangle. Let $A'$ be the annulus formed by cutting out $\overline{qx} \cup D$, and then gluing $T$ in its place, where one vertex is sent to $q$, the two adjacent sides are sent to the two arcs formed by cutting along $\overline{qx}$, and the final edge is glued to the arc of $\partial D$ joining $x$ to $x$. See \Cref{fig:deforming_a_lollipop}. \Cref{topological annulus} applies in $A'$, showing that the simple closed curve replacing $\gamma_q$ has arcs homotopic to the two curves in $\gamma_p$ from $q$ to $r$. There is a continuous surjective map $\phi:A' \to A$ that collapses $T$ back to $\overline{qx} \cup D$, obtained by collapsing each leaf of a partial foliation of $T$ (omitting a disk sent to $D$). Post-composing the homotopies with $\phi$ gives the desired statement in $A$. \compat{This argument was rewritten. The original idea to collapse the stick of the lollipop didn't work in the case when $r$ is also a point on the stick of the lollipop (since $q$ and $r$ would be the same point when the stick is collapsed).}
\end{proof}

\begin{figure}[htb]
\centering
\includegraphics[width=5in]{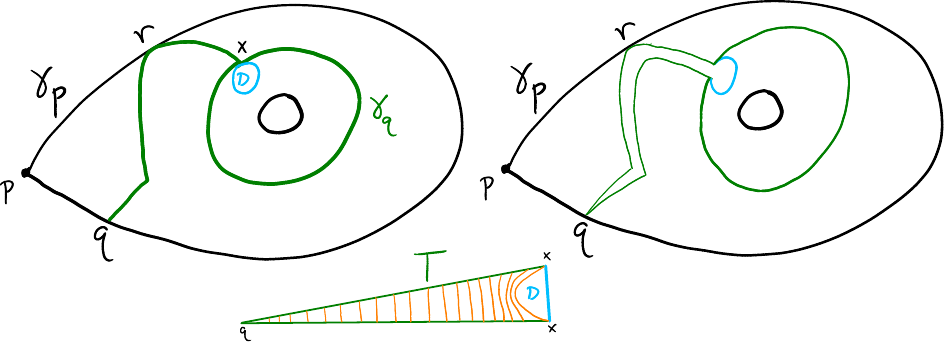}
\caption{Top left: A lollipop $\gamma_q$ in an annulus $A$ before surgery. Top right: The annulus $A'$ formed from $A$ by a surgery that makes $\gamma_q$ a simple closed curve. Bottom: The triangle $T$ glued to $A$ with $\overline{qx} \cup D$ cut out to form $A'$. To reconstruct $A$ from $A'$, we collapse leaves in the orange partial foliation of $T$ to points and send the unfoliated topological disk to $D$. \commb{Figure 25 looks like it’s taken out of an opthalomogist’s surgery manual.}\compat{That's funny. I can't tell if you were bothered or if you feel that we should change it. Eventually almost all of these figures should be improved, but I don't want to do it before the first arxiv post. Moving the point $p$ so it was at the top of the figure might help, then they'd look like teardrops rather than eyes.}}
\label{fig:deforming_a_lollipop}
\end{figure}



\begin{lemma}
\label{nesting}
Assume $\gamma_p$ exists and is a simple closed curve, $q \in \gamma_p \setminus \{p\}$ and $\gamma_q$ is defined. Suppose $r \in \gamma_p \cap \gamma_q$ and $\gamma_r$ exists. Then $\gamma_p \cap \gamma_r \subset \gamma_p \cap \gamma_q$.
\end{lemma}

\begin{proof}
Assume the hypotheses described in the lemma. We break into several cases.

If $\gamma_p$ is a closed trail, then $\gamma_p = \gamma_q=\gamma_r$ so the conclusion is trivially true. Therefore, we may assume that $\gamma_p$ is a simple closed curve but not a closed trail. Let $A_p$ be the subannulus bounded by $\gamma_p$ such that the interior angle at $p$ is less than $\pi$ as above. \Cref{containment} guarantees that $\gamma_q$ and $\gamma_r$ are contained in $A_p$.

Now suppose that $\gamma_q$ is a simple closed curve. If $\gamma_q$ is a closed trail, then we'd have $\gamma_r=\gamma_q$ and so the conclusion is trivially true. So, assume that $\gamma_q$ is not a closed trail. In this case we can define the region $A_q$ as above. Since $q \in \gamma_p \setminus \{p\}$ and $\gamma_p$ is a trail, the exterior angle of $A_p$ at $q$ is at least $\pi$. Since $\gamma_q \subset A_p$, the annular region $A_q$ must be on the same side of $\gamma_q$ at $q$ as $A_p$. Thus, $A_q \subset A_p$.
\Cref{containment} guarantees that $\gamma_r \subset A_q$. To see
$\gamma_p \cap \gamma_r \subset \gamma_p \cap \gamma_q$
observe that if $x \in \gamma_r \cap \gamma_p$, then $x \in \gamma_r \cap \partial A_p$. Since $\gamma_r \subset A_q$, we have $x \in A_q \cap \partial A_p$. But since $A_q \subset A_p$, any point in both $A_q$ and $\partial A_p$ must lie in $\partial A_q=\gamma_q$.

It remains to handle the case when $\gamma_q$ is not simple. By \Cref{lollipop}, $\gamma_q$ is a lollipop. Then $\gamma_q=\alpha \bullet \gamma_{q'} \bullet \alpha^{-1}$ for some $q' \in \gamma_q$. It is possible that $r \in \alpha$. But in this case $\gamma_r$ is the lollipop formed by removing the arc of $\alpha$ from $q$ to $r$, since this yields a trail. Thus in this case $\gamma_r \subset \gamma_q$ and the conclusion is trivially true. So, we may assume that $r \in \gamma_{q'}$. It is possible that $\gamma_{q'}$ is a closed trail, but in this case we have $\gamma_r=\gamma_{q'} \subset \gamma_q$ which again trivially leads to the desired conclusion. Otherwise $\gamma_{q'}$ is a simple closed curve but not a closed trail, and so $A_{q'}$ is well defined. Since $\gamma_q$ is a trail, the side of $\gamma_{q'}$ where the angle at $q'$ appears with measure less than $\pi$ must not contain $\alpha$. Thus, $A_{q'} \subset A_p$. We have $\gamma_r \subset A_{q'} \subset A_p$, so repeating the argument from the end of the previous paragraph shows that $\gamma_r \cap \gamma_p \subset \gamma_{q'} \cap \gamma_p$ and we also have $\gamma_{q'} \subset \gamma_q$.
\end{proof}

\begin{proof}[Proof of \Cref{closed trails exist}]
Let $p \in A$. Assume $\gamma_p$ exists and that $\gamma_q$ exists for all $q \in \gamma_p$. We will show that some $\gamma_q$ is a closed trail.

First of all, we can assume without loss of generality that $\gamma_p$ is a simple curve. (Otherwise replace $\gamma_p$ by the simple closed subarc guaranteed by \Cref{lollipop}, and observe that the hypotheses still hold.) We may also assume that $\gamma_p$ is not a closed trail, or else the conclusion is trivial.

For $q \in \gamma_p$, let $I_q=\gamma_p \cap \gamma_q$, which is a closed connected subset of $\gamma_p$.
Define
$${\mathcal I} = \{I_q:~q \in \gamma_p\}.$$
This collection is partially ordered by inclusion. \Cref{nesting} guarantees that $r \in I_q$ implies that $I_r \subset I_q$, so any nested family in ${\mathcal I}$ has a lower bound. Thus Zorn's Lemma tells us that ${\mathcal I}$ has a minimal element $I_{\min}$.

We break into several cases. First, it could be that $I_{\min}=\gamma_p$. In this case choose a $q \in \gamma_p \setminus \{p\}$. We see that $I_q=I_{\min}$ and thus $\gamma_q=\gamma_p$. Thus, $\gamma_p$ is a closed trail by \Cref{equality implies trail}.

Now suppose that $I_{\min}$ is a nondegenerate interval. Choose distinct $q$ and $r$ from the interior of $I_{\min}$. Then both $\gamma_q$ and $\gamma_r$ contain $I_{\min}$. Since $\gamma_q$ and $\gamma_r$ contain arcs in two directions leaving $q$ and $r$, respectively, these curves are not lollipops. Let $q'$ and $r'$ denote the endpoints of $I_{\min}$. Each of $\gamma_q$ and $\gamma_r$ is the union of two arcs joining $q'$ to $r'$, namely $I_{\min}$ and its complement. By \Cref{topological annulus}, these two arcs that are complements of $I_{\min}$ are homotopic rel endpoints. Thus by \Cref{no bigons}, they are equal and so we must have $\gamma_q=\gamma_r$. This common curve must be a trail by \Cref{equality implies trail}.

The last possibility is that $I_{\min}$ consists of a single point, call it $q$. We claim that $\gamma_q$ is a trail. Suppose to the contrary that $\gamma_q$ is not a closed trail. We break into subcases.

First, it could be that $\gamma_q$ is a lollipop, so $\gamma_q=\alpha \bullet \gamma_{r} \bullet \alpha^{-1}$. This case is illustrated on the left side of \Cref{fig:closed_trail_argument}. Because $\gamma_q \cap \gamma_p = \{q\}$, the curves $\gamma_p$ and $\gamma_{r}$ are disjoint and bound an annulus $A'$. We will apply the \gaussbonnet to $A'$. The only interior angle whose measure is less than $\pi$ occurs at $p$, and the interior angle at $r$ is at least $2\pi$. So, the quantity on the left side of \eqref{eq:Gauss-Bonnet} is negative. But the Euler characteristic of the annulus is zero, so this is a contradiction. It follows that $\gamma_q$ could not have been a lollipop after all.

\begin{figure}[htb]
\centering
\includegraphics[width=\textwidth]{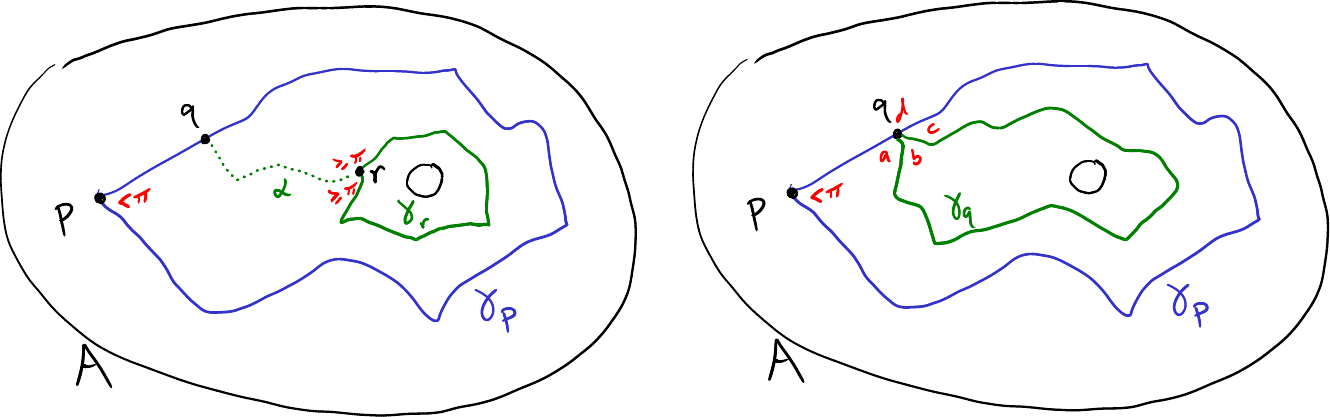}
\caption{Cases of the proof of \Cref{closed trails exist} when $I_{\min}=\{q\}$.}
\label{fig:closed_trail_argument}
\end{figure}

\compat{I stopped using the term ``interior angle'' in this paragraph except when referring to the region $R$. Instead when I refer to measuring angles along $\gamma_p$ and $\gamma_q$, I explicitly state which side the angle should be measured on. Dec 15, 2022.}
Otherwise $\gamma_q$ is a simple closed curve, but not a closed trail. This case is illustrated on the right side of \Cref{fig:closed_trail_argument}. This case is ruled out by another Gauss-Bonnet computation, this time involving the region $R$ between $\gamma_p$ and $\gamma_q$. This region is a topological disk that touches itself at the point $q$, but it has a lift to the zebra plane that does not touch itself so the \gaussbonnet applies. The right hand side of \eqref{eq:Gauss-Bonnet} is therefore $2\pi$. We will show the left side sums to at most $\pi$, giving us our desired contradiction. To get this upper bound, we must sum $\pi-\theta$ where $\theta$ varies over the interior angles in the boundary of the region $R$. Let $\theta_1, \ldots, \theta_m$ denote the measures of angles over all points where $\gamma_p$ bends, measured on the side of $\gamma_p$ containing $R$. Let $\eta_1, \ldots, \eta_n$ denote the measures of angles over all points where $\gamma_q$ bends, again taken using the side of $\gamma_q$ where $R$ resides. Note that in making these definitions, we have intentionally included angles at $q$ that do not appear in the boundary of $R$, and omitted the angles at $q$ that do appear. To fix this, observe that $\gamma_p$ and $\gamma_q$ come together at $q$ to form four angles. We will denote the measures of these angles by $a$, $b$, $c$ and $d$, as illustrated in \Cref{fig:closed_trail_argument}.
If $\theta_1$ is the angle of $\gamma_p$ based at $q$ measured on the side of $\gamma_p$ containing $R$, then $\theta_1=a+b+c$. Similarly, if $\eta_1$ is the angle of $\gamma_q$ based at $q$ measured on the side containing $R$, we have $\eta_1=a+c+d$. The interior angles of $R$ based at $q$ that appear in the boundary of $R$ are given by $a$ and $c$. Therefore, the total contribution of the boundary to the left side of \eqref{eq:Gauss-Bonnet} is given by
\begin{equation}
\star = \sum_{i=1}^m (\pi-\theta_i) + \sum_{i=1}^n (\pi-\eta_i) - (\pi-\theta_1) - (\pi-\eta_1) + (\pi-a) + (\pi - c).
\label{eq:star}
\end{equation}
Suppose $\theta_2$ is the angle of $\gamma_p$ measured at $p$ on the side containing $R$. Since $\gamma_p$ is simple but not a closed trail, $\theta_2<\pi$. Since $\gamma_p$ is a trail, this is the only positive term in the sum $\sum_{i=1}^m (\pi-\theta_i)$ and since this sum adds to an integer multiple of $\pi$ by \Cref{loop contribution}, we see that $\sum_{i=1}^m (\pi-\theta_i) \leq 0$. Since $\gamma_q$ is simple but not a closed trail, we have that $b<\pi$. (The angle $b$ must be the one less than $\pi$, because the complement contains $d$ which has measure at least $\pi$ because $\gamma_p$ is a trail.) Therefore, $\eta_1=a+c+d=(a+b+c+d)-b>a+b+c+d-\pi$, and $\pi-\eta_1<2\pi-a-b-c-d$. The other $\eta_i$ contribute nonpositively to the sum, and since again the sum must be an integer multiple of $\pi$, we have
$$\sum_{i=1}^n (\pi-\eta_i) \leq \pi -a -b -c -d.$$
Also we can simplify the remaining terms in \eqref{eq:star}:
$$- (\pi-\theta_1) - (\pi-\eta_1) + (\pi-a) + (\pi - c)
= a+b+c+d.$$
Plugging all these quantities into \eqref{eq:star}, we see that
$$\star \leq 0 + (\pi -a -b -c -d) + (a+b+c+d) = \pi.$$
This proves that the contribution of the boundary to the left side of \eqref{eq:Gauss-Bonnet} is at most $\pi$. The contribution of any singularities in the interior of $R$ is nonpositive, so the left side of \eqref{eq:Gauss-Bonnet} is at most $\pi$, while the right side is $2\pi$ as indicated above. This is a contradiction, so $\gamma_q$ cannot be simple and not a closed trail.
\end{proof}

\subsection{Uniqueness of closed trails}

As in the previous two subsections, we let $Z$ be a zebra plane, $\delta:Z \to Z$ be a fixed-point free zebra automorphism, and $A=Z/\langle\delta\rangle$ be the annular quotient.

\begin{theorem}
\label{bounded annulus}
If $\gamma_1$ and $\gamma_2$ are distinct closed trails that are core curves of $A$, then they are disjoint and bound a compact annulus $K \subset A$ such that there are no singularities in the interior of $K$ and all interior angles at singularities in $\partial K=\gamma_1 \cup \gamma_2$ are of measure $\pi$.
\end{theorem}
\begin{proof}
That $\gamma_1$ and $\gamma_2$ are simple and disjoint follows from \Cref{simple and disjoint}. Because they are simple disjoint core curves, they bound a compact annulus $K$ as described. We apply the \gaussbonnet. The Euler characteristic of the annulus $K$ is zero.
All singularities in the interior of $A$ contribute negatively to the left side of \eqref{eq:Gauss-Bonnet}, because there are no $\pi$-singularities. Also because $\gamma_1$ and $\gamma_2$ are trails, the boundary contributes nonpositively and positively unless all interior angles are $\pi$. We conclude that there can be no singularities in the interior of $K$ and all interior angles on the boundary have measure $\pi$.
\end{proof}

We have the following trivial consequence:

\begin{corollary}
\label{unique case}
If $\gamma$ is a closed trail in a zebra surface $S$ such that $\gamma$ has at least one bending angle larger than $\pi$ on each side, then $\gamma$ is the unique closed trail in the conjugacy class associated to $\gamma$ in $\prpi(S,p_0)$.
\end{corollary}
\begin{proof}
Let $Z$ be the PRU cover of $S$.
Then $\gamma$ has a lift to a trail $\tilde \gamma \subset Z$, which is bi-infinite by \Cref{proper}. The deck transformation $\Delta_\gamma$ preserves $\tilde \gamma$, and since $\tilde \gamma$ is bi-infinite, $\Delta_\gamma$ cannot be order two. It then follows from results in \Cref{sect:curves and deck transformations} that $A=Z/\langle \Delta_\gamma \rangle$ is an annulus. If $\eta$ were a distinct closed trail determining the same conjugacy class, then we could lift $\eta$ to a distinct closed trail on $A$ and \Cref{bounded annulus} would lead to a contradiction to the hypothesized property about angles on each side of $\gamma$.
\end{proof}

\subsection{Foliating quadrilaterals}
We will describe a result that produces a foliation of a quadrilateral. We do this to produce the foliations in cylinders. Concretely, consider an immersed trapezoid in a zebra surface, where the restriction to the interior is an embedding and the two parallel sides have the same image. Assuming the other sides do not meet each other, this results in an embedded annulus in the surface and there is a natural homeomorphism from one parallel side to the other coming from the edge identification. This homeomorphism determines a foliation, because of the following lemma. The statement is depicted on the left side of \Cref{fig:edge_homeomorphism}.
\compat{This paragraph was expanded to explain the purpose of the lemma. Dec 15, 2022.}

\begin{lemma}[Edge homeomorphism foliation]
\label{edge homeomorphism}
Let $P$ be the quadrilateral $abcd$ in a zebra plane $Z$ whose interior angles add to $2\pi$ and whose interior angles are all less than or equal to $\pi$. Let $h:\overline{ab} \to \overline{dc}$ be a homeomorphism such that $h(a)=d$ and $h(b)=c$. Then the collection of segments of leaves
$$\{ \overline{x~h(x)}~:~ x\in \overline{ab}\}$$
is pairwise disjoint, covers $P$, and foliates the interior of $P$. Moreover, the function $\mu:P \to \hat \R$ that sends a point $p \in P$ to the slope of the segment $\overline{x~ h(x)}$ containing $p$ is continuous.
\end{lemma}

\begin{figure}[htb]
\centering
\includegraphics[height=1.5in]{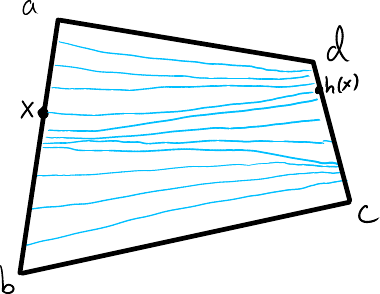}
\caption{A depiction of the statement of \Cref{edge homeomorphism}.}
\label{fig:edge_homeomorphism}
\end{figure}

\begin{proof}
\Cref{quadrilateral} guarantees that there are no singularities in the interior of $P$. We can find a Euclidean polygon $P' \subset \R^2$ with sides of the same slopes. Choosing an appropriate homeomorphism $\partial P \to \partial P'$, we may apply the \surgery to construct a new zebra plane $Z'=(\R^2 \setminus P') \cup P$. Observe that $Z'$ has no singularities. For the proof we consider $P$ to be a subset of $Z'$.  This will simplify the logic below (but is not essential for this argument), because $P$ considered as a subset of $Z$ could have singularities in $\partial P$.

First we will make some remarks about the geometry of $P$. Since all the interior angles are less than $\pi$, $P$ is convex by \Cref{thm:polygonal convexity}. Convexity guarantees that any two points in $P$ can be joined by a trail, but because $Z'$ does not contain singularities, any two points can be joined by segments of leaves. We will be repeatedly use that a distinct pair of leaves can intersect in at most one point.

The above remarks show that for each $x \in \overline{ab}$, the segment $\overline{x~h(x)}$ exists. Furthermore, we claim that distinct segments $\overline{x~h(x)}$ and $\overline{x' h(x')}$ cannot intersect. If they did intersect, then they intersect at exactly one point transversely, because of the stellar structure at the point and \Cref{no bigons}. Therefore, if $\overline{x~h(x)} \cap \overline{x' h(x')} \neq \emptyset$, then the points $x$ and $h(x)$ would have to lie on opposite sides of $\overline{x' h(x')}$, but that is impossible from our definition of $h$.

The restriction of $\mu$ to $\overline{ab}$ is continuous as a consequence of \Cref{slope continuity}.

Now we claim that the collection of segments $\overline{x~h(x)}$ with $x \in \overline{ab}$ covers $P$. It is clear that every point in $\partial P$ is covered. Let $q$ be a point in the interior of $P$. We will show $q$ lies in one of the segments.
Let $m_0$ denote the slope of $\overline{aq}$ and $m_1$ denote the slope of $\overline{bq}$.
Assume without loss of generality that the slope of $\overline{ab}=+\infty$. Then $m_0,m_1 \in \R$ and and $m_0<m_1$ by \Cref{triangle1}. By considering the order of segments emanating from $a$, we see that $m_0<\mu(a)$. Similar considerations at $b$ tell us that $\mu(b)<m_1$.
By \Cref{triangle2}, $\triangle abq$ is foliated by leaves of slope in $[m_0, m_1]$ through $q$, and the function $\nu:\overline{ab} \to [m_0,m_1]$ sending $x$ to the slope of $\overline{xq}$ is continuous. Since both $\nu$ and $\mu|_{\overline{ab}}$ are continuous, take real values, and satisfy
$$\nu(a)=m_0 < \mu(a) \quad \text{and} \quad \mu(b)<\nu(b)=m_1,$$
there must be an $x \in \overline{ab}$ such that $\mu(x)=\nu(x)$. But then the leaf
$\overline{x~h(x)}$ has slope $\nu(x)$ and so passes through $q$. This completes the proof of our claim that $P$ is covered.

Since the collection $\{\overline{x~h(x)}\}$ is pairwise disjoint and covers $P$, there is a well defined function $f_1:P \to \overline{ab}$ that sends $p \in P$ to the $x \in \overline{ab}$ such that $p \in \overline{x~h(x)}$. Choose a foliation $\sF'$ by segments of leaves joining $\overline{ad}$ to $\overline{bc}$. (If $\overline{ab}$ and $\overline{cd}$ are parallel, $\sF'$ can be taken to be a directional foliation, otherwise we can use \Cref{quadrilateral foliation} to produce such a foliation.) Define $f_2:P \to \overline{bc}$ to send $p \in P$ to the intersection of the leaf of $\sF'$ through $p$ with $\overline{bc}$. We claim that
$$f_1 \times f_2: P \to \overline{ab} \times \overline{bc}$$
is a homeomorphism. Note that the codomain is naturally homeomorphic to a rectangle in $\R^2$. The map is one-to-one because each $p \in P$ lies in exactly one
$\overline{x~h(x)}$ and exactly one leaf of $\sF'$. The map $f_1 \times f_2$ is surjective  since every $\overline{x~h(x)}$ intersects every leaf of $\sF'$.
Since $\sF'$ is a foliation, we see that $f_2$ is continuous. To see that $f_1$ is continuous observe that
$f_1^{-1}(\overline{xy}\setminus \{x,y\})$
is the quadrilateral with vertices $x$, $h(x)$, $h(y)$ and $h(x)$ with the edges $\overline{x~h(x)}$ and $\overline{y~h(y)}$ removed. This set is open as a subset of $P$. We've shown that $f_1 \times f_2$ is a continuous bijection. Since the domain is compact and the codomain is Hausdorff, $f_1 \times f_2$ is a homeomorphism.
Since $f_1 \times f_2$ is a homeomorphism, the collection of segments of the form $\overline{x~h(x)}=f_1^{-1}(x)$ foliate $P$.
\end{proof}

\subsection{The closed trail theorem for annuli}
\label{sect:cylinders}
We continue the notation from the previous subsection. The following result guarantees that we can extend certain closed trails to open sets foliated by closed leaves.

\begin{lemma}
\label{open set foliated by closed leaves}
Let $\gamma$ be a closed trail that is a core curve of an annulus $A=Z/\langle \delta \rangle$. Let $A_+$ denote one of the connected components of $A \setminus \gamma$. Suppose all interior angles in $A_+$ at singularities along $\gamma$ have measure $\pi$. Then there is an open annulus $U \subset A_+$ one of whose boundary components is $\gamma$ that is foliated by closed leaves (of possibly varying slope). Furthermore the function $\mu:U \cup \gamma \to \hat \R$ sending a point $p$ to the slope of the closed trail through $p$ is continuous.
\end{lemma}
\begin{proof}
Let $\pi:Z \to A$ denote the covering by the zebra plane. Choose a lift $\tilde \gamma:[0,1] \to Z$, which is an arc of a trail.
Let $\tilde A_+=\pi^{-1}(A_+)$. Using \Cref{trapezoid construction}, we can construct a rectangle $\tilde R \subset \tilde A_+$ such that one edge is given by $\tilde \gamma$. Such a rectangle is convex by
\Cref{thm:polygonal convexity} and cannot contain singularities by \Cref{quadrilateral}. Let $a$ and $b$ be the other vertices of $\tilde R$ so that the cyclically ordered vertices are $a$, $\tilde \gamma(0)$, $\tilde \gamma(1)$, and $b$. Because the measure of the interior angle of $A_+$ at the common point $\gamma(0)=\gamma(1)$ is $\pi$, the image under the covering $\pi$ of the rays $\overrightarrow{\tilde \gamma(0) a}$ and $\overrightarrow{\tilde \gamma(1) b}$ must coincide. So, we can choose points $a'$ from the interior of $\overline{\tilde \gamma(0) a}$ and $b'$ from the interior of $\overline{\tilde \gamma(1) b}$ such that the images of $\overrightarrow{\tilde \gamma(0) a'}$ and $\overrightarrow{\tilde \gamma(1) b'}$ coincide in $A$. This necessarily implies that there are no singularities on the interiors of these paths.
Also since they map to the same path in $A$, there is a natural homeomorphism
$h:\overline{\tilde \gamma(0) a'} \to \overline{\tilde \gamma(1) b'}$ obtained by restricting a deck transformation. By convexity of $\tilde R$, we can construct the arc of a trail $\overline{a' b'}$. Let $\tilde Q$ denote the quadrilateral with vertices $a'$, $\tilde \gamma(0)$, $\tilde \gamma(1)$, and $b'$.
Let $\tilde Q^-$ denote $\tilde Q$ with $\overline{a'b'}$ removed.
Then $\tilde Q^- \setminus \tilde \gamma$ contains no singularities.
We define $U=\pi(\tilde Q^-) \setminus \gamma$.
Observe that the restriction of $\pi$ to $\tilde Q^\circ$ factors through the quotient $\tilde Q/h$.
\Cref{edge homeomorphism} guarantees that we can foliate $\tilde R$ by leaves joining $x \in \overline{\tilde \gamma(0) a'}$ to $h(x) \in \overline{\tilde \gamma(1) b'}$ and that the slopes of these leaves vary continuously. For any $x \in \overline{\tilde \gamma(0) a'} \setminus \{\tilde \gamma(0)\}$, the segment $\overline{x~h(x)}$ does not pass through any singularities and therefore projects to a closed leaf in $U$. These closed leaves therefore foliate $U$.
\end{proof}

\begin{lemma}
\label{cylinders closed}
If $K \subset A$ is compact, then the union of all closed trails contained in $K$ is closed.
\end{lemma}

\begin{proof}
Let $T$ denote union of all closed trails contained in $K$. We need to show that $T$ is closed, so let $p_\infty \in \overline{T} \setminus T$. We will find a closed trail $\ell_\infty \subset K$ containing $p_\infty$.

\compat{The argument here was modified. We now lift two periods to the zebra plane $Z$. This makes it easier to prove that the limit is a closed trail. Dec 18, 2022.}
Since $p_\infty \in \overline{T}$, there is a sequence $p_n \in T$ converging to $p_\infty$. Let $\ell_n \subset T$ denote the closed trail through $p_n$. By \Cref{bounded annulus}, there can be at most two closed trails in $A$ that are not closed leaves, so we may assume that each $\ell_n$ is a closed leaf. Let $\tilde p_\infty \in \pi^{-1}(p_\infty)$ be a preimage. Then we likewise select preimages $\tilde p_n \in \pi^{-1}(p_n)$ such that $\lim \tilde p_n = \tilde p_\infty$. Let $\tilde \ell_n$ denote the portion of the bi-infinite leaf $\pi^{-1}(\ell_n)$ running from $\delta^{-1}(\tilde p_n)$ to $\delta(\tilde p_n)$. Then $\tilde \ell_n$ is a lift of two periods of $\ell_n$. Assume that
\begin{equation}
\label{closed assumption}
\bigcup_n \tilde \ell_n \quad \text{is contained in a compact subset of $Z$.}
\end{equation}
(We will verify this assumption at the end of the proof.)
Then \Cref{compact extension} guarantees that there is a limiting arc of a trail $\tilde \ell_\infty$ joining $\delta^{-1}(\tilde p_\infty)$ with $\delta(\tilde p_\infty)$.
Observe also that $\tilde p_\infty \in \tilde \ell_\infty$.
(If $\tilde p_\infty \not \in \tilde \ell_\infty$ then there is a compact neighborhood $K$ of $\tilde p_\infty$ that is disjoint from $\ell_\infty$. Since $\tilde \ell_n \to \tilde \ell_\infty$ in $\Cl(Z)$, we also must have that $\tilde \ell_n \cap K = \emptyset$ for $n$ large enough. But this is impossible because $\tilde p_n \in \tilde \ell_n$ converges to $\tilde p_\infty$.)
Thus $\tilde \ell_\infty$ is the concatenation of trail arcs:
$$\tilde \ell_\infty = \overline{\delta^{-1}(\tilde p_\infty) \tilde p_\infty} \bullet \overline{\tilde p_\infty \delta(\tilde p_\infty)}=\overline{\delta^{-1}(\tilde p_\infty) \tilde p_\infty} \bullet \delta\Big(\overline{\delta^{-1}(\tilde p_\infty) \tilde p_\infty}\Big).$$
It follows that the image $\ell_\infty \subset A$ of $\tilde \ell_\infty$ is naturally a parameterized closed curve, with a parameterization $\varphi$ coming from the composition of a homeomorphism $[0,1] \to \overline{\delta^{-1}(\tilde p_\infty) p_\infty}$ and the covering map $\pi:Z \to A$.
The parameterized closed curve $\varphi:\R/\Z \to \ell_\infty$ is a closed trail, because for every $t \in \R/\Z$ there is an open interval $I$ containing $t$ such that $\varphi|_I$ lifts to an injective map $I \to \tilde \ell_\infty$. Therefore, the closed curve $\ell_\infty$ satisfies the angle condition at all points $\varphi(t)$ and so is a closed trail.

It remains to verify \eqref{closed assumption}.
Let $\pi:Z \to A$ be the covering map.
By compactness and using \Cref{generalized rectangles} we construct a finite collection of generalized rectangles $\{\tilde P_1, \ldots, \tilde P_k\}$ in $Z$ such that the restriction of $\pi$ to each $\tilde P_i$ is injective and such that $K \subset \bigcup_{i=1}^k
\pi(P_i^\circ)$. Define
$$\tilde \sP = \big\{\delta^m(\tilde P_i):~\text{$i\in \{1, \ldots, k\}$ and $m \in \Z$}\big\}.$$
\compat{This line used to define $\sP$ as the collection of interiors of polygons. I changed it for consistency with the later proof of \Cref{thm:closed trails}.}

Say that a {\em polygonal chain} is a finite sequence of elements $\tilde Q_1, \tilde Q_2, \ldots, \tilde Q_c \in \tilde \sP$ such that for each $j \in \{1, \ldots, c-1\}$, we have $\tilde Q_j^\circ \cap \tilde Q_{j+1}^\circ \neq \emptyset$. We claim that for each $n$, $\tilde \ell_n$ is covered by the union of interiors of polygons in a polygonal chain of no more than $2k+1$ polygons. Fix $n$. Let $\tilde \gamma_n:\R \to Z$ be a parameterization of the bi-infinite leaf $\pi^{-1}(\ell_n)$ such that $\tilde \gamma_n(t+1)=\delta \circ \tilde \gamma_n(t)$ for all $t \in \R$.
Observe that \Cref{no returning trails} guarantees that $\tilde \gamma_n^{-1}(\tilde P_i^\circ)$ is either the empty set or an open interval.
By the periodicity of $\tilde \gamma_n$, we have
$$\tilde \gamma_n^{-1} \big(\delta^m(\tilde P_i^\circ)\big)=m+\tilde \gamma_n^{-1}(\tilde P_i^\circ) \quad \text{for all $m \in \Z$.}$$
The images of these sets under $\pi$ are all the same, so $\ell_n \cap \pi(\tilde P_i^\circ)$ is either empty or is homeomorphic to an open interval.
Choose a minimal subset ${\mathcal C}$ of $\sP=\{\pi(\tilde P_i^\circ):~i=1, \ldots, k\}$ that covers $\ell_n$. Since $\ell_n$ is homeomorphic to a circle and is being covered by open intervals, we can index the collection ${\mathcal C}$ as
$${\mathcal C}=\{Q_j^\circ:~j \in \Z/c\Z\} \quad \text{with $c=|\mathcal C|$}$$
such that $Q_j^\circ \cap Q_{j'}^\circ \cap \ell_n \neq \emptyset$ if and only if $j-j' \equiv \pm 1\pmod{c}$.
We can iteratively lift elements of $\mathcal C$ to cover $\ell_n$ by a polygonal chain of preimages. Recalling that $\tilde \ell_n$ is a lift of two periods of $\ell_n$, we see that $\tilde \ell_n$ can be covered using the interiors of no more than $2c+1$ elements of $\tilde \sP$. (It may be necessary to lift three copies of the polygon whose interior covers the image in $A$ of the endpoints of $\tilde \ell_n$, because there are three preimages of this point in $\tilde \ell_n$.) We have $c \leq k$, so this gives a covering of $\tilde \ell_n$ as described above with at most $2k+1$ polygons as desired.

We will use the above observation about polygonal chains to produce our compact set containing $\bigcup_n \tilde \ell_n$, and thus verifying \eqref{closed assumption}. We need a basic observation about the deck group $\{\delta^m:~m \in \Z\}$ of the covering $Z \to A$. Observe that for every two compact sets $L_1, L_2 \subset Z$, the collection
\begin{equation}
\label{eq:proper discontinuity}
\{n \in \Z:~\delta^n(L_1) \cap L_2 \neq \emptyset\} \quad \text{is finite.}
\end{equation}
(To see this, note that there is a homeomorphism from $A$ to $\R^2$ modulo a nontrivial translation. Then $Z$ can be identified with $\R^2$ and $\delta$ acts by translation, so \eqref{eq:proper discontinuity} holds.) In our setting, it follows that for any compact set $K \subset Z$, the collection
\begin{equation}
\label{eq:proper discontinuity2}
\{\tilde P \in \tilde \sP:~\tilde P \cap K \neq \emptyset\} \quad \text{is finite,}
\end{equation}
because $\tilde \sP$ is a finite collection of compact sets and their images under the deck group. Now let $K_0=\{\delta^{-1}(\tilde p_n)\} \cup \{\delta^{-1}(\tilde p_\infty)\}$, which is compact because $\lim \delta^{-1}(\tilde p_n)=\delta^{-1}(\tilde p_\infty)$. Then \eqref{eq:proper discontinuity2} guarantees that
the collection $\tilde \sP_0 \subset \tilde \sP$ of all polygons in $\sP$ intersecting $K_0$ is finite.
Then for $j \geq 0$, inductively define
$$\tilde \sP_{j+1}=\{\tilde Q \in \tilde \sP:~\text{there is a $\tilde P \in \tilde \sP_j$ such that $\tilde Q \cap \tilde P \neq \emptyset$}\}.$$
Then \eqref{eq:proper discontinuity2} tells us that if $\tilde \sP_j$ is finite then $\tilde \sP_{j+1}$ is finite. So, by induction we conclude that $\tilde \sP_j$ is finite for all $j \geq 0$. Now observe that $\tilde \sP_{2k}$ contains all polygonal chains that start by intersecting $K_0$ and include at most $2k+1$ polygons. Therefore, from the claim in the previous paragraph, we see that the union of closures of the polygon interiors in $\tilde \sP_{2k}$ is a compact set containing $\bigcup \tilde \ell_n$, verifying \eqref{closed assumption} and completing the proof.
\end{proof}

Recall the definitions related to the standard cylinder $C=[-1,1] \times \bbS^1$ given in the paragraph before the theorem on closed trails, \Cref{thm:closed trails}. We will state a version of this theorem that holds for annuli.

\begin{theorem}[Closed trails in annuli]
\label{thm:closed trails2}
Let $A=Z/\langle\delta\rangle$ be an annulus as above. Then one of the following mutually exclusive statements holds:
\begin{enumerate}
\item[(NR)] (Non-realization case) There is no closed trail in $A$.
\item[(Cyl)] (Cylinder case) There is an embedding $\epsilon:C^\circ \to A$ such that the closed leaves in $A$ are precisely the image under the embedding of the vertical closed leaves of $C^\circ$.
\item[(UT)] (Unique trail case) There is a unique closed trail in $A$, and this closed trail has at least one bending angle greater than $\pi$ on each side.
\end{enumerate}
Furthermore,
\begin{enumerate}
\item If $Z$ is convex, then case {\em (NR)} cannot occur.
\item In case {\em (Cyl)}, define $\sigma$ to be the collection of signs $s \in \{\pm \}$ such that
$\epsilon(H^\circ_s)$ has compact closure. Set
$$\ddot{C} = C^\circ \cup \bigcup_{s \in \sigma} \partial_s C.$$ Then there is an embedding $\ddot{\epsilon}: \ddot{C} \to A$ whose restriction $\ddot{\epsilon}|_{C^\circ}$ satisfies {\em (Cyl)} such that for each $s \in \sigma$, the parameterized curve $\ddot{\epsilon}|_{\partial_s C}$ is a closed trail in $[\gamma]$ passing through a nonempty collection of singularities, and every bending angle made when passing through such a singularity on the side of $\ddot{\epsilon}(C^\circ)$ has measure $\pi$. Furthermore, all closed trails in $A$ are obtained as restrictions of $\ddot{\epsilon}$ to vertical circles in $\ddot{C}$.
\end{enumerate}
\end{theorem}
\begin{proof}
The three statements (NR), (Cyl) and (UT) are mutually exclusive, because they correspond to different cardinalities of the set of closed trails in $A$. If there is a unique closed trail, it cannot have a side where all bending angles are $\pi$ by \Cref{open set foliated by closed leaves}. So the second assertion in (UT) holds whenever the closed trail is unique. Statement (1) is a consequence of  \Cref{cor: closed trails exist}.

It remains to show that if there is more than one closed trail in $A$, then statement (Cyl) applies and that statement (2) holds. First suppose $\gamma_1$ and $\gamma_2$ are distinct closed trails. Then they bound a compact subannulus $K(\gamma_1, \gamma_2)$. The union of closed trails in $K(\gamma_1, \gamma_2)$ is both open (by \Cref{open set foliated by closed leaves}) and closed (by \Cref{cylinders closed}) and is therefore all of $K(\gamma_1, \gamma_2)$. Furthermore \Cref{bounded annulus} guarantees that all closed trails in $K(\gamma_1, \gamma_2)$ are closed leaves except possibly for $\gamma_1$ and $\gamma_2$. It also follows from \Cref{open set foliated by closed leaves} that the union $U$ of all closed leaves in $A$ is an open sub-annulus foliated by these closed leaves with leaf space homeomorphic to $(-1,1)$. Therefore, there is a homeomorphism $\epsilon:C^\circ \to U$ as described in statement (Cyl).
\compat{Based on Ferran's comments. I added these last two sentences and commented out a paragraph that formally proved these sentences are correct. I guess it is sufficiently obvious. Dec 16, 2022.}

Now consider statement (2). Fix a sign $s$ and suppose $\epsilon(H^\circ_s)$ has compact closure. Since $H^\circ_s$ is not compact and $\epsilon$ is an embedding, there is a point $p \in \overline{\epsilon(H^\circ_s)} \setminus \epsilon(H^\circ_s)$. Then \Cref{cylinders closed} guarantees there is a closed trail $\gamma_s$ through $p$. Furthermore since $\gamma_s$ is not a subset of $U$, we know that $\gamma_s$ is not a closed leaf. Then,
\Cref{unique case} guarantees that bending angles are all $\pi$ on one side of $\gamma_s$,
and \Cref{open set foliated by closed leaves} guarantees that there is an open set on this side with boundary $\gamma_s$ as one boundary that is foliated by closed leaves. Therefore, the foliation extends to include $\gamma_s$, and the leaf space in a neighborhood of $\gamma_s$ is homeomorphic to a half-open interval in $\R$. Let $\ddot{U} = U \cup \bigcup_{s \in \sigma} \gamma_s$ where $\sigma$ is defined as in the statement. We see that $\ddot{U}$ is a surface with boundary $\partial \ddot{U}=\bigcup_{s \in \sigma} \gamma_s$. The set $\ddot{U}$ is foliated by closed trails, where the boundary components are some of the leaves. Thus $\ddot{U}$ is homeomorphic to the space $\ddot{C}$, which is a trivial $\bbS^1$ over an interval. This homeomorphism is $\ddot{\epsilon}$. We have $\ddot{\epsilon}(\partial_s C)=\gamma_s$, and the bending angles have already been discussed above.

Finally, to see all closed trails are images of vertical leaves in $\ddot{C}$ under $\ddot{\epsilon}$, suppose $\gamma'$ is any closed trail. If it is a closed leaf, then it is contained in $U$ and $\ddot{\epsilon}(C^\circ)$ contains $U$. Since distinct leaves are disjoint, $\gamma'$ must be the image of a vertical leaf in $C^\circ$. Otherwise $\gamma'$ passes through singularities. Again \Cref{unique case} guarantees that bending angles are all $\pi$ on one side of $\gamma'$ and \Cref{open set foliated by closed leaves} guarantees that there is an open set on this side with boundary $\gamma'$ as one boundary that is foliated by closed leaves, so $\gamma'$ is contained in the closure of $U$,
and so must be a boundary component. This means that $\gamma'=\gamma_s$ for some sign $s \in \sigma$ as desired.
\end{proof}

\subsection{The closed trail theorem}
In this section, we prove \Cref{thm:closed trails}.

Let $(S,\{\sF_m\})$ be a surface with a zebra structure, and fix a PR free homotopy class $\conj{\gamma}$ which is nontrivial, non-polar, and not a power.
Choose a nonsingular basepoint $p_0$ and let $[\gamma] \in \prpi(S,p_0)$ be a representative of the conjugacy class. Let $\tilde S$ be the PRU cover, and let $\Delta_\gamma$ denote the deck transformation associated to $[\gamma]$. Then $A=\tilde S/\langle\Delta_\gamma\rangle$ is an annulus by \Cref{annulus}. Observe that we have the sequence of covers
$$\tilde S \xrightarrow{\tilde \pi} A \xrightarrow{\hat \pi} S.$$
Let $\hat p_0 \in A$ be a lift of the basepoint $p_0 \in S$ and let $\tilde p_0$ be a lift of $\hat p_0$ to $\tilde S$.

\begin{proposition}
\label{lifting closed trails}
Every closed trail $\tau$ representing $\conj{\gamma}$ has a lift to $A$.
\end{proposition}
\begin{proof}
Suppose that $\tau:[0,1] \to S$ is a parameterized closed trail representing $\conj{\gamma}$ with $\tau(0)$ nonsingular. Fix a path $\eta$ starting at $p_0$ and ending at $\tau(0)$. Then $[\eta \bullet \tau \bullet \eta^{-1}]\in \prpi(S, p_0)$ lies in the conjugacy class $\conj{\gamma}$. Since $[\gamma]$ lies in the same conjugacy class, there is a $[\beta_0] \in \prpi(S, p_0)$ such that if $\beta_0 \in [\beta_0]$ is a representative then
\begin{equation}
\gamma=\beta_0 \bullet \eta \bullet \tau \bullet \eta^{-1} \bullet \beta_0^{-1} \quad \text{is in} \quad [\gamma].
\label{eq:gamma}
\end{equation}
This curve $\gamma$ lifts to $A$ and the portion of the lift corresponding to the subpath $\tau$ is the desired lift that is a closed trail.
\end{proof}

This proposition enables us to prove part of \Cref{thm:closed trails}.

\begin{lemma}
\label{NR and UT}
If $A$ contains no closed trails then statement (NR) of \Cref{thm:closed trails} holds.
If $A$ contains only one closed trail, then statement (UT) of \Cref{thm:closed trails} holds.
\end{lemma}
\begin{proof}
If there are no closed trails in $A$, then there can be no trails in $\conj{\gamma}$ in $S$. Also if there is a unique closed trail in $A$, then the image of this trail in $S$ must be the unique trail in $S$. The statement involving bending angles of the (UT) case from \Cref{thm:closed trails2} implies the bending angle statement for (UT) in \Cref{thm:closed trails}.
\end{proof}

Recall from \Cref{sect:curves and deck transformations} that $\prpi(S,p_0)$ acts on $\tilde S$ as the group of deck transformations of the cover $\tilde S \to S$. For $[\beta] \in \prpi(S,p_0)$, the corresponding deck transformation $\Delta_{\beta}:\tilde S \to \tilde S$
descends to a well-defined deck transformation $\hat \Delta_{\beta}:A \to A$ of the cover $\hat \pi:A \to S$ if and only if $[\beta] \in N(\Gamma)$ where
$$\Gamma=\langle [\gamma] \rangle \quad \text{and} \quad
N(\Gamma)=\{g \in \prpi(S,p_0)~:~ g \Gamma g^{-1}=\Gamma\} \quad \text{is the normalizer of $\Gamma$.}$$
\commf{This should be $\Gamma=\langle [\beta]\rangle$.}\compat{No, this is correct. The fundamental group of $A$ is identified with $\Gamma=\langle [\gamma] \rangle$ since $A=\tilde S/\langle\Delta_\gamma\rangle$. A deck transformation of $\tilde S$ descends to a deck transformation of $A$ if and only if it normalizes the fundamental group of $A$.}
The group $N(\Gamma)$ contains $\Gamma$ as a normal subgroup, and $\beta_1, \beta_2 \in N(\Gamma)$ induce the same deck transformation of $A$ if and only if they lie in the same coset of the quotient group $N(\Gamma)/\Gamma$. Thus the deck group of the covering $\hat \pi$ is $\hat \Delta \cong N(\Gamma)/\Gamma$. For background on this see \cite{Bredon}.

After \Cref{NR and UT}, it remains to consider the case when $A$ contains a cylinder, case (Cyl) of \Cref{thm:closed trails2}. In this case $A$ contains closed leaves.

\begin{proposition}
\label{closed trail image1}
If $\hat \ell$ is a closed leaf in $A$, then either:
\begin{enumerate}
\item The image $\hat \pi(\hat \ell)$ is a simple closed curve and the restriction of $\hat \pi$ to $\hat \ell$ is a homeomorphism onto its image.
\item The image $\hat \pi(\hat \ell)$ is a saddle connection whose endpoints are distinct poles,
and there is a deck transformation $\hat \iota:A \to A$ that is an involution and restricts to an orientation reversing homeomorphism $\hat \ell \to \hat \ell$ that fixes the preimages of the poles and $\hat \pi|_{\hat \ell}$ descends to a homeomorphism $\hat \ell/\hat \iota \to \hat \pi(\hat \ell)$.
\end{enumerate}
\end{proposition}
\begin{proof}
First suppose $\hat \pi(\hat \ell)$ contains no poles. Then $\hat \pi(\hat \ell)$ must be a closed leaf of one of the directional foliations of $S$. Therefore, $\hat \pi(\hat \ell)$ is a simple closed curve. The restriction of $\hat \pi$ to $\hat \ell$ gives a covering map to $\hat \pi(\hat \ell)$. If this map were of degree $d > 1$, then the parameterized curve $\hat \pi \circ \hat \ell$ which lies in $\conj{\gamma}$ would be a $d$-fold power of the parameterization of the image, making $\conj{\gamma}$ a power, a contradiction. Therefore $\phi|_{\hat \ell}$ is injective and since it is clearly a local homeomorphism, it is a homeomorphism. \compat{In general, \href{https://math.stackexchange.com/questions/55138/is-a-bijective-local-homeomorphism-a-global-homeomorphism-what-about-diffeomorp}{a bijective local homeomorphism is a homeomorphism}.}

Now suppose that $p \in \hat \pi(\hat \ell)$ is a pole. Let $\hat p \in \hat \ell$ be a preimage on $\hat \ell$. The preimage $\tilde \ell=\tilde \pi^{-1}(\hat \ell)$ in $\tilde S$ is a bi-infinite leaf, and the deck transformation $\Delta_\gamma$ translates along $\tilde \ell$. Let $\tilde p \in \tilde S$ be a preimage of $\hat p$. Since $\tilde p$ projects to a pole on $S$, there is an involutive deck transformation $\tilde \iota:\tilde S \to \tilde S$ of the cover $\tilde S \to S$ that fixes $\tilde p$. Observe that $\tilde \iota$ reverses the orientation of $\tilde \ell$, and so conjugates $\Delta_\gamma$ to its inverse. Thus from remarks above, $\tilde \iota$ descends to a deck transformation $\hat \iota:A \to A$ of $\hat \pi$. The transformation $\hat \iota$ induces an orientation reversing homeomorphism of $\hat \ell$, and therefore must have exactly two fixed points: $\hat p$ and some other point $\hat q \in \hat \ell$. Then since $\hat \iota$ is nontrivial and fixes the regular point $\hat q$, the image $q=\hat \pi(\hat q)$ must also be a pole. If there were more poles on $\hat \pi(\hat \ell)$, then we'd get more involutive deck transformations reversing the orientation of $\hat \ell$, and the composition of two would give a nontrivial deck transformation acting as an orientation preserving homeomorphism of $\hat \ell$. But then again the map $\hat \pi|_{\hat \ell}:\hat \ell \to \hat \pi(\hat \ell)$ would factor through a finite covering of the circle making $\conj{\gamma}$ a power. Thus, $p$ and $q$ must be the only poles on $\hat \pi(\hat \ell)$, which must be a saddle connection joining them. Again,
the natural map from $\hat \ell/\hat \iota$ to $\hat \pi(\hat \ell)$ is an injective local homeomorphism and so is a homeomorphism.
\end{proof}

\begin{proposition}
\label{closed trail image2}
If $\hat \ell_1$ and $\hat \ell_2$ are closed leaves in $A$, then the images $\hat \pi(\hat \ell_1)$ and $\hat \pi(\hat \ell_2)$ are either disjoint or they coincide. If $\hat \pi(\hat \ell_1)=\hat \pi(\hat \ell_2)$ then there is a deck transformation $\hat \delta \in \hat \Delta$ such that the restriction of $\hat \delta$ to $\hat \ell_1$ is a homeomorphism to $\hat \ell_2$.
\end{proposition}
\begin{proof}
Suppose $q \in \hat \pi(\hat \ell_1) \cap \hat \pi(\hat \ell_2)$ is a point common to the images.

First, if $\hat \ell_1$ and $\hat \ell_2$ have distinct slopes then because $\hat \pi(\hat \ell_1)$ and $\hat \pi(\hat \ell_2)$ intersect transversely, they would have to bound a bigon
\cite[Bigon Criterion]{FM}. This bigon lifts to $\tilde S$ in contradiction to \Cref{no bigons}. Thus they must have the same slope. It also follows that
$\hat \pi(\hat \ell_1)=\hat \pi(\hat \ell_2)$ since both these are closed leaves or saddle connections joining poles through a common point and of the same slope.

Parameterize $\hat \ell_1$ and $\hat \ell_2$ so that they start at preimages of $q$ and satisfy
\begin{equation}
\label{eq:parameterized the same}
\hat \pi \circ \hat \ell_1(t)=\hat \pi \circ \hat \ell_2(t) \quad \text{for all $t$}.
\end{equation}
For $i \in \{1,2\}$, let $\hat \eta_i$ be a curve in $A$ from the basepoint $\hat p_0$ to $\hat \ell_i(0)$. Set $\hat \gamma_i=\hat \eta_i \bullet \hat \ell_i \bullet \hat \eta_i^{-1}$ for $i=1,2$. Because of our choice of parameterizations for the $\hat \ell_i$, the two curves $\hat \pi(\hat \gamma_i)$ both represent $[\gamma]$ or $[\gamma^{-1}]$.
Now consider the closed curve $\beta=\hat \pi(\hat \eta_2) \bullet \hat \pi(\hat \eta_1^{-1})$. From \eqref{eq:parameterized the same} it follows that $\beta \bullet \hat \pi(\hat \gamma_1) \bullet \beta^{-1}$ is homotopic to $\hat \pi(\hat \gamma_2)$. Since each $[\hat \pi(\hat \gamma_i)] \in \{[\gamma_i^{\pm 1}]\}$, the deck transformation $\Delta_\beta$ descends to a deck transformation of $A$. Furthermore, since $\beta \bullet \pi(\hat \eta_1)$ is homotopic to $\pi(\hat \eta_2)$, it carries $\hat \ell_1(0)$ to $\hat \ell_2(0)$. Thus, because the closed leaves $\hat \ell_1$ and $\hat \ell_2$ have the same slope and this deck transformation of $A$ sends a point on one to a point on the other, it must send $\hat \ell_1$ to $\hat \ell_2$.
\end{proof}

\begin{lemma}
\label{quotient lemma}
Assume $A$ contains a closed leaf. Let $\hat \epsilon:C^\circ \to A$ be the embedding guaranteed to exist from statement (Cyl) of \Cref{thm:closed trails2}.
The deck group $\hat \Delta$ of the covering $\hat \pi:A \to S$ preserves $\hat \epsilon(C^\circ)$. The map
$$\varphi: \hat \epsilon(C^\circ) / \hat \Delta \to S; \quad [\hat p] \mapsto \hat \pi(\hat p)$$
is a homeomorphism onto its image.
\end{lemma}

\commb{Lemma 9.24. I am probably missing something simple but I don’t understand the statement that the image of $\hat{\epsilon}$ is $\hat{\Delta}$ invariant. Couldn’t the image of $\hat{\epsilon}$ be sent away from itself by some elements of $\hat{\Delta}$? I don’t see in the proof where this point is discussed. I would think this is what happens in the second case of Prop. 9.23. Similarly if the image of $\hat{epsilon}$ is not preserved by $\hat{Delta}$ then I don’t understand the quotient appearing in the displayed equation. This confusion is also preventing me from understand Prop. 9.25.}

\compat{I neglected to explain why $\hat \epsilon(C^\circ)$ is $\hat \Delta$-invariant. I added a new first paragraph to the proof. Thanks for pointing this out!}

\begin{proof}
By statement (Cyl) of \Cref{thm:closed trails2}, we know $\hat \epsilon(C^\circ)$ is the union of all closed leaves in $A$. Since $\hat \Delta$ consists of homeomorphisms of $A$ that preserve the zebra structure, each deck transformation must preserve the set $\hat \epsilon(C^\circ)$. \compat{Added this paragraph Dec 31, 2022.}

We now claim that $\varphi$ is a local homeomorphism. Because $\hat \Delta$ is the deck group of the branched covering $\hat \pi$, $\varphi$ is a local homeomorphism except possibly at the preimages of poles. If $p \in \hat \epsilon(C^\circ)$ is such that $\hat \pi(p)$ is a pole, then statement (2) of \Cref{closed trail image1} guarantees that there is a nontrivial deck transformation in $\hat \Delta$ that fixes $p$. Thus, $\varphi$ is also a local homeomorphism at preimages of poles and so is a local homeomorphism.

Then to prove that $\varphi$ is a homeomorphism onto its image, it suffices to show that it is injective. Let $\hat p_1$ and $\hat p_2$ be two points of $\hat \epsilon(C^\circ)$ such that $\hat \pi(\hat p_1)=\hat \pi(\hat p_2)$. Let $\hat \ell_i$ denote the closed leaf in $A$ through $\hat p_i$ for $i \in \{1, 2\}$. Then there is a deck transformation $\hat \delta$ carrying $\hat \ell_1$ to $\hat \ell_2$ by homeomorphism by \Cref{closed trail image2}. If $\hat \pi(\hat \ell_1)$ is simple closed curve then restrictions of $\hat \pi$ to $\hat \ell_1$ and $\hat \ell_2$
are injective by statement (1) of \Cref{closed trail image1}. Thus $\hat \delta(\hat p_1)=\hat p_2$. If $\hat \pi(\hat \ell_1)$ is a saddle connection joining poles then there is an involutive deck transformation $\hat \iota$ preserving $\hat \ell_1$ and by statement (2) of
\Cref{closed trail image1}, we either have $\hat \delta(\hat p_1)=\hat p_2$ or $\hat \delta \circ \hat \iota(\hat p_1)=\hat p_2$.
\end{proof}

\begin{lemma}
\label{TF and C}
If $A$ contains a closed leaf then either statement (TF) or statement (Cyl) of \Cref{thm:closed trails} holds.
\end{lemma}
\begin{proof}
Suppose $A$ contains a closed leaf. Then statement (Cyl) of \Cref{thm:closed trails2} applies and we get an embedding $\epsilon:C^\circ \to A$ that sends vertical closed leaves to closed leaves. The deck group $\hat \Delta$ of the covering $\hat \pi:A \to S$ acts on $ \epsilon(C^\circ) \subset A$ and sends closed leaves to closed leaves by \Cref{closed trail image2}. We will argue that the conclusion of \Cref{quotient lemma} implies one of the statements must hold. To do this we break into cases depending on the nature of the action of $\hat \Delta$.

First suppose the deck group $\hat \Delta$ of the covering $\hat \pi:A \to S$ is trivial. Then we have $\hat \epsilon(C^\circ) / \hat \Delta=\hat \epsilon(C^\circ)$. The composition
$\varphi \circ \hat \epsilon=\hat \pi \circ \hat \epsilon:C^\circ \to S$ is the desired embedding satisfying (Cyl). It includes every closed leaf by \Cref{lifting closed trails}.

The action of $\hat \Delta$ on $\epsilon(C^\circ)$ induces an action on the space of leaves, which is homeomorphic to an open interval. Let $I$ denote this leaf space.

The next simplest case is when $\hat \Delta=\langle \hat \iota \rangle$ where $\hat \iota$ acts on $I$ as an orientation-reversing homeomorphism. In this case $\hat \iota$ must fix a unique point on $I$, and so there is a closed leaf $\ell \subset A$ that is fixed by $\hat \iota$. Furthermore $\hat \iota$ acts on $\ell$ as an orientation-reversing homeomorphism so there are two fixed points in $\ell$. The images of these fixed points must be poles, and $\hat \pi(\ell)$ must be a saddle connection joining these poles. Let $H$ denote one of the two components of $\epsilon(C^\circ) \setminus \ell$. Then we can define $C' = \epsilon^{-1}(H)$ which is a sub-cylinder of $C^\circ$. Using \Cref{quotient lemma}, we see that
$$\hat \pi \circ \epsilon|_{C'}:C' \to S$$
is an embedding of a cylinder satisfying statement (Cyl) of \Cref{thm:closed trails}. In this case, one of the boundaries must be the saddle connection $\hat \pi(\ell)$.

Now suppose there is a nontrivial element $\hat \delta \in \hat \Delta$ that acts on $I$ as an orientation preserving homeomorphism. We claim that the action of $\hat \delta$ on $I$ is fixed-point free. If $\hat \delta$ has a fixed point in $I$, then because it is a deck transformation, it would have to act on the corresponding closed leaf $\hat \ell \subset A$ as an orientation-preserving homeomorphism. Let $\tilde \ell=\tilde \pi^{-1}(\hat \ell)$, which is a bi-infinite leaf, and let $[\delta] \in \prpi(S,p_0)$ denote a preimage of $\hat \delta$. Then the deck transformations associated to $[\gamma]$ and $[\delta]$ both act as orientation-preserving homeomorphisms of $\tilde \ell$, and the quotient $\tilde \ell/\langle [\gamma], [\delta]\rangle$ must be homeomorphic to a circle and so $\langle [\gamma], [\delta] \rangle \cong \Z$. But then, either $[\gamma]$ is a power (violating a hypothesis on $[\gamma]$), or $[\delta] \in \Gamma$ (violating that $\hat \delta$ was nontrivial).

Assuming it is nontrivial, the group of all $\hat \delta \in \hat \Delta$ that act on $I$ as an orientation preserving homeomorphism must be isomorphic to $\Z$. Letting $\hat \delta$ be a generator, we see that $\hat \epsilon(C^\circ) / \langle \hat \delta \rangle$ is homeomorphic to a torus. If this is the full $\hat \Delta$, then \Cref{quotient lemma} tells us that statement (TF) holds. It could be that in addition there is a $\hat \iota \in \hat \Delta$ which acts on $I$ as an orientation reversing homeomorphism. In this case, $\hat \iota$ and $\hat \delta$ must generate $\hat \Delta$, since the composition of any two elements whose actions on $I$ reverse orientation would have to lie in the group $\langle \hat \delta \rangle$. Again $\hat \iota$ must have a unique fixed point in $I$ and $\hat \iota$ must act on the corresponding closed leaf $\hat \ell_0$ as an orientation-reversing homeomorphism with two fixed points which are mapped to poles under $\hat \pi$. The image $\hat \pi(\hat \ell_0)$ is a saddle connection $\sigma_0$ joining the poles. Recalling that $\hat \delta$ was a generator for the orientation-preserving subgroup of $\hat \Delta$, let $\hat j=\hat \delta \circ \hat \iota$ which also acts as an orientation reversing homeomorphism on $I$, and so is an involution
fixing some leaf $\hat \ell_1$. Again $\hat \pi(\hat \ell_1)$ is a saddle connection joining two poles of $S$. Observe that $\hat j$ sends $\hat \ell_0$ to $\hat \delta(\ell_0)$ and so the fixed leaf $\hat \ell_1$ must lie strictly between $\hat \ell_0$ to $\hat \delta(\ell_0)$ . The region from $\hat \ell_0$ to $\hat \ell_1$ forms a fundamental domain for the action of $\langle \hat \delta \rangle$, and so $S=\hat \epsilon(C^\circ) / \langle \hat \delta, \hat \iota \rangle$ is a sphere with four poles. Let $R$ be the region between $\hat \ell_0$ and $\hat \ell_1$, and let $C' = \hat \epsilon^{-1}(R)$ be a subcylinder. Then the map
$$\epsilon':C' \to S; \quad \epsilon'=\hat \pi \circ \hat \epsilon$$
satisfies statement (Cyl) of \Cref{thm:closed trails}, and the boundary of $\epsilon'(C')$ consists of the two saddle connections $\sigma_0$ and $\sigma_1$.
\end{proof}

To complete the proof of \Cref{thm:closed trails}, it remains to prove statements (1) and (2). Statement (1) follows directly from \Cref{closed trails exist}, so (2) remains. We'll prove (2) using the corresponding statement (2) from \Cref{thm:closed trails2}. The difficulty is that it is easier for a subset of the smaller surface $S$ to be compact than for its preimage in $A$. The following result is the main tool we use for dealing with this difficulty.

\begin{lemma}
\label{intersection bound}
Let $P \subset S$ be a generalized rectangle with $n$ vertices. Let $\gamma_1, \gamma_2 \subset S$ be two disjoint closed leaves that bound an annulus $C$ in $S$ containing no singularities. Let $\# (P \cap \gamma_i)$ denote the number of connected components of $P \cap \gamma_i$. Then,
$$\Big|\# (P \cap \gamma_1) - \# (P \cap \gamma_2)\Big| \leq n.$$
\end{lemma}
\begin{proof}
This has to do with the ways the annulus $C$ can intersect $P$. Let $C_0$ be a connected component of $C \cap P$. Say a boundary arc of $C_0$ is a maximal arc of $\gamma_1$ or $\gamma_2$ contained in $\partial C_0$. We claim that at least one of the following statements holds for each $C_0$.
\begin{enumerate}
\item The component $C_0$ either has exactly two boundary arcs with one from $\gamma_1$ and one from $\gamma_2$.
\item The component $C_0$ has fewer than two boundary arcs and $C_0$ contains a vertex of $P$.
\end{enumerate}
The conclusion is immediate, from these statements, because components $C_0$ with two boundary arcs contribute equally to both $\# (P \cap \gamma_1)$ and $\# (P \cap \gamma_2),$ while there can be at most $n$ components with fewer than two boundary arcs.

Now fix $C_0$. Suppose first that $C_0$ has at least two boundary arcs from $\gamma_1$. Then, there is a curve $\beta \subset C_0$ joining these two boundary arc from $\gamma_1$. Since $C$ is an annulus, $C$ is isotopic within $C$ to an arc $\beta'$ of $\gamma_1$ joining the two arcs. Lift $P$ to a polygon $\tilde P \subset \tilde S$. Since $\beta'$ is isotopic to $\beta$, there is a lift $\tilde \beta'$ of $\beta'$ which is a segment of a leaf joining the two lifts of arcs of $\beta_1$. But then $\tilde \beta'$ must exit $\tilde P$ and later return to $\tilde P$ in violation of \Cref{no returning trails}. The same argument works of course when $C_0$ has two boundary arcs from $\gamma_2$, so this proves that (1) must hold when $C_0$ has two or more boundary arcs.

Now suppose $C_0$ has at most one boundary arc. If it has zero boundary arcs then $C_0$ must equal $P$, and so (2) holds trivially. Now suppose $C_0$ has one boundary arc and contains no vertices. Then the one boundary arc must join an edge of $P$ to itself. But this violates \Cref{no bigons}. Thus if $C_0$ has only one boundary arc, it must contain a vertex.
\end{proof}

\begin{proof}[Proof of \Cref{thm:closed trails}]
As discussed above, it remains to prove (2). Let $\epsilon:C^\circ \to S$ be as in statement (Cyl). Let $H_s^\circ \subset C^\circ$ be one of the two halves and suppose that $\epsilon(H_s^\circ) \subset K$ where $K \subset S$ is compact. For $t \in (-1,1)$ let $\ell_t=\epsilon(\{t\} \times \bbS^1)$ be the corresponding closed leaf. The set $\overline{\epsilon(H_s^\circ)} \setminus \epsilon(H_s^\circ)$ consists of accumulation points of $\ell_t$ as $t \to s$. To simplify notation, assume $s=+1$ so that $H_+^\circ=[0,1) \times \bbS^1$. The other case will work in a symmetric way.

Let ${\mathcal F}=\{P_0, \ldots, P_k\}$ be a finite collection of generalized rectangles in $S$ whose interiors cover $K$. Choose a point $q \in \overline{\epsilon(H_+^\circ)} \setminus \epsilon(H_+^\circ)$. We may assume that $q \in P_0^\circ$. We claim that $\ell_t$ intersects $P_0$ for $t$ sufficiently close to $+1$. Observe that $q \not \in \ell_0$ since $\ell_0 \subset \epsilon(H_+^\circ)$. Therefore, we can find a stellar neighborhood $N$ of $q$ with
$N \subset P_0$ and $N \cap \ell_0 =\emptyset$. Since $q$ is an accumulation point of $\ell_t$ as $t \to 1$, we must have that $\ell_{t_0} \cap N \neq \emptyset$ for some $t_0$. Let $\gamma$ be a segment in $N$ of a leaf joining $q$ to $\ell_{t_0} \cap N$, which is minimal in the sense that $\gamma \cap \ell_{t_0}$ consists only of the other endpoint of $\gamma$ which we'll denote by $r$. We claim that if $t$ is between $t_0$ and $1$, then $\ell_t$ must pass through $\gamma$. To see this consider the set $B=\epsilon([0,t] \times \bbS^1)$. Observe that $q \not \in B$ but $r \in B$, so there must be a point in $\gamma \cap \partial B$. Since $\gamma \subset N$ and $N \cap \ell_0=\emptyset$, the point of $\gamma \cap \partial B$ must lie in $\gamma \cap \ell_t$. Therefore, $\ell_t$ passes through $N$ and hence also $P_0$ as desired.

Let $\hat \epsilon:C^\circ \to A$ be a lift of $\epsilon$ to $A$. This lift exists, because we can lift a closed leaf by \Cref{lifting closed trails} and the fundamental group lifting criterion tells us that we can extend to the lift $\hat \epsilon$. (We remark that $\hat \epsilon$ may not be the same map as the map from (Cyl) of \Cref{thm:closed trails2} because it may have a smaller image; see the cases in the proof of \Cref{TF and C}.) Define $\hat \ell_t=\hat \epsilon(\{t\} \times \bbS^1)$, which is a lift of $\ell_t$. We can also lift $r$ to a point on $\hat \ell_{t_0}$ and lift $\gamma$ to a segment $\hat \gamma$. This determines a lift $\hat q$ of $q$.

Now define $\hat {\mathcal P}$ to be the collection of all lifts of the $P_i$ to $A$. Let $\hat P_0$ be the lift of $P_0$ to $A$ such that the lift carries $q$ to $\hat q$. Then for $t$ sufficiently close to $1$, we have $\hat \ell_t \cap \hat P_0 \neq \emptyset$.

We can think of $\hat {\mathcal P}$ as a graph ${\mathcal G}$ where the vertices are elements of $\hat {\mathcal P}$ and there is an edge between two lifts whenever they intersect. Considering the covering $\tilde S \to S$, observe that for any two lifts of the generalized polygons $\tilde P_i, \tilde P_j \subset \tilde S$, the collection
$$\{[\delta] \in \prpi(S,p_0)~:~ \Delta_\delta(\tilde P_i) \cap \tilde P_j \neq \emptyset\}$$
is finite. This implies that all vertices of the graph ${\mathcal G}$ have finite degree.

Observe that by compactness, $\ell_0 \subset S$ passes through finitely many of the polygons in ${\mathcal F}$ counting multiplicity. Then \Cref{intersection bound} gives us an $N$ such that each $\ell_t$ with $t \in (0,1)$ intersects at most $N$ elements of ${\mathcal F}$ counting multiplicity.

The union $\hat U$ of all the element of $\hat {\mathcal P}$ that can be reached from $\hat P_0$ by a path in ${\mathcal G}$ passing through at most $N$ edges is a finite union, and is therefore compact. Also observe that $\hat U$ contains each $\hat \ell_t$ for $t$ sufficiently close to $1$. This is because for $t$ large, $\hat \ell_t$ passes through $\hat P_0$. Then, it moves through at most $N$ sequentially overlapping elements of $\hat {\mathcal P}$ counting multiplicity before closing up. Thus $\hat \ell_t \subset \hat U$. Thus
$\hat \epsilon\big([t_0,1) \times \bbS^1\big) \subset \hat U$ and so
$\hat \epsilon(H_+^\circ) \subset \hat U \cup \hat \epsilon\big([0,t_0] \times \bbS^1\big),$
which is compact.

We've shown that if $\epsilon(H_+^\circ)$ has compact closure, then so does $\hat \epsilon(H_+^\circ)$. The same works for negative signs. As in statement (2), let $\sigma$ denote the collection of signs so that $\epsilon(H_s^\circ)$ has compact closure, or equivalently now $\hat \epsilon(H_s^\circ)$ has compact closure. Then by statement (2) of \Cref{thm:closed trails2}, there is an embedding $\ddot{\hat \epsilon}:\ddot{C} \to A$ of
the partial closure
$\ddot{C}=C^\circ \cup \bigcup_{s \in \sigma} \partial_s C$
into $A$ whose image is $\hat \epsilon(C^\circ) \cup \bigcup_{s \in \sigma} \hat \epsilon(H_s^\circ)$. (Note that the image of $\hat \epsilon$ may not include all closed leaves of $A$, but is always a sub-cylinder, and so we can define $\ddot{\hat \epsilon}$ regardless.)
We define $\ddot{\epsilon}=\hat \pi \circ \ddot{\hat \epsilon}:\ddot{C} \to S$, and observe that it satisfies statement (2), because it satisfies the corresponding statement of \Cref{thm:closed trails2}.\end{proof}

\subsection{Noncompact translation surfaces}
\label{sect:noncompact translation surface}

\begin{proof}[Proof of \Cref{noncompact translation surface}]
Let $S$ be a noncompact translation surface whose universal cover $\tilde S$ is geodesically convex.
Fix a nontrivial deck transformation $\delta:\tilde S \to \tilde S$. We will show $\delta$ is a hyperbolic isometry. Let $p \in S$ be a basepoint and $\tilde p \in \tilde S$ be a preimage.
From covering space theory, there is loop $\gamma \subset S$ based at $p$ such that given any point $\tilde q \in \tilde S$, the point $\delta(\tilde q)$ is the endpoint of the lifted concatenation $\widetilde{\gamma \bullet \alpha}$ starting at $\tilde p$, where $\tilde \alpha$ is a path from $\tilde p$ to $\tilde q$ whose image in $S$ is $\alpha$. From our hypotheses and \Cref{thm:closed trails}, there is a closed geodesic $g:[0,1] \to S$ in the free homotopy class $[\gamma]$, and so we can replace $\gamma$ by a concatenation of the form $\beta \bullet g \bullet \beta^{-1}$ where $\beta$ is a path joining $p$ to $g(0)$. Let $\tilde q$ be any endpoint of the lift $\widetilde{\beta \bullet \bar g}$ where $\bar g$ starts at $g(0)$ and wraps around $g$ any number of times, stopping anywhere on the curve.
The set of all $\tilde q$ that are attainable in this way is a lift of the universal cover of $g$ and is therefore a geodesic.
We have that $\delta(\tilde q)$ is the endpoint of the lift $\widetilde{\beta \bullet g \bullet \bar g}$, so $\delta$ translates along this geodesic. By the equivalence between translating along a geodesic and an isometry being hyperbolic described in \Cref{sect:translation surfaces}, we see that $\delta$ is hyperbolic.
\end{proof}

\section{Triangulating closed surfaces}
\label{sect:loop}

The goal of this section is to show that closed zebra surfaces without full cylinders can be leaf triangulated. Once this is stablished, we give a short proof of \Cref{conj:zebra case} in~\Cref{secc:proof-conj-zebra case}.

\subsection{Triangulating polygons and cylinders}
    \label{ssec:triangulating-polygons-cylinders}
Polygons in zebra surfaces were defined in \Cref{sect:polygons}. A {\em diagonal} of a polygon is a segment of a leaf whose interior is contained in the interior of the polygon and whose endpoints are vertices of the polygon.

\begin{lemma}
\label{triangulating polygons}
Suppose $P$ is a polygon in a zebra surface $S$ whose interior contains only nonsingular points. Then $P$ has a triangulation such that the vertex set is the same as the vertices of $P$ and the edge set consists of the edges of $P$ together with a collection of diagonals with disjoint interiors.
\end{lemma}
\begin{proof}
By \Cref{ngons}, $P$ has at least $3$ sides. If the polygon has 3 sides, we are done. By induction, it suffices to prove that given any polygon $P$ with four or more sides, a diagonal can be found.

We will now observe that by lifting to a branched cover $S'$ of $S$, we can assume $P$ is convex. To construct this cover, first construct a double branched cover, with double branching over the vertices, and with no branching in the interior of $P$. (Such a cover can always be produced if the number of points branched over is even, but you can always add an additional point in the interior of an edge.) Then $P$ can be lifted to $S'$, and all exterior angles will measure at least $\pi$, because the exterior angles at a vertex $v \in P$ have all been increased by at least the cone angle at $v$. Therefore $P$ is convex by \Cref{thm:polygonal convexity}.

Suppose $P$ has $n \geq 4$ edges. By \Cref{ngons}, $P$ has at least three
angles of measure less than $\pi$. Choose a pair of distinct vertices $(v,w)$ of $P$ such that both arcs of $\partial P \setminus \{v,w\}$ contain points where the interior angle is less than $\pi$. Since $P$ is convex, there is a trail arc $\overline{vw}$ is contained in $P$.
From our choice of points, $\overline{vw}$ can be neither boundary arc, and so must pass through the interior of the polygon. A maximal segment through the interior of $P$ gives
the desired diagonal, completing the argument in this case and proving that any polygon without interior singularities has a triangulation. \compatnew{Proof improved based on both your comments and questions. June 6, 2023.}
\end{proof}

We now turn our attention to triangulating cylinders.
Recall the standard cylinder $C=[-1,1] \times \bbS^1$ where $\bbS^1=\R/\Z$ with boundary components $\partial_\pm C=\{\pm 1\} \times \bbS^1$. We'll say that a {\em subcylinder} in a zebra surface $S$ is an immersion $\iota:C \to S$ such that
\begin{itemize}
\item The restriction of $\iota$ to $C^\circ=(-1,1) \times \bbS^1$ is an embedding
\item For $x \in (-1,1)$, $\iota(\{x\} \times \bbS^1)$ is a closed leaf.
\item For $s \in \{\pm\}$, $\iota(\partial_s)$ is a closed trail such that all angles on the side of $\iota(C^\circ)$ are $\pi$. \compat{I had written $\ddot{\epsilon}$ for $\iota$, but this was an error. Both of you requested `remind the notation of theorem 1.3', but I think this request stemmed from this error. I don't think this is needed now, but please let me know if you disagree.}
\end{itemize}
Note that for a subcylinder, we require that the immersion is defined on both boundary components but do not require that the boundary components pass through singularities. Assuming $S$ has at least one singularity, \Cref{thm:closed trails} guarantees that any subcylinder is contained in a cylinder, and for closed surfaces cylinders are naturally maximal subcylinders in a fixed PR free homotopy class by \Cref{cor:closed trails}.

As with cylinders, a subcylinder is {\em full} if every slope $m \in \hat \R$ is realized by some closed leaf or trail $\iota(\{x\} \times \bbS^1)$ with $x \in [-1,1]$.

Let $V \subset C$ be a finite set that intersects both boundary components and let $\iota:C \to S$ be a subcylinder. We'll say that a {\em triangulation of the subcylinder} with vertex set $V$ is a topological triangulation of $C$ with vertex set $V$ such that restrictions of $\iota$ to edges give parameterizations of arcs of trails in $S$. A subcylinder cannot contain singularities in its interior and as the interior angles at singularities in the boundary are always $\pi$, it does not make sense to think of a subcylinder as having any singularities in terms of its intrinsic geometry. So, we do not require that $V$ have any relation to the singularities of the surface.
With this definition, each component of $\partial_\pm C \setminus V$ is necessarily an edge. These are the only edges that might be arcs of trails on $S$; the others are segments of leaves.

\begin{lemma}
\label{triangulating cylinders}
Let $\iota:C \to S$ be a subcylinder and let $V$ be a finite subset of $\partial C$ that intersects both boundary components of $C$.
If the subcylinder $\iota$ is not full, then it has a triangulation whose vertex set is $V$.
\end{lemma}
\begin{proof}
Choose a slope $m$ that is not realized by the closed trails making up the subcylinder. Let $p \in V$. Consider the ray $r$ of slope $m$ emanating from $p$ and traveling into the interior $\iota(C^\circ)$.

We claim that $r$ must reach the boundary component that does not contain $p$.
It cannot exit through the boundary component containing $p$ because it would create a bigon. So, if it does not exit as claimed, then it never exits. Assume this is the case, and we will derive a contradiction. Pulling back the ray, we get a path $\gamma:[0,+\infty) \to C$ such that $\iota \circ \gamma\big([0,+\infty)\big)$ is the ray and $\iota \circ \gamma(0)=p$. Let $\pi_x:C \to [-1,1]$ be the projection from $C$ onto the $x$-coordinate. Note that whenever $r$ crosses a fiber $\pi_x^{-1}(t)$, it cannot return because this would create a bigon. Therefore $\pi_x \circ \gamma$ is strictly monotonic. Assume without loss of generality that $\pi_x \circ \gamma$ is strictly increasing. Then $\pi_x \circ \gamma\big([0,+\infty)\big)=[-1, c)$ for some $c \in (-1,1]$. Choose any point $q \in \pi_x^{-1}(c)$ and let $s$ be a segment of a leaf of slope $m$ that contains $q$ in its interior. Since $s$ is transverse to $\pi_x^{-1}(c)$, there is a $c'<c$ such that $\pi_x^{-1}(c')$ intersects $s$. Let $Q$ be the quadrilateral formed from arcs of $s$ and
$\pi_x^{-1}(c')$ and $\pi_x^{-1}(c)$ viewed as segments joining $s$ to itself.
Since $r$ crosses $\pi_x^{-1}(c')$, $r$ eventually crosses into $Q$. But, $r$ can never exit because $c \not \in \pi_x \circ r\big([0,+\infty)\big)$ and can't cross $s$. This violates the properness of the lift of $r$ to the PRU cover, proving the claim; see \Cref{leaves are proper maps}.

Let $p'$ be the point at which $r$ exits $C$. Then $p$ and $p'$ are from distinct boundary components of $C$. \commb{They could be the same in the topology of $S$. There are distinct on $C$.} \compat{I clarified. We are triangulating $C$ not $\iota(C)$...}
Cutting along $r=\overline{pp'}$ makes the cylinder into a polygon with vertex set consists of $V \setminus \{p,p'\}$ and two copies of both $p$ and $p'$, which can be triangulated by \Cref{triangulating polygons}. This won't be the triangulation we want unless $p' \in V$.
Choose a $q \in V$ in the boundary component containing $p'$. Let $T \subset C$ be a triangle of
the triangulation with vertex $q$. Then one edge of $T$ must join $q$ to a point on the other boundary component (because all trails to points on the same component are parallel). Therefore, there is an edge $\overline{qy}$ of $T$ with $y$ in the opposite boundary from $q$. This $y$ must lie in $V$ because $p'$ lies on the component containing $q$. Now we cut the cylinder along $\overline{qy}$ and see a polygon and take $V$ to be the vertex set. We can triangulate this polygon using \Cref{triangulating polygons} to obtain the desired triangulation of $C$.
\end{proof}

\subsection{Minimal triangulations}
\label{sect:minimal triangulations}
We devote this subsection to a proof of the following result: \compat{In response to Barak: I've been teaching too much `Intro to Proofs'. In such a course, a {\em direct proof} of the implication $A \implies B$ assumes that $A$ is true and proves that $B$ is true. I deleted the word.}

\begin{theorem}
\label{no full implies leaf}
Let $S$ be a closed zebra surface that has no poles but has at least one singularity.
If $S$ has no full cylinders, then it has a leaf triangulation.
\end{theorem}

Fix a closed zebra surface $S$ that has no poles and at least one singularity for the remainder of this subsection.
Let $\Sigma \subset S$ denote the singular set.
The hypothesis that there are no full cylinders will only appear in the proof of \Cref{no full implies leaf}.

We consider a class of triangles more general than leaf triangulations. Say that a triangulation $\sT$ of $S$ is {\em preleaf} if
its vertex set $\sV(\sT)$ contains $\Sigma$ and all edges are segments of leaves. Then $\sV_0(\sT) = \sV(\sT) \setminus \Sigma$ is a finite collection of nonsingular points. Observe:

\begin{proposition}
The surface $S$ has a preleaf triangulation.
\end{proposition}
\begin{proof}
Construct triangles covering each singular point with a vertex at the singular point, and cover the remainder of the surface with generalized rectangles. By compactness, this covering can be taken to be finite. Then the union of the boundaries of the polygons in the cover divides the surface into polygons. Each such polygon can can be triangulated using \Cref{triangulating polygons}.
\end{proof}

We will say that a preleaf triangulation $\sT$ is {\em minimal} if it minimizes the cardinality of $\sV_0(\sT)$ among all preleaf triangulations of $S$. To prove \Cref{no full implies leaf}, we will end up showing that on a surface without full cylinders, $\sV_0(\sT)=\emptyset$.

\begin{proposition}
Suppose $\sT$ is a minimal preleaf triangulation of $S$ and $p \in \sV_0(\sT)$.
Then there is an edge of $\sT$ that joins $p$ to itself.
\end{proposition}
\begin{proof}
If this is not true, then the union of triangles with vertex $p$ forms a polygon in $S$ with $p$ in its interior. This polygon can be triangulated using only its boundary vertices by \Cref{triangulating polygons}. This removes $p$ from the triangulation, which contradicts minimality.
\end{proof}

Say a vertex $p \in \sV_0(\sT)$ of the triangulation has {\em minimal complexity} if $\sT$ realizes the minimum number of edges emanating from $p$ among all preleaf triangulations with vertex set $\sV(\sT)$. We count an edge from $p$ to itself twice in this count, since it emanates from $p$ in two ways.

\begin{proposition}
\label{exactly one edge}
If $\sT$ is a minimal preleaf triangulation of $S$ and $p \in \sV_0(\sT)$ has minimal complexity, then there is exactly one edge joining $p$ to itself.
\end{proposition}

\begin{proof}
\compat{Improvements were made to the wording of this proof based on Barak's comments. June 6.}
Suppose to the contrary that there are at least two such edges. We break into cases depending on how these edges are arranged.

First suppose there is a triangle $T$ all of whose vertices are $p$. Let $\alpha$, $\beta$ and $\gamma$ denote the angles of $T$, and use $m_\alpha, m_\beta, m_\gamma \in (0, \pi)$ to denote their respective measures. Consider the circle's worth of rays emanating from $p$, $U_p$. (This is the zebra analog of the unit tangent space.) We use angles on $U_p$ to explain the contradiction in this case. Because $p$ is nonsingular, $U_p$ can be identified with $\R/2 \pi \Z$. There is a closed interval $I(\alpha) \subset U_p$ of length $m_\alpha$
consisting of rays that enter $T$ through the angle $\alpha$. Similarly, there are also intervals $I(\beta), I(\gamma) \subset U_p$ associated to $\beta$ and $\gamma$, respectively.
Observe that for any edge $e$ of $T$, the two rays that emanate from $p$ and travel along $e$ constitute an antipodal pair of points in $U_p$. This is because these rays have the same slope, but are distinct as rays on the surface.
\compat{They could fail to be distinct if there was a pole along the edge, but we ruled that out. I deleted this comment from the text, because it probably makes things more confusing...}
Therefore
$$U_p \setminus \big(I(\alpha) \cup I(\beta) \cup I(\gamma)\big) =
\big(\pi + I(\alpha)^\circ\big)  \cup
\big(\pi + I(\beta)^\circ\big)  \cup
\big(\pi + I(\gamma)^\circ\big),$$
i.e., this complementary region consists of points antipodal to points in the union
of the interiors of the intervals. See \Cref{fig:triangle tangent space}. Now choose the edge $e$ of $T$ joining the angles $\alpha$ and $\beta$. Let $T'$ be the triangle opposite $e$ from $T$ and let $q \in S$ be the vertex of $T'$ opposite $e$.
One possibility is that $q=p$, but in this case, because the total of the interior angles of $T$ and $T'$ is $2\pi$, all rays in $U_p$ must immediately enter $T \cup T'$. Therefore, the intervals associated to the angles of $T'$ must be the complement of the intervals associated to $T$. It follows that each edge of $T$ is the same as an edge of $T'$, and therefore $S$ is a torus with no singularities contradicting the assumption that $S$ had at least one singularity. Now suppose $q \neq p$. This is the situation depicted in \Cref{fig:triangle tangent space}. Let $\alpha'$ and $\beta'$ be the angles of $T'$ adjacent to the angles $\alpha$ and $\beta$.
Observe that $m_{\alpha'} \leq m_\beta$, because the interval $I(\alpha') \subset U_p$ associated to $\alpha'$ is contained in $\pi+I(\beta)$.
Thus $m_\alpha + m_\alpha' < \pi$ and by a similar argument we also have $m_\beta+m_{\beta'}<\pi$. By \Cref{thm:polygonal convexity}, $Q=T \cup T'$ is convex in the sense that a connected preimage in the universal cover is convex. The edge $e$ is one diagonal, and the other diagonal $e'$ is also a segment of a leaf because the trail joining these two vertices cannot be a trail running in either direction around $\partial Q$, because the two sums $m_\alpha+m_{\alpha'}$ and $m_\beta+m_{\beta'}$ are both less than $\pi$. Replacing the edge $e$ with $e'$ reduces the number of edges emanating from $p$ by one, which contradicts the assumption that $p$ has minimal complexity. This rules out the possibility that there is a triangle all of whose vertices are $p$.

\begin{figure}
\includegraphics[height=1in]{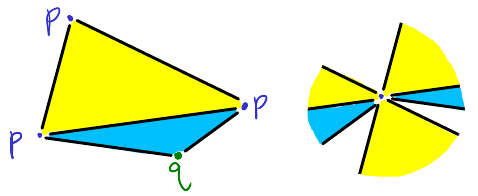}
\caption{Left: A triangle $T$ all of whose vertices are the same point is shown in yellow, and a triangle $T'$ is shown in blue.
Right: The two triangles from the viewpoint of $p$.}
\label{fig:triangle tangent space}
\end{figure}

The second possibility is that there are distinct edges $e$ and $e'$ joining $p$ to itself, but that no such pair of edges are adjacent in the cyclic ordering around $p$. Focus on one such edge $e$ and let $T_1$ and $T_2$ be the triangles sharing this edge. Let $\alpha_1$ and $\beta_1$ be the angles of $T_1$ along $e$, and let $\alpha_2$ and $\beta_2$ be the same for $T_2$, with $(\alpha_1, \alpha_2)$ and $(\beta_1, \beta_2)$ being adjacent angles. From the argument above with $Q=T_1 \cup T_2$, if $m_{\alpha_1}+m_{\alpha_2}<\pi$ and $m_{\beta_1}+m_{\beta_2}<\pi$, then the diagonal $e$ of $Q$ can be replaced with the other diagonal of $Q$. Note that the condition that edges joining $p$ to itself cannot be adjacent implies that this other diagonal does not have $p$ as either vertex, so this again contradicts the hypothesis that $p$ has minimal complexity.
We conclude that one of the pair of angles sums to at least $\pi$, and so the total of all four angles adjacent to $e$, $m_{\alpha_1}+m_{\alpha_2}+m_{\beta_1}+m_{\beta_2}$, exceeds $\pi$. The same holds for the four angles adjacent to $e'$. Moreover, since $e$ and $e'$ are not adjacent in the cyclic ordering around $p$, these collections of four angles must be disjoint. This gives eight distinct angles based at $p$ whose sum exceeds $2 \pi$, and this contradicts the hypothesis that $p$ is nonsingular ($p \in \sV_0(\sT)$).
\end{proof}

Now suppose we are in the setting where there is exactly one edge $e$ of $\sT$ that joins $p \in \sV_0(\sT)$ to itself. Then $e$ is a closed leaf. Orient $e$ and let $\conj{e}$ denote the PR free homotopy class of $e$. Since $S$ also has a singularity, the homotopy class $\conj{e}$ lies in the ``Cylinder case'' of \Cref{thm:closed trails}. Thus, $e$ then lies in the interior of the cylinder associated to $\conj{e}$.

Let $U \subset S$ be the open subset consisting of the union of $\{p\}$, the interiors of edges with $p$ as an endpoint, and the interiors of triangles with vertex $p$. Observe that there is a retraction of $U$ onto $e$, which sends the two triangles with edge $e$ onto $e$, and maps every other point to $p$. Thus $U$ is an annulus with the loop $e$ being a core curve.

\begin{proposition}
\label{maximal subcylinder}
There is a unique maximal subcylinder $\iota:C \to S$ of the cylinder determined by $\conj{e}$
such that $e \subset \iota(C^\circ) \subset U$. Both boundary components of this subcylinder contain vertices of $\sT$.
Every edge $e'$ of $\sT$ in the boundary of $U$ is either contained entirely in $\iota(\partial C)$, is disjoint from $\iota(\partial C)$, or we have $e' \cap \iota(\partial C) \subset \partial e'$.
\end{proposition}
\begin{proof}
\compat{This proof was completely rewritten based on the issue raised by Barak. June 7, 2022.}
As noted above, the homotopy class $\conj{e}$ is realized by core curves of a cylinder.  The surface $S$ is closed, so statement (2) of \Cref{thm:closed trails} guarantees that there is a continuous map $\ddot{\epsilon}:C \to S$ defined on the full standard cylinder $C=[-1, 1] \times \bbS^1$ such that:
\begin{itemize}
\item The restriction of $\ddot{\epsilon}$ to $C^\circ$ is an embedding.
\item For $t \in (-1,1)$, $\ddot{\epsilon}(\{t\} \times \bbS^1)$ is a closed leaf.
\item For each sign, $\ddot{\epsilon}(\{\pm 1\} \times \bbS^1)$ is a closed trail passing through a nonempty set of singularities.
\end{itemize}
Let $\gamma_t$ denote $\ddot{\epsilon}(\{t\} \times \bbS^1)$.

Let $t_0$ be such that $\gamma_{t_0}=e \subset U$. Set
$$b=\sup J_+ \quad \text{where} \quad J_+ = \{t \in [t_0,1]:~\ddot{\epsilon}([t_0, t] \times \bbS^1) \subset U\}.$$
Observe that the sets $\ddot{\epsilon}([t_0, t])$ are increasing so $J_+$ is an interval containing $t_0$ and $\ddot{\epsilon}\big([t_0, b)\big) \subset U$. On the other hand, $1 \not \in J_+$ because $\gamma_1$ contains a singularity and $U$ does not. More generally, $J_+$ is a half-open interval $[t_0, b)$. (If $t \in J_+$, then by applying \Cref{open set foliated by closed leaves} to the annulus $U$, we can extend the curve $\gamma_t$ to a foliation of a neighborhood of the curve contained in $U$.) Therefore, $\gamma_b$ must intersect $\partial U$, which consists of edges of the triangles making up $U$. Since the edges are segments of leaves, each edge is either contained in $\gamma_b$ or is intersected transversely. But, transverse intersections cannot occur because they would imply that the edge also intersects some $\gamma_t$ with $t \in [t_0, b)$ but by construction each such $\gamma_t \subset U$. We conclude that the collection of vertices $V_b = \gamma_b \cap \sV(\sT)$ is nonempty. Similarly, define
$$a=\inf J_- \quad \text{where} \quad J_- = \{t \in [-1,t_0]:~\ddot{\epsilon}([t, t_0] \times \bbS^1) \subset U\}.$$
By a symmetric argument $V_a = \gamma_a \cap \sV(\sT)$ is nonempty.

Define $\iota$ to be a reparameterization of $\ddot{\epsilon}$ restricted to $[a,b] \times \bbS^1$:
$$\iota:C \to S \quad \text{by} \quad \iota(t,y)=\ddot{\epsilon}(\tfrac{a+b}{2}+ \tfrac{(b-a)t}{2}, y).$$
We claim that $\iota$ is the maximal subcylinder desired. To see this observe that because $\iota(C)$ touches both boundaries of $U$, the difference $U \setminus \iota(C)$ is a finite collection of interiors of polygons. Such polygons cannot contain closed leaves, so all the closed leaves in $U$ must be contained in $\iota(C^\circ)=\ddot{\epsilon}\big((a,b) \times \bbS^1\big)$. The other two statements in the proposition follow from the analysis of $\gamma_a$ and $\gamma_b$ and the discussion of edges in the previous paragraph.
\end{proof}

\begin{proof}[Proof of \Cref{no full implies leaf}]
Let $S$ be a closed zebra surface that has no poles but has at least one singularity. Also assume that $S$ has no full cylinders. We will show it has a leaf triangulation.

Suppose to the contrary that there is no leaf triangulation. Then if $\sT$ is a minimal preleaf triangulation, the set $\sV_0(\sT)$ of nonsingular vertices is nonempty. As above, select a $p \in \sV_0(\sT)$. We can assume without loss of generality that $\sT$ is chosen such that $p$ has minimal complexity. By \Cref{exactly one edge}, there is exactly one edge $e$ from $p$ to itself. We define the neighborhood $U$ of $p$ as above \Cref{maximal subcylinder} and this proposition gives us that there is a subcylinder $\iota:C \to S$
such that $e \subset \iota(C^\circ) \subset U$ and that $\sV(\sT) \cap \iota(\partial_\pm C) \neq \emptyset$ for each choice of sign.

By hypothesis, the subcylinder $\iota$ cannot be full. Therefore, \Cref{triangulating cylinders} guarantees that $\iota$ can be triangulated where the vertex set is $\sV(\sT) \cap \iota(\partial C)$. Note that $p$ is not in the vertex set of this triangulation, but all vertices are elements of $\sV(\sT)$. Also observe that the difference $U \setminus \iota(C)$ is a finite union of interiors of polygons with vertices in the set $\sV(\sT) \setminus \{p\}$. By \Cref{triangulating polygons}, we can triangulate each such polygon, thus obtaining an alternate triangulation of $U$ that does not use $p$ as a vertex. By changing the triangulation $\sT$ by replacing the portion of the triangulation covering $U$ with our new triangulation of $U$, we reduce the number of vertices by one thereby showing that $\sT$ is not a minimal preleaf triangulation. This is our desired contradiction.
\end{proof}

\subsection{Producing new triangulations}
    \label{ssec:producing-new-triangulations}

In \Cref{sect:minimal triangulations} we showed that there are leaf triangulations of closed zebra surfaces with no poles and at least one singularity with angle $3 \pi$ or more.
Here our goal is to show that there are many triangulations as long as the closed surface meets a few basic requirements which now allow poles.

\compat{The definitions in the next three paragraphs were made more formal. Hopefully it helps. June 10, 2023.}
To state our results, we will generalize the notions of a leaf saddle connection and a leaf triangulation given in the introduction. Let $\sV \subset S$ be a finite set containing all singularities with angle greater than $2 \pi$, but possibly with some additional nonsingular points and poles that we treat as marked points. Let $\tilde S$ denote the PRU-cover of $S$, and let $\tilde \sV \subset \tilde S$ denote the collection of preimages of points in $\sV$.

A {\em $\tilde \sV$-saddle connection} $\tilde \sigma \subset \tilde S$ is a closed segment of a leaf (possibly the full leaf) such that $\tilde \sigma \cap \tilde \sV$ consists of the two endpoints of $\tilde \sigma$. Denote this pair of endpoints $\partial \tilde \sigma$. The interior of $\tilde \sigma$ is $\tilde \sigma \smallsetminus \partial \tilde \sigma$. A {\em $\tilde \sV$-triangulation} is a triangulation of $\tilde S$ whose edges are $\tilde \sV$-saddle connection. By \Cref{triangle1}, the triangles in a $\tilde \sV$-triangulation have no singularities in their interiors.

A {\em $\sV$-saddle connection} on $S$ is a path $\sigma:[0,1] \to S$ whose lifts to $\tilde S$ are $\tilde \sV$-saddle connections. The {\em interior} of $\sigma$ is $\sigma\big((0,1)\big)$, the image of the interior of a lift. Note that the interior of $\sigma$ can contain a pole $p \in S \setminus \sV$, but in this case $\sigma$ ``bounces off $p$,'' i.e., $\sigma$ has a reparameterization $\sigma'$ such that $\sigma'(\frac{1}{2})=p$ and $\sigma'(t)=\sigma'(1-t)$ for all $t \in [0,1]$. If the interior of $\sigma$ does not contain a pole, then $\sigma$ is a segment of a leaf of $S$ and so the restriction $\sigma|_{(0,1)}$ is injective. A {\em $\sV$-triangulation} is a collection $\sT$ of $\sV$-saddle connections with pairwise disjoint interiors such that the collection $\tilde \sT$ of lifts to $\tilde \sV$-saddle connections on $\tilde S$ is the edge set of a $\tilde \sV$-triangulation.
Note that since the complementary regions of $\tilde \sT$ are required to be triangles, every pole $p \in S \setminus \sV$ must lie in the interior of a $\sV$-saddle connection in $\sT$. (Further, this saddle connection in $\sT$ containing the pole $p$ must be unique because interiors are required to be disjoint.) Examples of $\sV$-triangulations are shown in \Cref{fig:v triangulation}.

\begin{figure}
\includegraphics[height=1in]{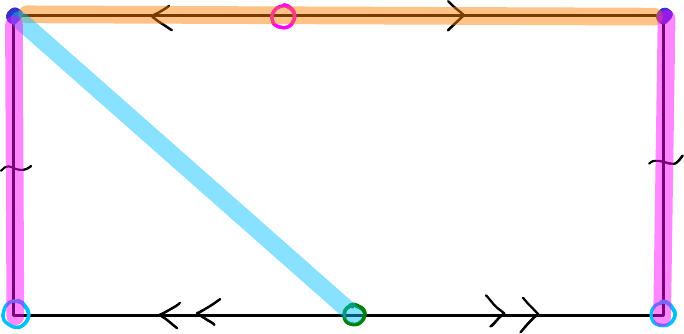}
\hspace{0.5in}
\includegraphics[height=1in]{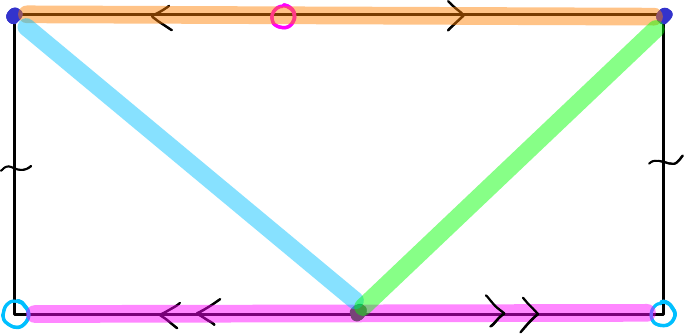}
\caption{Two $\sV$-triangulations of a dilation pillowcase (a sphere with four poles).
The set $\sV$ is indicated by filled in circles. On the left $\sV$ consists of just one pole, and on the right it consists of two poles. The other poles are marked with open circles. \compat{Added June 10, 2023.}}
\label{fig:v triangulation}
\end{figure}

\begin{theorem}
\label{thm:triangulation}
Let $S$ be a closed zebra surface and let $\sV \subset S$ be a finite nonempty set including every singularity of $S$ with angle greater than $2 \pi$. Assume that any full subcylinder in $S$ has a point of $\sV$ in its interior. \compat{I changed cylinder to subcylinder in the previous sentence. Barak, thanks for pointing this out.} If $\sS_0$ is any collection of $\sV$-saddle connections with pairwise disjoint interiors, then there is a $\sV$-triangulation $\sT$ containing $\sS_0$.
\end{theorem}

Let $\sS_0$ be a collection of $\sV$-saddle connections with pairwise disjoint interiors. A {\em maximal} collection of $\sV$-saddle connections with pairwise disjoint interiors is a collection that is maximal with respect to the inclusion partial order. The fact that there are maximal collections, follows from the fact that there is an upper bound on the size of such a collection that only depends on $S$ and $\sV$:
\compat{Previously there was a Zorn's Lemma argument that said that maximal collections exist. But, we need not just that the collection exists but that it is finite so that we can analyze the complementary regions. This argument is painful but it gives precisely the number of edges in a maximal collection of $\sV$-saddle connections. There might be a softer argument that gives finiteness without giving the exact value... June 10, 2023.}

\begin{proposition}
\label{bound implies triangulation}
Let $S$ and $\sV$ be as in \Cref{thm:triangulation}. As in this theorem, assume that
$\sS_0$ is any collection of $\sV$-saddle connections with pairwise disjoint interiors.
Let $\alpha:S \to \Z_{\geq -1}$ be the singular data function of $S$. Let $P \subset S$ be the collection of poles. Then,
\begin{equation}
\label{eq:edge bound}
|\sS_0| \leq 3 |\sV|  + \tfrac{1}{2}|P \setminus \sV| + \tfrac{3}{2} \sum_{v \in \sV} \alpha(v).
\end{equation}
Moreover, equality is attained if and only if $\sS_0$ is a $\sV$-triangulation.
\end{proposition}
\begin{proof}
Suppose that $|\sS_0|$ is greater than or equal to the expression on the right side of
\eqref{eq:edge bound}. Let $C$ be the subsurface in the complement of the union of the saddle connections in $\sS_0$. Based on our hypothesis, it suffices to show the comparison in \eqref{eq:edge bound} is equality, that every element of $\sV$ is a vertex of some $\sS_0$, and that all components of $C$ are triangles that contain no poles in their interiors.

Let $c$ denote the number of connected components of $C$, and write $C=\bigsqcup_{i=1}^c C_i$. Each $C_i \subset C$ has a boundary consisting of one or more boundary components. As a consequence of \Cref{loop contribution}, the sum of the interior angles of $C_i$ must be an integer multiple of $\pi$, say $n_i \pi$ for some $n_i \geq 1$. When a boundary component is traversed, we alternate between edges and vertices. Let $e_i$ denote the number of edges in $\partial C_i$ counting multiplicity, which also equals the number of vertices in $\partial C_i$. By an application of the \gaussbonnet, we see that
$$e_i-n_i - \sum_{v \in \Sigma \cap C_i^\circ} \alpha(v) = 2 \chi(C_i),$$
where $\Sigma \subset S$ denotes the singular set.

Now consider summing over all components of $C$. We have $\sum_{i} \chi(C_i)=\chi(C)$.
If a $\sigma \in \sS_0$ does not contain a pole in its interior, it appears with multiplicity two as an edge in the union $\bigcup_{i=1}^c \partial C_i$. On the other hand, if $\sigma$ has a pole in its interior, then it will have multiplicity one in the union. \compat{In response to Ferran's concerns, I clarified the previous two sentences. I hope it is more clear now? I'm trying to use the same notion of edges and vertices as the previous paragraph.} So,
$$\sum e_i=2|\sS_0| -|P \cap (\bigcup \sS_0 \setminus \sV)|.$$
We therefore have
\begin{equation}
\label{eq:S_0}
2|\sS_0| -|P \cap (\bigcup \sS_0 \setminus \sV)| - \sum_i n_i - \sum_{v \in \Sigma \cap C^\circ} \alpha(v) = 2 \chi(C).
\end{equation}
To simplify this equation, we make a few observations. First $\pi \sum_{i=1}^c n_i$, gives the sum of the cone angles over points $\sV \setminus C^\circ$. (The points in $\sV \cap C^\circ$ do not appear as endpoints of saddle connections in $\sS_0$.) So, we have
$$\sum_{i=1}^c n_i = \sum_{v \in \sV \setminus C^\circ} \big(\alpha(v)+2\big)=2 |\sV \setminus C^\circ| + \sum_{v \in \sV \setminus C^\circ} \alpha(v).$$
Also, $\Sigma \cap C^\circ$ consists of the disjoint union of $\sV \cap C^\circ$ with its nonsingular points removed and the set $P \cap C^\circ$. If $p$ is nonsingular, then $\alpha(v)=0$ while $\alpha(p)=-1$ for poles, so
$$\sum_{v \in \Sigma \cap C^\circ} \alpha(v) = - |P \cap C^\circ| + \sum_{v \in \sV \cap C^\circ} \alpha(v).$$
Combining these identities and the fact that $P \setminus \sV$ splits into $P \cap C^\circ$
and $P \cap (\bigcup \sS_0 \setminus \sV)$, we see that
$$|P \cap (\bigcup \sS_0 \setminus \sV)| + \sum_{i=1}^c n_i + \sum_{v \in \Sigma \cap C^\circ} \alpha(v) =
2 |\sV \setminus C^\circ| + |P \setminus \sV|-2 |P \cap C^\circ| +
\sum_{v \in \sV} \alpha(v).
$$
Plugging this into \eqref{eq:S_0}, we obtain the identity
\begin{equation}
\label{eq:S_1}
2 |\sS_0| = 2 \chi(C) + 2 |\sV \setminus C^\circ| + |P \setminus \sV|-2 |P \cap C^\circ| +
\sum_{v \in \sV} \alpha(v).
\end{equation}
Therefore, by hypothesis, we have
\begin{equation}
\label{inequaltity 1}
2 \chi(C) + 2 |\sV \setminus C^\circ| + |P \setminus \sV|-2 |P \cap C^\circ| +
\sum_{v \in \sV} \alpha(v) \geq 6 |\sV|  + |P \setminus \sV| + 3 \sum_{v \in \sV} \alpha(v).
\end{equation}

Using that $\sV = (\sV \setminus C^\circ) \sqcup (\sV \cap C^\circ)$ and simplifying \eqref{inequaltity 1} yields
$$\chi(C) \geq |\sV \cap C^\circ|+ |P \cap C^\circ| + 2 | \sV | + \sum_{v \in \sV} \alpha(v) .$$
Now note that because each component $C_i$ is a surface with boundary, we have $\chi(C_i) \leq 1$. Thus $\chi(C)=\sum_{i=1}^c \chi(C_i) \leq c$. Furthermore by \Cref{loop contribution}, the sum of the interior angles of each component $C_i$ is at least $\pi$. The total cone angle over points in $\sV$ is $\pi \sum_{v \in \sV} \big(\alpha(v)+2)\big)$. Therefore, we have
$c \leq 2 |\sV| + \sum_{v \in \sV} \alpha(v)$. We conclude that
$$2 |\sV| + \sum_{v \in \sV} \alpha(v) \geq c \geq \chi(C) \geq |\sV \cap C^\circ|+ |P \cap C^\circ| + 2 | \sV | + \sum_{v \in \sV} \alpha(v)$$
and so all inequalities above must be equality and $\sV \cap C^\circ=P \cap C^\circ=\emptyset$. It follows that each component $C_i$ has no elements of $\sV$ or $P$ in its interior and $\chi(C_i)=1$. Thus $C_i$ is a topological disk. Since $\pi c$ is the sum of the cone angles over $\sV$, the sum of the interior angles of $C_i$ must be $\pi$. It follows that each $C_i$ is a triangle as desired.
\end{proof}

\Cref{thm:triangulation} follows directly by taking $\sS$ to be a maximal collection containing $\sS_0$ and using the following:
\commf{Why do we need \Cref{lem:triangulation}? Theorem \Cref{thm:triangulation} follows from  \Cref{bound implies triangulation}, right?} \compat{At this point we know that a collection of $\sV$-saddle connections with cardinality equal to the upper bound from \Cref{bound implies triangulation} is a $\sV$-triangulation. But, a priori, a maximal collection (in the sense of the partial order on collections) may not attain this upper bound.}

\begin{lemma}
\label{lem:triangulation}
Under the hypotheses of \Cref{thm:triangulation}, if $\sS$ is a maximal collection of $\sV$-saddle connections with pairwise disjoint interiors, then $\sS$ is a $\sV$-triangulation of $S$. \commb{Can they intersect at a pole which is in their interior.} \compat{No. I'm not sure why I didn't use the term $\sV$-triangulation here. This should now be clear from the definition above: They cannot intersect at an interior pole.}
\end{lemma}

The remainder of the section will be devoted to proving this lemma, so we fix $S$, $\sV$, and $\sS$ maximal. A key observation is the following:

\begin{proposition}
\label{convex cover}
Let $S$ be a closed zebra surface. Let $\sV$ be a finite nonempty subset of $S$ that contains all singularities with angle larger than $2 \pi$. Suppose that any full cylinder in $S$ has a point of $\sV$ in its interior. There is a finite branched cover $\phi:S' \to S$, such that the local degrees of $\phi$ satisfy the following statements:
\begin{enumerate}
\item The local degree of $\phi$ at any preimage of a pole in $S \setminus \sV$ is two.
\item The local degree at any preimage of a point in $\sV$ is at least two.
\item The local degree at any preimage of a pole in $\sV$ is at least three.
\item The local degree at any other point is one.
\end{enumerate}
\end{proposition}
\begin{proof}
We use a sequence of covers. We can always construct a double branched cover branched over an even number of points. (To do this, cut the surface open along disjoint curves joining the points in pairs, make two copies, and glue each copy to the other along the edges formed by the cuts. \compat{In response to Barak, I spelled out the construction.})
Assuming $S$ has poles, let $S_1$ be a double branched cover branched over the poles, and also branched over a singular point with angle greater than $2 \pi$ if necessary to make the number of branched points even. If $S$ has no poles, take $S_1=S$. Then in either case, $S_1$ is a branched cover that has no poles.

Let $\sV_1 \subset S_1$ be the preimage of $\sV$. Since $S_1$ has no poles, the Gauss-Bonnet Theorem guarantees that the genus of $S_1$ is at least one. If $\sV_1$ has odd cardinality, then let $S_2$ be an unbranched double cover of $S_1$. If the cardinality is even, take $S_2=S_1$. Letting $\sV_2 \subset S_2$ be the preimage of $\sV_1$, we see that the cardinality of $\sV_2$ is even in either case.

To complete the construction, let $S'$ be a double cover of $S_2$ branched over the points in $\sV_2$. It is a simple exercise to check that the statements (1)-(4) hold.
\end{proof}

The surface $S'$ of \Cref{convex cover} inherits a pullback zebra structure. The first virtue of $S'$ is as follows:

\begin{corollary}
\label{virtue1}
The singular set of $S'$ is $\phi^{-1}(\sV)$, and $S'$ has no poles. The surface $S'$ has a leaf triangulation, and the universal cover $\tilde S'$ is convex.
\end{corollary}
\begin{proof}
Since all preimages of poles have local degree at least two, $S'$ has no poles. Statement (1) guarantees that preimages of poles are nonsingular. Statements (2) and (3) guarantee that preimages of points in $\sV$ are singular. Statement (4) guarantees that $S'$ has no other singular points, so $\phi^{-1}(\sV)$ is the singular set. This completes the proof of the first statement.

Note that every cylinder in $S'$ covers a subcylinder in $S$. Because the preimages of points in $\sV$ are all singular, no covered subcylinder in $S$ can have a point of $\sV$ in its interior. Therefore, the hypothesis on full cylinders in $S$ guarantees that $S'$
has no full cylinders. Then $S'$ has a leaf triangulation by \Cref{no full implies leaf}. The universal cover $\tilde S'$ is convex by \Cref{thm:convex}.
\end{proof}

The second virtue of $S'$ is:
\begin{proposition}
\label{virtue2}
Let $R \subset S$ be a connected component of $S \setminus \bigcup_{\sigma \in \sS} \sigma$.
Let $R' \subset S'$ be a connected component $\phi^{-1}(R)$. Then, the external angles of $R'$ all measure at least $\pi$.
\end{proposition}
\begin{proof}
\compat{This proof was rewritten. Jun 10, 2023.}
Let $\alpha'$ be an internal angle of $R'$, and let $\mathit{ext}(\alpha')$ denote the measure of the associated external angle. Let $v' \in \phi^{-1}(\sV)$ denote the vertex of $\alpha'$. Let $\alpha$ and $v$ be the respective images of $\alpha'$ and $v'$ under $\phi$. Then $\alpha$ is an internal angle of $R$ based at $v \in \sV$. Let $m_\alpha$ denote the measure of $\alpha$ and let $m_\beta$ denote the measure of the cone angle at $v$. Then, $m_\alpha \leq m_\beta$ and and there is at least one edge $\sigma \in \sS$ that has $v$ as an endpoint. The measure of $\alpha'$ is also $m_\alpha$. Since $v$ is a vertex, the local degree of $\phi$ at $v'$ is at least two, so the cone angle at $v'$ is at least $2 m_\beta$. We conclude
$$\mathit{ext}(\alpha') \geq 2 m_\beta - m_\alpha \geq m_\beta \geq \pi.$$
\end{proof}

We have the following direct consequence of the above two propositions:

\begin{corollary}
\label{R convexity}
Let $R' \subset S'$ be as in \Cref{virtue2}, and let $\tilde R' \subset \tilde S'$ be a connected component of the preimage of $R'$ in the universal cover $\tilde S'$ of $S'$. Then the closure of $\tilde R'$ is convex.
\end{corollary}
\begin{proof}
Let $X \subset \tilde S'$ denote the closure of $\tilde R'$.
Let $x,y \in X$. We need to show that $x$ and $y$ can be joined by an arc of a trail contained entirely in the closure of $\tilde R'$. We know $\tilde S'$ is convex by \Cref{virtue1}, so there is an arc of a trail $\tau$ joining $x$ to $y$ in $\tilde S'$.
If $\tau \not \subset X$, then we can select a connected component $\tau_0$ of $\tau \setminus \overline{\tilde R'}$. Together with an arc in $\partial \tilde R'$, $\tau_0$ bounds a polygonal disk $D$. We will obtain a contradiction by arguing that $D$ has at most two interior angles whose measure are less than $\pi$ in violation of \Cref{ngons}. Those two possible interior angles are the two intersection points of $\tau_0$ with $\partial \tilde R'$. The interior angles based in the interior of $\tau_0$ measure at least $\pi$ because $\tau_0$ is an arc of a trail.
The interior angles in the interior of the arc $\partial D \cap \partial \tilde R'$ measure at least $\pi$ by \Cref{virtue2} because they are internal angles of $\tilde R'$.
\end{proof}

The following is a first step towards \Cref{lem:triangulation}.

\begin{proposition}
\label{no interior vertices}
If $v \in \sV$, then $v$ is the endpoint of a $\sV$-saddle connection in $\sS$. If $p \in S \smallsetminus \sV$ is a pole, then $p$ is in the interior of a $\sV$-saddle connection in $\sS$.
\end{proposition}
\begin{proof}
Suppose the first statement involving $v \in \sV$ is false. Then $v$ lies in the interior of $S \setminus \bigcup_{\sigma \in \sS} \sigma$. Let $R$ be the component containing $v$. Let $R'$ be a connected preimage in $S'$, and let $\tilde R'$ be a connected preimage of $R'$ in $\tilde S'$. Let $\tilde v' \in \tilde R'$ be a preimage of $v$. Let $X$ be the closure of $\tilde R'$, and let $\tilde w' \in X$ be a preimage of a vertex in $\overline{R}$ that is distinct from $\tilde v'$. (Such a point exists if $R$ has a boundary. If $\partial R=\emptyset$ then $\tilde R'=\tilde S'$ which has infinitely many preimages of $v$.) By \Cref{R convexity}, there is an arc of a trail $\tau \subset X$ whose endpoints are $\tilde v'$ and $\tilde w'$. The initial leaf $\tilde \ell'$ of $\tau$ that emanates from $\tilde v'$ must be in the interior $\tilde R'$ of $X$ and terminate at a singularity $\tilde x' \in \partial \tilde R'$. The image $\ell$ of $\tilde \ell'$ in $S$ is a segment of a leaf starting at $v$ that travels through $R$ and terminates at a point $x \in \sV \cap \partial R$ by \Cref{virtue1}. Therefore $\ell$ is a $\sV$ saddle connection whose interior is disjoint from the interiors of connections in $\sS$.
This violates the maximality of $\sS$.

If $p \in S \smallsetminus \sV$ is a pole that is not contained in the interior of an $\sV$ saddle connection, then $p$ also lies lies in the interior of a component $R$ as above. The same argument as above gives a leaf segment $\ell$ from $p$ to an $x \in \sV \cap \partial R$ whose interior is contained in $R$. Then we can add to $\sS$ the $\sV$-saddle connection which travels from $x$ to $p$ along $\ell$ and then doubles back to $x$. This again violates the maximality of $\sS$.
\end{proof}

\begin{proposition}
\label{R simply connected}
If $R$ is a connected component of $S \setminus \bigcup_{\sigma \in \sS} \sigma$, then $R$ is simply connected.
\end{proposition}
\begin{proof}
By \Cref{no interior vertices}, the open set $R$ contains no vertices or poles, and $\partial R$ consists of a finite nonempty union of edges. Thus $R$ is homeomorphic to a punctured finite type surface. We break into two cases.

If $R$ has genus zero, then it is simply connected if and only if it has only one boundary component. Suppose to the contrary that $R$ has two boundary components. Take a connected preimage $\tilde R' \subset \tilde S'$. Then $\tilde R'$ also has at least two boundary components. Let $X$ denote the closure of $\tilde R'$, and choose singularities $\tilde v'$ and $\tilde w'$ from distinct boundary components. By convexity of $X$, there is an arc of a trail $\tilde \tau' \subset X$ that joins the two points. Let $\tilde \tau'_0 \subset \tilde \tau'$ be a minimal subarc joining distinct boundary components of $\tilde R'$. Then $\tilde \tau'_0$ starts and ends at singularities, whose images in $R$ are in $\sV$ by \Cref{virtue1}. The rest of the image $\tau_0$ in $\bar R$ is contained in $R$ which has no singularities or poles so must be a segment of a leaf. This violates the maximality of $\sS$.

The second case covers when $R$ has positive genus. Thus, the fundamental group of $R$ contains a free group of rank two. Let $\alpha$ be an interior angle of $R$, and let $v \in \sV$ be the point at which the angle is based. Let $R^\ast$ denote $R \cup \{v\}$ topologized in such a way that continuous paths emanating from $v$ must enter $R$ through the interior angle $\alpha$. (It is a priori possible that $R^\ast$ has multiple interior angles based at $v$.) This doesn't change the homotopy type of $R$. Loops starting at $v$ and returning to $v$ that are homotopic to the boundary lie in a subgroup of the fundamental group isomorphic to $\Z$, so there is a loop $\gamma$ starting at $v$ that is not homotopic to the boundary. Fix such a $\gamma$. Let $\tilde \gamma'$ be a lift of $\gamma$ to $\tilde R'$.
Since $\gamma$ is not homotopic to a boundary, $\tilde \gamma'$ connects distinct boundary components of $\tilde R'$. Let $\tilde \tau'$ be the trail arc joining the endpoints of $\tilde \gamma'$. Therefore, there is a minimal subarc $\tilde \tau'_0 \subset  \tilde \tau'$ that joins distinct boundary components and has singular endpoints. Repeating the argument from the previous paragraph, we see that the image $\tau_0$ is a $\sV$-saddle-connection whose interior is contained in $R$, violating the maximality of $\sS$.
\end{proof}

To prove \Cref{thm:triangulation}, it remains to complete the following.

\begin{proof}[Proof of \Cref{lem:triangulation}]
Assume $S$ and $\sV \subset S$ satisfy the hypotheses of \Cref{thm:triangulation}. Let $\sS$ be a maximal collection of $\sV$-saddle connections with disjoint interiors. We already showed that every pole in $S \setminus \sV$ is contained in the interior of a connection in $\sS$ in \Cref{no interior vertices}. Let $R$ be a component of $S \setminus \bigcup_{\sigma \in \sS} \sigma$. It remains to show that each such $R$ is a triangle. By {no interior vertices}, the open set $R$ contains no singularities and is bounded by a finite union of connections from $\sS$. By \Cref{R simply connected}, $R$ is simply connected. Therefore, its connected preimage $R' \subset S'$ is a polygon. (This follows from \Cref{virtue2}, which together with \Cref{ngons} guarantees that $\partial R'$ is a simple closed curve. The boundary of $R$ is naturally a loop but is not a priori simple.) Note that because of the aforementioned properties of $R$, the restriction of the covering $S' \to S$ to $R' \to R$ is a homeomorphism.

Suppose to the contrary that $R$ was not a triangle. then $R'$ would not be a triangle. But, we can triangulate $R'$ using \Cref{triangulating polygons}. Any diagonals added descend to $\sV$-saddle connections whose interiors are contained in $R$ and thus disjoint from the interiors of connections in $\sV$. This would violate maximality of $\sS$, giving us our contradiction. This proves that
$\sS$ is the collection of edges in a triangulation as needed.
\end{proof}

\subsection{Proof of \texorpdfstring{\Cref{conj:zebra case}}{Theorem \ref{conj:zebra case}}}
    \label{secc:proof-conj-zebra case}

\begin{proof}[Proof of \Cref{conj:zebra case}]
Let $S$ be a closed zebra surface with at least one singularity that is not a pole.
We are tasked to prove that the four statements (a)-(d) are logically equivalent.
The implication (a) implies (b) is \Cref{thm:convex}. The implication (b) implies (c) is \Cref{cor:closed trails}. The implication (c) implies (d) follows directly from \Cref{full cylinder}. It remains to prove that (d) implies (a), i.e., if $S$ has no full cylinders, then it has a leaf triangulation.

If $S$ is a closed surface with no full cylinders, then we can take $\sV$ to be the set of singularities of $S$ with angle at least $3 \pi$. By hypothesis, $\sV$ is nonempty. \Cref{thm:triangulation} guarantees that the empty collection of $\sV$-saddle connections can be extended to a $\sV$-triangulation, which from our choice of $\sV$ is the same as a leaf triangulation.
\end{proof}

\section{Questions}
\label{sect:questions}

Hopf tori give examples of zebra structures on the torus without singularities such that all closed leaves lie in one homotopy class. \Cref{torus cover convex} says that if there are two nonhomotopic closed leaves then every nontrivial free homotopy class of closed curves contains a closed leaf. But we are not sure of the following:

\begin{question}
Is there a zebra torus without singularities that has no closed leaves?
\end{question}

Every dilation structure on a closed surface has a cylinder \cite{BGT}.

\begin{question}
Does every zebra structure on a closed surface contain a cylinder?
\end{question}

Two zebra structures $(S_1, \{\sF_m^1\})$ and $(S_2, \{\sF_m^2\})$ should be considered {\em isomorphic} if there is an orientation-preserving homeomorphism $\phi:S_1 \to S_2$ such that  $\sF_m^1$ coincides with the pullback of $\sF_m^2$ for each $m \in \hat \R$. The following question has been studied in the context of dilation surfaces (and more general complex affine structures) in \cite{ABW} building off work in \cite{V93}.

\begin{question}
Given a closed surface $S$ and a singular data function $\alpha$, is there a way to understand the ``stratum'' of all compatible zebra structures on $S$ up to isomorphism? Is there a natural topology? Is this stratum naturally an orbifold modeled on a function space?
\end{question}

\begin{question}
\label{q:uniformization}
Is there a notion of uniformization for zebra surfaces? That is, given a zebra surface, is there a homeomorphism to a Riemann surface that respects angles?
\end{question}

We now make several definitions that we consider analogous to the affine automorphism group of a translation or dilation surface, the derivative mapping from the affine automorphism group to $\GL(2,\R)$, and the Veech group of these surfaces.

Let $S$ be an oriented topological surface. Let $Z(S)$ denote the collection of all zebra structures on $S$. Then each element of $Z(S)$ is an indexed family of singular foliations $\{\sF_m:~m \in \hat \R\}$ on $S$. \Cref{sect:homeo action} described an action of $\Homeo_+(\hat \R)$ on $Z(S)$, which we will denote by $\{\sF_m\} \mapsto \varphi(\{\sF_m\})$
for $\varphi \in \Homeo_+(\hat \R)$. Also if $h:S \to S$ is an orientation-preserving homeomorphism, then there is a natural action of $h$ on $Z(S)$ that sends $\{\sF_m\}$ to $h(\{\sF_m\})=\{h_\ast(\sF_m)\}$, where $h_\ast(\sF_m)$ denotes the pushforward of the foliation under $h$. We call $h$ a {\em stellar automorphism} of $(S, \{\sF_m\})$ if there is a $\varphi \in \Homeo_+(\hat \R)$ such that $h(\{\sF_m\})=\varphi(\{\sF_m\})$. Let $\Aut^\medstar_+(S, \{\sF_m\})$ denote the group of stellar automorphisms of $(S, \{\sF_m\})$. Observing that $\varphi$ is uniquely determined by $h$, we see there is a natural group homomorphism
\begin{equation}
\label{eq:derivative}
D:\Aut^\medstar_+(S, \{\sF_m\}) \to \Homeo_+(\hat \R); \quad h \mapsto \varphi,
\end{equation}
that we call the {\em stellar derivative}. The {\em stellar group} of $(S, \{\sF_m\})$
is the image of $D$, which we denote by $G^\medstar_+(S, \{\sF_m\}) \subset \Homeo_+(\hat \R)$. The kernel of $D$ is the {\em zebra automorphism group} $\Aut(S, \{\sF_m\})$, the group of zebra automorphisms as defined in \Cref{sect:topological}.

The following question is answered in the context of closed dilation surfaces by \cite[Theorem 1]{DFG19}.
\begin{question}
When is $G^\medstar_+(S, \{\sF_m\})$ a discrete subgroup of $\Homeo_+(\hat \R)$?
\end{question}
We also wonder about discreteness of $\Aut^\medstar_+(S, \{\sF_m\})$.
Note however that discreteness of both $G^\medstar_+(S, \{\sF_m\})$ and $\Aut(S, \{\sF_m\})$ implies the discreteness of $\Aut^\medstar_+(S, \{\sF_m\})$ since the following sequence is exact:
$$1 \to \Aut(S, \{\sF_m\}) \to \Aut^\medstar_+(S, \{\sF_m\}) \xrightarrow{D} G^\medstar_+(S, \{\sF_m\}) \to 1.$$
Note also that if the answer to \Cref{q:uniformization} is affirmative, then $\Aut(S, \{\sF_m\})$ is isomorphic to a subgroup of a group of automorphisms of a Riemann surface and is therefore finite when $S$ is a closed surface of genus at least two.

The examples we understand of interesting subgroups of stellar groups all come from translation and dilation surfaces and are therefore contained in $\SL(2,\R) \subset \Homeo_+(\hat \R)$.
\begin{question}
Is every $G^\medstar_+(S, \{\sF_m\})$ conjugate within $\Homeo_+(\hat \R)$ to a subgroup of $\SL(2,\R)$? Is $G^\medstar_+(S, \{\sF_m\})$ always conjugate into a group of homeomorphism acting smoothly on the circle?
\end{question}

The map $D$ of \eqref{eq:derivative} gives an action of $\Aut^\medstar_+(S, \{\sF_m\})$ on the circle $\hat \R$.

\begin{question}
Because we can deform the zebra structure, it is natural to wonder when the stellar derivative $D$ of a structure $(S, \{\sF_m\})$ is {\em locally rigid up to deformations of the zebra structure}, meaning that there is an open set of representations from $\Aut^\medstar_+(S, \{\sF_m\})$ to $\Homeo_+(\hat \R)$ containing $D$ such that for any representation $D'$ in the open set, there is another zebra surface $(S', \{\sF_m'\})$ and an injective homomorphism $\psi:\Aut^\medstar_+(S, \{\sF_m\}) \to \Aut^\medstar_+(S', \{\sF_m'\})$ such that $D' \circ \psi^{-1}$ is the restriction of the stellar derivative of $(S', \{\sF_m'\})$ to the image of $\psi$. This is an instance of a collection of natural questions regarding the rigidity of the action of $\Aut_+^\medstar(S,\{\sF_m\})$ on the circle. See \cite{Mann} for background on rigidity questions for group actions on the circle.
\compat{I think it is okay as phrased, but there are some issues. For example, it would be surprising (in light of facts about Veech groups of closed translation surfaces) if this group was always finitely generated. I think in Mann's survey finitely generated is assumed. My other issue is that I'm not clear what the topology should be on representations. Is the $C^0$ topology the only one that makes sense?}
\end{question}

There is a natural group homomorphism sending elements of $\Aut^\medstar_+(S, \{\sF_m\})$ to the mapping class group of $S$.

\begin{question}
Which mapping classes arise from elements of $\Aut^\medstar_+(S, \{\sF_m\})$? In particular, are all reducible mapping classes realizable as the stellar automorphism of a zebra surface?
\end{question}

All finite-order and pseudo-Anosov elements of the mapping class group are realized by affine automorphisms of half-translation surfaces. In contrast, the only reducible mapping classes that arise from affine automorphisms of translation and dilation surfaces are multitwists \cite{W21}.

\compat{I tried to think about what we might ask about the dynamics of stellar automorphisms. This was the old (vaguely phrased) question 10.7. The only thing that comes to mind is the proving a stellar automorphism in a pseudo-Anosov mapping class is topologically conjugate to a standard (Markov) pseudo-Anosov. But we intend to do this, so maybe it is better not to ask it here.}

The forgetful maps from half-translation and half-dilation structures to zebra structures leads to interesting questions.

\begin{question}
Is there ever a homeomorphism $\phi:S \to S$ with $S$ a closed half-translation surface that gives a stellar automorphism of the induced zebra structure, but does not give an affine automorphism of $S$? What about for half-dilation surfaces?
\end{question}

\begin{question}
Characterize the zebra structures that arise from half-translation (or half-dilation) structures.
\end{question}

There is a conjectural geometric characterization of dilation structures that arise from translation structures \cite[\S 1.5]{BGT}.

\section*{Acknowledgments}

\compat{Inserted two sentences to thank Guillaume.}
We'd like to thank Jon Chaika who directed us in exploring the natural wonders in Utah, which partially inspired this project. We'd also like to thank David Aulicino and Jane Wang for helpful conversations. We'd like to thank Guillaume Tahar for suggestions for the proof of \Cref{conj:zebra case}. We feel these ideas could have lead to an alternate approach to this result. W. P. Hooper was supported by a grant from the Simons Foundation and by a PSC-CUNY Award, jointly funded by The Professional Staff Congress and The City University of New York. F. Valdez would like to thank the following grants: CONACYT Ciencia B\'asica CB-2016 283960 and UNAM PAPIIT IN-101422. B. Weiss was supported by grants BSF 2016256, ISF 2019/19, and ISF-NSFC 3739/21.

\bibliographystyle{amsalpha}
\bibliography{references}

\end{document}